\documentclass[preprint,12pt,authoryear]{elsarticle}

\usepackage[letterpaper,margin=1in]{geometry}
\usepackage{amsmath,amssymb,amsthm,mathtools}
\usepackage{mathrsfs}
\biboptions{authoryear,round}
\usepackage{url}
\usepackage{adjustbox}
\usepackage{array,booktabs,longtable}
\usepackage[colorlinks=true,linkcolor=blue,citecolor=blue,urlcolor=blue]{hyperref}
\hypersetup{
  pdftitle={Monoidal Alphabets for Generalized Harmonic Sums},
  pdfauthor={Jayanta Phadikar}
}
\numberwithin{equation}{section}

\journal{Journal of Symbolic Computation}

\newtheorem{theorem}{Theorem}[section]
\newtheorem{lemma}[theorem]{Lemma}
\newtheorem{proposition}[theorem]{Proposition}
\newtheorem{corollary}[theorem]{Corollary}

\theoremstyle{definition}
\newtheorem{definition}[theorem]{Definition}
\newtheorem{example}[theorem]{Example}
\theoremstyle{remark}
\newtheorem{remark}[theorem]{Remark}
\theoremstyle{plain}

\DeclareMathOperator{\Li}{Li}
\DeclareMathOperator{\ord}{ord}
\DeclareMathOperator{\lcm}{lcm}
\DeclareMathOperator{\trunc}{trunc}

\newcommand{\idpoly}{\mathrm X}
\newcommand{\Hyp}{\operatorname{Hyp}}
\begin{document}

\begin{frontmatter}

\title{\texorpdfstring{\vspace*{1.0cm}}{}Monoidal Alphabets for Generalized Harmonic Sums}

\author[aff1]{Jayanta Phadikar\texorpdfstring{\corref{cor1}}{}}
\ead{jayantap@wolfram.com}

\address[aff1]{Wolfram Research}

\cortext[cor1]{Corresponding author.}

\begin{abstract}
	We develop a general finite-alphabet framework for Euler-type sums based on the notion of a
	\emph{monoidal alphabet}. An alphabet of summand letters is called monoidal when it is closed
	under pointwise multiplication, thereby inducing the usual stuffle, or quasi-shuffle, algebra on
	the associated nested sums. This viewpoint places classical multiple harmonic number/sum objects, colored
	harmonic sums, and several generalized Euler sums under a common structural mechanism. We
	focus on three fundamental families of monoidal alphabets: the ordinary power alphabet generated
	by \(n\), the affine alphabet generated by linear factors \(an+b\), and the polynomial-base alphabet
	generated by polynomial factors \(P(n)\). The resulting classes of multiple harmonic numbers, multiple
	affine harmonic numbers, and multiple polynomial-base harmonic numbers provide systematic containers for a
	wide range of finite and infinite Euler-type sums. We prove closure and lifting results showing
	that nested sums whose summands are built from these alphabets, possibly multiplied by
	harmonic-sum factors, reduce to the corresponding finite harmonic-sum objects. As consequences,
	the framework recovers many known Euler-sum identities and produces many new identities in a
	uniform way. While reduction to simpler functions remains a separate and often
	difficult problem, the monoidal-alphabet perspective provides a unified algebraic language for
	organizing, transforming, and extending harmonic-sum identities.
\end{abstract}

\begin{keyword}
Euler sums \sep monoidal alphabets \sep harmonic sums \sep multiple harmonic 
numbers \sep multiple polylogarithms \sep symbolic summation \sep quasi-shuffle algebra
\end{keyword}

\end{frontmatter}

\clearpage

\renewcommand{\contentsname}{Contents}
{\hypersetup{linkcolor=blue}\setcounter{tocdepth}{2}\tableofcontents}
\clearpage

\section*{Notation and conventions}
\addcontentsline{toc}{section}{Notation and conventions}

Throughout, $\mathbb N=\{1,2,\ldots\}$, $\mathbb N_0=\mathbb N\cup\{0\}$, and all vector spaces are over $\mathbb C$.  Complex powers are taken after a branch has been fixed; for positive integers we use $n^q=\exp(q\log n)$ with the real logarithm.  The zeta and polylogarithmic symbols introduced below are used in an extended sense: their weight parameters, colors, and affine or polynomial data may be complex.  Thus they should be regarded as functions of complex variables, unlike the classical multiple-zeta-value and multiple-polylogarithm notation where the weight indices are usually positive integers.

We freely identify a letter with the corresponding one-letter word.  Thus, for example, $\mathcal H_a(N)$ means $\mathcal H_{(a)}(N)$, and $\mathcal H_{a,\alpha}(N)$ means $\mathcal H_{(a,\alpha)}(N)$; the same convention is used for $\mathcal G$ and $\mathcal P$.  When the upper argument $x$ of $H_x^{(r)}(s)$ is not a non-negative integer, $H_x^{(r)}(s)$ denotes the Lerch continuation $\Li_r(s)-s^{x+1}\Phi(s,r,x+1)$, with branches fixed where needed; in particular, $H_x=H_x^{(1)}(1)=\psi(x+1)+\gamma$.

{\small
\renewcommand{\arraystretch}{1.25}
\setlength{\tabcolsep}{5pt}
\begin{longtable}{@{}>{\raggedright\arraybackslash}p{0.26\textwidth}>{\raggedright\arraybackslash}p{0.68\textwidth}@{}}
\toprule
\textbf{Notation} & \textbf{Meaning} \\
\midrule
\endfirsthead
\toprule
\textbf{Notation} & \textbf{Meaning} \\
\midrule
\endhead
$n,N,k$ & Summation indices and finite upper limits; $k$ is used especially in finite identities that later pass to $k\to\infty$. \\

$(r,s)$ & A colored harmonic letter, with power $r\in\mathbb C$ and color $s\in\mathbb C$. \\

$\alpha=(a_1,\ldots,a_d)$, $\emptyset$ & A word in colored letters; $d$ is its depth and $\emptyset$ is the empty word. \\

$\displaystyle \mathcal H_\alpha(N)$ & Multiple harmonic number.  If $\alpha=((r_1,s_1),\ldots,(r_d,s_d))$, then
$\displaystyle \mathcal H_\alpha(N)=\sum_{N\ge n_1>\cdots>n_d\ge1}\prod_{j=1}^d s_j^{n_j}n_j^{-r_j}$, with $\mathcal H_\emptyset(N)=1$. \\

$\displaystyle (r,s)\circ(r',s')$ & Merged letter: $\displaystyle (r,s)\circ(r',s')=(r+r',ss')$. \\

$\mathscr H_N$ & The finite span of all $\mathcal H_\alpha(N)$ with upper limit $N$. \\

$\mathscr E_n$ & The summand space spanned by terms $z^n n^q\mathcal H_\alpha(n)$, where $z,q\in\mathbb C$. \\

$H_n^{(r)}(s)$, $H_n^{\mathbf r}(\mathbf s)$, $H_n^{\star,\mathbf r}(\mathbf s)$ & Colored depth-one, strict multiple, and star multiple harmonic numbers.  These are identified with $\mathcal H$-sums, after expanding star sums into strict sums by merging equal indices. \\

$\omega_m$ & The primitive $m$th root of unity $e^{2\pi i/m}$, used in residue-class and scaled-upper-limit filters. \\

$M$ & A level or modulus in residue-class constructions, especially in $R_{M,a;k}^{(r)}(s)$ and $\omega_M$; this is distinct from the finite upper limit $N$. \\

$\theta_j$ & A rational shifted-denominator parameter, usually $\theta_j=b_j/a_j$, used only in cancellation arguments. \\

$\mathcal L$ & The limiting functional $\mathcal L(X)=\lim_{N\to\infty}X(N)$, when the limit exists. \\

$\mathcal G_\Gamma(N)$, $\mathscr G_N$ & Multiple affine harmonic numbers and their span.  An affine letter $L=(\boldsymbol\rho,\sigma,\mathbf A)$ has value $L(n)=\sigma^n\prod_\nu(a_\nu n+b_\nu)^{-\rho_\nu}$. \\

$\mathcal P_\Omega(N)$, $\mathscr P_N$ & Multiple polynomial-base harmonic numbers and their span.  A polynomial letter $L=(\boldsymbol\rho,\sigma,\mathbf P)$ has value $L(n)=\sigma^n\prod_\nu P_\nu(n)^{-\rho_\nu}$. \\

$\Li_{\mathbf r}(\mathbf s)$, $\zeta(\mathbf r)$ & Ordinary colored multiple polylogarithmic and zeta values.  The value $\Li_{\mathbf r}(\mathbf s)$ is the convergent limit associated with $\mathcal H_{((r_1,s_1),\ldots,(r_d,s_d))}(N)$; the zeta specialization is $\zeta(\mathbf r)=\Li_{\mathbf r}(1,\ldots,1)$. \\

$\Li^{\mathrm{aff}}_\Gamma$, $\zeta^{\mathrm{aff}}_\Gamma$ & Affine multiple polylogarithmic and affine zeta values associated with a word $\Gamma$ in affine letters.  The zeta specialization corresponds to trivial colors $\sigma=1$ in all letters. \\

$\Li^{\mathrm{pb}}_\Omega$, $\zeta^{\mathrm{pb}}_\Omega$ & Polynomial-base multiple polylogarithmic and polynomial-base zeta values associated with a word $\Omega$ in polynomial letters.  The zeta specialization corresponds to trivial colors $\sigma=1$ in all letters. \\
\bottomrule
\end{longtable}
}
\noindent Auxiliary symbols such as $A$, $B$, and $A_{\mathrm{prod}}$ are local letters used in examples and tables; they should not be confused with the global monoidal alphabet $\mathcal A$.

\section{Introduction}

\subsection{Origin of Euler sums}

Euler sums belong to a long line of questions that began with Euler's attempt
to understand infinite series not as isolated numerical curiosities, but as
objects with hidden structure.  In his work on zeta-type series, and in his
correspondence with Goldbach, Euler was led to sums in which harmonic numbers
occur inside another infinite series.  In modern notation, the basic linear
examples have the form
\[
  S_{p,q}
  =
  \sum_{n=1}^{\infty}\frac{H_n^{(p)}}{n^q},
  \qquad q>1.
\]
At first sight these look like modest variations of the zeta function: one
takes the partial sums
\[
  H_n^{(p)}=\sum_{k=1}^{n}\frac{1}{k^p}
\]
and sums them again against $n^{-q}$.  The surprise, already visible in
Euler's calculations, is that many such expressions collapse to combinations
of zeta values.  Thus a finite, elementary-looking operation of forming
harmonic numbers and summing them once more opens a door to a rich algebra
of special constants.  Nonlinear Euler sums, obtained by replacing a single
harmonic number by a product of harmonic numbers with the same upper limit,
make this structure even more apparent
\citep{Euler1917,Berndt1985,BorweinBorweinGirgensohn1995,FlajoletSalvy1998}.
The modern language of weight, depth, and degree reflects the fact that these
sums are not merely examples, but members of large and highly organized
families.

\subsection{Euler sums in number theory and quantum field theory}

Euler sums are important because they sit at a meeting point of three
closely related structures: finite harmonic sums, multiple zeta values, and
the special constants arising from perturbative quantum field theory.  In
number theory, the classical Euler sums
\[
  S_{p_1,\ldots,p_k;q}
  =
  \sum_{n=1}^{\infty}
  \frac{H_n^{(p_1)}\cdots H_n^{(p_k)}}{n^q}
\]
may be viewed as one-dimensional projections of the multiple-zeta algebra.
The product
\[
  H_n^{(p_1)}\cdots H_n^{(p_k)}
\]
is governed by the same stuffle, or quasi-shuffle, mechanism that governs
finite multiple harmonic numbers.  Hence, after expanding the product into
nested finite sums and letting the outer index tend to infinity, one is led
naturally to multiple zeta values and their alternating or colored variants.
This explains why many Euler sums collapse to combinations of zeta values,
and why the general problem of evaluating Euler sums is inseparable from the
structure theory of multiple zeta values
\citep{BorweinBorweinGirgensohn1995,BorweinBradleyBroadhurst1997,FlajoletSalvy1998,Xu2017,XuWang2020}.

The same objects appear, from a different direction, in Feynman integral
calculations.  In perturbative quantum field theory, loop integrals are often
converted into Mellin moments, hypergeometric sums, or parameter integrals.
After expansion in the dimensional regularization parameter, gamma functions
and hypergeometric terms produce harmonic sums and nested sums.  Vermaseren's
work on harmonic sums, Mellin transforms, and integrals gave an early
algorithmic framework for symbolic sums over harmonic series, binomial
coefficients, and denominators arising in Feynman diagram calculations
\citep{Vermaseren1999}.  Bl\"umlein and Kurth showed that finite harmonic
sums form a natural basis for Mellin transforms appearing in two-loop
massless QED and QCD calculations \citep{BluemleinKurth1999}.  Later
symbolic summation and difference-field methods, together with the packages
\textsc{Sigma}, \textsc{EvaluateMultiSums}, \textsc{SumProduction}, and
\textsc{HarmonicSums}, made this connection systematic for large classes of
multi-sums arising from loop integrals
\citep{Schneider2014,Ablinger2014,AblingerBluemleinRaabSchneider2014,Bekavac2006}.

There is also a more direct bridge through digamma and polygamma functions.
Since
\[
  \psi(n+1)=H_n-\gamma,
  \qquad
  \psi^{(m)}(n+1)
  =
  (-1)^{m+1}m!\bigl(\zeta(m+1)-H_n^{(m+1)}\bigr),
\]
series involving $\psi$ and $\psi^{(m)}$ at integer or rationally shifted
arguments naturally produce Euler sums and their affine or residue-class
variants.  Coffey studied one-dimensional digamma and polygamma series of
the forms
\[
  \sum_{n\ge1}\frac{(\pm1)^n\psi(n+p/q)}{n^r},
  \qquad
  \sum_{n\ge1}\frac{(\pm1)^n\psi^{(m)}(n+p/q)}{n^r},
\]
developing integral representations and explicit examples, with motivation
from the evaluation of Feynman amplitudes \citep{Coffey2005}.  Related work
of Ogreid and Osland evaluated one-, two-, and three-dimensional series
appearing in Feynman diagram calculations, many of them Euler-series-like,
in terms of constants such as $\zeta(2)$, $\zeta(3)$, Catalan's constant, and
Clausen values \citep{OgreidOsland2002}.  Thus the passage from harmonic
numbers to shifted harmonic numbers, polygamma values, cyclotomic sums, and
binomially weighted sums arises naturally in analytic Feynman integral calculations.

\subsection{Recent extensions and motivation}

The modern theory of Euler sums has developed in several parallel directions,
all of which point toward the usefulness of working with finite summand
alphabets.  Classical Euler sums involving \(H_n^{(r)}\) are governed by
finite harmonic sums and their quasi-shuffle product; alternating and colored
Euler sums enlarge the same mechanism by allowing letters of the form
\(s^n n^{-r}\), where \(s\) is typically a sign or a root of unity.  These
extensions occur naturally in the study of alternating multiple zeta values,
colored multiple zeta values, and multiple polylogarithms at roots of unity
\citep{BorweinBorweinGirgensohn1995,BorweinBradleyBroadhurst1997,
XuWang2020}.  Au's polylogarithmic-integral method, implemented in the accompanying
\texttt{MultipleZetaValues} package, and his later WZ-pair approach provide
systematic evaluations for large classes of ordinary, colored, and affine
Euler-type sums \citep{Au2020PolylogIntegral,AuWZ2025}.

Closely related families arise from congruence restrictions on the summation
indices.  Odd harmonic numbers, Hoffman multiple \(t\)-values,
Kaneko--Tsumura multiple \(T\)-values, finite mixed values, and level-\(N\)
multiple zeta values are all obtained by imposing parity or residue-class
conditions on the summation indices.  Such restrictions may be expressed by
finite Fourier filters, and hence reduce to colored finite harmonic numbers.
Hoffman's multiple \(t\)-values and the Kaneko--Tsumura level-two values are
basic examples of this phenomenon \citep{Hoffman2019,KanekoTsumura2020}.
Dirichlet-type Euler sums involving odd harmonic numbers and their
alternating variants were studied by Xu and Wang, who obtained explicit
formulas and parity theorems by residue computations \citep{XuWang2022}.
Further parity-restricted and mixed constructions have also been studied in
connection with variants of multiple zeta values of level two
\citep{XuZhao2022,Zhao2024FiniteMMV}.

Shifted and rationally shifted harmonic quantities provide another important
source of examples.  Digamma and polygamma values satisfy
\[
  \psi(n+1)=H_n-\gamma,\qquad
  \psi^{(m)}(n+1)
  =
  (-1)^{m+1}m!\bigl(\zeta(m+1)-H_n^{(m+1)}\bigr),
\]
and their rational shifts naturally lead to denominators of the form
\((an+b)^r\).  Coffey studied one-dimensional digamma and polygamma series
related to Feynman-diagram evaluations, including sums with rationally shifted
arguments \citep{Coffey2005}.  Recent work of Olaikhan evaluates Euler sums
involving harmonic numbers with rational arguments, such as
\[
  \sum_{k=1}^{\infty}\frac{H_{k/n}^{(p)}}{k^q},
  \qquad
  \sum_{k=1}^{\infty}\frac{(-1)^kH_{k/(2n)}^{(p)}}{k^q},
\]
expressing the relevant odd-weight cases in terms of Riemann and Hurwitz
zeta values \citep{Olaikhan2026}.  Cyclotomic and Hurwitz-type cyclotomic
Euler sums have also been treated by contour integration and residue
calculus, producing explicit formulas and parity results
\citep{RuiXu2025,Rui2026}.  The recent work of Xu on cyclotomic multiple
Hurwitz zeta values gives a multiple-parameter Hurwitz setting that fits
naturally into the same shifted-index viewpoint
\citep{Xu2026CyclotomicMultipleHurwitz}.

Multiple-argument and polynomial-denominator sums give further evidence that
the natural object is not a single special function but the alphabet of its
summand.  Sofo's multiple-argument Euler sums involve harmonic numbers such
as \(H_{qn}^{(r)}\), and hence require systematic control of scaled
truncation indices \citep{Sofo2025}.  Polynomial zeta functions,
one-dimensional Epstein--Hurwitz zeta functions, and Mathieu-type series
supply related denominator structures involving polynomial expressions in
the summation index
\citep{EieChen1999,Dabrowski2000,Elizalde1994,
ElizaldeEtAl1994,Mathieu1890,PoganySrivastavaTomovski2006,
PoganyTomovski2006}.

\subsection{The present alphabetic convolution framework}
\label{subsec:present-convolution-framework}

This paper introduces a general alphabetic and convolution framework for
harmonic sums.  The central idea is to extract from a summand a multiplicative
alphabet of one-variable letters and to regard finite nested sums as word sums
over that alphabet.  When the alphabet is closed under pointwise multiplication,
the corresponding word sums are stable under the stuffle, or quasi-shuffle,
product: coincident summation indices are resolved by multiplying the
corresponding letters.  This gives a finite algebraic mechanism for converting
products, convolutions, shifted sums, scaled sums, and nested summation domains
into linear combinations of structured finite word sums.

The first alphabet considered here is the colored harmonic alphabet, with
letters
\[
  n\longmapsto s^n n^{-r},
  \qquad r,s\in\mathbb C.
\]
It contains ordinary harmonic sums, alternating harmonic sums, colored harmonic
sums, finite multiple harmonic numbers, and the finite forms underlying multiple
zeta values and multiple polylogarithms.  The second is the affine alphabet,
with letters
\[
  n\longmapsto s^n\prod_j(a_jn+b_j)^{-\rho_j},
\]
which accommodates shifted denominators, rational shifts, residue-class
filters, level constructions, truncated Hurwitz-type sums, and truncated
Lerch-type sums.  The third is the polynomial-base alphabet,
\[
  n\longmapsto s^n\prod_j P_j(n)^{-\rho_j},
\]
which provides a natural setting for polynomial-denominator sums such as finite
Mathieu-type sums, one-dimensional Epstein--Hurwitz-type truncations,
polynomial zeta functions, and polynomial-base polylogarithmic sums.

The convolution step follows a simple pattern.  Products of inner
harmonic-type factors are expanded by stuffle relations.  The remaining outer
factor is then absorbed as an additional letter in the appropriate alphabet.
Finite summation is handled by decomposing the index domain into strict
inequalities, diagonal contributions, boundary terms, shifts, scalings, and
affine or polynomial changes of the summand letters.  In this way, sums such as
\[
  \sum_{n=1}^{N} z^n n^q
  \prod_i \mathcal H_{\alpha_i}(n)
\]
and their affine or polynomial analogues are converted into finite linear
combinations of word sums in the corresponding alphabet.  This is parallel in
spirit to the convolution approach to special-function integrals developed by
Adamchik and Marichev \citep{AdamchikMarichev1990}: one embeds a calculation in
a stable class of special functions, performs the convolution inside that
class, and then applies further reductions to obtain simpler forms.

The framework unifies several previously separate families of identities.  It
places ordinary, colored, alternating, affine, residue-class, level, shifted,
scaled, nested, and polynomial-denominator Euler-type sums inside one
multiplicative alphabetic formalism.  It also separates two tasks that often
occur together in closed-form evaluations: first converting a sum into its
natural finite word-sum alphabet, and then reducing the resulting word sums to
smaller bases or limiting constants when such reductions are available.  The
closure theorems therefore serve both as structural results and as mechanisms
for explicit evaluation implementable in a computer algebra system.  Examples of explicit evaluations for
each of the main classes are provided in the supplementary material.

The paper is organized accordingly.  Section~2 formalizes monoidal alphabets
and introduces the colored, affine, and polynomial-base harmonic-number spaces.
Section~3 proves the basic finite convolution theorem in the colored alphabet.
Sections~4 and~5 develop the affine-letter and polynomial-letter extensions.
Section~6 treats scaled index sums, and Section~7 lifts the method to nested
summation domains.  Section~8 passes to infinite sums and limiting constants.
Section~9 discusses normal forms and reductions, Section~10 records limitations
and future directions, and Section~11 concludes the paper.  Appendix~A lists further monoidal alphabets that fit the same mechanism.
Appendix~B records convergence criteria for the multiple-polylogarithmic variants used in the paper.
Appendix~C gives additional infinite-sum reductions for special limiting cases not included in the main finite closure theory.

\section{Monoidal harmonic-number alphabets}

This section introduces the alphabetic framework used throughout the paper.  We first
formulate the general notion of a monoidal alphabet of one-variable summand letters and
the associated universal finite nested sums.  We then specialize this general construction to the three basic harmonic-number
alphabets used in the sequel: the basic colored harmonic-number alphabet, the affine harmonic-number alphabet, and the polynomial-base harmonic-number alphabet.
For each alphabet we record the principal classes of summand factors that can be
represented in it.  These reduction principles supply the input for the finite convolution theorems and for the examples and consequences developed in the later sections.

\subsection{Monoidal alphabets and universal nested sums}
\label{subsec:monoidal-alphabets-universal-nested-sums}

The organizing principle of this paper is that many Euler-type sums are
governed by the multiplicative alphabet from which its one-variable summand letters are drawn. The required algebraic structure is very small: the alphabet must be closed under pointwise multiplication, because this is exactly what is needed when two summation indices collide in a stuffle product.

\begin{definition}[Monoidal alphabet and monoidal sequence]
Let $R$ be a commutative ring, field, or algebra, and let
$\mathcal A$ be a family of sequences $a:\mathbb N\to R$.  We call
$\mathcal A$ a \emph{monoidal alphabet} if
\[
1\in\mathcal A,
\qquad
 a,b\in\mathcal A\Longrightarrow ab\in\mathcal A,
\qquad
 (ab)(n):=a(n)b(n).
\]
The elements of $\mathcal A$ are called \emph{monoidal sequences}, or
\emph{letters}.  Given letters $a_1,\ldots,a_d\in\mathcal A$, define the
associated universal finite nested sum by
\begin{equation}
\label{eq:monoidal-universal-nested-sum}
\mathfrak H_{\mathcal A}(N;a_1,\ldots,a_d)
:=
\sum_{N\ge n_1>\cdots>n_d\ge1}
 a_1(n_1)a_2(n_2)\cdots a_d(n_d),
\qquad
\mathfrak H_{\mathcal A}(N;\emptyset)=1.
\end{equation}
\end{definition}

For words $u,v$ in the alphabet $\mathcal A$, the usual decomposition of the
two ordered index sets gives
\begin{equation}
\label{eq:monoidal-stuffle-product}
\mathfrak H_{\mathcal A}(N;u)\mathfrak H_{\mathcal A}(N;v)
=
\sum_{w\in u*v}\mathfrak H_{\mathcal A}(N;w),
\end{equation}
where $u*v$ denotes the quasi-shuffle, or stuffle, of words, and a collision of
letters $a$ and $b$ is replaced by their product $ab\in\mathcal A$.  Thus every
product-closed alphabet of one-variable sequences gives a stable nested-sum
space.

The first and smallest choice is the ordinary harmonic alphabet
\[
\phi_r(n)=n^{-r},
\qquad r\in\mathbb C.
\]
It gives the usual finite multiple harmonic numbers.  Adjoining an exponential
color gives the colored harmonic alphabet
\begin{equation}
\label{eq:monoidal-colored-letter}
\phi_{r,s}(n)=s^n n^{-r},
\qquad r,s\in\mathbb C,
\qquad
\phi_{r,s}\phi_{r',s'}=\phi_{r+r',ss'}.
\end{equation}
Then $\mathfrak H_{\mathcal A}$ is exactly the colored multiple harmonic number
$\mathcal H_\alpha(N)$ used below.  The specialization $s=1$ gives the
uncolored case, while convergent limits give the corresponding multiple zeta
values and multiple polylogarithms.  Star versions are obtained either by using
weak inequalities or, equivalently, by converting weak inequalities to strict
ones through the same stuffle refinements.

The second choice is the affine alphabet.  Here a letter may carry a finite
product of affine powers,
\begin{equation}
\label{eq:monoidal-affine-letter}
\phi(n)=s^n\prod_{j=1}^{m}(a_jn+b_j)^{-r_j},
\qquad s,a_j,b_j,r_j\in\mathbb C,
\end{equation}
with branches fixed and with no zero denominator on the relevant summation
range.  The product of two such letters again has the same form: colors
multiply and the exponent lists are concatenated, with equal affine bases
combined by adding exponents.  The resulting nested sums are the multiple affine harmonic numbers $\mathcal G_\Gamma(N)$ developed in
Section~\ref{sec:affine-letter-extensions}.  They contain the ordinary colored
sums by taking the single affine base $n$.

The third choice is the polynomial-base alphabet.  Its elementary letters have
the form
\begin{equation}
\label{eq:monoidal-polynomial-letter}
\phi(n)=s^n\prod_{j=1}^{m}P_j(n)^{-r_j},
\qquad s,r_j\in\mathbb C,\quad P_j\in\mathbb C[x].
\end{equation}
Again the product law is closed by multiplying colors and adding or
concatenating exponent data.  The corresponding nested sums are the multiple polynomial-base harmonic numbers $\mathcal P_\Omega(N)$ developed in
Section~\ref{sec:polynomial-letter-extensions}; the affine alphabet is the
special case in which all $P_j$ are linear.

Thus the three main alphabets of this paper are the three simplest
product-closed families obtained by allowing, respectively, powers of $n$,
powers of affine forms, and powers of polynomial bases.  More generally, fixed
scalar bases $f_1(n),\ldots,f_m(n)$ may be adjoined by using letters of the
schematic form
\begin{equation}
\label{eq:monoidal-general-letter}
\phi(n)=s^n\prod_{j=1}^{m}f_j(n)^{q_j},
\qquad s,q_j\in\mathbb C,
\end{equation}
with branches fixed whenever needed.  Their collision rule merely adds exponent
vectors and multiplies colors.  Hence products of inner harmonic-type sums can
first be expanded by the stuffle rule, after which the remaining outer factors
become additional monoidal letters.  The monoidal construction supplies the
natural closed word space in which such Euler-type sums live; reduction to
smaller alphabets, special constants, or classical functions is a separate
problem.  \hyperref[app:further-monoidal-alphabets]{Appendix A} records further
examples that can be handled in the same manner.

\subsection{Basic colored harmonic-number alphabet}
\label{subsec:basic-colored-harmonic-sum-alphabet}
Let a \emph{letter} be a pair $(r,s)\in\mathbb C^2$, and let a \emph{word} be
a finite sequence of letters.  For a word
\[
  \alpha=((r_1,s_1),\ldots,(r_d,s_d)),
\]
write
\begin{equation}
\label{eq:standalone-Halpha}
\mathcal H_{\alpha}(N)
=
\sum_{N\ge n_1>\cdots>n_d\ge 1}
\prod_{j=1}^d s_j^{n_j}n_j^{-r_j},
\qquad
\mathcal H_{\emptyset}(N)=1.
\end{equation}
For two letters we use the merge operation
\begin{equation}
\label{eq:standalone-letter-merge}
(r,s)\circ(r',s')=(r+r',ss').
\end{equation}
For each positive integer $n$, define the summand space
\begin{equation}
\label{eq:standalone-En-space}
\mathscr E_n=
\operatorname{span}_{\mathbb C}
\left\{z^n n^q\mathcal H_{\alpha}(n):
z,q\in\mathbb C,\ \alpha\text{ a word}\right\}.
\end{equation}
Similarly, for the upper limit $N$, let
\begin{equation}
\label{eq:standalone-HN-space}
\mathscr H_N=
\operatorname{span}_{\mathbb C}
\left\{\mathcal H_{\alpha}(N):\alpha\text{ a word}\right\}.
\end{equation}

\begin{definition}[Harmonic-sum reducibility]
A sequence $f(n)$ is called \emph{harmonic-sum reducible} if
$f(n)\in\mathscr E_n$, that is, if it has a finite representation
\begin{equation}
\label{eq:standalone-reducible-form}
f(n)=
\sum_{\nu=1}^{L}
c_\nu z_\nu^n n^{q_\nu}\mathcal H_{\alpha_\nu}(n),
\qquad
c_\nu,z_\nu,q_\nu\in\mathbb C.
\end{equation}
\end{definition}

\begin{remark}[Use of the closure theorem]
The closure assertion used throughout this section is Theorem~\ref{thm:finite-euler-sum}: products of elements of $\mathscr E_n$ remain in $\mathscr E_n$,
and finite summation maps $\mathscr E_n$ into $\mathscr H_N$.  Thus the purpose
of the present section is not to reprove closure, but to record natural
families of factors which belong to $\mathscr E_n$.
\end{remark}

\begin{proposition}[Depth-one harmonic and alternating harmonic numbers]
\label{prop:standalone-depth-one-harmonic}
The ordinary, generalized, colored, and alternating harmonic numbers are
harmonic-sum reducible.  More precisely,
\begin{align}
H_n^{(r)}(s)&:=\sum_{k=1}^{n}\frac{s^k}{k^r}=\mathcal H_{(r,s)}(n),
\label{eq:standalone-colored-harmonic}\\
A_n^{(r)}(s)&:=\sum_{k=1}^{n}\frac{(-1)^{k-1}s^k}{k^r}
=-\mathcal H_{(r,-s)}(n).
\label{eq:standalone-colored-alternating}
\end{align}
In particular $H_n^{(r)}=H_n^{(r)}(1)$ and the classical alternating harmonic
number is $A_n^{(r)}=A_n^{(r)}(1)$.
\end{proposition}

\begin{proof}
The identities are immediate from the definition of $\mathcal H_{\alpha}(n)$
for words of depth one.
\end{proof}

\begin{example}
The following are basis elements of $\mathscr E_n$:
\[
H_n=\mathcal H_{(1,1)}(n),\qquad
H_n^{(2)}\!\left(\frac12\right)=\mathcal H_{(2,1/2)}(n),
\qquad
A_n^{(3)}=-\mathcal H_{(3,-1)}(n).
\]
Therefore, for arbitrary $z,q\in\mathbb C$,
\[
z^n n^q H_n^{(2)}\!\left(\frac12\right)A_n^{(3)}\in\mathscr E_n.
\]
\end{example}

\begin{proposition}[Strict and star multiple harmonic numbers]
\label{prop:standalone-multiple-star}
For vectors $\mathbf r=(r_1,\ldots,r_d)$ and
$\mathbf s=(s_1,\ldots,s_d)$, the colored multiple harmonic number
\begin{equation}
\label{eq:standalone-colored-harmonic-sum}
H_n^{\mathbf r}(\mathbf s)
:=
\sum_{n\ge n_1>\cdots>n_d\ge1}
\prod_{j=1}^{d}\frac{s_j^{n_j}}{n_j^{r_j}}
=
\mathcal H_{((r_1,s_1),\ldots,(r_d,s_d))}(n)
\end{equation}
is harmonic-sum reducible.  The star variant
\begin{equation}
\label{eq:standalone-star-harmonic-sum}
H_n^{\star,\mathbf r}(\mathbf s)
:=
\sum_{n\ge n_1\ge\cdots\ge n_d\ge1}
\prod_{j=1}^{d}\frac{s_j^{n_j}}{n_j^{r_j}}
\end{equation}
is a finite linear combination of strict sums $\mathcal H_{\beta}(n)$, and is
therefore harmonic-sum reducible.
\end{proposition}

\begin{proof}
The strict case is the definition.  For the star sum, decompose the weakly
ordered region according to the equality pattern among adjacent indices.
Whenever two adjacent indices are equal, the corresponding letters merge by
\eqref{eq:standalone-letter-merge}.  The result is a finite sum over
coarsenings of the word $((r_1,s_1),\ldots,(r_d,s_d))$.
\end{proof}

\begin{example}
At depth two the star-to-strict conversion is especially transparent:
\[
H_n^{\star,(r_1,r_2)}(s_1,s_2)
=
\mathcal H_{(r_1,s_1),(r_2,s_2)}(n)
+
\mathcal H_{(r_1+r_2,s_1s_2)}(n).
\]
For example,
\[
H_n^{\star,(1,2)}(1,-1)
=
\mathcal H_{(1,1),(2,-1)}(n)+\mathcal H_{(3,-1)}(n).
\]
\end{example}

\begin{proposition}[Rational and floored harmonic upper limits]
\label{prop:standalone-rational-harmonic}
Let $p\in\mathbb Z_{>0}$ and $\omega_p=e^{2\pi i/p}$.  Choose a branch
$\rho=s^{1/p}$.  Then
\begin{align}
\label{eq:standalone-floor-H-reduction}
H_{\lfloor n/p\rfloor}^{(r)}(s)
&=
p^{r-1}\sum_{j=0}^{p-1}
\mathcal H_{(r,\rho\omega_p^j)}(n),\\
\label{eq:standalone-rational-H-reduction}
H_{n/p}^{(r)}(s)
&=
\Li_r(s)+p^{r-1}\sum_{j=0}^{p-1}\omega_p^{-jn}
\left(\mathcal H_{(r,\rho\omega_p^j)}(n)
-\Li_r(\rho\omega_p^j)\right).
\end{align}
Consequently $H_{n/p}^{(r)}(s)$ and
$H_{\lfloor n/p\rfloor}^{(r)}(s)$ are harmonic-sum reducible.  The same is
true for the alternating analogues $A_{n/p}^{(r)}(s)$ and
$A_{\lfloor n/p\rfloor}^{(r)}(s)$, since $A_x^{(r)}(s)=-H_x^{(r)}(-s)$.
\end{proposition}

\begin{proof}
For \eqref{eq:standalone-floor-H-reduction}, expand the right-hand side and
use
\[
\sum_{j=0}^{p-1}(\omega_p^m)^j=
\begin{cases}p,&p\mid m,\\0,&p\nmid m.
\end{cases}
\]
Only indices $m=pk$ remain, giving $\sum_{k\le\lfloor n/p\rfloor}s^k/k^r$.
For the rational upper argument, use the Lerch continuation
\[
H_x^{(r)}(s)=\Li_r(s)-s^{x+1}\Phi(s,r,x+1)
\]
and apply the same root-of-unity distribution to the tail.  The alternating
form follows by the substitution $s\mapsto -s$ and multiplication by $-1$.
\end{proof}

\begin{example}
For $p=2$, $r=2$, and $s=1/3$, one obtains
\begin{align*}
H_{n/2}^{(2)}\!\left(\frac13\right)
&=
\Li_2\!\left(\frac13\right)
-2(-1)^n\Li_2\!\left(-\frac1{\sqrt3}\right)
-2\Li_2\!\left(\frac1{\sqrt3}\right) \\
&\quad
+2(-1)^n\mathcal H_{(2,-1/\sqrt3)}(n)
+2\mathcal H_{(2,1/\sqrt3)}(n),\\[1mm]
H_{\lfloor n/2\rfloor}^{(2)}\!\left(\frac13\right)
&=
2\mathcal H_{(2,-1/\sqrt3)}(n)
+2\mathcal H_{(2,1/\sqrt3)}(n).
\end{align*}
Similarly,
\[
A_{\lfloor n/2\rfloor}^{(1)}
=-H_{\lfloor n/2\rfloor}^{(1)}(-1)
=-\mathcal H_{(1,i)}(n)-\mathcal H_{(1,-i)}(n),
\]
where $i^2=-1$.
\end{example}

\begin{remark}[Rationally scaled harmonic arguments]
\label{rem:rationally-scaled-harmonic-arguments}
The reduction in Proposition~\ref{prop:standalone-rational-harmonic} remains valid
after replacing $n$ by $qn$.  Hence factors such as $H_{qn/p}^{(r)}(s)$ and
$A_{qn/p}^{(r)}(s)$ reduce to finite sums involving polylogarithmic constants
and colored harmonic numbers with scaled upper limit $qn$.  We use this observation
in Section~\ref{sec:scaled-upper-limit-lifting}.

For example, let $\omega=e^{2\pi i/3}$ and choose $\rho=s^{1/3}$.  Then
\[
H_{2n/3}^{(r)}(s)
=
\Li_r(s)
+3^{r-1}\sum_{\varepsilon^3=1}\varepsilon^n
\left(
\mathcal H_{(r,\varepsilon\rho)}(2n)-\Li_r(\varepsilon\rho)
\right).
\]
For the alternating analogue, putting $\eta=(-s)^{1/3}$ gives
\[
A_{2n/3}^{(r)}(s)
=
-\Li_r(-s)
-3^{r-1}\sum_{\varepsilon^3=1}\varepsilon^n
\left(
\mathcal H_{(r,\varepsilon\eta)}(2n)-\Li_r(\varepsilon\eta)
\right).
\]
Thus both examples reduce to the same colored harmonic-number alphabet, with scaled
upper limit $2n$.
\end{remark}

\begin{remark}[Residue-class harmonic sums and finite level-\(M\) values]
\label{rem:residue-class-level-M-values}
Odd harmonic numbers are the level-two residue-class case.  For
\[
 O_k^{(r)}(s):=\sum_{j=1}^{k}\frac{s^{2j-1}}{(j-\frac12)^r},
\]
we have, since \(\mathbf 1_{m\,\mathrm{odd}}=(1-(-1)^m)/2\),
\[
\begin{aligned}
 O_k^{(r)}(s)
 &=2^r\sum_{\substack{1\le m\le 2k\\ m\,\mathrm{odd}}}\frac{s^m}{m^r}
  =2^{r-1}\sum_{m=1}^{2k}\frac{s^m-(-s)^m}{m^r}  \\
 &=2^{r-1}\left(
    \mathcal H_{(r,s)}(2k)-\mathcal H_{(r,-s)}(2k)
   \right).
\end{aligned}
\]
Thus \(O_k^{(r)}(s)\in\mathscr H_{2k}\).  The same filter applied at each
index gives, for finite colored Hoffman multiple \(t\)-values \citep{Hoffman2019},
\[
 t_k(\mathbf r;\mathbf s):=
 \sum_{k\ge j_1>\cdots>j_d\ge1}
 \prod_{\ell=1}^{d}\frac{s_\ell^{2j_\ell-1}}{(2j_\ell-1)^{r_\ell}},
\]
that
\[
 t_k(\mathbf r;\mathbf s)
 =2^{-d}\sum_{\sigma_1,\ldots,\sigma_d=\pm1}
 \left(\prod_{\ell=1}^{d}\sigma_\ell\right)
 \mathcal H_{(r_1,\sigma_1s_1),\ldots,(r_d,\sigma_ds_d)}(2k).
\]
The analogous parity-filter argument, with the alternating parity pattern,
applies to finite Kaneko--Tsumura multiple \(T\)-values \citep{KanekoTsumura2020}.
More generally, the multiple mixed values of Xu and Zhao allow arbitrary even/odd
patterns and include both Hoffman multiple \(t\)-values and Kaneko--Tsumura
multiple \(T\)-values as special cases \citep{XuZhao2022}; their finite analogues
are covered by the same parity filters \citep{Zhao2024FiniteMMV}.

For \(M\ge2\), \(1\le a\le M\), and \(\omega_M=e^{2\pi i/M}\), define the
level-\(M\) residue-class harmonic number by
\[
 R_{M,a;k}^{(r)}(s):=
 M^r\sum_{\substack{1\le m\le Mk\\ m\equiv a\pmod M}}
 \frac{s^m}{m^r}.
\]
The finite Fourier filter
\[
 \mathbf 1_{m\equiv a\pmod M}=
 \frac1M\sum_{\nu=0}^{M-1}\omega_M^{\nu(m-a)}
\]
gives
\[
 R_{M,a;k}^{(r)}(s)=
 M^{r-1}\sum_{\nu=0}^{M-1}\omega_M^{-a\nu}
 \mathcal H_{(r,s\omega_M^\nu)}(Mk).
\]
Applying this filter independently to each index gives the corresponding finite
multiple zeta values of level \(M\), as studied for example in
\citep{YuanZhao2016}.  Hence residue-class harmonic sums and finite multiple
zeta values of level \(M\) are obtained by the same finite Fourier filter and
lie in \(\mathscr H_{Mk}\).
\end{remark}

\begin{proposition}[Integral-upper hyperharmonic numbers]
\label{prop:standalone-integral-hyperharmonic}
Fix $m\in\mathbb Z_{\ge0}$.  Define
\[
h_n^{[0]}(r;s)=H_n^{(r)}(s),
\qquad
h_n^{[m]}(r;s)=\sum_{k=1}^{n}h_k^{[m-1]}(r;s)\quad(m\ge1).
\]
Then $h_n^{[m]}(r;s)$ is harmonic-sum reducible.  More generally, if
$\alpha$ is a word and
\[
\Hyp_{\alpha}^{[0]}(n)=\mathcal H_{\alpha}(n),
\qquad
\Hyp_{\alpha}^{[m]}(n)=\sum_{k=1}^{n}\Hyp_{\alpha}^{[m-1]}(k),
\]
then every fixed-order multiple hyperharmonic number
$\Hyp_{\alpha}^{[m]}(n)$ is harmonic-sum reducible.  The same holds
for the star version obtained by replacing the initial term by
$H_n^{\star,\mathbf r}(\mathbf s)$.
\end{proposition}

\begin{proof}
This is repeated finite summation applied to elements already in
$\mathscr E_n$.  Star initial data are first converted to strict sums by
Proposition~\ref{prop:standalone-multiple-star}.
\end{proof}

\begin{example}
For $m=1$,
\[
h_n^{[1]}(r;s)=\sum_{k=1}^n H_k^{(r)}(s)
=\mathcal H_{(-1,1),(r,s)}(n)+\mathcal H_{(r-1,s)}(n),
\]
because summing $H_k^{(r)}(s)$ splits the region $k>j$ and $k=j$.  In
particular,
\[
\sum_{k=1}^n H_k
=\mathcal H_{(-1,1),(1,1)}(n)+\mathcal H_{(0,1)}(n)
=(n+1)H_n-n.
\]
\end{example}

\begin{proposition}[Analytically continued hyperharmonic upper limits]
\label{prop:standalone-analytic-hyperharmonic-prop}
Let $h_x^{[m]}(r;s)$ denote the analytic continuation, in the upper argument
$x$, of the $m$-fold hyperharmonic number built from $H_x^{(r)}(s)$, normalized
by $h_x^{[0]}(r;s)=H_x^{(r)}(s)$.  Using the analytic continuation of
\citet{Mezo2009}, if
\[
\binom{y+m}{m}=\sum_{j=0}^{m}c_{m,j}y^j,
\]
then
\begin{equation}
\label{eq:standalone-analytic-hyperharmonic}
h_x^{[m]}(r;s)
=
\sum_{j=0}^{m}c_{m,j}
\sum_{a=0}^{j}\binom{j}{a}(-1)^a x^{j-a}H_x^{(r-a)}(s).
\end{equation}
Consequently, for $p\in\mathbb Z_{>0}$,
\[
h_{n/p}^{[m]}(r;s),
\qquad
h_{\lfloor n/p\rfloor}^{[m]}(r;s)
\]
are harmonic-sum reducible.
\end{proposition}

\begin{proof}
Formula \eqref{eq:standalone-analytic-hyperharmonic} reduces the non-integral
upper argument to a finite polynomial combination of harmonic numbers with the
same upper argument.  Proposition~\ref{prop:standalone-rational-harmonic} then
applies to $n/p$ and $\lfloor n/p\rfloor$.
\end{proof}

\begin{example}
For the first repeated hyperharmonic continuation one obtains
\begin{align*}
h_{n/2}^{[1]}(2;1)
&=-\frac{\pi^2}{6}+\frac{(-1)^n\pi^2}{6}
  -\frac{n\pi^2}{12}+\frac{(-1)^n n\pi^2}{12}
  +(1-(-1)^n)\log 2\\
&\quad-(-1)^n\mathcal H_{(1,-1)}(n)-\mathcal H_{(1,1)}(n)
 +(n+2)(-1)^n\mathcal H_{(2,-1)}(n)
 +(n+2)\mathcal H_{(2,1)}(n),\\[1mm]
h_{\lfloor n/2\rfloor}^{[1]}(2;1)
&=-\mathcal H_{(1,-1)}(n)-\mathcal H_{(1,1)}(n)
 +\left(n+\frac{3+(-1)^n}{2}\right)
 \left(\mathcal H_{(2,-1)}(n)+\mathcal H_{(2,1)}(n)\right).
\end{align*}
Also, from \eqref{eq:standalone-analytic-hyperharmonic},
\[
h_x^{[1]}(r;s)=(x+1)H_x^{(r)}(s)-H_x^{(r-1)}(s),
\]
which displays the reduction to ordinary harmonic upper-limit reductions.
\end{example}

\begin{remark}[Rationally scaled hyperharmonic arguments]
\label{rem:rationally-scaled-hyperharmonic-arguments}
Combining the preceding analytic-continuation formula with
Remark~\ref{rem:rationally-scaled-harmonic-arguments} gives the corresponding
scaled hyperharmonic reductions.  For instance, with
$\omega=e^{2\pi i/3}$ and $\rho=s^{1/3}$,
\begin{align*}
h_{2n/3}^{[1]}(r;s)
&=
\left(\frac{2n}{3}+1\right)\Li_r(s)-\Li_{r-1}(s) \\
&\quad
+3^{r-1}\left(\frac{2n}{3}+1\right)
\sum_{\varepsilon^3=1}\varepsilon^n
\left(
\mathcal H_{(r,\varepsilon\rho)}(2n)-\Li_r(\varepsilon\rho)
\right) \\
&\quad
-3^{r-2}
\sum_{\varepsilon^3=1}\varepsilon^n
\left(
\mathcal H_{(r-1,\varepsilon\rho)}(2n)-\Li_{r-1}(\varepsilon\rho)
\right).
\end{align*}
Hence this hyperharmonic example also lies in the same closure class after the
preliminary reduction.  This case will be used in
Section~\ref{sec:scaled-upper-limit-lifting} together with the harmonic and
alternating cases.
\end{remark}

\begin{proposition}[Polygamma tails]
\label{prop:standalone-polygamma}
For $m\ge1$, polygamma values at integer-shifted arguments are
harmonic-sum reducible:
\[
\psi(n+1)+\gamma=H_n,
\qquad
\psi^{(m)}(n+1)=(-1)^{m+1}m!\bigl(\zeta(m+1)-H_n^{(m+1)}\bigr).
\]
\end{proposition}

\begin{proof}
The displayed identities express the polygamma values as constants plus
ordinary harmonic numbers.
\end{proof}

\begin{example}
For instance,
\[
\psi'(n+1)=\zeta(2)-H_n^{(2)},
\qquad
\psi''(n+1)=-2\bigl(\zeta(3)-H_n^{(3)}\bigr).
\]
Thus $z^n n^q\psi'(n+1)H_n$ lies in $\mathscr E_n$.
\end{example}

\begin{proposition}[Stirling and fixed-degree combinatorial factors]
\label{prop:standalone-stirling-combinatorial}
The normalized Stirling numbers
\[
\frac{s(n+1,k)}{n!}
=(-1)^{n+1-k}\mathcal H_{\underbrace{(1,1),\ldots,(1,1)}_{k-1}}(n),
\]
where $s(n,k)$ denotes the signed Stirling number of the first kind, are
harmonic-sum reducible for fixed $k$.  The Stirling numbers of the second kind
are also reducible for fixed $k$, since
\[
S(n,k)=\frac1{k!}\sum_{j=0}^{k}(-1)^{k-j}\binom{k}{j}j^n.
\]
Moreover, fixed-degree polynomial, binomial, and factorial-power factors such
as $P(n)$, $\binom{an+b}{k}$, $(an+b)^{\underline{k}}$, and
$(an+b)^{\overline{k}}$ are harmonic-sum reducible when $k$ is fixed.
\end{proposition}

\begin{proof}
The first formula is the standard relation between Stirling numbers of the
first kind and elementary symmetric functions in $1,1/2,\ldots,1/n$; the
second is the finite exponential formula for $S(n,k)$.  Fixed-degree binomial
and factorial-power factors are polynomials in $n$.
\end{proof}

\begin{example}
At small fixed degrees,
\[
\frac{s(n+1,2)}{n!}=(-1)^{n-1}H_n,
\qquad
S(n,2)=2^{n-1}-1,
\qquad
\binom{n+3}{2}=\frac{n^2+5n+6}{2}.
\]
Hence all three can be multiplied into a summand of the form
$z^n n^q\mathcal H_{\alpha}(n)$ without leaving $\mathscr E_n$.
\end{example}

\begin{proposition}[Zeta and Lerch tails]
\label{prop:standalone-zeta-lerch}
The Hurwitz zeta and Lerch tails
\[
\zeta(s,n+1)=\zeta(s)-H_n^{(s)}\quad(s\ne1),
\]
\[
\Phi(z,s,n+1)=z^{-n-1}\bigl(\Li_s(z)-H_n^{(s)}(z)\bigr)
\quad(z\ne0,1)
\]
are harmonic-sum reducible.
\end{proposition}

\begin{proof}
Both displayed identities write the tails as constants, exponential factors,
and colored harmonic numbers.
\end{proof}

\begin{example}
For example,
\[
\zeta(2,n+1)=\zeta(2)-H_n^{(2)},
\qquad
\Phi\!\left(\frac12,2,n+1\right)
=2^{n+1}\left(\Li_2\!\left(\frac12\right)
-H_n^{(2)}\!\left(\frac12\right)\right).
\]
\end{example}

\begin{proposition}[Constant-coefficient recurrence sequences]
\label{prop:standalone-recurrences}
Any sequence satisfying a homogeneous linear recurrence with constant
coefficients is harmonic-sum reducible.  Indeed, such a sequence has the
exponential-polynomial form
\[
a(n)=\sum_{\lambda}P_\lambda(n)\lambda^n.
\]
This includes Fibonacci, Lucas, Pell, Jacobsthal, Padovan, Perrin,
Narayana's cows, $m$-bonacci sequences, fixed-parameter Chebyshev sequences
$T_n(x)$ and $U_n(x)$, and companion sequences governed by fixed
constant-coefficient recurrences.
\end{proposition}

\begin{proof}
The general solution of a fixed homogeneous constant-coefficient recurrence is
a finite sum of polynomial multiples of exponentials.  Each term belongs to
$\mathscr E_n$ with the empty word.
\end{proof}

\begin{example}
Binet's formula gives
\[
F_n=\frac{\phi^n-\widehat\phi^{\,n}}{\sqrt5},
\qquad
\phi=\frac{1+\sqrt5}{2},\quad \widehat\phi=\frac{1-\sqrt5}{2}.
\]
The Padovan sequence satisfies $P_n=P_{n-2}+P_{n-3}$, and Narayana's cows
satisfy $C_n=C_{n-1}+C_{n-3}$; hence both are finite sums of
exponential-polynomial terms determined by their characteristic polynomials.
\end{example}

\begin{proposition}[Fixed-periodic functions]
\label{prop:standalone-periodic}
Every fixed-modulus periodic function is harmonic-sum reducible.  If $M\ge1$
and $\omega_M=e^{2\pi i/M}$, then
\[
\mathbf1_{n\equiv a\,({\rm mod}\,M)}
=\frac1M\sum_{j=0}^{M-1}\omega_M^{j(n-a)}.
\]
Therefore every function depending only on $n\bmod M$ is a finite linear
combination of exponentials.
\end{proposition}

\begin{proof}
The displayed finite Fourier expansion expresses the residue-class indicator
as a finite exponential sum.  Any fixed-periodic function is a finite linear
combination of such indicators.
\end{proof}

\begin{example}
For parity,
\[
\mathbf1_{2\mid n}=\frac{1+(-1)^n}{2},
\qquad
\mathbf1_{2\nmid n}=\frac{1-(-1)^n}{2}.
\]
For modulus $3$,
\[
\mathbf1_{n\equiv1\,({\rm mod}\,3)}
=\frac13\left(1+\omega_3^{n-1}+\omega_3^{2(n-1)}\right).
\]
\end{example}

\begin{proposition}[Fixed-modulus arithmetic factors]
\label{prop:standalone-fixed-modulus-arithmetic}
Fixed-modulus residue, quotient, character, order, divisibility, gcd, and lcm
functions are harmonic-sum reducible under the usual fixed-modulus hypotheses.
For example, if $r_M(t)$ is the least non-negative residue of $t$ modulo $M$,
then $r_M(an+b)$ is periodic in $n$, and
\[
q_M(an+b)=\frac{an+b-r_M(an+b)}{M}
\]
is affine plus periodic.  Similarly, Dirichlet characters, Jacobi and
Kronecker symbols with fixed modulus, $\ord_M(an+b)$ on the coprime residue
classes, $r_M(u^{an+b})$ when the modulus is fixed, $\gcd(n,M)$, and
$\lcm(n,M)=nM/\gcd(n,M)$ are affine, polynomial, or periodic combinations.
\end{proposition}

\begin{proof}
With the modulus fixed, these functions are determined by finitely many
residue classes, except for the explicitly affine or polynomial factors shown
above.  Proposition~\ref{prop:standalone-periodic} applies to the periodic
parts.
\end{proof}

\begin{example}
For $M=2$,
\[
r_2(n)=\frac{1-(-1)^n}{2},
\qquad
\left\lfloor\frac n2\right\rfloor=\frac n2-\frac{1-(-1)^n}{4}.
\]
For a fixed Dirichlet character $\chi$ modulo $M$,
\[
\chi(n)=\sum_{a=0}^{M-1}\chi(a)\mathbf1_{n\equiv a\,({\rm mod}\,M)},
\]
so $\chi(n)$ is a finite Fourier sum.
\end{example}

\begin{proposition}[Rational affine integral-part functions]
\label{prop:standalone-integral-part}
Integral and nearest-integer functions of rational affine arguments are
harmonic-sum reducible.  Let $\ell(n)=an+b\in\mathbb Q n+\mathbb Q$, and choose
$M$ so that $M\ell(n)=An+B$ with $A,B\in\mathbb Z$.  Then
\[
\lfloor\ell(n)\rfloor=\frac{An+B-r_M(An+B)}{M},
\]
so floors are affine plus periodic.  Ceilings, fractional parts, truncations
on fixed-sign ranges, and nearest-integer functions with a fixed tie rule are
obtained from the same residue-class correction.
\end{proposition}

\begin{proof}
The formula expresses the floor as an affine function of $n$ minus a fixed
periodic residue correction.  The remaining integral-part functions differ
from the floor by affine, periodic, or bounded fixed-periodic corrections.
\end{proof}

\begin{example}
For instance,
\[
\left\lfloor\frac{3n+1}{4}\right\rfloor
=\frac{3n+1-r_4(3n+1)}{4},
\qquad
\left\{\frac{3n+1}{4}\right\}=\frac{r_4(3n+1)}{4}.
\]
Both are affine-periodic expressions and therefore lie in $\mathscr E_n$.
\end{example}

\begin{proposition}[Prime-power arithmetic functions]
\label{prop:standalone-prime-power}
For a fixed prime $p$, the functions
\[
\varphi(p^n),\quad \lambda(p^n),\quad \sigma_k(p^n),\quad
\Lambda(p^n),\quad \lambda_{\mathrm L}(p^n),\quad
\Omega(p^n),\quad \omega_{\mathrm{arith}}(p^n),\quad v_p(cp^{\ell(n)})
\]
Here \(\omega_{\mathrm{arith}}\) denotes the arithmetic function counting distinct
prime divisors; it is unrelated to the root-of-unity notation \(\omega_m\).
These functions are polynomial, exponential, or affine in $n$ under the usual fixed-parameter
assumptions.  Moreover, Ramanujan's tau function on prime powers satisfies
\[
\tau(p^{n+1})=\tau(p)\tau(p^n)-p^{11}\tau(p^{n-1}),
\]
and hence is recurrence-defined.  Therefore all these prime-power factors are
harmonic-sum reducible.
\end{proposition}

\begin{proof}
The standard prime-power formulas for these arithmetic functions are
polynomial, affine, or exponential in $n$.  The tau values on prime powers
satisfy the displayed constant-coefficient recurrence, so Proposition
\ref{prop:standalone-recurrences} applies.
\end{proof}

\begin{example}
For fixed $p$,
\[
\varphi(p^n)=p^n-p^{n-1},
\qquad
\sigma_k(p^n)=\frac{p^{k(n+1)}-1}{p^k-1},
\qquad
\Omega(p^n)=n.
\]
Thus $\varphi(p^n)H_n$ and $\sigma_k(p^n)H_n^{(r)}(s)$ are admissible
summand factors.
\end{example}

\begin{proposition}[Elementary exponential, trigonometric, and hyperbolic factors]
\label{prop:standalone-elementary-exponential}
For rational or complex affine $\ell(n)=an+b$, the functions
$e^{\ell(n)}$, $\cos(\ell(n))$, $\sin(\ell(n))$, $\cosh(\ell(n))$, and
$\sinh(\ell(n))$ are harmonic-sum reducible.  For example,
\[
\cos(\ell(n))=
\frac{e^{ib}(e^{ia})^n+e^{-ib}(e^{-ia})^n}{2},
\qquad
\sin(\ell(n))=
\frac{e^{ib}(e^{ia})^n-e^{-ib}(e^{-ia})^n}{2i}.
\]
\end{proposition}

\begin{proof}
Euler's formulas express trigonometric and hyperbolic functions with affine
arguments as finite linear combinations of exponentials in $n$.
\end{proof}

\begin{example}
For example,
\[
\cos\left(\frac{\pi n}{3}\right)
=\frac12\left(e^{i\pi/3}\right)^n+
 \frac12\left(e^{-i\pi/3}\right)^n,
\]
so
\[
z^n n^q\cos\left(\frac{\pi n}{3}\right)H_n^{(r)}(s)
\]
is a sum of two basis-type elements of $\mathscr E_n$.
\end{example}

The following table summarizes the principal families in compact form.  In
the table, $k,m$ are fixed non-negative integers, $p\in\mathbb Z_{>0}$ when
it occurs as a denominator, $p$ is prime in the prime-power row, $M$ is a
fixed modulus, and $\ell(n)=an+b$ is rational affine unless otherwise stated.

\begin{center}
\scriptsize
\setlength{\tabcolsep}{3pt}
\renewcommand{\arraystretch}{1.05}
\begin{adjustbox}{max width=\textwidth}
\begin{tabular}{@{}>{\raggedright\arraybackslash}p{0.25\textwidth}>{\raggedright\arraybackslash}p{0.48\textwidth}>{\raggedright\arraybackslash}p{0.21\textwidth}@{}}
\hline
\textbf{Family} & \textbf{Typical objects} & \textbf{Reason}\\
\hline
Harmonic and colored harmonic numbers
& $H_n^{(r)}$, $H_n^{(r)}(s)$, $A_n^{(r)}(s)$
& depth-one $\mathcal H$ sums\\
Rational and floored harmonic upper limits
& $H_{n/p}^{(r)}(s)$, $A_{n/p}^{(r)}(s)$, $H_{\lfloor n/p\rfloor}^{(r)}(s)$, $A_{\lfloor n/p\rfloor}^{(r)}(s)$
& root-of-unity distribution/filter\\
Odd and parity-restricted finite sums
& $O_k^{(r)}(s)$, finite Hoffman $t_k(\mathbf r;\mathbf s)$, finite Kaneko--Tsumura $T$-values, finite mixed values
& parity filters at upper limit $2k$\\
Residue-class level-$M$ finite sums
& $R_{M,a;k}^{(r)}(s)$, finite level-$M$ residue-class multiple zeta values
& finite Fourier filters at upper limit $Mk$\\
Multiple harmonic and star sums
& $H_n^{\mathbf r}(\mathbf s)$, $H_n^{\star,\mathbf r}(\mathbf s)$
& strict sums; star sums split by equalities\\
Integral-upper hyperharmonic families
& $h_n^{[m]}(r;s)$, $\Hyp_{\alpha}^{[m]}(n)$, $\Hyp_{\mathbf r,\mathbf s}^{\star,[m]}(n)$
& repeated finite summation\\
Rational and floored hyperharmonic upper limits
& $h_{n/p}^{[m]}(r;s)$, $h_{\lfloor n/p\rfloor}^{[m]}(r;s)$
& Mez\H{o} analytic continuation and harmonic upper-limit reductions\\
Polygamma values
& $\psi(n+1)$, $\psi^{(m)}(n+1)$
& harmonic-number tails\\
Stirling and fixed-degree combinatorial factors
& $s(n+1,k)/n!$, $|s(n+1,k)|/n!$, $S(n,k)$, $\binom{\ell(n)}{k}$
& multiple harmonic numbers, finite exponential sums, or polynomials\\
Zeta and Lerch tails
& $\zeta(s,n+1)$ $(s\ne1)$, $\Phi(z,s,n+1)$
& finite harmonic or colored-harmonic tails\\
Constant-coefficient recurrence sequences
& Fibonacci/Lucas, Pell/Pell--Lucas, Jacobsthal/Jacobsthal--Lucas, $m$-bonacci, Padovan/Perrin, Narayana's cows, $T_n(x)$, $U_n(x)$
& exponential-polynomial form\\
Fixed-periodic functions
& $\mathbf1_{n\equiv a\,({\rm mod}\,M)}$ and any fixed-periodic factor
& finite Fourier expansion\\
Fixed-modulus arithmetic factors
& $r_M(\ell(n))$, $q_M(\ell(n))$, $\chi(\ell(n))$, $\left(\frac{\ell(n)}{M}\right)$, $\ord_M(\ell(n))$, $\gcd(n,M)$, $\lcm(n,M)$
& affine, polynomial, or periodic form\\
Integral and nearest-integer functions
& $\lfloor\ell(n)\rfloor$, $\lceil\ell(n)\rceil$, $\{\ell(n)\}$, $\trunc(\ell(n))$, $N_\tau(\ell(n))$
& affine plus periodic correction\\
Prime-power arithmetic functions
& $\varphi(p^{\ell(n)})$, $\lambda(p^{\ell(n)})$, $\sigma_k(p^{\ell(n)})$, $\Lambda(p^{\ell(n)})$, $\tau(p^n)$
& polynomial, exponential, periodic, or recurrence form\\
Elementary exponential functions
& $z^n$, $\cos(an+b)$, $\sin(an+b)$, $\cosh(an+b)$, $\sinh(an+b)$
& finite exponential combinations\\
\hline
\end{tabular}
\end{adjustbox}
\end{center}

\subsection{Multiple affine harmonic number alphabet}
\label{subsec:affine-harmonic-sum-alphabet}

We next introduce the affine analogue of the basic colored harmonic-number
alphabet.  The purpose is to replace scaled or shifted upper arguments by
finite sums whose truncation index is again $n$, at the cost of allowing affine
letters of the form $(an+b)^{-r}$.  Throughout this subsection all displayed
reductions are to affine sums with upper limit $n$.

An \emph{affine letter} is a triple
\[
  L=(\boldsymbol\rho,\sigma,\mathbf A),
  \qquad
  \boldsymbol\rho=(\rho_1,\ldots,\rho_t),
  \qquad
  \mathbf A=((a_1,b_1),\ldots,(a_t,b_t)),
\]
where $\sigma\in\mathbb C$, $\rho_\nu\in\mathbb C$, and each affine form
$a_\nu n+b_\nu$ is nonzero on the relevant summation range.  Its value at a
positive integer $n$ is
\begin{equation}
\label{eq:affine-letter-value-final}
  L(n)=\sigma^n\prod_{\nu=1}^{t}(a_\nu n+b_\nu)^{-\rho_\nu}.
\end{equation}
If $L$ and $M$ are affine letters, we write $L\circ M$ for the affine letter
whose value is $L(n)M(n)$; explicitly this multiplies the colors and
concatenates the affine factors.

For a word $\Gamma=(L_1,\ldots,L_d)$ in affine letters, define
\begin{equation}
\label{eq:affine-Gamma-definition-final}
  \mathcal G_{\Gamma}(N)
  :=
  \sum_{N\ge n_1>\cdots>n_d\ge1}
  \prod_{j=1}^{d}L_j(n_j),
  \qquad
  \mathcal G_{\emptyset}(N)=1.
\end{equation}
The corresponding finite affine value space is
\begin{equation}
\label{eq:affine-value-space-final}
  \mathscr G_N
  :=
  \operatorname{span}_{\mathbb C}
  \{\mathcal G_{\Gamma}(N):\Gamma\text{ a word in affine letters}\}.
\end{equation}
The basic colored alphabet is contained in this one, since the basic letter
$(r,s)$ is the affine letter $((r),s,((1,0)))$.  For readability, in examples we
write
\begin{equation}
\label{eq:affine-L-shorthand-final}
  L_{a,b}^{r}(\sigma):=((r),\sigma,((a,b))),
  \qquad
  L_{a,b}^{\mathbf r}(\sigma):=(\mathbf r,\sigma,((a,b))_{r\in\mathbf r})
\end{equation}
when no confusion can arise.

For summand-level reductions define
\begin{equation}
\label{eq:affine-summand-space-final}
  \mathscr E^{\rm aff}_n
  :=
  \operatorname{span}_{\mathbb C}
  \{L(n)G(n):L\text{ is an affine letter or }1,
  \;G(n)\in\mathscr G_n\}.
\end{equation}

\begin{definition}[Affine-harmonic-sum reducibility]
\label{def:affine-harmonic-sum-reducibility-final}
A sequence $f(n)$ is called \emph{affine-harmonic-sum reducible} if
$f(n)\in\mathscr E^{\rm aff}_n$.  Equivalently, it has a finite representation
\begin{equation}
\label{eq:affine-reducible-form-final}
  f(n)=\sum_{\nu=1}^{M} c_\nu\sigma_\nu^n
  \prod_{\mu=1}^{m_\nu}(a_{\nu\mu}n+b_{\nu\mu})^{-q_{\nu\mu}}
  \mathcal G_{\Gamma_\nu}(n),
\end{equation}
where $c_\nu,\sigma_\nu,q_{\nu\mu}\in\mathbb C$, the affine forms are nonzero on
positive integers in the relevant range, and $\Gamma_\nu$ is a word in affine
letters.  Thus the affine analogue of the basic factor $z^n n^q\mathcal
H_\alpha(n)$ is a finite product of affine powers times an affine word sum.
\end{definition}

\begin{proposition}[Affine upper arguments in harmonic, alternating harmonic, and hyperharmonic numbers]
\label{prop:scaled-harmonic-alternating-hyperharmonic-affine-final}
Let $p_1,p_2\in\mathbb Q$ with $p_1>0$, and put
\[
  x(n)=p_1n+p_2.
\]
Interpret non-integral upper arguments by the Lerch continuation
\[
  H_x^{(r)}(s)=\Li_r(s)-s^{x+1}\Phi(s,r,x+1),
\]
with branches fixed where needed.  Then $H_{x(n)}^{(r)}(s)$ is affine-harmonic-sum reducible.  Consequently the alternating analogue
\[
  A_{x(n)}^{(r)}(s):=-H_{x(n)}^{(r)}(-s)
\]
is affine-harmonic-sum reducible.  Moreover, for each fixed
$m\in\mathbb Z_{\ge0}$, the analytically continued hyperharmonic number
$h_{x(n)}^{[m]}(r;s)$ is affine-harmonic-sum reducible.
\end{proposition}

\begin{proof}
Write $p_1=A/Q$ and $p_2=B/Q$, where $A,Q\in\mathbb Z_{>0}$ and
$B\in\mathbb Z$.  Let $\omega_Q=e^{2\pi i/Q}$ and choose a branch
$\rho=s^{1/Q}$.  The root-of-unity distribution for the Lerch tail gives
\begin{equation}
\label{eq:rational-affine-upper-harmonic-reduction-final}
  H_{(An+B)/Q}^{(r)}(s)
  =
  \Li_r(s)+Q^{r-1}\sum_{j=0}^{Q-1}\omega_Q^{-j(An+B)}
  \left(
    H_{An+B}^{(r)}(\rho\omega_Q^j)-\Li_r(\rho\omega_Q^j)
  \right).
\end{equation}
It remains to reduce the integer-affine upper limit $An+B$.  For any fixed
integer $B$ and any color $\xi$, decompose the range up to $An+B$ into residue
classes modulo $A$:
\begin{equation}
\label{eq:integer-affine-upper-depth-one-final}
  H_{An+B}^{(r)}(\xi)
  =
  \sum_{\ell=1}^{A}\xi^{\ell-A}
  \mathcal G_{L_{A,\ell-A}^{r}(\xi^A)}(n)+E_B(n;\xi,r),
\end{equation}
where the endpoint correction is the finite sum
\[
E_B(n;\xi,r)=
\begin{cases}
\displaystyle\sum_{t=1}^{B}\dfrac{\xi^{An+t}}{(An+t)^r},& B\ge1,\\[3mm]
0,& B=0,\\[2mm]
\displaystyle-\sum_{t=B+1}^{0}\dfrac{\xi^{An+t}}{(An+t)^r},& B\le -1.
\end{cases}
\]
Each term of $E_B(n;\xi,r)$ is an affine letter evaluated at $n$.  Therefore
\eqref{eq:integer-affine-upper-depth-one-final} lies in the affine summand
space, and substitution into
\eqref{eq:rational-affine-upper-harmonic-reduction-final} proves the assertion
for $H_{x(n)}^{(r)}(s)$.  The alternating case follows from
$A_x^{(r)}(s)=-H_x^{(r)}(-s)$.

For fixed hyperharmonic order $m$, the analytic-continuation formula expresses
$h_x^{[m]}(r;s)$ as a finite polynomial combination, in $x$, of terms
$H_x^{(r-a)}(s)$.  Since $x(n)$ is affine in $n$, the polynomial factors are
allowed affine prefactors and the harmonic factors have already been reduced.
\end{proof}

\begin{example}\normalfont
For every positive integer $p$,
\[
  H_{pn}^{(r)}(s)
  =
  \sum_{\ell=1}^{p}s^{\ell-p}
  \mathcal G_{L_{p,\ell-p}^{r}(s^p)}(n).
\]
In particular,
\[
  H_{2n}^{(r)}
  =
  \mathcal G_{L_{2,-1}^{r}(1)}(n)+\mathcal G_{L_{2,0}^{r}(1)}(n),
\]
which is the even/odd decomposition written with truncation index $n$.
\end{example}

\begin{example}\normalfont
For the alternating harmonic number with integer-affine upper argument,
\[
\begin{aligned}
A_{3n+1}^{(r)}(s)
&=-s^{-2}\mathcal G_{L_{3,-2}^{r}(-s^3)}(n)
  +s^{-1}\mathcal G_{L_{3,-1}^{r}(-s^3)}(n)
  -\mathcal G_{L_{3,0}^{r}(-s^3)}(n)  \\
&\quad
  +s\,\frac{(-s^3)^n}{(3n+1)^r}.
\end{aligned}
\]
The last term is an affine letter evaluated at the outer index $n$.
\end{example}

\begin{example}\normalfont
The fractional upper argument used in computations, for instance
$H_{n/2+1/3}^{(2)}(1/4)$, is covered by
\eqref{eq:rational-affine-upper-harmonic-reduction-final} with
$A=3$, $B=2$, and $Q=6$.  Hence it becomes a finite linear combination of
constants, affine outer letters, and terms
\[
  \mathcal G_{L_{3,1}^{2}(\xi)}(n),
  \qquad
  \mathcal G_{L_{3,2}^{2}(\xi)}(n),
\]
where $\xi$ runs through finitely many sixth-root colored constants.  Similarly,
for first-order hyperharmonic numbers,
\[
 h_{2n}^{[1]}(r;s)=(2n+1)H_{2n}^{(r)}(s)-H_{2n}^{(r-1)}(s),
\]
and the preceding formula for $H_{2n}^{(r)}(s)$ gives an affine reduction with
upper limit $n$.
\end{example}

\begin{proposition}[Integer-affine upper arguments in colored multiple harmonic numbers]
\label{prop:scaled-multiple-harmonic-affine-final}
Let $p\in\mathbb Z_{>0}$ and $q\in\mathbb Z_{\ge0}$.  With the notation for
colored multiple harmonic numbers from the previous subsection,
\[
  H_{pn+q}^{\mathbf r}(\mathbf s)
\]
is affine-harmonic-sum reducible for all finite vectors
$\mathbf r=(r_1,\ldots,r_d)$ and $\mathbf s=(s_1,\ldots,s_d)$.
\end{proposition}

\begin{proof}
First consider $q=0$.  Split every index $k_j$ according to its residue class
modulo $p$:
\[
  k_j=pm_j+\ell_j-p,
  \qquad 1\le \ell_j\le p,
  \qquad 1\le m_j\le n.
\]
For fixed residues $(\ell_1,\ldots,\ell_d)$, the strict inequalities among the
$k_j$'s become finitely many alternatives among the $m_j$'s: either
$m_j>m_{j+1}$, or $m_j=m_{j+1}$ with $\ell_j>\ell_{j+1}$.  Decomposing by the
equality pattern among adjacent $m_j$'s gives a finite sum of affine words.  An
equality block contributes one affine letter whose value is the product of all
letters in that block.  The case $q>0$ is obtained by choosing which initial
indices lie in the finite tail $\{pn+1,\ldots,pn+q\}$; only finitely many such
choices occur, and each gives an affine prefactor times the case $q=0$.
\end{proof}

\begin{example}\normalfont
At depth one the proposition gives the already displayed formula
\[
  H_{pn}^{(r)}(s)=
  \sum_{\ell=1}^{p}s^{\ell-p}\mathcal G_{L_{p,\ell-p}^{r}(s^p)}(n).
\]
For $p=2$ and depth two one obtains the explicit affine decomposition
\[
\begin{aligned}
H_{2n}^{(r_1,r_2)}(s_1,s_2)
&=\sum_{\ell_1,\ell_2=1}^{2}
  s_1^{\ell_1-2}s_2^{\ell_2-2}
  \mathcal G_{L_{2,\ell_1-2}^{r_1}(s_1^2),
              L_{2,\ell_2-2}^{r_2}(s_2^2)}(n) \\
&\quad+
  s_1^{0}s_2^{-1}
  \mathcal G_{L_{2,0}^{r_1}(s_1^2)\circ
              L_{2,-1}^{r_2}(s_2^2)}(n),
\end{aligned}
\]
because the only equal-index residue possibility with $\ell_1>\ell_2$ is
$(\ell_1,\ell_2)=(2,1)$.
\end{example}

\begin{example}\normalfont
For a shifted upper argument, the tail is finite.  For example,
\[
  H_{2n+1}^{(r)}(s)
  =s^{-1}\mathcal G_{L_{2,-1}^{r}(s^2)}(n)
  +\mathcal G_{L_{2,0}^{r}(s^2)}(n)
  +s\,\frac{(s^2)^n}{(2n+1)^r}.
\]
The last term is again an affine letter evaluated at $n$.
\end{example}

\begin{proposition}[Truncated Hurwitz--Lerch sums and their multiple versions]
\label{prop:truncated-hurwitz-lerch-affine-final}
Let
\[
  \zeta_N(r,a):=\sum_{m=0}^{N-1}\frac1{(m+a)^r},
  \qquad
  \Phi_N(z,r,a):=\sum_{m=0}^{N-1}\frac{z^m}{(m+a)^r}.
\]
Then the truncated Hurwitz zeta value $\zeta_n(r,a)$ and the truncated Lerch
Phi value $\Phi_n(z,r,a)$ are affine-harmonic-sum reducible.  Their strict
multiple versions, viewed as finite analogues of the cyclotomic multiple
Hurwitz zeta values studied in \citep{Xu2026CyclotomicMultipleHurwitz},
\[
  \zeta_n(\mathbf r;\mathbf a)
  :=\sum_{n>m_1>\cdots>m_d\ge0}
  \prod_{j=1}^{d}\frac1{(m_j+a_j)^{r_j}}
\]
and
\[
  \Phi_n(\mathbf z;\mathbf r;\mathbf a)
  :=\sum_{n>m_1>\cdots>m_d\ge0}
  \prod_{j=1}^{d}\frac{z_j^{m_j}}{(m_j+a_j)^{r_j}}
\]
are also affine-harmonic-sum reducible.
\end{proposition}

\begin{proof}
Put $k_j=m_j+1$.  Then $m_j+a_j=k_j+a_j-1$, so each denominator is an affine
form in $k_j$.  The Lerch colors give only the harmless constants $z_j^{-1}$,
since $z_j^{m_j}=z_j^{-1}z_j^{k_j}$.
\end{proof}

\begin{example}\normalfont
The depth-one reductions are
\[
  \zeta_n(r,a)=\mathcal G_{L_{1,a-1}^{r}(1)}(n),
  \qquad
  \Phi_n(z,r,a)=z^{-1}\mathcal G_{L_{1,a-1}^{r}(z)}(n).
\]
Thus, for instance,
\[
  \zeta_n\!\left(r,\frac12\right)
  =\mathcal G_{L_{1,-1/2}^{r}(1)}(n).
\]
\end{example}

\begin{example}\normalfont
At depth two,
\[
  \Phi_n((z_1,z_2);(r_1,r_2);(a_1,a_2))
  =z_1^{-1}z_2^{-1}
  \mathcal G_{L_{1,a_1-1}^{r_1}(z_1),
              L_{1,a_2-1}^{r_2}(z_2)}(n).
\]
The star version is obtained, as usual, by decomposing weak inequalities into
strict inequalities and equality blocks; equality blocks simply merge affine
letters by the operation $\circ$.
\end{example}

\begin{proposition}[Residue-class harmonic sums and finite level-$M$ values]
\label{prop:residue-class-level-M-affine-final}
The residue-class entities discussed in
Remark~\ref{rem:residue-class-level-M-values} are affine-harmonic-sum reducible.  Explicitly, this includes the odd harmonic numbers
$O_n^{(r)}(s)$, finite colored Hoffman multiple $t$-values
$t_n(\mathbf r;\mathbf s)$, finite Kaneko--Tsumura multiple $T$-values,
finite multiple mixed values, level-$M$ residue-class harmonic numbers
$R_{M,a;n}^{(r)}(s)$, finite multiple zeta values of level $M$, and the
corresponding colored level-$M$ variants.  More precisely, let $M\ge2$,
$1\le a\le M$, and use the normalization of
Remark~\ref{rem:residue-class-level-M-values}:
\[
  R_{M,a;n}^{(r)}(s):=
  M^r\sum_{j=1}^{n}\frac{s^{Mj+a-M}}{(Mj+a-M)^r}.
\]
Then $R_{M,a;n}^{(r)}(s)\in\mathscr G_n$.  Applying the same residue-class
restriction independently at each index shows that the finite Hoffman,
Kaneko--Tsumura, mixed, and level-$M$ multiple values listed above are affine-harmonic-sum reducible.
\end{proposition}

\begin{proof}
The displayed residue-class sum is exactly
\[
  R_{M,a;n}^{(r)}(s)=M^r s^{a-M}\mathcal G_{L_{M,a-M}^{r}(s^M)}(n).
\]
At higher depth, each index has its own affine form $Mj+a_j-M$ and color
$s_j^M$, with the constant $s_j^{a_j-M}$ pulled out.  Thus the finite level-$M$
multiple sums are affine words.  The parity-restricted objects in
Remark~\ref{rem:residue-class-level-M-values} are the special case
$M=2$, and finite multiple mixed values are obtained by assigning, at each
summation level, the prescribed even or odd residue class.  Equivalently,
these residue restrictions may also be obtained by finite Fourier filters, but
the direct affine parameterization already places them in $\mathscr G_n$.
\end{proof}

\begin{example}\normalfont
Odd harmonic sums are the level-two case:
\[
  \sum_{j=1}^{n}\frac{s^{2j-1}}{(2j-1)^r}
  =s^{-1}\mathcal G_{L_{2,-1}^{r}(s^2)}(n).
\]
With the shifted normalization used for odd harmonic numbers,
\[
  \sum_{j=1}^{n}\frac{s^{2j-1}}{(j-\frac12)^r}
  =2^r s^{-1}\mathcal G_{L_{2,-1}^{r}(s^2)}(n).
\]
\end{example}

\begin{example}\normalfont
The finite Hoffman-type depth-two odd sum is already an affine word:
\[
  \sum_{n\ge j_1>j_2\ge1}
  \frac{1}{(2j_1-1)^{r_1}(2j_2-1)^{r_2}}
  =
  \mathcal G_{L_{2,-1}^{r_1}(1),L_{2,-1}^{r_2}(1)}(n).
\]
Similarly, the finite Kaneko--Tsumura parity pattern is obtained by assigning
one residue class modulo $2$ to each level.
\end{example}

\begin{example}\normalfont
For arbitrary level $M$ and residues $a_1,a_2$, the colored depth-two
residue-class sum is
\[
\begin{aligned}
&\sum_{n\ge j_1>j_2\ge1}
 \frac{s_1^{Mj_1+a_1-M}s_2^{Mj_2+a_2-M}}
 {(Mj_1+a_1-M)^{r_1}(Mj_2+a_2-M)^{r_2}} \\
&\qquad=
 s_1^{a_1-M}s_2^{a_2-M}
 \mathcal G_{L_{M,a_1-M}^{r_1}(s_1^M),
             L_{M,a_2-M}^{r_2}(s_2^M)}(n).
\end{aligned}
\]
\end{example}

\begin{proposition}[Affine-telescopic hypergeometric terms]
\label{prop:affine-telescopic-hypergeometric-final}
Let $(u_n)_{n\ge1}$ be a hypergeometric term for which there exist
$t\ge0$, $C,z\in\mathbb C$, affine forms $\ell_\nu(n)=a_\nu n+b_\nu$, and exponents
$q_\nu\in\mathbb C$ for $1\le\nu\le t$ such that
\begin{equation}
\label{eq:affine-telescopic-term-final}
  u_n=Cz^n\prod_{\nu=1}^{t}\ell_\nu(n)^{q_\nu}
\end{equation}
for all positive integers in the summation range, after fixing branches of the
complex powers.  Equivalently, such a term has telescopic quotient
\begin{equation}
\label{eq:affine-telescopic-quotient-final}
  \frac{u_{n+1}}{u_n}
  =z\prod_{\nu=1}^{t}
  \left(\frac{\ell_\nu(n+1)}{\ell_\nu(n)}\right)^{q_\nu}.
\end{equation}
Then $u_n$ is affine-harmonic-sum reducible.  More generally, if
$\alpha_1,\ldots,\alpha_m$ are words in the basic colored alphabet and
$e_1,\ldots,e_m\in\mathbb Z_{\ge0}$, then
\[
  u_n\prod_{j=1}^{m}\mathcal H_{\alpha_j}(n)^{e_j}
\]
is affine-harmonic-sum reducible, and finite summation of such terms lies in
$\mathscr G_N$.
\end{proposition}

\begin{proof}
Let
\[
  L=((-q_1,\ldots,-q_t),z,((a_1,b_1),\ldots,(a_t,b_t))).
\]
Then $L(n)=z^n\prod_\nu(a_\nu n+b_\nu)^{q_\nu}$, so $u_n=CL(n)$.  The product
of the basic sums $\mathcal H_{\alpha_j}(n)$ is a finite linear combination of
single basic sums by the quasi-shuffle product, and each basic letter is an
affine letter of the special form $((r),s,((1,0)))$.  Hence the summand lies in
$\mathscr E_n^{\rm aff}$.  Finite summation is obtained by adjoining the outer
affine letter and splitting the summation region into strict and equality
cases.
\end{proof}

\begin{example}\normalfont
The single sum
\[
  \sum_{m=1}^{n}\frac{z^m}{(2m+1)^a(3m-2)^b}
\]
is exactly
\[
  \mathcal G_{((-a,-b),z,((2,1),(3,-2)))}(n).
\]
\end{example}

\begin{example}\normalfont
Let $L=((-q),z,((2,1)))$ and $B=L_{1,0}^{r}(s)$.  Then
\[
  \sum_{m=1}^{n}z^m(2m+1)^q H_m^{(r)}(s)
  =\mathcal G_{L,B}(n)+\mathcal G_{L\circ B}(n).
\]
Thus products of affine-telescopic factors with harmonic factors remain in the
affine alphabet after summation.
\end{example}

\begin{proposition}[Cyclotomic harmonic sums]
\label{prop:cyclotomic-harmonic-sums-affine-final}
Finite cyclotomic harmonic sums, as used in Ablinger's HarmonicSums framework
\citep{AblingerThesis2012,Ablinger2014}, are affine-harmonic-sum reducible.
More explicitly, let $a_j,b_j,r_j\in\mathbb Z$ with $a_j>0$, $r_j>0$, and
$a_j n+b_j\ne0$ on the relevant summation range.  For
$\varepsilon_j\in\{\pm1\}$, define
\[
  C_N=
  \sum_{N\ge n_1\ge\cdots\ge n_d\ge1}
  \prod_{j=1}^{d}
  \frac{\varepsilon_j^{n_j}}{(a_j n_j+b_j)^{r_j}}.
\]
Then the sequence $C_n$ is affine-harmonic-sum reducible.  The same conclusion
holds for the strict version with $n_1>\cdots>n_d$, and for cyclotomic
$S$-sums with additional colors $x_j^{n_j}$.
\end{proposition}

\begin{proof}
For the strict version, take affine letters
\[
  L_j=L_{a_j,b_j}^{r_j}(\varepsilon_j)
  \qquad(1\le j\le d).
\]
Then the strict cyclotomic sum is exactly $\mathcal G_{L_1,\ldots,L_d}(N)$.
For the weakly ordered sum, decompose the region
$N\ge n_1\ge\cdots\ge n_d\ge1$ according to the equality pattern among adjacent
indices.  Each block of equal indices contributes one affine letter obtained by
multiplying the letters in that block; this only multiplies the colors and
concatenates the affine factors.  Thus $C_N$ is a finite linear combination of
strict affine sums.  Extra colors $x_j^{n_j}$ are absorbed by replacing
$\varepsilon_j$ with $\varepsilon_j x_j$ in the corresponding affine letters.
\end{proof}

\begin{example}\normalfont
A strict cyclotomic depth-two sum is literally an affine word:
\[
  \sum_{n\ge k_1>k_2\ge1}
  \frac{(-1)^{k_1}}{(2k_1+1)^2(3k_2-1)^3}
  =
  \mathcal G_{L_{2,1}^{2}(-1),L_{3,-1}^{3}(1)}(n).
\]
\end{example}

\begin{example}\normalfont
For a weakly ordered depth-two sum,
\[
\begin{aligned}
\sum_{n\ge k_1\ge k_2\ge1}
\frac{x^{k_1}y^{k_2}}{(2k_1+1)^r(3k_2+2)^s}
&=
\mathcal G_{L_{2,1}^{r}(x),L_{3,2}^{s}(y)}(n)  \\
&\quad+
\mathcal G_{L_{2,1}^{r}(x)\circ L_{3,2}^{s}(y)}(n).
\end{aligned}
\]
The second term is the equality block $k_1=k_2$.
\end{example}

\begin{remark}
The propositions above are input statements for the affine closure theorem
proved later.  Once a summand is placed in $\mathscr E^{\rm aff}_n$, finite
convolution-type summation is performed inside $\mathscr G_N$ by adjoining the
outer affine letter and splitting the summation region into strict and equality
cases.
\end{remark}

The following table summarizes the main affine-harmonic-sum reducible families
introduced in this subsection.  In the table, all parameters are fixed, affine
forms are assumed nonzero on the relevant summation range, and all reductions
are to affine sums with truncation index $n$.

\begin{center}
\scriptsize
\setlength{\tabcolsep}{3pt}
\renewcommand{\arraystretch}{1.08}
\begin{adjustbox}{max width=\textwidth}
\begin{tabular}{@{}>{\raggedright\arraybackslash}p{0.25\textwidth}>{\raggedright\arraybackslash}p{0.48\textwidth}>{\raggedright\arraybackslash}p{0.21\textwidth}@{}}
\hline
\textbf{Family} & \textbf{Typical objects} & \textbf{Reason}\\
\hline
Affine letters and affine words
& $\sigma^n\prod_\nu(a_\nu n+b_\nu)^{-\rho_\nu}$,
  $\mathcal G_\Gamma(n)$
& definition of the affine alphabet\\
Affine upper harmonic and alternating harmonic numbers
& $H_{p_1n+p_2}^{(r)}(s)$, $A_{p_1n+p_2}^{(r)}(s)$, with $p_1,p_2\in\mathbb Q$
& Lerch continuation plus root-of-unity distribution and residue-class splitting\\
Affine upper hyperharmonic numbers
& $h_{p_1n+p_2}^{[m]}(r;s)$ for fixed $m$
& finite polynomial combination of affine-upper harmonic terms\\
Integer-affine upper multiple harmonic numbers
& $H_{pn+q}^{\mathbf r}(\mathbf s)$, with $p\in\mathbb Z_{>0}$ and $q\in\mathbb Z_{\ge0}$
& residue classes modulo $p$ and finite tail corrections\\
Truncated Hurwitz and Lerch sums
& $\zeta_n(r,a)$, $\Phi_n(z,r,a)$, and strict multiple versions
& replace $m$ by $k-1$, giving affine denominators $k+a-1$\\
Residue-class, parity-restricted, and level-$M$ finite sums
& $O_n^{(r)}(s)$, finite colored Hoffman multiple $t$-values, finite Kaneko--Tsumura multiple $T$-values, finite multiple mixed values, $R_{M,a;n}^{(r)}(s)$, finite level-$M$ multiple zeta values, colored level-$M$ variants
& Remark~\ref{rem:residue-class-level-M-values}; each residue class gives an affine form $Mj+a-M$\\
Affine-telescopic hypergeometric terms
& $Cz^n\prod_\nu(a_\nu n+b_\nu)^{q_\nu}$ and products with harmonic factors
& the term is an affine letter; harmonic factors are basic affine words\\
Cyclotomic harmonic sums and cyclotomic $S$-sums
& $\displaystyle\sum_{n\ge k_1\ge\cdots\ge k_d\ge1}\prod_j\frac{x_j^{k_j}}{(a_jk_j+b_j)^{r_j}}$
& weak inequalities split into strict affine words and equality blocks\\
\hline
\end{tabular}
\end{adjustbox}
\end{center}

\subsection{Polynomial-base harmonic-number alphabet}
\label{subsec:polynomial-harmonic-sum-alphabet}

The next enlargement replaces affine linear factors by arbitrary polynomial
factors.  This is useful for finite polynomial zeta sums, polynomial-base
polylogarithms, one-dimensional Epstein--Hurwitz type sums, and finite Mathieu
series.  The point is not that these objects usually reduce to the ordinary or
affine alphabet; rather, after polynomial factors are admitted as letters, the
same stuffle and finite-summation mechanism applies without change.

A \emph{polynomial letter} is a triple
\[
  L=(\boldsymbol\rho,\sigma,\mathbf P),
  \qquad
  \boldsymbol\rho=(\rho_1,\ldots,\rho_t),
  \qquad
  \mathbf P=(P_1,\ldots,P_t),
\]
where \(\rho_\nu,\sigma\in\mathbb C\), \(P_\nu\in\mathbb C[x]\), and the
values \(P_\nu(n)\) are nonzero on the positive integers under consideration.
After branches for the powers have been fixed, the letter has value
\[
  L(n)=\sigma^n\prod_{\nu=1}^{t}P_\nu(n)^{-\rho_\nu}.
\]
For a word \(\Omega=(L_1,\ldots,L_d)\), define
\[
  \mathcal P_\Omega(N)
  :=
  \sum_{N\ge n_1>\cdots>n_d\ge1}
  \prod_{j=1}^{d}L_j(n_j),
  \qquad
  \mathcal P_\emptyset(N)=1.
\]
The identity polynomial \(\idpoly(x)=x\) embeds the basic colored harmonic
alphabet, while affine linear polynomials embed the affine alphabet.
Multiplication of letters is given by concatenating their lists of powers and
polynomials and multiplying their colors; hence the associated polynomial
harmonic sums satisfy the same quasi-shuffle product.

For summand-level reductions define
\begin{equation}
\label{eq:polynomial-summand-space-final}
  \mathscr E^{\rm pol}_n
  :=
  \operatorname{span}_{\mathbb C}
  \{L(n)\mathcal P_\Omega(n):L\text{ is a polynomial letter or }1,
  \;\Omega\text{ a word in polynomial letters}\}.
\end{equation}

\begin{definition}[Polynomial-harmonic-sum reducibility]
\label{def:polynomial-harmonic-sum-reducibility-final}
A sequence $f(n)$ is called \emph{polynomial-harmonic-sum reducible} if
$f(n)\in\mathscr E^{\rm pol}_n$.  Equivalently, it has a finite representation
\begin{equation}
\label{eq:polynomial-reducible-form-final}
  f(n)=\sum_{\nu=1}^{M}c_\nu\sigma_\nu^n
  \prod_{\mu=1}^{m_\nu}P_{\nu\mu}(n)^{-q_{\nu\mu}}
  \mathcal P_{\Omega_\nu}(n),
\end{equation}
where $c_\nu,\sigma_\nu,q_{\nu\mu}\in\mathbb C$, the polynomial factors are
nonzero on positive integers in the relevant range, and $\Omega_\nu$ is a word
in polynomial letters.  Thus the polynomial analogue of the basic factor
$z^n n^q\mathcal H_\alpha(n)$ is a finite product of polynomial powers times a
polynomial-base word sum.
\end{definition}

The general finite-summation step for polynomial letters is proved later in Theorem~\ref{thm:nested-polynomial-factor-closure}; the present subsection records the main finite families that belong to this summand space.

\begin{proposition}[Finite polynomial zeta sums and polynomial-base polylogarithms]
\label{prop:finite-polynomial-zeta-polylog-reduction}
Let \(P_1,\ldots,P_d\in\mathbb C[x]\), let \(r_1,\ldots,r_d\in\mathbb C\), and
let \(z_1,\ldots,z_d\in\mathbb C\).  The finite strict sums
\[
  \zeta_N(P_1^{r_1},\ldots,P_d^{r_d})
  :=
  \sum_{N\ge n_1>\cdots>n_d\ge1}
  \prod_{j=1}^{d}P_j(n_j)^{-r_j}
\]
and
\[
  \Li_N^{\mathbf P,\mathbf r}(\mathbf z)
  :=
  \sum_{N\ge n_1>\cdots>n_d\ge1}
  \prod_{j=1}^{d}z_j^{n_j}P_j(n_j)^{-r_j}
\]
are polynomial-harmonic-sum reducible.  In the convergent limit they give
polynomial zeta values and polynomial-base polylogarithmic values, a finite
counterpart of zeta functions associated with polynomials
\citep{EieChen1999,Dabrowski2000} and of the usual polylogarithmic construction
\citep{Lewin1981}.
\end{proposition}

\begin{proof}
Take the polynomial letters \(L_j=((r_j),1,(P_j))\) in the first case and
\(L_j=((r_j),z_j,(P_j))\) in the second.  The two displayed sums are precisely
\(\mathcal P_{L_1,\ldots,L_d}(N)\).
\end{proof}

\begin{example}\normalfont
For a quadratic polynomial \(P(x)=x^2+x+1\),
\[
  \sum_{n=1}^{N}\frac{z^n}{(n^2+n+1)^r}
  =
  \mathcal P_{((r),z,(x^2+x+1))}(N).
\]
At depth two,
\[
  \sum_{N\ge m>n\ge1}\frac{u^m v^n}{(m^2+1)^a(n^3+n+1)^b}
\]
is \(\mathcal P_{L_1,L_2}(N)\) with
\(L_1=((a),u,(x^2+1))\) and \(L_2=((b),v,(x^3+x+1))\).
\end{example}

\begin{proposition}[One-dimensional truncated Epstein--Hurwitz type sums]
\label{prop:truncated-epstein-hurwitz-polynomial-reduction}
Let \(A,B,C,z\in\mathbb C\), and assume that the quadratic polynomial
\(Q(x)=Ax^2+Bx+C\) does not vanish on the relevant positive integers.  Then
\[
  \sum_{n=1}^{N}\frac{z^n}{Q(n)^r}
  \quad\text{and, more generally,}\quad
  \sum_{N\ge n_1>\cdots>n_d\ge1}
  \prod_{j=1}^{d}\frac{z_j^{n_j}}{Q_j(n_j)^{r_j}}
\]
are polynomial-harmonic-sum reducible.  In particular, finite one-dimensional
Epstein--Hurwitz sums such as
\[
  \sum_{n=0}^{N-1}\frac{1}{\bigl((n+a)^2+b^2\bigr)^r}
\]
are included after the harmless index shift.  This is the finite polynomial-letter
analogue of the one-dimensional Epstein--Hurwitz zeta functions used in zeta
regularization and spectral applications \citep{Elizalde1994,ElizaldeEtAl1994}.
\end{proposition}

\begin{proof}
The displayed summands are products of colors and powers of polynomial factors.
After replacing \(n=0,\ldots,N-1\) by \(m=n+1\), the denominator
\((n+a)^2+b^2\) becomes the polynomial \((m+a-1)^2+b^2\) in \(m\).  The result is
therefore a polynomial-letter sum.
\end{proof}

\begin{example}\normalfont
The finite shifted quadratic tail
\[
  \sum_{n=0}^{N-1}\frac{(-1)^n}{\bigl((n+a)^2+b^2\bigr)^r}
\]
becomes
\[
  -\mathcal P_{((r),-1,((x+a-1)^2+b^2))}(N),
\]
up to the constant sign coming from \((-1)^{m-1}\).
\end{example}

\begin{proposition}[Finite Mathieu and generalized Mathieu series]
\label{prop:finite-mathieu-polynomial-reduction}
For fixed parameters for which the denominators do not vanish, the finite
Mathieu-type sums
\[
  \sum_{n=1}^{N}\frac{2n}{(n^2+a^2)^{r+1}},
  \qquad
  \sum_{n=1}^{N}\frac{2n z^n}{(n^2+a^2)^{r+1}},
\]
and their finite products with polynomial harmonic factors are polynomial-harmonic-sum reducible.  The same conclusion holds for generalized finite
Mathieu-type series in which the numerator and denominator are finite products
of fixed polynomial powers.  These finite sums mirror the classical and
generalized Mathieu series \citep{Mathieu1890,PoganySrivastavaTomovski2006,PoganyTomovski2006}.
\end{proposition}

\begin{proof}
The factor \(2n\) is a scalar times the polynomial power \(\idpoly(n)^1\), and
\((n^2+a^2)^{-(r+1)}\) is a polynomial power.  Thus the whole level factor is a
polynomial letter, and Theorem~\ref{thm:nested-polynomial-factor-closure}
handles finite summation and products with polynomial harmonic factors.
\end{proof}

\begin{example}\normalfont
The colored finite Mathieu sum is represented as
\[
  \sum_{n=1}^{N}\frac{2n z^n}{(n^2+a^2)^{r+1}}
  =
  2\,\mathcal P_{((-1,r+1),z,(x,x^2+a^2))}(N),
\]
where the first exponent \(-1\) records the numerator factor \(x^1\).
\end{example}

\begin{center}
\scriptsize
\renewcommand{\arraystretch}{1.35}
\setlength{\tabcolsep}{4pt}
\begin{adjustbox}{max width=\textwidth}
\begin{tabular}{@{}>{\raggedright\arraybackslash}p{0.28\textwidth}>{\raggedright\arraybackslash}p{0.43\textwidth}>{\raggedright\arraybackslash}p{0.23\textwidth}@{}}
\toprule
\textbf{Family} & \textbf{Typical finite object} & \textbf{Polynomial-letter reason} \\
\midrule
Polynomial zeta sums
& $\displaystyle \sum_{n=1}^{N}P(n)^{-r}$ and strict multiple versions
& Single polynomial letter; multiple version is a word. \\
Polynomial-base polylogarithms
& $\displaystyle \sum_{n=1}^{N}z^nP(n)^{-r}$ and colored strict multiple versions
& Color plus polynomial denominator. \\
One-dimensional truncated Epstein--Hurwitz sums
& $\displaystyle \sum_{n=0}^{N-1}\bigl((n+a)^2+b^2\bigr)^{-r}$
& Shift gives a quadratic polynomial in the new index. \\
Finite Mathieu-type series
& $\displaystyle \sum_{n=1}^{N}2n\,(n^2+a^2)^{-r-1}$, with optional color $z^n$
& Numerator and denominator are polynomial powers. \\
Products with polynomial harmonic factors
& $\displaystyle \sum_{n=1}^{N}z^n\prod_\nu P_\nu(n)^{q_\nu}\prod_j\mathcal P_{\Omega_j}(n)^{e_j}$
& Quasi-shuffle product plus one-step polynomial summation. \\
\bottomrule
\end{tabular}
\end{adjustbox}
\end{center}

\section{Finite convolution in the basic colored alphabet}

We now prove the basic finite convolution theorem for summands of the form
\[
  z^n n^q\,\mathcal H_\alpha(n),
\]
and then record limited affine and arithmetic-progression cases that still remain within
this basic colored harmonic-number alphabet. 

\subsection{Basic finite convolution}

\begin{theorem}[Closure under finite Euler-type summation]
\label{thm:finite-euler-sum}
Let a \emph{letter} be a pair $(r,s)\in \mathbb{C}^2$, and let a
\emph{word} be a finite sequence of letters. For a word
\[
\alpha=((r_1,s_1),\ldots,(r_d,s_d)),
\]
define
\begin{equation}
\label{eq:generalized-harmonic-sum}
\mathcal H_{\alpha}(N)
:=
\sum_{N\ge n_1>\cdots>n_d\ge 1}
\prod_{i=1}^{d}\frac{s_i^{n_i}}{n_i^{r_i}},
\end{equation}
with the convention
\[
\mathcal H_{\emptyset}(N)=1.
\]
Here powers of positive integers are interpreted by
\[
n^r:=\exp(r\log n),
\]
where $\log n$ denotes the real logarithm.

For each $N$, let
\begin{equation}
\label{eq:AN-space}
\mathscr H_N
:=
\operatorname{span}_{\mathbb C}
\{\mathcal H_{\alpha}(N):\alpha \text{ is a word}\}.
\end{equation}
Similarly, for each $n$, let
\begin{equation}
\label{eq:Bn-space}
\mathscr E_n
:=
\operatorname{span}_{\mathbb C}
\{z^n n^q\mathcal H_{\alpha}(n):
z,q\in\mathbb C,\ \alpha \text{ is a word}\}.
\end{equation}

Suppose
\[
F_j(n)\in \mathscr E_n,
\qquad j=1,\ldots,m.
\]
Equivalently, suppose
\begin{equation}
\label{eq:Hj-expansion}
F_j(n)
=
\sum_{\ell=1}^{M_j}
c_{j,\ell}
z_{j,\ell}^{\,n}
n^{q_{j,\ell}}
\mathcal H_{\alpha_{j,\ell}}(n),
\end{equation}
where
\[
c_{j,\ell},z_{j,\ell},q_{j,\ell}\in\mathbb C,
\]
and each $\alpha_{j,\ell}$ is a word.

Then, for every
\[
e_1,\ldots,e_m\in\mathbb Z_{\ge 0},
\]
the finite sum
\begin{equation}
\label{eq:finite-euler-sum}
S(N)
:=
\sum_{n=1}^{N}
\prod_{j=1}^{m}F_j(n)^{e_j}
\end{equation}
belongs to $\mathscr H_N$. Equivalently, $S(N)$ is a finite
$\mathbb C$-linear combination of generalized finite harmonic numbers with
upper limit $N$.
\end{theorem}

\begin{lemma}[Elementary summation]
\label{lem:elementary-summation}
Let $z,q\in\mathbb C$, let $\lambda=(-q,z)$, and define the merge of two
letters by
\begin{equation}
\label{eq:letter-merge}
(r,s)\circ(r',s')=(r+r',ss').
\end{equation}
Then, for every word $\alpha$,
\begin{equation}
\label{eq:elementary-summation-step}
\sum_{n=1}^{N} z^n n^q \mathcal H_{\alpha}(n)
=
\begin{cases}
\mathcal H_{\lambda}(N), & \alpha=\emptyset,\\[1mm]
\mathcal H_{\lambda,\alpha}(N)+\mathcal H_{\lambda\circ a_1,\alpha'}(N),
& \alpha=(a_1,\alpha').
\end{cases}
\end{equation}
In particular, summation maps $\mathscr E_n$ into $\mathscr H_N$.
\end{lemma}

\begin{proof}
The case $\alpha=\emptyset$ is the definition of $\mathcal H_{\lambda}(N)$.
If $\alpha=(a_1,\alpha')$ is nonempty, expand $\mathcal H_\alpha(n)$ and
split the region $N\ge n\ge n_1>\cdots$ into the disjoint cases $n>n_1$ and
$n=n_1$.  The first part gives $\mathcal H_{\lambda,\alpha}(N)$; in the
second part the letters $\lambda$ and $a_1$ merge, giving
$\mathcal H_{\lambda\circ a_1,\alpha'}(N)$.
\end{proof}

\begin{proof}[Proof of Theorem~\ref{thm:finite-euler-sum}]
The quasi-shuffle product shows that products of sums
$\mathcal H_{\alpha}(n)$ with the same upper limit are finite
$\mathbb C$-linear combinations of such sums.  Since
\[
z_1^n n^{q_1}z_2^n n^{q_2}=(z_1z_2)^n n^{q_1+q_2},
\]
the space $\mathscr E_n$ is closed under multiplication.  Therefore
$\prod_{j=1}^{m}F_j(n)^{e_j}\in\mathscr E_n$.  Applying
Lemma~\ref{lem:elementary-summation} termwise gives
\[
S(N)=\sum_{n=1}^{N}\prod_{j=1}^{m}F_j(n)^{e_j}\in\mathscr H_N.
\]
\end{proof}

\subsubsection{Consequences and examples}
\label{subsubsec:basic-finite-consequences}

The theorem immediately covers finite Euler-type sums whose summand is built
from an elementary factor $z^n n^q$ and any finite product of generalized harmonic numbers with the same upper limit $n$.  Products are first reduced by the
quasi-shuffle product, and the remaining outer summation is absorbed by adding
one leading letter, with a possible merge at the first inner index.  In the
first rows of the table we display smaller generated spans.  Here
$\mathscr H_k[\mathcal B]$ denotes the span of all $\mathcal H_\beta(k)$ whose
letters lie in the merge-closed alphabet generated by the finite set
$\mathcal B$.

\begin{center}
\scriptsize
\renewcommand{\arraystretch}{1.08}
\setlength{\tabcolsep}{4pt}
\begin{adjustbox}{max width=\textwidth}
\begin{tabular}{@{}>{\raggedright\arraybackslash}p{0.48\textwidth}>{\centering\arraybackslash}p{0.28\textwidth}>{\raggedright\arraybackslash}p{0.16\textwidth}@{}}
\toprule
\textbf{Sum} & \textbf{Span} & \textbf{Comment} \\
\midrule
$\displaystyle \sum_{n=1}^{k}2^n n^3 H_n^{(2)}(-1)$
& $\mathscr H_k[\{(-3,2),(2,-1)\}]$
& Concrete depth-one case. \\

$\displaystyle \sum_{n=1}^{k}\frac{(-1)^n}{n^2}H_n^{(1)}(1)H_n^{(3)}(i)$
& $\mathscr H_k[\{(2,-1),(1,1),(3,i)\}]$
& Product reduced by quasi-shuffle. \\

$\displaystyle \sum_{n=1}^{k}3^n n^4\zeta(2,n+1)H_n^{(1)}(-1)$
& $\mathscr H_k[\{(-4,3),(2,1),(1,-1)\}]$
& Hurwitz tail gives a finite colored correction. \\

$\displaystyle \sum_{n=1}^{k}i^n n^2\Phi\!\left(\frac13,3,n+1\right)H_n^{(1,2)}(-1,i)$
& $\mathscr H_k[\{(-2,i),(3,\frac13),(1,-1),(2,i)\}]$
& Lerch tail plus depth-two colored factor. \\

$\displaystyle \sum_{n=1}^{k} z^n n^q H_n^{\mathbf r}(\mathbf s)\,H_n^{\star,\mathbf u}(\mathbf w)$
& $\operatorname{span}_{\mathbb C}\{\mathcal H_\beta(k)\}_\beta$
& Symbolic strict and star case. \\

$\displaystyle \sum_{n=1}^{k} z^n n^q F_n\,L_n\,H_n^{\star,(r_1,r_2,r_3)}(s_1,s_2,s_3)$
& $\operatorname{span}_{\mathbb C}\{\mathcal H_\beta(k)\}_\beta$
& Fixed recurrence factors become exponential-polynomial. \\

$\displaystyle \sum_{n=1}^{k} z^n n^q\chi(n)\left\lfloor\frac{3n+2}{5}\right\rfloor^m H_n^{(r)}(s)H_n^{(u)}(w)$
& $\operatorname{span}_{\mathbb C}\{\mathcal H_\beta(k)\}_\beta$
& Periodic and polynomial parts. \\

$\displaystyle \sum_{n=1}^{k} z^n n^q h_n^{[m]}(r;s)\,H_n^{\star,\mathbf u}(\mathbf w)\,\mathbf1_{n\equiv a\,({\rm mod}\,M)}$
& $\operatorname{span}_{\mathbb C}\{\mathcal H_\beta(k)\}_\beta$
& Hyperharmonic factor with a residue-class filter. \\
\bottomrule
\end{tabular}
\end{adjustbox}
\end{center}

\subsection{The aligned affine case}

We now record a useful affine extension of the closure principle proved in
Theorem~\ref{thm:finite-euler-sum}.  The case considered here is the
\emph{aligned affine case}: the same affine expression occurs both in the
outer power and in the upper limit of the generalized harmonic numbers.

Throughout this section we assume
\[
a\in\mathbb Z_{>0},\qquad b\in\mathbb Z,\qquad a+b\ge 1.
\]
Then
\[
an+b\in\mathbb Z_{\ge 1}
\qquad (n\ge 1),
\]
so that $\mathcal H_{\alpha}(an+b)$ is a genuine finite generalized harmonic number.  For complex $q$, the power $(an+b)^q$ is interpreted using the real
logarithm of the positive integer $an+b$.

\begin{theorem}[Closure in the aligned affine case]
\label{thm:aligned-affine-closure}
Let
\[
a\in\mathbb Z_{>0},\qquad b\in\mathbb Z,\qquad a+b\ge 1.
\]
Let $z,q\in\mathbb C$, let $\alpha_1,\ldots,\alpha_m$ be words in the
generalized harmonic number alphabet, and let
\[
e_1,\ldots,e_m\in\mathbb Z_{\ge 0}.
\]
Then
\begin{equation}
\label{eq:aligned-affine-sum}
S(k)
:=
\sum_{n=1}^{k}
z^n(an+b)^q
\prod_{j=1}^{m}\mathcal H_{\alpha_j}(an+b)^{e_j}
\end{equation}
belongs to
\begin{equation}
\label{eq:aligned-affine-span}
\operatorname{span}_{\mathbb C}
\left\{
\mathcal H_{\beta}(ak+b)
\right\}_{\beta}.
\end{equation}
Equivalently, every aligned affine sum of the form
\eqref{eq:aligned-affine-sum} reduces to a finite linear combination of
generalized harmonic numbers with upper limit $ak+b$.
\end{theorem}

\begin{proof}
By the quasi-shuffle product,
\[
\prod_{j=1}^{m}\mathcal H_{\alpha_j}(an+b)^{e_j}
\]
is a finite $\mathbb C$-linear combination of single generalized harmonic numbers $\mathcal H_{\gamma}(an+b)$.  Hence it is enough to consider
\[
\sum_{n=1}^{k}z^n(an+b)^q\mathcal H_{\gamma}(an+b).
\]
Put
\[
M=an+b.
\]
Then $M$ runs through the arithmetic progression
\[
a+b,\ 2a+b,\ \ldots,\ ak+b,
\]
or equivalently through the integers $M\le ak+b$ satisfying
\[
M\equiv b\pmod a.
\]
If $z=0$, the sum is zero.  Otherwise choose $\xi\in\mathbb C$ such that
$\xi^a=z$.  On the progression $M=an+b$ we have
\[
z^n=z^{(M-b)/a}=\xi^{M-b}=\xi^{-b}\xi^M.
\]
Thus the sum becomes
\[
\xi^{-b}
\sum_{\substack{a+b\le M\le ak+b\\ M\equiv b\pmod a}}
\xi^M M^q\mathcal H_{\gamma}(M).
\]
The congruence condition is imposed by the root-of-unity filter
\[
\mathbf 1_{M\equiv b\pmod a}
=
\frac1a\sum_{\ell=0}^{a-1}\omega_a^{\ell(M-b)},
\qquad
\omega_a=e^{2\pi i/a}.
\]
Therefore the preceding expression is a finite linear combination of sums of
the form
\[
\sum_{M=1}^{ak+b}
(\xi\omega_a^\ell)^M M^q\mathcal H_{\gamma}(M),
\]
together with finitely many lower-limit correction terms independent of $k$.
By Theorem~\ref{thm:finite-euler-sum}, each such sum belongs to
$\mathscr H_{ak+b}$, that is, to the span of generalized harmonic numbers with
upper limit $ak+b$.  The correction terms are constants and are included by
the empty word.  Hence $S(k)$ belongs to the same span.
\end{proof}

\subsubsection{Consequences and examples}

The table records aligned affine examples, where the same affine expression
appears in the power and in the harmonic upper limit.  Each entry is reduced to
harmonic sums at the corresponding affine terminal point.

\begin{center}
\scriptsize
\renewcommand{\arraystretch}{1.08}
\setlength{\tabcolsep}{4pt}
\begin{adjustbox}{max width=\textwidth}
\begin{tabular}{@{}>{\raggedright\arraybackslash}p{0.52\textwidth}>{\centering\arraybackslash}p{0.22\textwidth}>{\raggedright\arraybackslash}p{0.18\textwidth}@{}}
\toprule
\textbf{Sum} & \textbf{Span} & \textbf{Condition} \\
\midrule
$\displaystyle \sum_{n=1}^{k} z^n(2n+1)^q\mathcal H_\alpha(2n+1)$
& $\operatorname{span}_{\mathbb C}\{\mathcal H_\beta(2k+1)\}_\beta$
& None. \\

$\displaystyle \sum_{n=1}^{k} z^n(3n-2)^q H_{3n-2}^{(r_1,r_2)}(s_1,s_2)$
& $\operatorname{span}_{\mathbb C}\{\mathcal H_\beta(3k-2)\}_\beta$
& $r_1,r_2\in\mathbb C$. \\

$\displaystyle \sum_{n=1}^{k} z^n(2n+1)^q\zeta(u,2n+2)\,H_{2n+1}^{(r)}(s)\,A_{2n+1}^{(v)}(w)$
& $\operatorname{span}_{\mathbb C}\{\mathcal H_\beta(2k+1)\}_\beta$
& $u\ne1$. \\

$\displaystyle \sum_{n=1}^{k} z^n(3n-1)^q\Phi(\xi,u,3n)\,H_{3n-1}^{\star,(r_1,r_2)}(s_1,s_2)$
& $\operatorname{span}_{\mathbb C}\{\mathcal H_\beta(3k-1)\}_\beta$
& $\xi\ne0,1$. \\

$\displaystyle \sum_{n=1}^{k} z^n(4n-3)^q\bigl(H_{4n-3}^{(r)}(s)\bigr)^2A_{4n-3}^{(u)}(w)\,h_{4n-3}^{[m]}(v;\xi)$
& $\operatorname{span}_{\mathbb C}\{\mathcal H_\beta(4k-3)\}_\beta$
& $m\in\mathbb N_0$. \\

$\displaystyle \sum_{n=1}^{k} z^n(5n-4)^q\chi(5n-4)\,H_{5n-4}^{\mathbf r}(\mathbf s)\,H_{5n-4}^{\star,\mathbf u}(\mathbf w)$
& $\operatorname{span}_{\mathbb C}\{\mathcal H_\beta(5k-4)\}_\beta$
& Fixed modulus for $\chi$. \\

$\displaystyle \sum_{n=1}^{k} z^n(2n)^q F_{2n}\,H_{2n}^{(r_1,r_2,r_3)}(s_1,s_2,s_3)$
& $\operatorname{span}_{\mathbb C}\{\mathcal H_\beta(2k)\}_\beta$
& $F_m$ has a fixed recurrence. \\

$\displaystyle \sum_{n=2}^{k} z^n(n-1)^q H_{n-1}^{(r)}(s)\,\psi^{(m)}(n)$
& $\operatorname{span}_{\mathbb C}\{\mathcal H_\beta(k-1)\}_\beta$
& Endpoint avoids $0$. \\
\bottomrule
\end{tabular}
\end{adjustbox}
\end{center}

\subsection{Arithmetic-progression shifted powers}

Let $c$ be a positive integer.  This section treats finite sums in which the
outer power is attached to one member of the arithmetic progression
$cn-1,cn,cn+1$, while the generalized harmonic numbers are evaluated at a
neighboring member of the same progression.  The resulting sums are controlled
by the standard root-of-unity filter for a fixed residue class modulo $c$.

Put
\[
  \omega_c=\exp(2\pi i/c).
\]
For an integer residue $a$, define
\begin{equation}
\label{eq:residue-filter}
\delta_{c,a}(M)
=
\begin{cases}
1, & M\equiv a \pmod c,\\
0, & M\not\equiv a \pmod c,
\end{cases}
=
\frac1c\sum_{\ell=0}^{c-1}\omega_c^{\ell(M-a)}.
\end{equation}
Thus a sum over the residue class $M\equiv a\pmod c$ is a finite linear
combination of unrestricted sums with modified colors.  If $\xi^c=z$, then
for $M=cn+a$ one has
\[
  z^n=\xi^{cn}=\xi^{M-a}.
\]
The restriction to an arithmetic progression therefore only changes the colors
appearing in the resulting harmonic sums.

We shall also use the elementary boundary relation
\begin{equation}
\label{eq:affine-one-step-up}
\mathcal H_{(r,s),\alpha'}(M+1)
=
\mathcal H_{(r,s),\alpha'}(M)
+
\frac{s^{M+1}}{(M+1)^r}\mathcal H_{\alpha'}(M),
\end{equation}
valid for every nonempty word $((r,s),\alpha')$.
For the family involving $cn-1$, set
\[
\nu_c=
\begin{cases}
1,& c\ge 2,\\
2,& c=1.
\end{cases}
\]
This avoids the endpoint $cn-1=0$ when $c=1$.

\begin{theorem}[Arithmetic-progression lower shifts]
\label{thm:arithmetic-progression-lower-shifts}
Let $c\in\mathbb Z_{>0}$, let $z,q\in\mathbb C$, let
$\alpha_1,\ldots,\alpha_m$ be words in the generalized harmonic number alphabet,
and let $e_1,\ldots,e_m\in\mathbb Z_{\ge0}$.  Then
\begin{equation}
\label{eq:cn-plus-one-lower-shift}
\sum_{n=1}^{k}
 z^n(cn+1)^q
 \prod_{j=1}^{m}\mathcal H_{\alpha_j}(cn)^{e_j}
\in
\mathscr H_{ck+1}.
\end{equation}
Similarly,
\begin{equation}
\label{eq:cn-lower-shift}
\sum_{n=\nu_c}^{k}
 z^n(cn)^q
 \prod_{j=1}^{m}\mathcal H_{\alpha_j}(cn-1)^{e_j}
\in
\mathscr H_{ck}.
\end{equation}
\end{theorem}

\begin{proof}
By the quasi-shuffle product, products of generalized harmonic numbers with the
same upper limit are finite linear combinations of single generalized harmonic numbers with that upper limit.  It is therefore enough to treat one factor
$\mathcal H_\alpha$.

For \eqref{eq:cn-plus-one-lower-shift}, put $M=cn+1$.  Then
$M\equiv1\pmod c$, $M\le ck+1$, and, after choosing $\xi$ with $\xi^c=z$,
\[
  z^n(cn+1)^q\mathcal H_\alpha(cn)
  =
  \xi^{-1}\xi^M M^q\mathcal H_\alpha(M-1).
\]
Using the filter \eqref{eq:residue-filter}, the sum is a finite linear
combination, up to lower-end constants independent of $k$, of sums of the form
\[
  \sum_{M=1}^{ck+1}\eta^M M^q\mathcal H_\alpha(M-1),
  \qquad \eta\in\mathbb C.
\]
If $\lambda=(-q,\eta)$, this is $\mathcal H_{\lambda,\alpha}(ck+1)$ when
$\alpha$ is nonempty, and $\mathcal H_{\lambda}(ck+1)$ when $\alpha$ is empty.
Thus it belongs to $\mathscr H_{ck+1}$.

For \eqref{eq:cn-lower-shift}, put $M=cn$.  Then $M\equiv0\pmod c$ and
\[
  z^n(cn)^q\mathcal H_\alpha(cn-1)
  =
  \xi^M M^q\mathcal H_\alpha(M-1).
\]
The same residue-filter and leading-letter argument give an element of
$\mathscr H_{ck}$, with only lower-end constants added when the range starts at
$\nu_c$.
\end{proof}

The remaining two shifted families contain boundary terms produced by
\eqref{eq:affine-one-step-up}.  These boundary terms involve mixed rational
factors such as $M^q(M+1)^{-r}$.  Hence the closure statement is most naturally
formulated for integer exponents, unless the alphabet is enlarged further.

\begin{theorem}[Arithmetic-progression upper shifts in the integer-weight case]
\label{thm:arithmetic-progression-upper-shifts-integer}
Let $c\in\mathbb Z_{>0}$.  Assume that the first components of all letters
occurring in $\alpha_1,\ldots,\alpha_m$ are integers, and let $q\in\mathbb Z$.
Let $z\in\mathbb C$ and $e_1,\ldots,e_m\in\mathbb Z_{\ge0}$.  Then
\begin{equation}
\label{eq:cn-plus-one-upper-shift}
\sum_{n=1}^{k}
 z^n(cn)^q
 \prod_{j=1}^{m}\mathcal H_{\alpha_j}(cn+1)^{e_j}
\in
\mathscr H_{ck}+\mathscr H_{ck+1}.
\end{equation}
Similarly,
\begin{equation}
\label{eq:cn-minus-one-upper-shift}
\sum_{n=\nu_c}^{k}
 z^n(cn-1)^q
 \prod_{j=1}^{m}\mathcal H_{\alpha_j}(cn)^{e_j}
\in
\mathscr H_{ck-1}+\mathscr H_{ck}.
\end{equation}
\end{theorem}

\begin{proof}
We prove \eqref{eq:cn-plus-one-upper-shift}.  Put $M=cn$.  Applying
\eqref{eq:affine-one-step-up} to each shifted factor and expanding the product,
we obtain a finite linear combination of terms of the form
\[
  \eta^M R(M)
  \prod_{\ell}\mathcal H_{\beta_\ell}(M)^{f_\ell},
\]
where $\eta\in\mathbb C$, $f_\ell\in\mathbb Z_{\ge0}$, and $R(M)$ is a rational
function whose possible poles occur only at $M=0$ and $M=-1$.  The rationality
of $R(M)$ uses the hypotheses that $q$ and the first components of the letters
are integers.  After another quasi-shuffle expansion, it remains to sum terms
of the form
\[
  \eta^M R(M)\mathcal H_\beta(M)
\]
over the residue class $M\equiv0\pmod c$, with $M\le ck$.
By partial fractions, $R(M)$ is a finite linear combination of integral powers
of $M$ and of $M+1$.  The terms involving powers of $M$ are handled by the
finite Euler-type closure theorem together with the residue filter
\eqref{eq:residue-filter}; they lie in $\mathscr H_{ck}$.  For the terms
involving powers of $M+1$, set $L=M+1$.  Then $L\equiv1\pmod c$,
$L\le ck+1$, and $\mathcal H_\beta(M)=\mathcal H_\beta(L-1)$.  The
residue-filter and leading-letter argument used in
Theorem~\ref{thm:arithmetic-progression-lower-shifts} place these terms in
$\mathscr H_{ck+1}$.  Hence the whole sum lies in
$\mathscr H_{ck}+\mathscr H_{ck+1}$.

For \eqref{eq:cn-minus-one-upper-shift}, put $M=cn-1$.  Then
$M\equiv -1\pmod c$ and $\mathcal H_{\alpha_j}(cn)=\mathcal H_{\alpha_j}(M+1)$.
Using \eqref{eq:affine-one-step-up}, quasi-shuffle expansion, and partial
fractions gives powers of $M$ and $M+1$.  The $M$-terms are residue-filtered
sums with upper limit $ck-1$, hence lie in $\mathscr H_{ck-1}$.  For the
$(M+1)$-terms, set $L=M+1$; then $L\equiv0\pmod c$, $L\le ck$, and the
summands contain $\mathcal H_\beta(L-1)$, which gives elements of
$\mathscr H_{ck}$.  This proves
\eqref{eq:cn-minus-one-upper-shift}.
\end{proof}

\subsubsection{Consequences and examples}

The table records representative lower- and upper-shift arithmetic-progression
cases.  Lower-shift entries allow $q\in\mathbb C$; upper-shift entries use the
integer-power hypothesis from Theorem~\ref{thm:arithmetic-progression-upper-shifts-integer}.

\begin{center}
\scriptsize
\renewcommand{\arraystretch}{1.08}
\setlength{\tabcolsep}{4pt}
\begin{adjustbox}{max width=\textwidth}
\begin{tabular}{@{}>{\raggedright\arraybackslash}p{0.52\textwidth}>{\centering\arraybackslash}p{0.22\textwidth}>{\raggedright\arraybackslash}p{0.18\textwidth}@{}}
\toprule
\textbf{Sum} & \textbf{Span} & \textbf{Condition} \\
\midrule
$\displaystyle \sum_{n=1}^{k} z^n(2n+1)^q H_{2n}^{(r)}(s)$
& $\mathscr H_{2k+1}$
& $q\in\mathbb C$. \\

$\displaystyle \sum_{n=1}^{k} z^n(3n+1)^q\mathcal H_{\alpha_1}(3n)^2\mathcal H_{\alpha_2}(3n)$
& $\mathscr H_{3k+1}$
& $q\in\mathbb C$. \\

$\displaystyle \sum_{n=1}^{k} z^n(5n+1)^q h_{5n}^{[m]}(r;s)\,H_{5n}^{\star,\mathbf u}(\mathbf w)$
& $\mathscr H_{5k+1}$
& $q\in\mathbb C$. \\

$\displaystyle \sum_{n=1}^{k} z^n(3n)^q\zeta(u,3n)\,H_{3n-1}^{(r_1,r_2)}(s_1,s_2)$
& $\mathscr H_{3k}$
& $q\in\mathbb C$, $u\ne1$. \\

$\displaystyle \sum_{n=1}^{k} z^n(2n)^q H_{2n+1}^{(r)}(s)\,A_{2n+1}^{(u)}(w)$
& $\mathscr H_{2k}+\mathscr H_{2k+1}$
& $q\in\mathbb Z$. \\

$\displaystyle \sum_{n=1}^{k} z^n(3n-1)^q H_{3n}^{(r_1,r_2)}(s_1,s_2)\,H_{3n}^{\star,\mathbf u}(\mathbf w)$
& $\mathscr H_{3k-1}+\mathscr H_{3k}$
& $q\in\mathbb Z$. \\

$\displaystyle \sum_{n=1}^{k}\frac{z^n}{(4n-1)^2}\mathcal H_{\alpha_1}(4n)^2\mathcal H_{\alpha_2}(4n)\,\psi'(4n+1)$
& $\mathscr H_{4k-1}+\mathscr H_{4k}$
& Integer first components. \\

$\displaystyle \sum_{n=1}^{k} z^n(6n)^q\mathcal H_{\alpha_1}(6n+1)^2\mathcal H_{\alpha_2}(6n+1)\,\mathbf1_{n\equiv a\,({\rm mod}\,M)}$
& $\mathscr H_{6k}+\mathscr H_{6k+1}$
& $q\in\mathbb Z$. \\
\bottomrule
\end{tabular}
\end{adjustbox}
\end{center}

\section{Affine-letter extensions}
\label{sec:affine-letter-extensions}

\subsection{Affine letters and the closure theorem}

We now pass from the special cases considered above to finite sums with
general affine powers
\[
(an+b)^q,
\]
where $a,b,q\in\mathbb C$.  For arbitrary complex $q$, these factors are not
reducible in general to the basic colored harmonic-number alphabet.  Closure
therefore requires an enlarged alphabet in which affine powers may appear as
letters.  This enlargement is structural: quasi-shuffle multiplication and the
outer-index splitting used in the closure proof create diagonal contributions
in which several affine factors occur at the same summation level.  The target
algebra must therefore allow one letter to carry a finite product of affine
powers.

The affine-letter harmonic-number alphabet used below is strictly more general
than the cyclotomic harmonic sums in the HarmonicSums framework of Ablinger
\citep{Ablinger2014}.  A cyclotomic harmonic-number letter carries one affine
denominator at a level, whereas one affine letter here may carry several factors
\[
a_1n+b_1,\ldots,a_tn+b_t
\]
at the same level, with complex powers and an independent color.  This
level-wise product structure is precisely what the closure argument requires.

An \emph{affine letter} is a triple
\[
L=(\boldsymbol\rho,\sigma,\mathbf A),
\]
where
\[
\boldsymbol\rho=(\rho_1,\ldots,\rho_t),
\qquad
\mathbf A=((a_1,b_1),\ldots,(a_t,b_t)).
\]
Its value at a positive integer $n$ is
\begin{equation}
\label{eq:affine-letter-value}
L(n)
=
\sigma^n
\prod_{\nu=1}^{t}(a_\nu n+b_\nu)^{-\rho_\nu}.
\end{equation}
For a word $\Gamma=(L_1,\ldots,L_d)$ in affine letters, define
\begin{equation}
\label{eq:affine-G-definition}
\mathcal G_{\Gamma}(N)
:=
\sum_{N\ge n_1>\cdots>n_d\ge 1}
\prod_{j=1}^{d}L_j(n_j),
\end{equation}
with the convention $\mathcal G_{\emptyset}(N)=1$.  Ordinary generalized harmonic numbers are recovered by taking, for each letter, a single affine factor
$(1,0)$:
\[
\mathcal H_{((r_1,s_1),\ldots,(r_d,s_d))}(N)
=
\mathcal G_{\Gamma}(N),
\]
where
\[
L_j=((r_j),s_j,((1,0))).
\]
Complex powers are interpreted after fixing a branch of the logarithm.  We
assume that no affine factor appearing in a denominator vanishes at the relevant
positive integers.

\begin{theorem}[Closure for affine-harmonic-sum reducible summands]
\label{thm:affine-power-summand-closure}
Let
\[
  F_j(n)\in \mathscr E^{\rm aff}_n,
  \qquad j=1,\ldots,m,
\]
be affine-harmonic-sum reducible sequences in the sense of
Definition~\ref{def:affine-harmonic-sum-reducibility-final}.  Let
$t\ge0$, let $a_\nu,b_\nu,q_\nu,z\in\mathbb C$ for
$1\le\nu\le t$, and assume that
\[
  a_\nu n+b_\nu\ne0
  \qquad (1\le n\le k,
  \;1\le\nu\le t).
\]
After fixing branches for the powers, for every
$e_1,\ldots,e_m\in\mathbb Z_{\ge0}$ the finite sum
\begin{equation}
\label{eq:affine-power-main-sum}
S(k)
=
\sum_{n=1}^{k}
 z^n
 \prod_{\nu=1}^{t}(a_\nu n+b_\nu)^{q_\nu}
 \prod_{j=1}^{m}F_j(n)^{e_j}
\end{equation}
belongs to
\[
\mathscr G_k
=
\operatorname{span}_{\mathbb C}
\{\mathcal G_\Gamma(k):\Gamma\text{ a word in affine letters}\}.
\]
Here the empty product is interpreted as $1$ when $t=0$.  Thus the
finite-summation operator sends affine-harmonic-sum reducible summands, after
multiplication by an additional finite product of affine powers, back to the
multiple affine harmonic-number span with upper limit $k$.
\end{theorem}

\begin{lemma}[Affine elementary summation]
\label{lem:affine-elementary-summation}
Let $A$ be an affine letter and let $\Gamma$ be a word in affine letters.  If
$\Gamma=\emptyset$, then
\[
  \sum_{n=1}^{k}A(n)=\mathcal G_A(k).
\]
If $\Gamma=(B_1,\Gamma')$ is nonempty, then
\begin{equation}
\label{eq:affine-elementary-summation-step}
\sum_{n=1}^{k}A(n)\mathcal G_{B_1,\Gamma'}(n)
=
\mathcal G_{A,B_1,\Gamma'}(k)
+
\mathcal G_{A\circ B_1,\Gamma'}(k).
\end{equation}
Consequently, finite summation maps $\mathscr E^{\rm aff}_n$ into
$\mathscr G_k$.
\end{lemma}

\begin{proof}
The empty-word case is immediate from the definition.  In the nonempty case,
expand $\mathcal G_{B_1,\Gamma'}(n)$.  The summation region is
\[
  k\ge n\ge n_1>n_2>\cdots\ge1.
\]
The part with $n>n_1$ gives $\mathcal G_{A,B_1,\Gamma'}(k)$, while the
part with $n=n_1$ merges the two letters at that level and gives
$\mathcal G_{A\circ B_1,\Gamma'}(k)$.
\end{proof}

\begin{proof}[Proof of Theorem~\ref{thm:affine-power-summand-closure}]
The affine sums $\mathcal G_\Gamma(n)$ form a quasi-shuffle algebra, and the
pointwise product of affine letters is again an affine letter.  Hence
$\mathscr E^{\rm aff}_n$ is closed under multiplication, and therefore
\[
  \prod_{j=1}^{m}F_j(n)^{e_j}\in\mathscr E^{\rm aff}_n.
\]
The additional factor
\[
  z^n\prod_{\nu=1}^{t}(a_\nu n+b_\nu)^{q_\nu}
\]
is itself the value of one affine letter, namely
\[
A=
\bigl((-q_1,\ldots,-q_t),z,
((a_1,b_1),\ldots,(a_t,b_t))\bigr),
\]
with the evident empty list when $t=0$.  Multiplying by this letter still gives
an element of $\mathscr E^{\rm aff}_n$.  Applying
Lemma~\ref{lem:affine-elementary-summation} termwise proves that
$S(k)\in\mathscr G_k$.
\end{proof}

\begin{corollary}[Closure for products of affine harmonic numbers]
\label{thm:affine-product-powers-summand}
Let $\Gamma_1,\ldots,\Gamma_m$ be words in affine letters, and let
$e_1,\ldots,e_m\in\mathbb Z_{\ge0}$.  Under the same nonvanishing and branch
assumptions as in Theorem~\ref{thm:affine-power-summand-closure}, the sum
\begin{equation}
\label{eq:affine-product-main-sum}
\sum_{n=1}^{k}
 z^n
 \prod_{\nu=1}^{t}(a_\nu n+b_\nu)^{q_\nu}
 \prod_{j=1}^{m}\mathcal G_{\Gamma_j}(n)^{e_j}
\end{equation}
belongs to $\mathscr G_k$.
\end{corollary}

\begin{proof}
Each $\mathcal G_{\Gamma_j}(n)$ belongs to $\mathscr E^{\rm aff}_n$.  The
claim is therefore the corresponding special case of
Theorem~\ref{thm:affine-power-summand-closure}.
\end{proof}

\begin{corollary}[Ordinary harmonic-number specialization]
\label{cor:affine-self-closure}
Let $\alpha_1,\ldots,\alpha_m$ be words in the generalized harmonic-number
alphabet and let $e_1,\ldots,e_m\in\mathbb Z_{\ge0}$.  Under the same
nonvanishing and branch assumptions as in
Theorem~\ref{thm:affine-power-summand-closure}, the sum
\begin{equation}
\label{eq:affine-self-closure-sum}
\sum_{n=1}^{k}
z^n
\prod_{\nu=1}^{t}(a_\nu n+b_\nu)^{q_\nu}
\prod_{j=1}^{m}\mathcal H_{\alpha_j}(n)^{e_j}
\end{equation}
belongs to $\mathscr G_k$.
\end{corollary}

\begin{proof}
Embed the basic colored alphabet into the affine alphabet by identifying the
colored letter $(r,s)$ with the affine letter $((r),s,((1,0)))$.  Then each
$\mathcal H_{\alpha_j}(n)$ is a special case of a multiple affine harmonic
number, and hence is affine-harmonic-sum reducible.  The claim follows from
Theorem~\ref{thm:affine-power-summand-closure}.
\end{proof}

\subsection{Consequences and examples}

Recall that
\[
\mathscr G_k:=
\operatorname{span}_{\mathbb C}
\{\mathcal G_\Gamma(k):\Gamma\text{ a word in affine letters}\}.
\]
We also use the local refinement $\mathscr G_k[\mathcal B]$ for the subspan
whose affine letters lie in the merge-closed alphabet generated by the finite
set $\mathcal B$.  For the concrete rows put
\[
\begin{aligned}
\mathcal B_{\mathrm{a},1}&:=\{((3),-1,((2,-1))),((2),i,((1,0)))\},\\
\mathcal B_{\mathrm{a},2}&:=\{((-2,1),3,((2,1),(5,-2))),((1),-1,((1,0))),((2),i,((1,0)))\}.
\end{aligned}
\]
The table records representative consequences of
Theorem~\ref{thm:affine-power-summand-closure} and
Corollary~\ref{thm:affine-product-powers-summand}; the first rows display concrete
letter data, while the later rows keep the symbolic span.

\begin{center}
\scriptsize
\renewcommand{\arraystretch}{1.08}
\setlength{\tabcolsep}{3pt}
\begin{longtable}{@{}>{\raggedright\arraybackslash}p{0.48\textwidth}>{\raggedright\arraybackslash}p{0.30\textwidth}>{\raggedright\arraybackslash}p{0.16\textwidth}@{}}
\toprule
\textbf{Convergent infinite sum} & \textbf{Resulting value space} & \textbf{Reason for convergence} \\
\midrule
\endfirsthead
\toprule
\textbf{Convergent infinite sum} & \textbf{Resulting value space} & \textbf{Reason for convergence} \\
\midrule
\endhead
\bottomrule
\endlastfoot

$\displaystyle \sum_{n=1}^{k}2^n(3n+1)^5$
& $\mathcal G_{((-5),2,((3,1)))}(k)$
& $3n+1\ne0$. \\

$\displaystyle \sum_{n=1}^{k}\frac{(-1)^n}{(2n-1)^3}H_n^{(2)}(i)$
& $\mathscr G_k[\mathcal B_{\mathrm{a},1}]$
& Nonzero factor. \\

$\displaystyle \sum_{n=1}^{k}3^n\frac{(2n+1)^2}{5n-2}H_n^{\star,(1,2)}(-1,i)$
& $\mathscr G_k[\mathcal B_{\mathrm{a},2}]$
& Star split into strict sums. \\

$\displaystyle \sum_{n=1}^{k} z^n(an+b)^q(\alpha n+\beta)^\lambda H_n^{\mathbf r}(\mathbf s)\,H_n^{\star,\mathbf u}(\mathbf w)$
& $\mathscr G_k$
& Both affine factors nonzero. \\

$\displaystyle \sum_{n=1}^{k}\frac{z^n\,A_n^{(r)}(s)\,\Phi(\xi,u,n+1)\,H_n^{(v)}(w)}{(a n+b)^p(c n+d)^\lambda}$
& $\mathscr G_k$
& Nonzero denominators, $\xi\ne0,1$. \\

$\displaystyle \sum_{n=1}^{k} z^n\chi(n)\left\lfloor\frac{3n+2}{5}\right\rfloor^m(an+b)^q\mathcal H_\alpha(n)$
& $\mathscr G_k$
& Fixed modulus; $m\in\mathbb N_0$. \\

$\displaystyle \sum_{n=1}^{k} z^n\prod_{\nu=1}^{t}(a_\nu n+b_\nu)^{q_\nu}\mathcal G_\Gamma(n)^2\mathcal G_\Delta(n)$
& $\mathscr G_k$
& Affine self-closure. \\

$\displaystyle \sum_{n=2}^{k} z^n(n-1)^q(n+1)^\lambda\mathcal H_\alpha(n)\,\psi^{(m)}(n+1)$
& $\mathscr G_k$
& Endpoint avoids $n-1=0$. \\

\end{longtable}
\end{center}

\section{Polynomial-letter extensions}
\label{sec:polynomial-letter-extensions}

\subsection{Polynomial letters and the closure theorem}

We next enlarge the affine alphabet to a polynomial-letter alphabet in order to
handle finite sums with polynomial powers
\[
P(n)^q,
\]
where $P\in\mathbb C[x]$ and $q\in\mathbb C$.  Special polynomial cases were
treated earlier, either by direct expansion or by reduction to affine factors.
For a general polynomial with a complex exponent, however, $P(n)^q$ is not
reducible in general to the affine-letter alphabet.  Closure therefore requires
a polynomial-letter alphabet.

This enlargement is again structural.  Quasi-shuffle multiplication and the
outer-index split produce diagonal terms in which several polynomial factors
occur at the same summation level.  Hence one polynomial letter is allowed to
carry a finite list of polynomial factors, together with their complex powers
and an independent color.  This is the polynomial analogue of the affine-letter
alphabet above and the natural target for polynomial-power convolution.

A \emph{polynomial letter} is a triple
\[
L=(\boldsymbol\rho,\sigma,\mathbf P),
\]
where
\[
\boldsymbol\rho=(\rho_1,\ldots,\rho_t),
\qquad
\mathbf P=(P_1,\ldots,P_t),
\qquad
P_\nu\in\mathbb C[x].
\]
Its value at a positive integer $n$ is
\begin{equation}
\label{eq:poly-letter-value}
L(n)
=
\sigma^n
\prod_{\nu=1}^{t}P_\nu(n)^{-\rho_\nu}.
\end{equation}
For a word $\Omega=(L_1,\ldots,L_d)$ in polynomial letters, define
\begin{equation}
\label{eq:poly-P-definition}
\mathcal P_{\Omega}(N)
:=
\sum_{N\ge n_1>\cdots>n_d\ge 1}
\prod_{j=1}^{d}L_j(n_j),
\end{equation}
with the convention $\mathcal P_{\emptyset}(N)=1$.  Ordinary generalized harmonic numbers are recovered by taking the identity polynomial
\[
\idpoly(x)=x.
\]
Indeed,
\[
\mathcal H_{((r_1,s_1),\ldots,(r_d,s_d))}(N)
=
\mathcal P_{\Omega}(N),
\]
where
\[
L_j=((r_j),s_j,(\idpoly)),
\qquad j=1,\ldots,d.
\]
Complex powers are interpreted after fixing branches of the logarithm.  We
assume throughout that no polynomial appearing in a denominator vanishes at the
positive integers in the summation range.

If
\[
L=(\boldsymbol\rho,\sigma,\mathbf P),
\qquad
M=(\boldsymbol\eta,\tau,\mathbf Q),
\]
we write
\[
L\circ M
=
(\boldsymbol\rho\mathbin{\Vert}\boldsymbol\eta,
\sigma\tau,
\mathbf P\mathbin{\Vert}\mathbf Q),
\]
where $\Vert$ denotes concatenation.  Then
\[
(L\circ M)(n)=L(n)M(n).
\]

\begin{theorem}[Closure for polynomial-harmonic-sum reducible summands]
\label{thm:poly-power-summand-closure}
Let
\[
  F_j(n)\in \mathscr E^{\rm pol}_n,
  \qquad j=1,\ldots,m,
\]
be polynomial-harmonic-sum reducible sequences in the sense of
Definition~\ref{def:polynomial-harmonic-sum-reducibility-final}.  Let
$t\ge0$, let $P_1,\ldots,P_t\in\mathbb C[x]$, let
$q_1,\ldots,q_t,z\in\mathbb C$, and assume that
\[
  P_\nu(n)\ne0
  \qquad (1\le n\le k,
  \;1\le\nu\le t).
\]
After fixing branches for the powers, for every
$e_1,\ldots,e_m\in\mathbb Z_{\ge0}$ the finite sum
\begin{equation}
\label{eq:poly-power-main-sum}
S(k)
=
\sum_{n=1}^{k}
 z^n
 \prod_{\nu=1}^{t}P_\nu(n)^{q_\nu}
 \prod_{j=1}^{m}F_j(n)^{e_j}
\end{equation}
belongs to
\[
\mathscr P_k
=
\operatorname{span}_{\mathbb C}
\{\mathcal P_\Omega(k):\Omega\text{ a word in polynomial letters}\}.
\]
Here the empty product is interpreted as $1$ when $t=0$.  Thus the
finite-summation operator sends polynomial-harmonic-sum reducible summands,
after multiplication by an additional finite product of polynomial powers, back
to the polynomial-base harmonic-number span with upper limit $k$.
\end{theorem}

\begin{lemma}[Polynomial elementary summation]
\label{lem:polynomial-elementary-summation}
Let $A$ be a polynomial letter and let $\Omega$ be a word in polynomial letters.
If $\Omega=\emptyset$, then
\[
  \sum_{n=1}^{k}A(n)=\mathcal P_A(k).
\]
If $\Omega=(B_1,\Omega')$ is nonempty, then
\begin{equation}
\label{eq:polynomial-elementary-summation-step}
\sum_{n=1}^{k}A(n)\mathcal P_{B_1,\Omega'}(n)
=
\mathcal P_{A,B_1,\Omega'}(k)
+
\mathcal P_{A\circ B_1,\Omega'}(k).
\end{equation}
Consequently, finite summation maps $\mathscr E^{\rm pol}_n$ into
$\mathscr P_k$.
\end{lemma}

\begin{proof}
The empty-word case is immediate.  In the nonempty case, expand
$\mathcal P_{B_1,\Omega'}(n)$.  The summation region is
\[
  k\ge n\ge n_1>n_2>\cdots\ge1.
\]
The part with $n>n_1$ gives $\mathcal P_{A,B_1,\Omega'}(k)$, while the part
with $n=n_1$ merges the two polynomial letters at that level and gives
$\mathcal P_{A\circ B_1,\Omega'}(k)$.
\end{proof}

\begin{proof}[Proof of Theorem~\ref{thm:poly-power-summand-closure}]
The polynomial-letter sums $\mathcal P_\Omega(n)$ form a quasi-shuffle algebra,
and the pointwise product of polynomial letters is again a polynomial letter.
Hence $\mathscr E^{\rm pol}_n$ is closed under multiplication, and therefore
\[
  \prod_{j=1}^{m}F_j(n)^{e_j}\in\mathscr E^{\rm pol}_n.
\]
The additional factor
\[
  z^n\prod_{\nu=1}^{t}P_\nu(n)^{q_\nu}
\]
is itself the value of one polynomial letter, namely
\[
A=
\bigl((-q_1,\ldots,-q_t),z,(P_1,\ldots,P_t)\bigr),
\]
with the evident empty list when $t=0$.  Multiplying by this letter still gives
an element of $\mathscr E^{\rm pol}_n$.  Applying
Lemma~\ref{lem:polynomial-elementary-summation} termwise proves that
$S(k)\in\mathscr P_k$.
\end{proof}

\begin{corollary}[Closure for products of polynomial-base harmonic numbers]
\label{thm:poly-product-powers-summand}
Let $\Omega_1,\ldots,\Omega_m$ be words in polynomial letters, and let
$e_1,\ldots,e_m\in\mathbb Z_{\ge0}$.  Under the same nonvanishing and branch
assumptions as in Theorem~\ref{thm:poly-power-summand-closure}, the sum
\begin{equation}
\label{eq:poly-product-main-sum}
\sum_{n=1}^{k}
 z^n
 \prod_{\nu=1}^{t}P_\nu(n)^{q_\nu}
 \prod_{j=1}^{m}\mathcal P_{\Omega_j}(n)^{e_j}
\end{equation}
belongs to $\mathscr P_k$.
\end{corollary}

\begin{proof}
Each $\mathcal P_{\Omega_j}(n)$ belongs to $\mathscr E^{\rm pol}_n$.  The claim
is therefore the corresponding special case of
Theorem~\ref{thm:poly-power-summand-closure}.
\end{proof}

\begin{corollary}[Ordinary harmonic-number specialization]
\label{cor:poly-self-closure}
Let $\alpha_1,\ldots,\alpha_m$ be words in the generalized harmonic-number
alphabet and let $e_1,\ldots,e_m\in\mathbb Z_{\ge0}$.  Under the same
nonvanishing and branch assumptions as in
Theorem~\ref{thm:poly-power-summand-closure}, the sum
\begin{equation}
\label{eq:poly-self-closure-sum}
\sum_{n=1}^{k}
z^n
\prod_{\nu=1}^{t}P_\nu(n)^{q_\nu}
\prod_{j=1}^{m}\mathcal H_{\alpha_j}(n)^{e_j}
\end{equation}
belongs to $\mathscr P_k$.
\end{corollary}

\begin{proof}
Embed the basic colored alphabet into the polynomial alphabet by identifying the
colored letter $(r,s)$ with the polynomial letter $((r),s,(\idpoly))$.  Then
each $\mathcal H_{\alpha_j}(n)$ is a special case of a multiple
polynomial-base harmonic number, and hence is polynomial-harmonic-sum
reducible.  The claim follows from
Theorem~\ref{thm:poly-power-summand-closure}.
\end{proof}

\subsection{Consequences and examples}

Recall that
\[
\mathscr P_k:=
\operatorname{span}_{\mathbb C}
\{\mathcal P_\Omega(k):\Omega\text{ a word in polynomial letters}\}.
\]
We also write $\mathscr P_k[\mathcal B]$ for the subspan generated by a finite
merge-closed polynomial-letter set $\mathcal B$.  For the concrete rows put
\[
\begin{aligned}
\mathcal B_{\mathrm{p},1}&:=\{((2),-1,(x^2+1)),((3),i,(x))\},\\
\mathcal B_{\mathrm{p},2}&:=\{((-2,1),3,(x^2+x+1,x^2+4)),((1),-1,(x)),((2),i,(x))\},\\
\mathcal B_{\mathrm{p},3}&:=\{((-1,p+1),z,(x,x^2+a^2)),((r),s,(x))\}.
\end{aligned}
\]
The table records representative consequences of Theorem~\ref{thm:poly-power-summand-closure}
and Corollary~\ref{thm:poly-product-powers-summand}.  Nonvanishing is understood for the
summation range, and branches are fixed once and for all.

\begin{center}
\scriptsize
\renewcommand{\arraystretch}{1.08}
\setlength{\tabcolsep}{4pt}
\begin{adjustbox}{max width=\textwidth}
\begin{tabular}{@{}>{\raggedright\arraybackslash}p{0.47\textwidth}>{\centering\arraybackslash}p{0.31\textwidth}>{\raggedright\arraybackslash}p{0.14\textwidth}@{}}
\toprule
\textbf{Sum} & \textbf{Span} & \textbf{Condition} \\
\midrule
$\displaystyle \sum_{n=1}^{k}2^n(n^2+n+1)^3$
& $\mathcal P_{((-3),2,(x^2+x+1))}(k)$
& Polynomial nonzero. \\

$\displaystyle \sum_{n=1}^{k}\frac{(-1)^n}{(n^2+1)^2}H_n^{(3)}(i)$
& $\mathscr P_k[\mathcal B_{\mathrm{p},1}]$
& Depth-one factor. \\

$\displaystyle \sum_{n=1}^{k}3^n\frac{(n^2+n+1)^2}{n^2+4}H_n^{\star,(1,2)}(-1,i)$
& $\mathscr P_k[\mathcal B_{\mathrm{p},2}]$
& Star split into strict sums. \\

$\displaystyle \sum_{n=1}^{k}\frac{2n\,z^n H_n^{(r)}(s)}{(n^2+a^2)^{p+1}}$
& $\mathscr P_k[\mathcal B_{\mathrm{p},3}]$
& Mathieu-type summand. \\

$\displaystyle \sum_{n=1}^{k} z^n\bigl(n^2+n+1\bigr)^q H_n^{\star,\mathbf r}(\mathbf s)\,A_n^{(u)}(w)$
& $\mathscr P_k$
& Symbolic polynomial case. \\

$\displaystyle \sum_{n=1}^{k} z^n\prod_{\nu=1}^{t}P_\nu(n)^{q_\nu}\,\bigl(H_n^{(r)}(s)\bigr)^2H_n^{\star,\mathbf u}(\mathbf w)$
& $\mathscr P_k$
& $P_\nu(n)\ne0$. \\

$\displaystyle \sum_{n=1}^{k} z^nP(n)^q\mathcal P_\Omega(n)^2\mathcal P_\Theta(n)$
& $\mathscr P_k$
& Polynomial self-closure. \\

$\displaystyle \sum_{n=2}^{k} z^nP(n-1)^q(n^2+1)^\lambda\mathcal H_\alpha(n)$
& $\mathscr P_k$
& Nonzero factors on range. \\
\bottomrule
\end{tabular}
\end{adjustbox}
\end{center}

\section{Scaled-index sums}
\label{sec:scaled-upper-limit-lifting}

This section collects the scaled-index principles for the three alphabets
developed above.  The common idea is to synchronize all upper limits \(p_j n\)
by lifting them to the common limit \(pn\), where
\(p=\operatorname{lcm}(p_1,\ldots,p_m)\), and then to recover the original
summation over \(n\) by a divisibility filter and one new leading letter.  Thus
scaled-index sums are reduced to the same finite convolution mechanism, with the
only extra factor being \(z^n n^q\) in the basic colored case, a finite product
of affine powers in the affine case, and a finite product of polynomial powers
in the polynomial case.

\subsection{Scaled upper-limit lifting theorems}
\label{subsec:scaled-upper-limit-lifting-theorems}

\begin{theorem}[Scaled upper limits in the basic colored alphabet]
\label{thm:scale-unification}
Let
\[
 p_1,\ldots,p_m\in\mathbb Z_{>0},
 \qquad
 p=\lcm(p_1,\ldots,p_m),
\]
and let \(\alpha_1,\ldots,\alpha_m\) be words in the generalized harmonic number alphabet.  Let
\[
 e_1,\ldots,e_m\in\mathbb Z_{\ge0},
 \qquad z,q\in\mathbb C.
\]
Then the finite sum
\begin{equation}
\label{eq:scaled-sum}
S(k)
=
\sum_{n=1}^{k}
z^n n^q
\prod_{j=1}^{m}\mathcal H_{\alpha_j}(p_jn)^{e_j}
\end{equation}
belongs to the finite linear span
\begin{equation}
\label{eq:scaled-span-inclusion}
S(k)
\in
\operatorname{span}_{\mathbb C}
\left\{
\mathcal H_\beta(pk)
\right\}_{\beta}.
\end{equation}
Thus products of generalized harmonic numbers evaluated at different integer
multiples of the summation index, even after multiplication by the elementary
factor \(z^n n^q\), reduce to generalized harmonic numbers at the common upper
limit \(pk\).
\end{theorem}

\begin{proof}
We prove the result in three steps.

First, we describe the scale-up mechanism.  Let
\[
\alpha=((r_1,s_1),\ldots,(r_d,s_d))
\]
be a word, and suppose that \(c,p\in\mathbb Z_{>0}\) with \(c\mid p\).  Put
\[
M=\frac pc,
\qquad
\omega_M=e^{2\pi i/M}.
\]
Choose \(M\)-th roots
\[
t_i^M=s_i,
\qquad i=1,\ldots,d.
\]
Also write
\[
|\alpha|=d,
\qquad
\operatorname{wt}(\alpha)=r_1+\cdots+r_d.
\]
Then
\begin{equation}
\label{eq:scale-up-identity}
\mathcal H_\alpha(cn)
=
M^{\operatorname{wt}(\alpha)-|\alpha|}
\sum_{a_1,\ldots,a_d=0}^{M-1}
\mathcal H_{\alpha(a_1,\ldots,a_d)}(pn),
\end{equation}
where
\[
\alpha(a_1,\ldots,a_d)
=
\left(
(r_1,t_1\omega_M^{a_1}),
\ldots,
(r_d,t_d\omega_M^{a_d})
\right).
\]
For the empty word this identity is interpreted as
\(\mathcal H_\emptyset(cn)=\mathcal H_\emptyset(pn)=1\).  Indeed, after
expanding the right-hand side and interchanging the finite root sums with the
nested sum, the factor
\[
\sum_{a_i=0}^{M-1}\omega_M^{a_iq_i}
\]
forces \(q_i=M\ell_i\).  On these surviving terms,
\[
t_i^{q_i}=t_i^{M\ell_i}=s_i^{\ell_i},
\qquad
q_i^{r_i}=M^{r_i}\ell_i^{r_i},
\]
and the scalar in \eqref{eq:scale-up-identity} exactly cancels the powers of
\(M\) introduced by the filter and the denominator.  This proves
\eqref{eq:scale-up-identity}.

Applying \eqref{eq:scale-up-identity} with \(c=p_j\) shows that every
\(\mathcal H_{\alpha_j}(p_jn)\) is a finite linear combination of generalized harmonic numbers with common upper limit \(pn\).  The quasi-shuffle product at this
common upper limit then gives
\begin{equation}
\label{eq:product-common-scale-reduction-new}
\prod_{j=1}^{m}\mathcal H_{\alpha_j}(p_jn)^{e_j}
=
\sum_\gamma C_\gamma\mathcal H_\gamma(pn),
\qquad C_\gamma\in\mathbb C.
\end{equation}

It remains to sum the elementary factor against \(\mathcal H_\gamma(pn)\).  Fix
a \(p\)-th root \(\rho^p=z\), put \(\zeta_p=e^{2\pi i/p}\), and define
\[
\lambda_a=(-q,\rho\zeta_p^a),
\qquad 0\le a\le p-1.
\]
For nonempty \(\gamma=(c_1,\gamma')\), using
\[
\mathbf 1_{p\mid Q}=\frac1p\sum_{a=0}^{p-1}\zeta_p^{aQ}
\]
and the change of variable \(Q=pn\) gives
\begin{equation}
\label{eq:scaled-elementary-summation-formula}
\sum_{n=1}^{k}z^n n^q\mathcal H_\gamma(pn)
=
p^{-q-1}\sum_{a=0}^{p-1}
\left(
\mathcal H_{\lambda_a,\gamma}(pk)
+
\mathcal H_{\lambda_a\circ c_1,\gamma'}(pk)
\right).
\end{equation}
Here \(p^{-q}=\exp(-q\log p)\), and the merge is the usual colored-letter merge.
Indeed, the divisibility filter converts the left hand side into a finite
linear combination of sums
\[
\sum_{Q=1}^{pk}(\rho\zeta_p^a)^Q Q^q\mathcal H_\gamma(Q),
\]
up to the scalar \(p^{-q-1}\).  The summation region
\(pk\ge Q\ge Q_1>\cdots\) splits into \(Q>Q_1\) and \(Q=Q_1\), giving the two
terms in \eqref{eq:scaled-elementary-summation-formula}.  If \(\gamma=\emptyset\),
the same argument gives
\begin{equation}
\label{eq:scaled-empty-elementary-summation-formula}
\sum_{n=1}^{k}z^n n^q
=
p^{-q-1}\sum_{a=0}^{p-1}\mathcal H_{\lambda_a}(pk).
\end{equation}
Combining \eqref{eq:product-common-scale-reduction-new} with
\eqref{eq:scaled-elementary-summation-formula} and
\eqref{eq:scaled-empty-elementary-summation-formula} proves the theorem.
\end{proof}

\begin{theorem}[Scaled upper limits in the affine-letter alphabet]
\label{thm:affine-scaled-upper-limits}
Let
\[
 p_1,\ldots,p_m\in\mathbb Z_{>0},
 \qquad
 p=\lcm(p_1,\ldots,p_m).
\]
Let \(\Gamma_1,\ldots,\Gamma_m\) be words in affine letters, and let
\(e_1,\ldots,e_m\in\mathbb Z_{\ge0}\).  Let also \(z\in\mathbb C\), let \(t\in\mathbb Z_{\ge0}\), and let
\[
 a_\nu,b_\nu,q_\nu\in\mathbb C,
 \qquad 1\le \nu\le t,
\]
where the case \(t=0\) is allowed.  Assume that all affine factors appearing
below are nonzero at the positive integers at which they are evaluated, after
all scale changes, and fix compatible branches of the powers.  Then
\begin{equation}
\label{eq:affine-scaled-upper-limit-sum}
S(k)=
\sum_{n=1}^{k}
z^n\prod_{\nu=1}^{t}(a_\nu n+b_\nu)^{q_\nu}
\prod_{j=1}^{m}\mathcal G_{\Gamma_j}(p_jn)^{e_j}
\end{equation}
belongs to the finite linear span
\begin{equation}
\label{eq:affine-scaled-upper-limit-span}
S(k)\in
\operatorname{span}_{\mathbb C}\{\mathcal G_\Delta(pk)\}_\Delta,
\end{equation}
where \(\Delta\) ranges over words in affine letters.  Thus affine-letter sums
with different integer-scaled upper limits can be unified at the common upper
limit \(pk\), and the outer affine factor is absorbed into the same alphabet.
\end{theorem}

\begin{proof}
Let \(c\mid p\), put \(M=p/c\), and let
\(\Gamma=(L_1,\ldots,L_d)\), where
\[
L_i=(\boldsymbol\rho_i,\sigma_i,\mathbf A_i),
\qquad
\mathbf A_i=((a_{i,1},b_{i,1}),\ldots,(a_{i,t_i},b_{i,t_i})).
\]
Choose \(M\)-th roots \(\tau_i^M=\sigma_i\), and put
\(\omega_M=e^{2\pi i/M}\).  For
\(\mathbf h=(h_1,\ldots,h_d)\in\{0,\ldots,M-1\}^d\), define
\[
L_i^{[M,h_i]}
=
\left(
\boldsymbol\rho_i,
\tau_i\omega_M^{h_i},
\left((a_{i,1}/M,b_{i,1}),\ldots,(a_{i,t_i}/M,b_{i,t_i})\right)
\right)
\]
and
\[
\Gamma^{[M,\mathbf h]}
=
(L_1^{[M,h_1]},\ldots,L_d^{[M,h_d]}).
\]
Then
\begin{equation}
\label{eq:affine-scale-up-identity}
\mathcal G_\Gamma(cn)
=
M^{-d}\sum_{h_1,\ldots,h_d=0}^{M-1}
\mathcal G_{\Gamma^{[M,\mathbf h]}}(pn),
\end{equation}
with the evident interpretation when \(d=0\).  After expansion, the root sums
force the inner indices to be multiples of \(M\), and the scaled affine factors
satisfy
\[
(a_{i,\nu}/M)(M\ell_i)+b_{i,\nu}=a_{i,\nu}\ell_i+b_{i,\nu},
\qquad
(\tau_i\omega_M^{h_i})^{M\ell_i}=\sigma_i^{\ell_i}.
\]
The factor \(M^{-d}\) cancels the \(M^d\) produced by the root filters, proving
\eqref{eq:affine-scale-up-identity}.

Using \eqref{eq:affine-scale-up-identity} with \(c=p_j\), every
\(\mathcal G_{\Gamma_j}(p_jn)\) is converted into a finite linear combination
of affine-letter sums with upper limit \(pn\).  The quasi-shuffle product for
affine-letter sums then reduces
\[
\prod_{j=1}^{m}\mathcal G_{\Gamma_j}(p_jn)^{e_j}
\]
to a finite linear combination of single affine-letter sums \(\mathcal G_\Delta(pn)\).

Choose a \(p\)-th root \(\rho^p=z\), and define the affine letter
\[
A=
\left(
(-q_1,\ldots,-q_t),
\rho,
((a_1/p,b_1),\ldots,(a_t/p,b_t))
\right),
\]
with the empty-product convention when \(t=0\).  Then
\[
A(pn)=z^n\prod_{\nu=1}^{t}(a_\nu n+b_\nu)^{q_\nu}.
\]
Let \(\zeta_p=e^{2\pi i/p}\) and set
\[
C_a=((0),\zeta_p^a,((1,0))),
\qquad 0\le a\le p-1.
\]
For nonempty \(\Delta=(D_1,\ldots,D_d)\), inserting the divisibility filter for
\(p\mid Q\) and splitting the boundary \(Q=m_1\) gives
\[
\sum_{n=1}^{k}A(pn)\mathcal G_\Delta(pn)
=
\frac1p\sum_{a=0}^{p-1}
\left(
\mathcal G_{C_a\circ A,D_1,\ldots,D_d}(pk)
+
\mathcal G_{C_a\circ A\circ D_1,D_2,\ldots,D_d}(pk)
\right).
\]
If \(\Delta=\emptyset\), the same argument gives
\[
\sum_{n=1}^{k}A(pn)
=
\frac1p\sum_{a=0}^{p-1}\mathcal G_{C_a\circ A}(pk).
\]
All terms lie in the affine span with upper limit \(pk\), and the theorem
follows by linearity.
\end{proof}

\begin{theorem}[Scaled upper limits in the polynomial-letter alphabet]
\label{thm:poly-scaled-upper-limits}
Let
\[
 p_1,\ldots,p_m\in\mathbb Z_{>0},
 \qquad
 p=\lcm(p_1,\ldots,p_m).
\]
Let \(\Omega_1,\ldots,\Omega_m\) be words in polynomial letters, and let
\(e_1,\ldots,e_m\in\mathbb Z_{\ge0}\).  Let also \(z\in\mathbb C\), let \(t\in\mathbb Z_{\ge0}\), let
\(P_1,\ldots,P_t\in\mathbb C[x]\), and let \(q_1,\ldots,q_t\in\mathbb C\),
where \(t=0\) is allowed.  Assume that all polynomial factors appearing below
are nonzero at the positive integers at which they are evaluated, after all
scale changes, and fix compatible branches of the powers.  Then
\begin{equation}
\label{eq:poly-scaled-upper-limit-sum}
S(k)=
\sum_{n=1}^{k}
z^n\prod_{\nu=1}^{t}P_\nu(n)^{q_\nu}
\prod_{j=1}^{m}\mathcal P_{\Omega_j}(p_jn)^{e_j}
\end{equation}
belongs to the finite linear span
\begin{equation}
\label{eq:poly-scaled-upper-limit-span}
S(k)\in
\operatorname{span}_{\mathbb C}\{\mathcal P_\Theta(pk)\}_\Theta,
\end{equation}
where \(\Theta\) ranges over words in polynomial letters.  Thus polynomial-letter
sums with different integer-scaled upper limits can be unified at the common
upper limit \(pk\), and the outer polynomial factor is absorbed into the same
alphabet.
\end{theorem}

\begin{proof}
Let \(c\mid p\), put \(M=p/c\), and let
\(\Omega=(L_1,\ldots,L_d)\), where
\[
L_i=(\boldsymbol\rho_i,\sigma_i,\mathbf Q_i),
\qquad
\mathbf Q_i=(Q_{i,1},\ldots,Q_{i,t_i}).
\]
Choose \(M\)-th roots \(\tau_i^M=\sigma_i\), and put
\(\omega_M=e^{2\pi i/M}\).  For
\(\mathbf h=(h_1,\ldots,h_d)\in\{0,\ldots,M-1\}^d\), define
\[
L_i^{[M,h_i]}
=
\left(
\boldsymbol\rho_i,
\tau_i\omega_M^{h_i},
(Q_{i,1}(x/M),\ldots,Q_{i,t_i}(x/M))
\right)
\]
and
\[
\Omega^{[M,\mathbf h]}
=
(L_1^{[M,h_1]},\ldots,L_d^{[M,h_d]}).
\]
Then
\begin{equation}
\label{eq:poly-scale-up-identity}
\mathcal P_\Omega(cn)
=
M^{-d}\sum_{h_1,\ldots,h_d=0}^{M-1}
\mathcal P_{\Omega^{[M,\mathbf h]}}(pn).
\end{equation}
Again, the root filter forces \(Q_i=M\ell_i\), and then
\[
Q_{i,\nu}(Q_i/M)=Q_{i,\nu}(\ell_i),
\qquad
(\tau_i\omega_M^{h_i})^{Q_i}=\sigma_i^{\ell_i}.
\]
This proves \eqref{eq:poly-scale-up-identity}.

Using \eqref{eq:poly-scale-up-identity} with \(c=p_j\), every
\(\mathcal P_{\Omega_j}(p_jn)\) is converted into a finite linear combination
of polynomial-letter sums with upper limit \(pn\).  The quasi-shuffle product
for polynomial-letter sums then reduces
\[
\prod_{j=1}^{m}\mathcal P_{\Omega_j}(p_jn)^{e_j}
\]
to a finite linear combination of single polynomial-letter sums \(\mathcal P_\Theta(pn)\).

Choose a \(p\)-th root \(\rho^p=z\), and define the polynomial letter
\[
A=
\left(
(-q_1,\ldots,-q_t),
\rho,
(P_1(x/p),\ldots,P_t(x/p))
\right),
\]
with the evident empty-product convention if \(t=0\).  Then
\[
A(pn)=z^n\prod_{\nu=1}^{t}P_\nu(n)^{q_\nu}.
\]
Let \(\zeta_p=e^{2\pi i/p}\) and set
\[
C_a=((0),\zeta_p^a,(\idpoly)),
\qquad 0\le a\le p-1.
\]
For nonempty \(\Theta=(D_1,\ldots,D_d)\), inserting the divisibility filter for
\(p\mid Q\) and splitting the boundary \(Q=m_1\) gives
\[
\sum_{n=1}^{k}A(pn)\mathcal P_\Theta(pn)
=
\frac1p\sum_{a=0}^{p-1}
\left(
\mathcal P_{C_a\circ A,D_1,\ldots,D_d}(pk)
+
\mathcal P_{C_a\circ A\circ D_1,D_2,\ldots,D_d}(pk)
\right).
\]
If \(\Theta=\emptyset\), then
\[
\sum_{n=1}^{k}A(pn)
=
\frac1p\sum_{a=0}^{p-1}\mathcal P_{C_a\circ A}(pk).
\]
All resulting terms are polynomial-letter sums with upper limit \(pk\), and the
claim follows by linearity.
\end{proof}

\subsection{Consequences and examples}
\label{subsec:scaled-upper-limit-lifting-examples}

The table records representative scaled upper-limit sums in the three alphabets.
The target upper limit is always the least common multiple of the displayed
scales times \(k\).  In the basic colored case, the factor \(z^n n^q\) is part
of the theorem.  In the affine and polynomial cases, the corresponding outer
factors are absorbed as single affine or polynomial letters.

\begin{center}
\scriptsize
\renewcommand{\arraystretch}{1.08}
\setlength{\tabcolsep}{4pt}
\begin{adjustbox}{max width=\textwidth}
\begin{tabular}{@{}>{\raggedright\arraybackslash}p{0.47\textwidth}>{\centering\arraybackslash}p{0.25\textwidth}>{\raggedright\arraybackslash}p{0.20\textwidth}@{}}
\toprule
\textbf{Sum} & \textbf{Target span} & \textbf{Reason} \\
\midrule
$\displaystyle \sum_{n=1}^{k} z^n n^q H_{2n}^{(r)}(s)H_{3n}^{(u)}(w)$
& $\operatorname{span}_{\mathbb C}\{\mathcal H_\beta(6k)\}_\beta$
& Basic colored theorem; scales $2,3$. \\

$\displaystyle \sum_{n=1}^{k} z^n n^q H_{2n}^{(1,2)}(1,i)\,H_{5n}^{\star,(u,v)}(s,w)$
& $\operatorname{span}_{\mathbb C}\{\mathcal H_\beta(10k)\}_\beta$
& Star sums split into strict sums. \\

$\displaystyle \sum_{n=1}^{k} z^n n^q\zeta(u,2n+1)\,\Phi(\xi,v,3n+1)\,A_{5n}^{(r)}(s)$
& $\operatorname{span}_{\mathbb C}\{\mathcal H_\beta(30k)\}_\beta$
& Tails reduce to scaled colored sums. \\

$\displaystyle \sum_{n=1}^{k} z^n n^q F_n\chi(n)\,H_{2n}^{\mathbf r}(\mathbf s)\,H_{7n}^{\star,\mathbf u}(\mathbf w)$
& $\operatorname{span}_{\mathbb C}\{\mathcal H_\beta(14k)\}_\beta$
& Recurrence and periodic factors only modify colors. \\

$\displaystyle \sum_{n=1}^{k} z^n(a n+b)^q\mathcal G_\Gamma(2n)^2\mathcal G_\Delta(3n)$
& $\operatorname{span}_{\mathbb C}\{\mathcal G_\Upsilon(6k)\}_\Upsilon$
& Affine theorem; outer affine power is a letter. \\

$\displaystyle \sum_{n=1}^{k} z^n\prod_{\nu=1}^{t}(a_\nu n+b_\nu)^{q_\nu}\mathcal G_\Gamma(4n)\mathcal G_\Delta(6n)$
& $\operatorname{span}_{\mathbb C}\{\mathcal G_\Upsilon(12k)\}_\Upsilon$
& Product of affine powers in one outer letter. \\

$\displaystyle \sum_{n=1}^{k} z^n(n^2+a^2)^q\mathcal P_\Omega(2n)\mathcal P_\Theta(5n)$
& $\operatorname{span}_{\mathbb C}\{\mathcal P_\Xi(10k)\}_\Xi$
& Polynomial theorem; Epstein--Hurwitz type outer factor. \\

$\displaystyle \sum_{n=1}^{k} z^n\prod_{\nu=1}^{t}P_\nu(n)^{q_\nu}\mathcal P_\Omega(3n)^2\mathcal P_\Theta(4n)$
& $\operatorname{span}_{\mathbb C}\{\mathcal P_\Xi(12k)\}_\Xi$
& Product of polynomial powers in one outer letter. \\
\bottomrule
\end{tabular}
\end{adjustbox}
\end{center}

The first two rows are the direct scaled versions of the colored harmonic-number
closure theorem.  The middle two rows show that no separate reduction to the
basic alphabet is needed once affine letters are admitted.  The next two rows
are the polynomial-base analogue.  The last two rows indicate how the earlier
rationally scaled harmonic and hyperharmonic reductions feed into the same
scaled-lifting mechanism.

\section{Nested sums}

We now collect the nested versions of the three finite closure principles
developed above.  The summation region is the weak simplex
\[
1\le n_1\le n_2\le\cdots\le n_r\le k,
\]
and the factor attached to the level $n_\ell$ may belong to the basic colored
alphabet, the affine-letter alphabet, or the polynomial-letter alphabet.  The
first result treats nested sums whose level-wise factors are built from scaled
generalized harmonic numbers together with elementary factors $z^n n^q$.  The next
results record the affine and polynomial-base liftings of ordinary harmonic
sums, followed by the corresponding self-closure statements for the affine and
polynomial-base alphabets.

\subsection{Nested sums in the basic colored alphabet}

We now extend Theorem~\ref{thm:scale-unification} to nested sums in which a
possibly different scaled harmonic-number summand may occur at each level of the
nesting.  Thus, instead of summing only a function of the innermost variable,
we allow products of the form
\[
F_1(n_1)F_2(n_2)\cdots F_r(n_r)
\]
over the simplex
\[
1\le n_1\le n_2\le\cdots\le n_r\le k.
\]

For a positive integer $L$, write
\begin{equation}
\label{eq:VL-definition}
\mathcal V_L(k)
:=
\operatorname{span}_{\mathbb C}
\left\{
\mathcal H_{\beta}(Lk)
\right\}_{\beta}.
\end{equation}
Thus $\mathcal V_L(k)$ is the vector space generated by generalized harmonic numbers with common upper limit $Lk$.

\begin{theorem}[Nested closure of the basic harmonic alphabet]
\label{thm:nested-scale-closure}
For $1\le \ell\le r$, let
\begin{equation}
\label{eq:level-summand}
F_\ell(n)
=
z_\ell^n n^{q_\ell}
\prod_{j=1}^{m_\ell}
\mathcal H_{\alpha_{\ell,j}}(p_{\ell,j}n)^{e_{\ell,j}},
\end{equation}
where
\[
z_\ell,q_\ell\in\mathbb C,
\qquad
p_{\ell,j}\in\mathbb Z_{>0},
\qquad
 e_{\ell,j}\in\mathbb Z_{\ge 0},
\]
and where each $\alpha_{\ell,j}$ is a word in the generalized harmonic number
alphabet.  Let
\begin{equation}
\label{eq:global-scale-L}
L=
\operatorname{lcm}
\{p_{\ell,j}:1\le \ell\le r,\ 1\le j\le m_\ell\},
\end{equation}
with $L=1$ if the displayed set is empty.  Then the nested sum
\begin{equation}
\label{eq:nested-sum-Sr}
S_r(k)
=
\sum_{n_r=1}^{k}
\sum_{n_{r-1}=1}^{n_r}
\cdots
\sum_{n_1=1}^{n_2}
\prod_{\ell=1}^{r}F_\ell(n_\ell)
\end{equation}
belongs to the finite linear span
\begin{equation}
\label{eq:nested-span-inclusion}
S_r(k)
\in
\operatorname{span}_{\mathbb C}
\left\{
\mathcal H_{\beta}(Lk)
\right\}_{\beta}.
\end{equation}
In particular, if $F_1=\cdots=F_r=F$, then
\[
\sum_{n_r=1}^{k}
\sum_{n_{r-1}=1}^{n_r}
\cdots
\sum_{n_1=1}^{n_2}
F(n_1)F(n_2)\cdots F(n_r)
\in
\operatorname{span}_{\mathbb C}
\left\{
\mathcal H_{\beta}(Lk)
\right\}_{\beta}.
\]
\end{theorem}

\begin{proof}
We first record the one-step rule needed to handle the elementary factor
$z^n n^q$.  For every word $\gamma$ and every $z,q\in\mathbb C$,
\begin{equation}
\label{eq:basic-one-step-with-z-nq}
\sum_{n=1}^{k}z^n n^q\mathcal H_\gamma(Ln)
\in
\mathcal V_L(k).
\end{equation}
Indeed, write $m=Ln$ and choose an $L$-th root $\zeta$ of $z$ when $z\ne0$.
The condition $L\mid m$ is encoded by the root-of-unity filter
\[
\mathbf 1_{L\mid m}=\frac1L\sum_{\omega^L=1}\omega^m.
\]
Thus, up to the harmless scalar $L^{-q}$, the left hand side of
\eqref{eq:basic-one-step-with-z-nq} is a finite linear combination of sums
\[
\sum_{m=1}^{Lk}(\omega\zeta)^m m^q\mathcal H_\gamma(m).
\]
The factor $(\omega\zeta)^m m^q$ is a single colored letter of weight $-q$.
If $\gamma=\emptyset$, the last sum is a depth-one harmonic sum.  If
$\gamma=(L_1,\ldots,L_d)$, the region
$Lk\ge m\ge m_1>\cdots>m_d\ge1$ splits into $m>m_1$ and $m=m_1$, giving a
finite linear combination of generalized harmonic numbers with upper limit $Lk$.
This proves \eqref{eq:basic-one-step-with-z-nq}.

We now prove the theorem by induction on the depth of the outer nesting.  Since
every $p_{\ell,j}$ divides $L$, Theorem~\ref{thm:scale-unification} expresses
each scaled factor $\mathcal H_{\alpha_{\ell,j}}(p_{\ell,j}n)$ as a finite
linear combination of generalized harmonic numbers with common upper limit $Ln$.
Powers and products of these factors are then reduced by the quasi-shuffle
product.  Consequently the harmonic part of $F_\ell(n)$ lies in
$\mathcal V_L(n)$.

Define
\[
S_0(k)=1,
\qquad
S_t(k)=\sum_{n=1}^{k}F_t(n)S_{t-1}(n)
\quad(1\le t\le r).
\]
Assume that $S_{t-1}(n)\in\mathcal V_L(n)$.  Multiplying $S_{t-1}(n)$ by the
harmonic part of $F_t(n)$ and using the quasi-shuffle product gives a finite
linear combination of terms $\mathcal H_\gamma(Ln)$.  Hence
$F_t(n)S_{t-1}(n)$ is a finite linear combination of terms
\[
z_t^n n^{q_t}\mathcal H_\gamma(Ln).
\]
Applying the one-step rule \eqref{eq:basic-one-step-with-z-nq} to each such
term gives $S_t(k)\in\mathcal V_L(k)$.  The induction gives the result at
$t=r$.
\end{proof}

\paragraph{Consequences and examples}

The table records nested examples in which scaled harmonic factors occur at
several levels.  The displayed span uses the least common multiple of all
scales appearing in the summand.

\begin{center}
\scriptsize
\renewcommand{\arraystretch}{1.08}
\setlength{\tabcolsep}{4pt}
\begin{adjustbox}{max width=\textwidth}
\begin{tabular}{@{}>{\raggedright\arraybackslash}p{0.53\textwidth}>{\centering\arraybackslash}p{0.22\textwidth}>{\raggedright\arraybackslash}p{0.17\textwidth}@{}}
\toprule
\textbf{Sum} & \textbf{Span} & \textbf{Condition} \\
\midrule
$\displaystyle \sum_{1\le n_1\le n_2\le k} z_1^{n_1}n_1^p H_{2n_1}^{(1)}(a)\,z_2^{n_2}n_2^q H_{3n_2}^{(2)}(b)$
& $\operatorname{span}_{\mathbb C}\{\mathcal H_\beta(6k)\}_\beta$
& None. \\

$\displaystyle \sum_{1\le n_1\le n_2\le n_3\le k} H_{n_1}\,H_{2n_2}^{(2)}(i)\,H_{3n_3}^{(1)}(-1)$
& $\operatorname{span}_{\mathbb C}\{\mathcal H_\beta(6k)\}_\beta$
& None. \\

$\displaystyle \begin{aligned}[t]
&\sum_{1\le n_1\le n_2\le n_3\le k}
\bigl(H_{n_1}H_{2n_1}^{(2)}(i)\bigr)
\bigl(A_{3n_2}^{(1)}(s)\bigr)^2\\[-1mm]
&\qquad\qquad\times
H_{2n_3}^{\star,(1,2)}(1,i)
\end{aligned}$
& $\operatorname{span}_{\mathbb C}\{\mathcal H_\beta(6k)\}_\beta$
& Star and alternating factors. \\

$\displaystyle \sum_{1\le n_1\le n_2\le k}\bigl(H_{2n_1}^{(1)}(-1)\bigr)^2\bigl(H_{3n_2}^{(3)}(z)\bigr)^3$
& $\operatorname{span}_{\mathbb C}\{\mathcal H_\beta(6k)\}_\beta$
& Powers reduced by quasi-shuffle. \\

$\displaystyle \begin{aligned}[t]
&\sum_{1\le n_1\le n_2\le n_3\le k}
 z_1^{n_1}n_1^{q_1}\zeta(u,2n_1+1)\\[-1mm]
&\qquad\times z_2^{n_2}n_2^{q_2}\Phi(\xi,v,3n_2+1)
\,z_3^{n_3}n_3^{q_3}H_{5n_3}^{(r)}(s)
\end{aligned}$
& $\operatorname{span}_{\mathbb C}\{\mathcal H_\beta(30k)\}_\beta$
& Hurwitz and Lerch tails at scaled limits. \\

$\displaystyle \begin{aligned}[t]
&\sum_{1\le n_1\le n_2\le n_3\le n_4\le k}
F_{n_1}\,H_{2n_1}^{(r)}(s)\,
\chi(n_2)H_{3n_2}^{\star,\mathbf u}(\mathbf w)\\[-1mm]
&\qquad\times A_{5n_3}^{(v)}(\eta)\,H_{7n_4}^{(a,b)}(\lambda,\mu)
\end{aligned}$
& $\operatorname{span}_{\mathbb C}\{\mathcal H_\beta(210k)\}_\beta$
& Recurrence and fixed-periodic factors. \\

$\displaystyle \begin{aligned}[t]
&\sum_{1\le n_1\le\cdots\le n_4\le k}
H_{2n_1}^{\star,(1,2)}(1,i)
H_{3n_2}^{(2,1)}(-1,z)\\[-1mm]
&\qquad\times h_{4n_3}^{[m]}(r;s)\,H_{5n_4}^{(u)}(w)
\end{aligned}$
& $\operatorname{span}_{\mathbb C}\{\mathcal H_\beta(60k)\}_\beta$
& Hyperharmonic scaled factor. \\

$\displaystyle \begin{aligned}[t]
&\sum_{1\le n_1\le\cdots\le n_5\le k}
\prod_{j=1}^{5} z_j^{n_j}n_j^{q_j}\,
H_{p_j n_j}^{\mathbf r_j}(\mathbf s_j)
\end{aligned}$
& $\operatorname{span}_{\mathbb C}\{\mathcal H_\beta(Lk)\}_\beta$
& $L=\lcm(p_1,\ldots,p_5)$. \\
\bottomrule
\end{tabular}
\end{adjustbox}
\end{center}

\subsection{Nested sums in the affine-letter alphabet}

Recall that an affine letter is a triple
\[
L=(\boldsymbol\rho,\sigma,\mathbf A),
\qquad
\mathbf A=((a_1,b_1),\ldots,(a_t,b_t)),
\]
with value
\[
L(n)=\sigma^n\prod_{\nu=1}^{t}(a_\nu n+b_\nu)^{-\rho_\nu}.
\]
For a word $\Gamma=(L_1,\ldots,L_d)$, put
\[
\mathcal G_\Gamma(N)
=
\sum_{N\ge m_1>\cdots>m_d\ge1}
\prod_{j=1}^{d}L_j(m_j),
\qquad
\mathcal G_\emptyset(N)=1.
\]
If $L=(\boldsymbol\rho,\sigma,\mathbf A)$ and
$M=(\boldsymbol\eta,\tau,\mathbf B)$, write
\[
L\circ M=(\boldsymbol\rho\Vert\boldsymbol\eta,\sigma\tau,
\mathbf A\Vert\mathbf B),
\]
so that $(L\circ M)(n)=L(n)M(n)$.

\begin{theorem}[Affine lifting theorem for ordinary harmonic sums]
\label{thm:nested-affine-factor-closure}
For $1\le\ell\le r$, let
\begin{equation}
\label{eq:nested-affine-level-factor}
F_\ell(n)
=
z_\ell^n
\prod_{\nu=1}^{m_\ell}(a_{\ell,\nu}n+b_{\ell,\nu})^{q_{\ell,\nu}}
\prod_{j=1}^{h_\ell}\mathcal H_{\alpha_{\ell,j}}(n)^{e_{\ell,j}},
\end{equation}
where
\[
z_\ell,a_{\ell,\nu},b_{\ell,\nu},q_{\ell,\nu}\in\mathbb C,
\qquad
 e_{\ell,j}\in\mathbb Z_{\ge0}.
\]
Assume that no affine factor vanishes on the relevant positive integers, and fix
branches for the complex powers.  Then
\begin{equation}
\label{eq:nested-affine-main-sum}
S_r(k)
=
\sum_{n_r=1}^{k}
\sum_{n_{r-1}=1}^{n_r}
\cdots
\sum_{n_1=1}^{n_2}
\prod_{\ell=1}^{r}F_\ell(n_\ell)
\end{equation}
belongs to the finite linear span
\begin{equation}
\label{eq:nested-affine-span}
S_r(k)\in
\operatorname{span}_{\mathbb C}
\{\mathcal G_\Gamma(k)\}_\Gamma,
\end{equation}
where $\Gamma$ ranges over words in affine letters.
\end{theorem}

\begin{proof}
We first note the one-step summation rule.  If $A$ is an affine letter and
$\Gamma=(L_1,\ldots,L_d)$, then
\begin{equation}
\label{eq:affine-one-step-rule}
\sum_{n=1}^{k}A(n)\mathcal G_\Gamma(n)
=
\begin{cases}
\mathcal G_A(k), & d=0,\\[1mm]
\mathcal G_{A,L_1,\ldots,L_d}(k)
+\mathcal G_{A\circ L_1,L_2,\ldots,L_d}(k), & d\ge1.
\end{cases}
\end{equation}
Indeed, for $d\ge1$ the region
$k\ge n\ge m_1>\cdots>m_d\ge1$ splits into the two cases
$n>m_1$ and $n=m_1$.

Now define recursively
\[
S_0(k)=1,
\qquad
S_t(k)=\sum_{n=1}^{k}F_t(n)S_{t-1}(n)
\quad(1\le t\le r).
\]
We prove by induction that $S_t(k)$ lies in the affine span
$\operatorname{span}_{\mathbb C}\{\mathcal G_\Gamma(k)\}_\Gamma$.
Assume this for $t-1$.  The harmonic part of $F_t(n)$,
\[
\prod_{j=1}^{h_t}\mathcal H_{\alpha_{t,j}}(n)^{e_{t,j}},
\]
is a finite linear combination of single harmonic sums by the quasi-shuffle
product; each of these is an affine sum with ordinary letters
$((u),v,((1,0)))$.  Hence the product of this harmonic part with
$S_{t-1}(n)$ is, again by quasi-shuffle, a finite linear combination of affine
sums $\mathcal G_\Delta(n)$.

The remaining level factor is the single affine letter
\[
A_t=
\bigl((-q_{t,1},\ldots,-q_{t,m_t}),z_t,
((a_{t,1},b_{t,1}),\ldots,(a_{t,m_t},b_{t,m_t}))\bigr),
\]
with the evident empty-tuple interpretation when $m_t=0$.  Therefore
$F_t(n)S_{t-1}(n)$ is a finite linear combination of terms
$A_t(n)\mathcal G_\Delta(n)$.  Applying
\eqref{eq:affine-one-step-rule} to each such term gives
$S_t(k)\in\operatorname{span}_{\mathbb C}\{\mathcal G_\Gamma(k)\}_\Gamma$.
The induction gives the result at $t=r$.
\end{proof}

\begin{corollary}[Self-closure of the affine harmonic alphabet]
\label{cor:nested-affine-self-closure}
For $1\le\ell\le r$, let
\[
F_\ell(n)
=
z_\ell^n
\prod_{\nu=1}^{m_\ell}(a_{\ell,\nu}n+b_{\ell,\nu})^{q_{\ell,\nu}}
\prod_{j=1}^{h_\ell}\mathcal G_{\Gamma_{\ell,j}}(n)^{e_{\ell,j}},
\]
where the affine factors are nonzero on the relevant positive integers and
branches for all complex powers are fixed.  Then
\[
\sum_{n_r=1}^{k}
\sum_{n_{r-1}=1}^{n_r}
\cdots
\sum_{n_1=1}^{n_2}
\prod_{\ell=1}^{r}F_\ell(n_\ell)
\in
\operatorname{span}_{\mathbb C}\{\mathcal G_\Gamma(k)\}_\Gamma.
\]
\end{corollary}

\begin{proof}
Products of the multiple affine harmonic numbers $\mathcal G_{\Gamma_{\ell,j}}(n)$ with the
same upper limit reduce, by the quasi-shuffle product for affine letters, to
finite linear combinations of single multiple affine harmonic numbers
$\mathcal G_\Delta(n)$.  Multiplication by the elementary affine factor gives
terms of the form $A(n)\mathcal G_\Delta(n)$, and the one-step rule
\eqref{eq:affine-one-step-rule} closes the outer summation.  Iterating over the
levels gives the claim.
\end{proof}

\subsection{Nested sums in the polynomial-letter alphabet}

For the polynomial version, a polynomial letter is a triple
\[
L=(\boldsymbol\rho,\sigma,\mathbf P),
\qquad
\mathbf P=(P_1,\ldots,P_t),
\qquad P_\nu\in\mathbb C[x],
\]
with value
\[
L(n)=\sigma^n\prod_{\nu=1}^{t}P_\nu(n)^{-\rho_\nu}.
\]
For a word $\Omega=(L_1,\ldots,L_d)$, put
\[
\mathcal P_\Omega(N)
=
\sum_{N\ge m_1>\cdots>m_d\ge1}
\prod_{j=1}^{d}L_j(m_j),
\qquad
\mathcal P_\emptyset(N)=1.
\]
Ordinary harmonic letters correspond to the identity polynomial
$\idpoly(x)=x$.

\begin{theorem}[Polynomial-base lifting theorem for ordinary harmonic sums]
\label{thm:nested-polynomial-factor-closure}
For $1\le\ell\le r$, let
\begin{equation}
\label{eq:nested-polynomial-level-factor}
F_\ell(n)
=
z_\ell^n
\prod_{\nu=1}^{m_\ell}P_{\ell,\nu}(n)^{q_{\ell,\nu}}
\prod_{j=1}^{h_\ell}\mathcal H_{\alpha_{\ell,j}}(n)^{e_{\ell,j}},
\end{equation}
where
\[
P_{\ell,\nu}\in\mathbb C[x],
\qquad
z_\ell,q_{\ell,\nu}\in\mathbb C,
\qquad
 e_{\ell,j}\in\mathbb Z_{\ge0}.
\]
Assume that no polynomial factor vanishes on the relevant positive integers,
and fix branches for the complex powers.  Then
\begin{equation}
\label{eq:nested-polynomial-main-sum}
T_r(k)
=
\sum_{n_r=1}^{k}
\sum_{n_{r-1}=1}^{n_r}
\cdots
\sum_{n_1=1}^{n_2}
\prod_{\ell=1}^{r}F_\ell(n_\ell)
\end{equation}
belongs to the finite linear span
\begin{equation}
\label{eq:nested-polynomial-span}
T_r(k)\in
\operatorname{span}_{\mathbb C}
\{\mathcal P_\Omega(k)\}_\Omega,
\end{equation}
where $\Omega$ ranges over words in polynomial letters.
\end{theorem}

\begin{proof}
The proof is identical in structure to the affine case.  The one-step rule is
\begin{equation}
\label{eq:polynomial-one-step-rule}
\sum_{n=1}^{k}A(n)\mathcal P_\Omega(n)
=
\begin{cases}
\mathcal P_A(k), & \Omega=\emptyset,\\[1mm]
\mathcal P_{A,L_1,\ldots,L_d}(k)
+\mathcal P_{A\circ L_1,L_2,\ldots,L_d}(k),
& \Omega=(L_1,\ldots,L_d),\ d\ge1,
\end{cases}
\end{equation}
where $A\circ L_1$ is obtained by concatenating exponents and polynomial lists
and multiplying the colors.  This follows by splitting
$k\ge n\ge m_1>\cdots>m_d\ge1$ into $n>m_1$ and $n=m_1$.

Inductively write
\[
T_0(k)=1,
\qquad
T_t(k)=\sum_{n=1}^{k}F_t(n)T_{t-1}(n).
\]
By the induction hypothesis, $T_{t-1}(n)$ is a finite linear combination of
polynomial sums $\mathcal P_\Omega(n)$.  The harmonic part of $F_t(n)$ is first
reduced by the ordinary quasi-shuffle product and is then regarded as a product
of polynomial sums with identity-polynomial letters.  Multiplying these with
$T_{t-1}(n)$ and using the quasi-shuffle product for polynomial letters gives a
finite linear combination of $\mathcal P_\Delta(n)$.

The remaining level factor is the polynomial letter
\[
A_t=
\bigl((-q_{t,1},\ldots,-q_{t,m_t}),z_t,
(P_{t,1},\ldots,P_{t,m_t})\bigr).
\]
Thus $F_t(n)T_{t-1}(n)$ is a finite linear combination of terms
$A_t(n)\mathcal P_\Delta(n)$, and \eqref{eq:polynomial-one-step-rule} places
the outer sum in the polynomial span at $k$.  This completes the induction.
\end{proof}

\begin{corollary}[Self-closure of the polynomial-base harmonic alphabet]
\label{cor:nested-polynomial-self-closure}
For $1\le\ell\le r$, let
\[
F_\ell(n)
=
z_\ell^n
\prod_{\nu=1}^{m_\ell}P_{\ell,\nu}(n)^{q_{\ell,\nu}}
\prod_{j=1}^{h_\ell}\mathcal P_{\Omega_{\ell,j}}(n)^{e_{\ell,j}},
\]
where the polynomial factors are nonzero on the relevant positive integers and
branches for all complex powers are fixed.  Then
\[
\sum_{n_r=1}^{k}
\sum_{n_{r-1}=1}^{n_r}
\cdots
\sum_{n_1=1}^{n_2}
\prod_{\ell=1}^{r}F_\ell(n_\ell)
\in
\operatorname{span}_{\mathbb C}\{\mathcal P_\Omega(k)\}_\Omega.
\]
\end{corollary}

\begin{proof}
Products of multiple polynomial-base harmonic numbers with the same upper limit reduce, by
the quasi-shuffle product for polynomial letters, to finite linear combinations
of single multiple polynomial-base harmonic numbers $\mathcal P_\Delta(n)$.  Multiplication
by the elementary polynomial letter gives terms of the form
$A(n)\mathcal P_\Delta(n)$, and the one-step rule
\eqref{eq:polynomial-one-step-rule} closes the outer summation.  Iterating over
the levels gives the result.
\end{proof}

\subsection{Consequences and examples}

Put
\[
\Delta_r(k):=\{(n_1,\ldots,n_r):1\le n_1\le\cdots\le n_r\le k\},
\]
and recall the spans
\[
\begin{aligned}
\mathscr G_k&:=\operatorname{span}_{\mathbb C}
\{\mathcal G_\Gamma(k):\Gamma\text{ a word in affine letters}\},\\
\mathscr P_k&:=\operatorname{span}_{\mathbb C}
\{\mathcal P_\Omega(k):\Omega\text{ a word in polynomial letters}\}.
\end{aligned}
\]
The table records the nested affine and polynomial closures.  Nonvanishing is
understood throughout the relevant summation range, and branches are fixed once
and for all.  By Corollaries~\ref{cor:nested-affine-self-closure} and
\ref{cor:nested-polynomial-self-closure}, the same conclusions remain valid when the
ordinary harmonic factors are replaced by multiple affine or multiple polynomial-base harmonic numbers from the corresponding alphabet.  For compactness,
$\mathscr G_k[\mathcal B]$ and $\mathscr P_k[\mathcal B]$ denote the finite-generated
subspans determined by the indicated merge-closed letter sets.  In the concrete
rows below, put
\[
\begin{aligned}
\mathcal B_{\rm aff}&:=\{((-3),2,((2,1))),((2),-1,((1,0))),
((2),-1,((3,-1))),((1),i,((1,0)))\},\\
\mathcal B_{\rm poly}&:=\{((2),-1,(x^2+1)),((3),i,(x)),
((-2,1),3,(x^2+x+1,x^2+4)),((1),-1,(x))\}.
\end{aligned}
\]

\begin{center}
\scriptsize
\renewcommand{\arraystretch}{1.08}
\setlength{\tabcolsep}{4pt}
\begin{adjustbox}{max width=\textwidth}
\begin{tabular}{@{}>{\raggedright\arraybackslash}p{0.55\textwidth}>{\centering\arraybackslash}p{0.20\textwidth}>{\raggedright\arraybackslash}p{0.17\textwidth}@{}}
\toprule
\textbf{Sum} & \textbf{Span} & \textbf{Condition} \\
\midrule
$\displaystyle \sum_{1\le n_1\le n_2\le k}
2^{n_1}(2n_1+1)^3H_{n_1}^{(2)}(-1)
\frac{(-1)^{n_2}}{(3n_2-1)^2}H_{n_2}^{(1)}(i)$
& $\mathscr G_k[\mathcal B_{\rm aff}]$
& Concrete affine two-level case. \\

$\displaystyle \sum_{n_2=1}^{k}\sum_{n_1=1}^{n_2}
 z_1^{n_1}(a_1n_1+b_1)^p\mathcal H_{\alpha_1}(n_1)
 z_2^{n_2}(a_2n_2+b_2)^q\mathcal H_{\alpha_2}(n_2)$
& $\mathscr G_k$
& $a_i n+b_i\ne0$. \\

$\displaystyle \begin{aligned}[t]
&\sum_{\Delta_3(k)}\frac{z_1^{n_1}H_{n_1}^{(r)}(s)}{(2n_1+1)^p}
\,z_2^{n_2}(3n_2-1)^q\Phi(\xi,u,n_2+1)\\[-1mm]
&\qquad\times z_3^{n_3}(5n_3+2)^\lambda A_{n_3}^{(v)}(w)
\end{aligned}$
& $\mathscr G_k$
& Nonzero affine factors, $\xi\ne0,1$. \\

$\displaystyle \sum_{\Delta_r(k)}\prod_{\ell=1}^{r}
 z_\ell^{n_\ell}
 \prod_{\nu=1}^{m_\ell}(a_{\ell,\nu}n_\ell+b_{\ell,\nu})^{q_{\ell,\nu}}
 \prod_j\mathcal G_{\Gamma_{\ell,j}}(n_\ell)^{e_{\ell,j}}$
& $\mathscr G_k$
& Affine self-closure. \\

$\displaystyle \sum_{1\le n_1\le n_2\le k}
\frac{(-1)^{n_1}}{(n_1^2+1)^2}H_{n_1}^{(3)}(i)
\,3^{n_2}\frac{(n_2^2+n_2+1)^2}{n_2^2+4}H_{n_2}^{(1)}(-1)$
& $\mathscr P_k[\mathcal B_{\rm poly}]$
& Concrete polynomial two-level case. \\

$\displaystyle \begin{aligned}[t]
&\sum_{\Delta_3(k)}
\frac{z_1^{n_1}H_{n_1}^{(r)}(s)}{(n_1^2+a^2)^p}
\,z_2^{n_2}(n_2^2+n_2+1)^qH_{n_2}^{\star,\mathbf u}(\mathbf w)\\[-1mm]
&\qquad\times \frac{2n_3\,z_3^{n_3}A_{n_3}^{(v)}(\eta)}{(n_3^2+b^2)^{\lambda+1}}
\end{aligned}$
& $\mathscr P_k$
& Polynomial factors nonzero. \\

$\displaystyle \sum_{\Delta_r(k)}\prod_{\ell=1}^{r}
 z_\ell^{n_\ell}
 \prod_{\nu=1}^{m_\ell}P_{\ell,\nu}(n_\ell)^{q_{\ell,\nu}}
 \prod_j\mathcal H_{\alpha_{\ell,j}}(n_\ell)^{e_{\ell,j}}$
& $\mathscr P_k$
& $P_{\ell,\nu}(n)\ne0$. \\

$\displaystyle \sum_{\Delta_r(k)}\prod_{\ell=1}^{r}
 z_\ell^{n_\ell}
 \prod_{\nu=1}^{m_\ell}P_{\ell,\nu}(n_\ell)^{q_{\ell,\nu}}
 \prod_j\mathcal P_{\Omega_{\ell,j}}(n_\ell)^{e_{\ell,j}}$
& $\mathscr P_k$
& Polynomial self-closure. \\
\bottomrule
\end{tabular}
\end{adjustbox}
\end{center}

\section{Infinite sums}
\label{sec:infinite-sums}

The closure theorems of the preceding sections are finite statements.  They
reduce finite Euler-type sums to finite linear combinations of harmonic sums
with one of the three alphabets used in this paper: the basic colored alphabet,
the affine alphabet, and the polynomial alphabet.  Infinite sums are obtained by
passing to the limit of these finite reductions.

For a basic word
\[
\alpha=((r_1,s_1),\ldots,(r_d,s_d)),
\]
we write
\begin{equation}
\label{eq:section7-H-definition}
\mathcal H_{\alpha}(N)
=
\sum_{N\ge n_1>\cdots>n_d\ge 1}
\prod_{j=1}^{d}\frac{s_j^{n_j}}{n_j^{r_j}},
\qquad
\mathcal H_{\emptyset}(N)=1.
\end{equation}
When the limit exists, it is the corresponding multiple polylogarithmic value,
\begin{equation}
\label{eq:section7-multiple-polylogarithm-definition}
\Li_{r_1,\ldots,r_d}(s_1,\ldots,s_d)
:=
\sum_{n_1>\cdots>n_d\ge 1}
\prod_{j=1}^{d}\frac{s_j^{n_j}}{n_j^{r_j}}.
\end{equation}
Thus the finite closure space has a natural limiting value space.

\subsection{Limit value spaces}
\label{subsec:finite-closures-limiting-value-spaces}

All limit passages in this subsection are understood under termwise convergence
of the displayed finite reduction.  That is, whenever a finite identity has the
form
\begin{equation}
\label{eq:finite-reduction-general-limit}
S_N=
\sum_{\nu=1}^{M}c_\nu X_\nu(N),
\end{equation}
we pass to the limit only in the case where each $X_\nu(N)$ has an ordinary
finite limit.  Then
\begin{equation}
\label{eq:termwise-limit-general}
\lim_{N\to\infty}S_N=
\sum_{\nu=1}^{M}c_\nu\lim_{N\to\infty}X_\nu(N).
\end{equation}
The point of the closure theorems is that the possible $X_\nu$ belong to a
controlled finite alphabet, so that the limiting constants also belong to a
controlled value space.  Sufficient convergence criteria for the basic, affine,
and polynomial-base multiple polylogarithms used in these limit passages are
given in \ref{app:convergence-mpl-variants}.

\begin{center}
\small
\renewcommand{\arraystretch}{1.25}
\setlength{\tabcolsep}{5pt}
\begin{adjustbox}{max width=\textwidth}
\begin{tabular}{@{}p{0.25\textwidth}p{0.30\textwidth}p{0.34\textwidth}@{}}
\toprule
\textbf{finite object} & \textbf{limit at $N\to\infty$} & \textbf{value space} \\
\midrule
Uncolored basic multiple harmonic numbers
$\mathcal H_{((r_1,1),\ldots,(r_d,1))}(N)$
&
$\zeta(r_1,\ldots,r_d)$
&
Multiple zeta values, under the usual admissibility condition. \\
\addlinespace
Colored basic multiple harmonic numbers
$\mathcal H_{((r_1,s_1),\ldots,(r_d,s_d))}(N)$
&
$\Li_{r_1,\ldots,r_d}(s_1,\ldots,s_d)$
&
Multiple polylogarithms and their special values. \\
\addlinespace
Multiple affine harmonic numbers
$\mathcal G_{\Gamma}(N)$
&
$\displaystyle \Li^{\mathrm{aff}}_{\Gamma}:=\lim_{N\to\infty}\mathcal G_{\Gamma}(N)$
&
Affine multiple-polylogarithmic constants. \\
\addlinespace
Multiple polynomial-base harmonic numbers
$\mathcal P_{\Omega}(N)$
&
$\displaystyle \Li^{\mathrm{pb}}_{\Omega}:=\lim_{N\to\infty}\mathcal P_{\Omega}(N)$
&
Polynomial-base multiple-polylogarithmic constants. \\
\bottomrule
\end{tabular}
\end{adjustbox}
\end{center}

\begin{proposition}[Limit passage from finite closure]
\label{prop:limit-passage-finite-closure}
Let $\mathscr A_N$ be one of the finite closure spaces generated by basic,
affine, or polynomial letters.  Suppose that a finite Euler-type sum has been
reduced to
\begin{equation}
\label{eq:finite-closure-limit-prop}
S_N=
\sum_{\nu=1}^{M}c_\nu X_{\omega_\nu}(N),
\qquad X_{\omega_\nu}(N)\in \mathscr A_N.
\end{equation}
If each $X_{\omega_\nu}(N)$ has an ordinary finite limit, then
\begin{equation}
\label{eq:finite-closure-limit-prop-conclusion}
\lim_{N\to\infty}S_N=
\sum_{\nu=1}^{M}c_\nu\mathcal L(X_{\omega_\nu}),
\qquad
\mathcal L(X_{\omega}):=\lim_{N\to\infty}X_{\omega}(N).
\end{equation}
For the basic colored alphabet, the values $\mathcal L(X_\omega)$ are ordinary
multiple polylogarithms.  For uncolored basic letters they are multiple zeta
values.  For affine and polynomial letters they are, respectively, affine and
polynomial-base analogues of multiple-polylogarithmic constants.
\end{proposition}

\begin{proof}
The assertion is the ordinary termwise limit of the finite identity
\eqref{eq:finite-closure-limit-prop}.  The description of the value spaces is
exactly the limiting interpretation of the corresponding alphabet.
\end{proof}

For example, the basic depth-one reduction
\begin{equation}
\label{eq:section7-depth-one-finite}
\sum_{n=1}^{N}\frac{z^n}{n^q}H_n^{(r)}(s)
=
\mathcal H_{((q,z),(r,s))}(N)
+
\mathcal H_{(q+r,zs)}(N)
\end{equation}
gives
\begin{equation}
\label{eq:section7-depth-one-infinite}
\sum_{n=1}^{\infty}\frac{z^n}{n^q}H_n^{(r)}(s)
=
\Li_{q,r}(z,s)+\Li_{q+r}(zs)
\end{equation}
whenever the two limiting terms converge.  Products of harmonic sums are first
reduced by the quasi-shuffle product and then the same limiting rule is applied.

The general affine and polynomial cases are conceptually identical.  A finite
closure theorem gives
\[
S_N=\sum_\nu c_\nu\mathcal G_{\Gamma_\nu}(N)
\quad\text{or}\quad
S_N=\sum_\nu c_\nu\mathcal P_{\Omega_\nu}(N),
\]
and the infinite value, in the termwise-convergent case, is obtained by
replacing the finite sums by their corresponding affine or polynomial-base
limits.

\subsection{Divergence peeling}
\label{subsec:letter-peeling-cancellation}

The preceding subsection only covers reductions whose limiting terms converge
one by one.  In applications, the finite reduction may contain terms whose
individual limits diverge, although the original infinite sum is convergent.
A useful way to remove such artificial divergences is to work in the finite
quasi-shuffle algebra before taking the limit.  This gives a purely algebraic
``letter peeling'' procedure.

Let letters be denoted by $a=(r,s)$ and let the merged letter be
\[
a\circ b=(r+r',ss')
\quad\text{for}\quad b=(r',s').
\]
For words $u,v$, let $u*v$ denote the quasi-shuffle product.  Thus, for every
finite $N$,
\begin{equation}
\label{eq:finite-quasi-shuffle-product}
\mathcal H_u(N)\mathcal H_v(N)
=
\sum_{w\in u*v}m_w\mathcal H_w(N),
\end{equation}
where $m_w$ is the multiplicity with which the word $w$ occurs.  If $a$ is a
single letter and $\beta$ is a nonempty word, then one of the words in
$a*\beta$ is the concatenation $a\beta$.  Hence
\begin{equation}
\label{eq:finite-peeling-identity}
\mathcal H_{a\beta}(N)
=
\mathcal H_a(N)\mathcal H_\beta(N)
-
\sum_{\substack{w\in a*\beta\\w\ne a\beta}}
 m_w\mathcal H_w(N).
\end{equation}
This identity peels the first letter $a$ off the word $a\beta$.

\begin{proposition}[Finite letter-peeling criterion]
\label{prop:finite-letter-peeling-criterion}
Let
\begin{equation}
\label{eq:peeling-input-expression}
E_N=
\sum_{w}c_w\mathcal H_w(N)
\end{equation}
be a finite reduction of a convergent infinite sum.  Suppose that repeated use
of the finite identity \eqref{eq:finite-peeling-identity}, followed by
quasi-shuffle reduction of products, transforms $E_N$ into
\begin{equation}
\label{eq:peeling-output-expression}
E_N=R_N+D_N,
\end{equation}
where $R_N$ is a finite linear combination of harmonic sums with ordinary
limits and $D_N\to0$ as $N\to\infty$.  Then
\begin{equation}
\label{eq:peeling-output-limit}
\lim_{N\to\infty}E_N
=
\lim_{N\to\infty}R_N.
\end{equation}
In particular, if the peeling process rewrites $E_N$ as a finite linear
combination of convergent words after cancellation of the divergent parts, then
the infinite sum is obtained by taking the ordinary limits of those remaining
words.
\end{proposition}

\begin{proof}
All transformations are identities in the finite quasi-shuffle algebra, hence
hold for each finite $N$.  Taking $N\to\infty$ in
\eqref{eq:peeling-output-expression} gives the result, since the terms in
$R_N$ have ordinary limits and $D_N\to0$.
\end{proof}

The criterion should be read as a sufficient algebraic test.  It does not
assign values to divergent multiple polylogarithms.  The products and
cancellations are performed at finite $N$; only after the divergent part has
cancelled do we pass to the limit.

\begin{example}[A depth-one subtraction]
\label{ex:peeling-depth-one-subtraction}
Let $p>1$.  Consider
\begin{equation}
\label{eq:peeling-example-one-target}
S=
\sum_{n=1}^{\infty}\frac{H_n^{(p)}-\zeta(p)}{n}.
\end{equation}
Since $H_N^{(p)}\to\zeta(p)$ and
$(H_N^{(p)}-\zeta(p))H_N=o(1)$, this is the limit of
\begin{equation}
\label{eq:peeling-example-one-truncation}
S_N=
\sum_{n=1}^{N}\frac{H_n^{(p)}-H_N^{(p)}}{n}.
\end{equation}
The finite closure gives
\begin{equation}
\label{eq:peeling-example-one-finite}
S_N
=
\mathcal H_{(1,1),(p,1)}(N)
+
\mathcal H_{(p+1,1)}(N)
-
\mathcal H_{(1,1)}(N)\mathcal H_{(p,1)}(N).
\end{equation}
The term $\mathcal H_{(1,1),(p,1)}(N)$ and the product involving
$\mathcal H_{(1,1)}(N)$ are not separately convergent.  Peeling uses the finite
stuffle identity
\begin{equation}
\label{eq:peeling-example-one-stuffle}
\mathcal H_{(1,1)}(N)\mathcal H_{(p,1)}(N)
=
\mathcal H_{(1,1),(p,1)}(N)
+
\mathcal H_{(p,1),(1,1)}(N)
+
\mathcal H_{(p+1,1)}(N).
\end{equation}
Substitution into \eqref{eq:peeling-example-one-finite} gives the convergent
finite identity
\begin{equation}
\label{eq:peeling-example-one-reduced}
S_N=-\mathcal H_{(p,1),(1,1)}(N).
\end{equation}
Therefore
\begin{equation}
\label{eq:peeling-example-one-answer}
\sum_{n=1}^{\infty}\frac{H_n^{(p)}-\zeta(p)}{n}
=
-\zeta(p,1).
\end{equation}
For $p=2$ this gives $-\zeta(2,1)=-\zeta(3)$.
\end{example}

\begin{example}[Peeling a depth-two harmonic factor]
\label{ex:peeling-depth-two-subtraction}
Let $p>1$ and $q\in\mathbb N$.  Put
\[
\mathcal H_{p,q}(N):=
\mathcal H_{(p,1),(q,1)}(N).
\]
Then
\begin{equation}
\label{eq:peeling-example-two-target}
T=
\sum_{n=1}^{\infty}\frac{\mathcal H_{p,q}(n)-\zeta(p,q)}{n}
\end{equation}
is obtained as the limit of
\begin{equation}
\label{eq:peeling-example-two-truncation}
T_N=
\sum_{n=1}^{N}\frac{\mathcal H_{p,q}(n)-\mathcal H_{p,q}(N)}{n}.
\end{equation}
The summation over $n$ gives
\begin{equation}
\label{eq:peeling-example-two-finite}
T_N=
\mathcal H_{(1,1),(p,1),(q,1)}(N)
+
\mathcal H_{(p+1,1),(q,1)}(N)
-
\mathcal H_{(1,1)}(N)\mathcal H_{(p,1),(q,1)}(N).
\end{equation}
The product expands as
\begin{align}
\label{eq:peeling-example-two-stuffle}
\mathcal H_{(1,1)}(N)\mathcal H_{(p,1),(q,1)}(N)
={}&
\mathcal H_{(1,1),(p,1),(q,1)}(N)
+
\mathcal H_{(p,1),(1,1),(q,1)}(N)\notag\\
&+
\mathcal H_{(p,1),(q,1),(1,1)}(N)
+
\mathcal H_{(p+1,1),(q,1)}(N)
+
\mathcal H_{(p,1),(q+1,1)}(N).
\end{align}
Hence all terms with the divergent first letter cancel and
\begin{equation}
\label{eq:peeling-example-two-reduced}
T_N=
-
\mathcal H_{(p,1),(1,1),(q,1)}(N)
-
\mathcal H_{(p,1),(q,1),(1,1)}(N)
-
\mathcal H_{(p,1),(q+1,1)}(N).
\end{equation}
Passing to the limit gives
\begin{equation}
\label{eq:peeling-example-two-answer}
\sum_{n=1}^{\infty}\frac{\mathcal H_{p,q}(n)-\zeta(p,q)}{n}
=
-
\zeta(p,1,q)-\zeta(p,q,1)-\zeta(p,q+1).
\end{equation}
This example illustrates the general pattern: subtracting the limiting value of
a convergent inner word produces divergent finite pieces, but the first-letter
peeling identity moves the leading divergent letter into a product where it
cancels.
\end{example}

\subsection{Consequences and examples}
\label{subsec:infinite-sums-consequences-examples}

The following table records representative convergent infinite sums obtained
from the finite closure theorems.  Only the resulting value space is shown; no
explicit evaluation is displayed.  The examples deliberately include star sums,
periodic factors, Dirichlet characters, multiple affine harmonic numbers, and
multiple polynomial-base harmonic numbers.  We use
\begin{align*}
\mathcal Z_w&:=\operatorname{span}_{\mathbb Q}
 \{\text{ordinary multiple zeta values of weight }w\},\\
\mathcal M(\mathcal C)&:=\operatorname{span}_{\mathbb C}
 \{\Li_{\mathbf r}(\mathbf s):s_j\in\langle\mathcal C\rangle\},\\
\mathcal V_{\rm aff}(\Lambda)&:=\operatorname{span}_{\mathbb C}
 \left\{\Li^{\mathrm{aff}}_{\Gamma}:
 \begin{array}{l}
 \Gamma\text{ is a word in the merge-closed}\\
 \text{affine alphabet generated by }\Lambda
 \end{array}\right\},\\
\mathcal V_{\rm pb}(\Pi)&:=\operatorname{span}_{\mathbb C}
 \left\{\Li^{\mathrm{pb}}_{\Omega}:
 \begin{array}{l}
 \Omega\text{ is a word in the merge-closed}\\
 \text{polynomial alphabet generated by }\Pi
 \end{array}\right\}.
\end{align*}
Here \(\langle\mathcal C\rangle\) is the finite multiplicative color set
generated by \(\mathcal C\), enlarged when necessary by roots of unity coming
from residue-class filters or scaled upper limits.  We write
\(\boldsymbol\mu_m\) for the set of \(m\)-th roots of unity,
\(r_M(t)\) for the least non-negative residue of \(t\) modulo
\(M\), and \(\chi_M\) for a fixed Dirichlet character modulo \(M\).  For the
affine and polynomial examples below let
\[
\begin{gathered}
L_1=((2),1,((2,1))),\qquad
L_2=((1),-1,((3,-1))),\qquad
L_3=((3/2),1,((1,2))),\\
P_1=((1),1,(x^2+1)),
\qquad P_2=((2),-1,(x^2+x+1)),
\qquad P_3=((3/2),1,(x^2+2x+2)).
\end{gathered}
\]

\begin{center}
\scriptsize
\renewcommand{\arraystretch}{1.12}
\setlength{\tabcolsep}{5pt}
\begin{adjustbox}{max width=\textwidth}
\begin{tabular}{@{}>{\raggedright\arraybackslash}p{0.60\textwidth}>{\raggedright\arraybackslash}p{0.32\textwidth}@{}}
\toprule
\textbf{Convergent infinite sum} & \textbf{Resulting value space} \\
\midrule
\(\displaystyle
\sum_{n=1}^{\infty}
\frac{H_n^{\star,(2,1,2)}(1,1,1)\,H_n^{(3)}}{n^4}
\)
& \(\displaystyle \mathcal Z_{12}\). \\
\addlinespace
\(\displaystyle
\sum_{n=1}^{\infty}
\frac{(-1)^n H_n^{\star,(1,2)}(-1,i)\,H_n^{(3)}(1/2)}{n^{5+i/4}}
\)
& \(\displaystyle \mathcal M\bigl(\{-1,i,\tfrac12\}\bigr)\), with complex exponents. \\
\addlinespace
\(\displaystyle
\sum_{n=1}^{\infty}
\frac{r_5(3n+2)\,H_n^{\star,(2,1)}(1,-1)\,H_n^{(2)}(1/3)}{n^6}
\)
& \(\displaystyle \mathcal M\bigl(\{-1,\tfrac13\}\cup\boldsymbol\mu_5\bigr)\). \\
\addlinespace
\(\displaystyle
\sum_{n=1}^{\infty}
\frac{\chi_7(n)\,H_n^{(1,2)}(i,-1/2)\,H_n^{\star,(2,1)}(1,-1)}{n^5}
\)
& \(\displaystyle \mathcal M\bigl(\{i,-\tfrac12,-1\}\cup\boldsymbol\mu_7\bigr)\). \\
\addlinespace
\(\displaystyle
\sum_{n=1}^{\infty}
\frac{(1/2)^n H_{2n}^{\star,(1,2)}(1,-1)\,H_{3n}^{(2,1)}(i,1/3)}{n^4}
\)
& \(\displaystyle \mathcal M\bigl(\{\tfrac12,-1,i,\tfrac13\}\cup\boldsymbol\mu_6\bigr)\). \\
\addlinespace
\(\displaystyle
\sum_{n=1}^{\infty}
\frac{(1/3)^n\mathcal G_{L_1,L_2}(n)\,\mathcal G_{L_3}(n)}{n^{2+i/3}}
\)
& \(\displaystyle \mathcal V_{\rm aff}\bigl(\{((2+i/3),\tfrac13,((1,0))),L_1,L_2,L_3\}\bigr)\). \\
\addlinespace
\(\displaystyle
\sum_{n=1}^{\infty}
\frac{\chi_4(n)\,r_3(n)\,\mathcal G_{L_1}(n)\,H_n^{\star,(1,1)}(-1,1/2)}{n^5}
\)
& Affine multiple-polylogarithmic span generated by \(L_1\), \(\boldsymbol\mu_{12}\), \(-1\), and \(1/2\). \\
\addlinespace
\(\displaystyle
\sum_{n=1}^{\infty}
\frac{(2/3)^n\mathcal P_{P_1,P_2}(n)\,H_n^{\star,(1,2)}(-1,1/2)}{(n^2+1)^2}
\)
& \(\displaystyle \mathcal V_{\rm pb}\bigl(\{((2),\tfrac23,(x^2+1)),P_1,P_2,((1),-1,x),((2),\tfrac12,x)\}\bigr)\). \\
\addlinespace
\(\displaystyle
\sum_{n=1}^{\infty}
\frac{\mathcal P_{P_2}(n)\,\mathcal P_{P_3}(n)\,H_n^{(2)}(-1/3)}{(n^2+n+1)^3}
\)
& \(\displaystyle \mathcal V_{\rm pb}\bigl(\{((3),1,(x^2+x+1)),P_2,P_3,((2),-\tfrac13,x)\}\bigr)\). \\
\addlinespace
\(\displaystyle
\sum_{n_1=1}^{\infty}\sum_{n_2=1}^{n_1}
\frac{\chi_5(n_1)(1/2)^{n_2}H_{n_1}^{\star,(1,2)}(-1,1)\,\mathcal P_{P_1}(n_2)}{n_1^3n_2^2}
\)
& Mixed colored and polynomial-base multiple-polylogarithmic span generated by \(\boldsymbol\mu_5\), \(-1\), \(1/2\), and \(P_1\). \\
\bottomrule
\end{tabular}
\end{adjustbox}
\end{center}

The two additional mechanisms used for integer-affine multiple-polylogarithmic
values and for shifted-denominator cancellations are useful but somewhat
tangential to the main closure framework.  They are therefore recorded separately
in \ref{app:additional-infinite-sum-reductions}.

\section{Normal forms and further reductions}

The closure theorems above place the resulting finite sums in harmonic-sum,
multiple affine harmonic number, or multiple polynomial-base harmonic-number spaces, and place their
convergent limits in corresponding multiple-polylogarithmic spaces.  This
section records only the most useful post-processing reductions.  They are not
part of the closure proofs; they are normalization steps applied after the sums
have already been brought into the appropriate space.

\subsection{Finite sums}

For finite sums, the first simplification is to detect special multiple harmonic numbers that collapse to ordinary harmonic numbers, polynomials in $n$, or
polynomials in ordinary harmonic numbers.  For example,
\begin{equation}
\label{eq:section8-all-zero-simplification}
\mathcal H_{\underbrace{(0,1),\ldots,(0,1)}_{d}}(n)
=
\binom{n}{d}.
\end{equation}
The all-one word is the elementary symmetric polynomial in
$1,1/2,\ldots,1/n$; equivalently,
\begin{equation}
\label{eq:section8-all-one-bell}
\mathcal H_{\underbrace{(1,1),\ldots,(1,1)}_{d}}(n)
=
\frac{1}{d!}
Y_d\left(
H_n,
-1!H_n^{(2)},
2!H_n^{(3)},
\ldots,
(-1)^{d-1}(d-1)!H_n^{(d)}
\right),
\end{equation}
where $Y_d$ is the complete exponential Bell polynomial.  In depth two one also
has, for instance,
\begin{equation}
\label{eq:section8-depth-two-one-one}
\mathcal H_{((1,1),(1,1))}(n)
=
\frac12\left(H_n^2-H_n^{(2)}\right),
\end{equation}
and words containing zero powers frequently reduce to lower-depth expressions,
such as
\begin{equation}
\label{eq:section8-zero-r-special}
\mathcal H_{((0,1),(r,1))}(n)
=
nH_n^{(r)}-H_n^{(r-1)},
\qquad
\mathcal H_{((r,1),(0,1))}(n)
=
H_n^{(r-1)}-H_n^{(r)}.
\end{equation}
Star variants are treated similarly by resolving weak inequalities into strict
ones and merging equal indices.  Similar reductions for star harmonic sums and
their generalized versions are described in Ablinger's thesis and implemented in
the HarmonicSums framework \citep{AblingerThesis2012,Ablinger2014}.

The second finite normal form comes from the quasi-shuffle product.  Products of
finite harmonic sums may be rewritten as linear combinations of single finite
harmonic sums by the stuffle/quasi-shuffle algebra \citep{Hoffman1992}.  In
depth one,
\begin{equation}
\label{eq:section8-depth-one-stuffle}
\mathcal H_{(a,x)}(n)\mathcal H_{(b,y)}(n)
=
\mathcal H_{((a,x),(b,y))}(n)
+
\mathcal H_{((b,y),(a,x))}(n)
+
\mathcal H_{(a+b,xy)}(n).
\end{equation}
After fixing an order on the alphabet, the Chen--Fox--Lyndon factorization gives
a canonical factorization of words into Lyndon words \citep{ChenFoxLyndon1958};
Duval's algorithm gives an efficient construction of this factorization
\citep{Duval1983}.  The triangularity of the quasi-shuffle product then permits
recursive rewriting of non-Lyndon words in terms of products of Lyndon generators
and lower correction terms.  Thus a finite expression
\begin{equation}
\label{eq:section8-finite-linear-H-simplification}
\sum_{\alpha} c_{\alpha}\mathcal H_{\alpha}(n)
\end{equation}

can be converted to a Lyndon normal form, after which cancellations and
low-depth reductions are usually easier to see.

\begin{example}[Chen--Fox--Lyndon simplification in a finite Euler sum]
A small example illustrates the practical effect of this normal form.  Let
\[
  H_n=\mathcal H_{(1,1)}(n),
  \qquad
  A_n=A_n^{(1)}=\sum_{j=1}^{n}\frac{(-1)^{j-1}}{j}
  =-\mathcal H_{(1,-1)}(n),
\]
and consider
\begin{equation}
\label{eq:cfl-example-sum}
  S(k)=\sum_{n=1}^{k}H_n A_n^{2}.
\end{equation}
The direct finite-convolution reduction gives the following linear combination
of multiple harmonic numbers:
\begin{align}
S(k)=
&\,\mathcal H_{(3,1)}(k)
+\mathcal H_{((0,1),(3,1))}(k)
+2\mathcal H_{((1,-1),(2,-1))}(k)
+\mathcal H_{((1,1),(2,1))}(k) \notag\\
&+2\mathcal H_{((2,-1),(1,-1))}(k)
+\mathcal H_{((2,1),(1,1))}(k)
+2\mathcal H_{((0,1),(1,-1),(2,-1))}(k) \notag\\
&+\mathcal H_{((0,1),(1,1),(2,1))}(k)
+2\mathcal H_{((0,1),(2,-1),(1,-1))}(k)
+\mathcal H_{((0,1),(2,1),(1,1))}(k) \notag\\
&+2\mathcal H_{((1,-1),(1,-1),(1,1))}(k)
+2\mathcal H_{((1,-1),(1,1),(1,-1))}(k)
+2\mathcal H_{((1,1),(1,-1),(1,-1))}(k) \notag\\
&+2\mathcal H_{((0,1),(1,-1),(1,-1),(1,1))}(k)
+2\mathcal H_{((0,1),(1,-1),(1,1),(1,-1))}(k) \notag\\
&+2\mathcal H_{((0,1),(1,1),(1,-1),(1,-1))}(k).
\label{eq:cfl-example-raw-H}
\end{align}
After applying the quasi-shuffle relations and rewriting the result in
Chen--Fox--Lyndon normal form, the same sum collapses to the much smaller
boundary expression
\begin{equation}
\label{eq:cfl-example-closed-form}
  S(k)
  =\frac12\left(
    2(-1)^k A_k(H_k-1)
    +A_k^{2}\bigl(-1-2k+2(k+1)H_k\bigr)
    +H_k^{(2)}
  \right).
\end{equation}
Thus the normal form replaces a sixteen-term expression involving words of
length up to four by an expression involving only ordinary and alternating
depth-one harmonic numbers and the boundary factor $(-1)^k$.
\end{example}

\subsection{Infinite sums}

After a justified passage to a convergent limit, the simplest reduction occurs
when all colors are $1$:
\begin{equation}
\label{eq:section8-mpl-to-mzeta}
\Li_{r_1,\ldots,r_d}(1,\ldots,1)
=
\zeta(r_1,\ldots,r_d),
\end{equation}
provided the defining series is convergent.  These multiple zeta values may then
be reduced, in known weights, by shuffle, stuffle, duality, sum relations, and
data-mine bases; for example,
\begin{equation}
\label{eq:section8-mzeta-examples}
\zeta(2,1)=\zeta(3),
\qquad
\zeta(3,1)=\frac14\zeta(4).
\end{equation}
For systematic reduction data and bases for multiple zeta values, we cite the
multiple-zeta data mine \citep{BluemleinBroadhurstVermaseren2010}.

If the colors are roots of unity, the limiting values are level-$M$ multiple
polylogarithmic constants.  For example, with
$\mu_M=e^{2\pi i/M}$ and
$s_j\in\{1,\mu_M,\ldots,\mu_M^{M-1}\}$, one obtains constants of the form
\begin{equation}
\label{eq:section8-level-M-mpl}
\Li_{r_1,\ldots,r_d}(s_1,\ldots,s_d).
\end{equation}
These may be further reduced only when the corresponding level, weight, and
basis data are available.  We keep this point structural here; for algorithms and relations for multiple
polylogarithms at algebraic arguments, including root-of-unity specializations,
we refer to the work of K.~C.~Au \citep{Au2022}.

A useful bridge to classical functions is obtained by rewriting multiple
polylogarithms as generalized polylogarithms.  With
\begin{equation}
\label{eq:section8-G-definition}
G(a_1,\ldots,a_m;z)
=
\int_0^z\frac{dt}{t-a_1}G(a_2,\ldots,a_m;t),
\qquad
G(;z)=1,
\end{equation}
one has the standard conversion
\begin{equation}
\label{eq:section8-Li-to-G}
\Li_{r_1,\ldots,r_d}(z_1,\ldots,z_d)
=
(-1)^d
G\left(
\underbrace{0,\ldots,0}_{r_1-1},\frac1{z_1},
\underbrace{0,\ldots,0}_{r_2-1},\frac1{z_1z_2},
\ldots,
\underbrace{0,\ldots,0}_{r_d-1},\frac1{z_1\cdots z_d};1
\right),
\end{equation}
away from the usual singular cases and after compatible branch choices have
been fixed.  For low weights and suitable alphabets, generalized
polylogarithms can often be reduced to logarithms, classical polylogarithms,
zeta values, and products thereof.  In particular, GPLs through weight four
can be reduced to logarithms, classical polylogarithms $\Li_n$,
and $\Li_{2,2}$, as in the reduction framework of
\citet{FrellesvigTommasiniWever2016}.
For related algorithms and symbol-based reductions, see also
\citet{VollingaWeinzierl2005} and \citet{DuhrGanglRhodes2012}.

For letters in the harmonic-polylogarithm alphabet $\{-1,0,1\}$, the above
$G$-functions reduce directly to harmonic polylogarithms.  With the usual HPL
kernels
\[
 f_0(t)=\frac1t,
 \qquad
 f_1(t)=\frac1{1-t},
 \qquad
 f_{-1}(t)=\frac1{1+t},
\]
the conversion is
\begin{equation}
\label{eq:section8-G-to-HPL}
G(a_1,\ldots,a_m;z)
=
(-1)^{\#\{j:a_j=1\}}
H_{a_1,\ldots,a_m}(z),
\qquad a_j\in\{-1,0,1\}.
\end{equation}
Thus a root-of-unity or special-point multiple-polylogarithm reduction may pass through the chain
\[
\begin{aligned}
&\text{multiple polylogarithms}
  \longrightarrow G\text{-functions}
  \longrightarrow \operatorname{HPLs}\\
&\hspace{2.5em}\longrightarrow
  \text{multiple-zeta and harmonic-polylogarithm constants},
\end{aligned}
\]
with the final reductions again controlled by known HPL and multiple-zeta data,
including the data mine \citep{RemiddiVermaseren2000,BluemleinBroadhurstVermaseren2010}.

\section{Limitations and future directions}

\paragraph{Infinite limits}
The infinite-limit part of the present framework is still incomplete.  At
present, it applies mainly when the resulting multiple polylogarithms are
convergent, or become convergent after peeling off divergent first letters.  A
full theory is still needed for limits such as
\[
  \lim_{k\to\infty}\sum_{n=1}^{k} c_n\,\mathcal H_{\alpha}(n),
\]
especially when several shifted or scaled upper limits occur:
\[
  \mathcal H_{\alpha_1}(a_1n+b_1),\ldots,
  \mathcal H_{\alpha_r}(a_rn+b_r).
\]
Such a theory would have to combine synchronization, asymptotic expansion,
regularization, and cancellation of divergent parts.  Vermaseren's work on
harmonic sums and Mellin transforms, and the HarmonicSums framework of Ablinger
and collaborators, provide useful precedents
\citep{Vermaseren1999,Ablinger2014,AblingerBluemleinRaabSchneider2014}.

\paragraph{Reduction of constants}
A second challenge is reduction.  Infinite limits may produce multiple zeta
values, multiple polylogarithms, multiple affine zeta/polylog constants, and
multiple polynomial-base zeta/polylog constants, but the resulting expressions need not
be minimal.  For classical multiple zeta values of positive integral weight,
Bl\"umlein, Broadhurst, and Vermaseren give explicit proven reductions through
weight $22$, and extend the computational evidence, using modular arithmetic,
to weight $30$ \citep{BluemleinBroadhurstVermaseren2010}.  For multiple
polylogarithms with positive integral weights and root-of-unity colors, Au gives
a systematic iterated-integral method for converting many such values into
colored multiple zeta values and then expanding them in an explicit basis.  In
particular, his implementation favors bases built from elementary constants such
as ordinary logarithms, zeta values, Dirichlet $L$-values, and classical
polylogarithms like $\Li_n(\alpha)$, together with higher-depth constants when
needed \citep{Au2022}.  However, comparable reduction theory for complex
powers, complex colors, affine letters, and polynomial-base letters is still
missing.  Such reductions would make the final closed forms substantially more
useful.

\paragraph{Binomial and hypergeometric weights}
A third limitation is that binomial, inverse-binomial, and more general
hypergeometric weights are not yet included.  Thus sums involving factors such
as
\[
  \binom{2n}{n},\qquad
  \binom{3n}{n},\qquad
  \binom{2n}{n}^{-1}
\]
or, more generally,
\[
  z^n\frac{(a_1)_n\cdots(a_r)_n}{(b_1)_n\cdots(b_s)_n}\,
  \mathcal H_{\alpha}(n)
\]
lie outside the present framework.  Such weights occur naturally in
hypergeometric series, binomial Euler sums, and Feynman-integral expansions.
Weinzierl developed algorithms for binomial and inverse-binomial sums as an
extension of nested-sum techniques \citep{Weinzierl2004}.  More directly,
Ablinger studied infinite nested sums involving Pochhammer symbols and reduced
them, via generating series and root-valued or cyclotomic iterated integrals, to
constants such as powers of $\pi$, $\log 2$, and zeta values
\citep{Ablinger2019}.  Incorporating such weights would require enlarging the
alphabet from colored harmonic letters to hypergeometric, factorial-ratio, or
root-valued letters.

\paragraph{Multi-lattice zeta kernels}
Another useful extension would be to include Witten, Barnes, Shintani, and
conical-zeta-type kernels.  For example, one would like to treat sums such as
\[
  \sum_{i,j,k\ge 1}
  \frac{\mathcal H_{\alpha}(ijk)}{(2i+3j+4k)^4},
\]
or variants with several linear forms in the denominator.  This would move the
framework from one-dimensional Euler-type convolution sums toward multi-lattice
sums with harmonic-type factors.  Witten and Mordell--Tornheim type zeta
functions already show how positive linear forms arise in multiple zeta-type
problems \citep{Matsumoto2006,ZhaoZhou2011}; conical zeta values provide a
geometric generalization of multiple zeta values through sums over cones
\citep{GuoPaychaZhang2014}.  Adding harmonic-type factors such as
$\mathcal H_{\alpha}(ijk)$ would likely require diagonal, partition, or
cone-decomposition methods beyond the present quasi-shuffle framework.

\paragraph{$q$-analogues}
Finally, one can ask for a $q$-deformation of the theory.  One may replace
ordinary powers and harmonic factors by $q$-integer analogues, for instance
\[
  [n]_{\mathfrak q}=\frac{1-\mathfrak q^n}{1-\mathfrak q},
  \qquad
  \sum_{n\ge 1} z^n [n]_{\mathfrak q}^{\lambda}\,\mathcal H^{(\mathfrak q)}_{\alpha}(n).
\]
The infinite limits would then be expected to involve $q$-multiple zeta values,
$q$-multiple polylogarithms, and $q$-Euler sums.  Bradley introduced multiple
$q$-zeta values and showed that they satisfy $q$-stuffle and $q$-shuffle type
products \citep{Bradley2005}.  Hessami Pilehrood, Hessami Pilehrood, and Zhao
developed $q$-analogues of families of multiple harmonic-number and multiple
zeta-star identities, recovering the classical identities as $q\to 1$
\citep{HessamiPilehroodHessamiPilehroodZhao2016}.  Extending the present
framework in this direction would require a $q$-quasi-shuffle algebra,
$q$-limit analysis, and compatible reduction rules for the resulting
$q$-constants.

\section{Conclusion}

This paper has developed an alphabetic convolution framework for Euler-type
sums.  Finite nested sums are treated as word sums over multiplicative
alphabets of one-variable letters.  When the alphabet is closed under
pointwise multiplication, the word sums are stable under the stuffle, or
quasi-shuffle, product.  Products, convolutions, shifts, scalings, and nested
domains can therefore be converted into finite linear combinations of
structured word sums.

The framework was carried out for three elementary monoidal alphabets.  The
colored harmonic alphabet covers ordinary, alternating, colored, and finite
multiple harmonic numbers.  The affine alphabet covers shifted and rationally
shifted denominators, residue-class filters, level constructions, truncated
Hurwitz-type sums, and truncated Lerch-type sums.  The polynomial-base alphabet
covers polynomial denominator families, including finite Mathieu-type,
one-dimensional Epstein--Hurwitz-type, polynomial zeta, and polynomial-base
polylogarithmic sums.

The same closure mechanism handles scaled indices and nested summation domains.
Scaling is separated from the summand letters, while nested domains are treated
by repeated one-variable convolution with diagonal and boundary terms.  Passing
to convergent infinite limits gives the corresponding constants; in borderline
cases, divergence peeling and cancellation identities isolate the convergent
parts.

Thus the construction gives a systematic framework for studying a large class
of harmonic-sum extensions.  It supports closed-form evaluations by first
placing a sum in its natural finite alphabetic space, and then applying
Lyndon-basis normalization, multiple-zeta reductions, root-of-unity reductions,
affine specializations, or polynomial factorizations inside that space.  In the
same way, it provides a common language for reduction problems, structural
identities, and the special functions generated by these sums.  Open problems
include regularized limiting theories, larger monoidal alphabets such as
hypergeometric or recurrence-defined letters, and efficient canonical forms for
these larger spaces.

\appendix
\addtocontents{toc}{\protect\setcounter{tocdepth}{1}}
\renewcommand{\thesubsection}{\Alph{section}.\arabic{subsection}}
\renewcommand{\thetheorem}{\Alph{section}.\arabic{theorem}}

\section{Further monoidal alphabets}
\label{app:further-monoidal-alphabets}

The construction in the main text only requires a product-closed alphabet of
one-variable scalar sequences.  The
following two examples illustrate how this mechanism extends beyond the three
principal alphabets used in the main text.  In each example we first give the
finite superfunction, then record two simple evaluations as values of that
superfunction.  Further analytic reduction is a separate step and depends on the
larger class of constants and iterated integrals generated by the chosen
alphabet.

\subsection{Central-binomial letters}
\label{app:central-binomial-letters}

Let
\[
C_n=\binom{2n}{n}.
\]
For parameters \(r,s,\beta\in\mathbb C\), and with a fixed branch for
\(C_n^\beta\), define
\begin{equation}
\label{eq:central-binomial-letter}
\phi_{r,s,\beta}(n)=s^n C_n^{\beta} n^{-r}.
\end{equation}
These letters are closed under pointwise multiplication, since
\begin{equation}
\label{eq:central-binomial-letter-product}
\phi_{r,s,\beta}\phi_{r',s',\beta'}
=
\phi_{r+r',ss',\beta+\beta'}.
\end{equation}
Thus coincident indices in the stuffle product are resolved by the same
collision rule as before, with the extra exponent \(\beta\) added on collision.
The associated finite nested superfunction is
\begin{equation}
\label{eq:central-binomial-superfunction}
\mathfrak B_N\bigl((r_1,s_1,\beta_1),\ldots,(r_d,s_d,\beta_d)\bigr)
=
\sum_{N\ge n_1>\cdots>n_d\ge1}
\prod_{j=1}^{d}
s_j^{n_j}\binom{2n_j}{n_j}^{\beta_j}n_j^{-r_j},
\qquad
\mathfrak B_N(\emptyset)=1.
\end{equation}
Here \(\beta=1\) gives central-binomial weights, \(\beta=-1\) gives
inverse-binomial weights, and mixed binomial powers are handled without changing
the formal word algebra.

Two elementary depth-one limits are most naturally written first as
superfunction values:
\begin{align}
\lim_{N\to\infty}\mathfrak B_N\bigl((1,\tfrac14,1)\bigr)
&=
\sum_{n=1}^{\infty}\binom{2n}{n}\frac{1}{4^n n}
=
2\log 2,
\label{eq:central-binomial-example-one}\\[4pt]
\lim_{N\to\infty}\mathfrak B_N\bigl((2,\tfrac14,1)\bigr)
&=
\sum_{n=1}^{\infty}\binom{2n}{n}\frac{1}{4^n n^2}
=
\zeta(2)-2\log^2 2.
\label{eq:central-binomial-example-two}
\end{align}
These formulas show only the first reduction step in a very simple case.  For
more general binomially weighted nested sums, the analytic reduction commonly
leads to iterated integrals over root-valued alphabets and to the binomial and
inverse-binomial sum technologies developed in the literature
\citep{Weinzierl2004,Ablinger2014,AblingerBluemleinRaabSchneider2014,Ablinger2019}.

\subsection{Fractional powers of harmonic-number bases}
\label{app:fractional-harmonic-power-letters}

Let
\[
H_n^{(p)}=\sum_{k=1}^{n}\frac{1}{k^p},
\qquad
H_0^{(p)}=0,
\]
and fix a finite list of orders \(p_1,\ldots,p_m\).  For
\(r,s\in\mathbb C\) and exponent vectors
\(\mathbf q=(q_1,\ldots,q_m)\),
\(\mathbf u=(u_1,\ldots,u_m)\), define
\begin{equation}
\label{eq:fractional-harmonic-letter}
\psi_{r,s,\mathbf q,\mathbf u}(n)
=
s^n n^{-r}
\prod_{j=1}^{m}\bigl(H_n^{(p_j)}\bigr)^{q_j}
\prod_{j=1}^{m}\bigl(H_{n-1}^{(p_j)}\bigr)^{u_j},
\end{equation}
where the relevant branches are fixed on the summation range.  For the
examples below all fractional powers are taken on the positive real branch, and
\(0^a=0\) for the positive exponents that occur at \(n=1\).  The product rule is
again monoidal:
\begin{equation}
\label{eq:fractional-harmonic-letter-product}
\psi_{r,s,\mathbf q,\mathbf u}
\psi_{r',s',\mathbf q',\mathbf u'}
=
\psi_{r+r',ss',\mathbf q+\mathbf q',\mathbf u+\mathbf u'}.
\end{equation}
The corresponding finite nested superfunction is
\begin{equation}
\label{eq:fractional-harmonic-superfunction}
\mathfrak F_N\bigl((r_1,s_1,\mathbf q_1,\mathbf u_1),\ldots,
(r_d,s_d,\mathbf q_d,\mathbf u_d)\bigr)
=
\sum_{N\ge n_1>\cdots>n_d\ge1}
\prod_{i=1}^{d}\psi_{r_i,s_i,\mathbf q_i,\mathbf u_i}(n_i),
\qquad
\mathfrak F_N(\emptyset)=1.
\end{equation}

For \(p_1=2\), the following finite identity is a difference of two values of
this same superfunction:
\begin{align}
\mathfrak F_N\bigl((0,1,(\tfrac12),(0))\bigr)
-
\mathfrak F_N\bigl((0,1,(0),(\tfrac12))\bigr)
&=
\sum_{n=1}^{N}\left(\bigl(H_n^{(2)}\bigr)^{1/2}
-
\bigl(H_{n-1}^{(2)}\bigr)^{1/2}\right)
\notag\\
&=
\bigl(H_N^{(2)}\bigr)^{1/2}.
\label{eq:fractional-harmonic-finite-one}
\end{align}
Taking \(N\to\infty\) gives the closed-form superfunction limit
\begin{equation}
\label{eq:fractional-harmonic-example-one}
\lim_{N\to\infty}
\left[
\mathfrak F_N\bigl((0,1,(\tfrac12),(0))\bigr)
-
\mathfrak F_N\bigl((0,1,(0),(\tfrac12))\bigr)
\right]
=
\zeta(2)^{1/2}
=
\frac{\pi}{\sqrt6}.
\end{equation}

Similarly, for \((p_1,p_2)=(2,3)\),
\begin{align}
&\mathfrak F_N\bigl((0,1,(\tfrac12,\tfrac13),(0,0))\bigr)
-
\mathfrak F_N\bigl((0,1,(0,0),(\tfrac12,\tfrac13))\bigr)
\notag\\
&\qquad =
\sum_{n=1}^{N}\left(
\bigl(H_n^{(2)}\bigr)^{1/2}\bigl(H_n^{(3)}\bigr)^{1/3}
-
\bigl(H_{n-1}^{(2)}\bigr)^{1/2}\bigl(H_{n-1}^{(3)}\bigr)^{1/3}
\right)
\notag\\
&\qquad =
\bigl(H_N^{(2)}\bigr)^{1/2}\bigl(H_N^{(3)}\bigr)^{1/3}.
\label{eq:fractional-harmonic-finite-two}
\end{align}
Hence
\begin{equation}
\label{eq:fractional-harmonic-example-two}
\lim_{N\to\infty}
\left[
\mathfrak F_N\bigl((0,1,(\tfrac12,\tfrac13),(0,0))\bigr)
-
\mathfrak F_N\bigl((0,1,(0,0),(\tfrac12,\tfrac13))\bigr)
\right]
=
\zeta(2)^{1/2}\zeta(3)^{1/3}
=
\frac{\pi}{\sqrt6}\,\zeta(3)^{1/3}.
\end{equation}

The two displayed identities are deliberately elementary telescoping examples.
They show that products of fractional powers of harmonic-number bases can be
placed inside the same finite monoidal superfunction framework.  When the
exponents are nonnegative integers, the usual stuffle expansion reduces products
of harmonic numbers to multiple harmonic numbers, and many convergent Euler sums
then reduce to multiple zeta values or their alternating analogues
\citep{FlajoletSalvy1998,Xu2017,XuWang2020,XuWang2022}.  For arbitrary fractional
powers, no such general reduction is implied; the monoidal alphabet gives the
closed formal nested-sum space, while simplification to known constants or
functions remains a separate analytic problem.

\section{Convergence criteria for multiple polylogarithmic variants}
\label{app:convergence-mpl-variants}

We record convergence tests for the basic, affine, and polynomial-base multiple
polylogarithms used in the paper.  The purpose of this appendix is not to discuss
all possible borderline or regularized cases, but to give a simple and useful
sufficient criterion for the series that occur in the main text.

Throughout, let
\[
\Lambda_k=\lambda_1\lambda_2\cdots \lambda_k,
\qquad 1\leq k\leq d,
\]
and let \(\mathbf 1_{\mathcal P}\) denote \(1\) or \(0\) according as the
condition \(\mathcal P\) is true or false.

\subsection{The basic case}

Consider the ordinary multiple polylogarithm
\[
\Li_{s_1,\ldots,s_d}(\lambda_1,\ldots,\lambda_d)
=
\sum_{n_1>\cdots>n_d\geq 1}
\frac{\lambda_1^{n_1}\cdots \lambda_d^{n_d}}
     {n_1^{s_1}\cdots n_d^{s_d}} .
\]
A standard convergence criterion is the following.  The series converges if,
for every \(1\leq k\leq d\),
\[
|\Lambda_k|<1
\]
or
\[
|\Lambda_k|=1
\quad\text{and}\quad
\operatorname{Re}(s_1+\cdots+s_k)
>
\sum_{\ell=1}^k \mathbf 1_{\Lambda_\ell=1}.
\]
Equivalently, at the \(k\)-th boundary one compares the real part of the partial
weight
\[
s_1+\cdots+s_k
\]
with the number of non-oscillatory partial products \(\Lambda_\ell=1\) among
\(\Lambda_1,\ldots,\Lambda_k\).

This criterion recovers the usual multiple-zeta-value condition.  Indeed, if
\(\lambda_1=\cdots=\lambda_d=1\), then \(\Lambda_\ell=1\) for every \(\ell\),
and the condition becomes
\[
\operatorname{Re}(s_1+\cdots+s_k)>k,
\qquad 1\leq k\leq d.
\]

\subsection{The affine case}

Let
\[
L_j(n)=\alpha_j n+\beta_j,
\qquad \alpha_j\neq 0,
\]
and assume that \(L_j(n)\neq 0\) for all positive integers \(n\).  Consider the
affine multiple polylogarithm
\[
\Li^{\mathrm{aff}}_{\mathbf s}
(\boldsymbol\lambda;\boldsymbol\alpha,\boldsymbol\beta)
=
\sum_{n_1>\cdots>n_d\geq 1}
\prod_{j=1}^d
\frac{\lambda_j^{n_j}}
     {(\alpha_j n_j+\beta_j)^{s_j}} .
\]
A sufficient condition for convergence is the same as in the ordinary case:
for every \(1\leq k\leq d\),
\[
|\Lambda_k|<1
\]
or
\[
|\Lambda_k|=1
\quad\text{and}\quad
\operatorname{Re}(s_1+\cdots+s_k)
>
\sum_{\ell=1}^k \mathbf 1_{\Lambda_\ell=1}.
\]
Thus affine shifts do not change the tail-convergence condition.  They affect
the finite initial values of the summand, and may introduce excluded singular
values if some \(\alpha_j n+\beta_j\) vanishes, but they do not change the
degree-one decay at infinity.

\subsection{The polynomial-base case}

Now let \(P_j(n)\in\mathbb C[n]\) be nonzero polynomials of positive degrees
\[
\nu_j=\deg P_j\geq 1,
\]
and assume that
\[
P_j(n)\neq 0
\qquad\text{for all } n\geq 1 .
\]
Consider the polynomial-base multiple polylogarithm
\[
\Li^{\mathrm{pb}}_{\mathbf s}
(\boldsymbol\lambda;\mathbf P)
=
\sum_{n_1>\cdots>n_d\geq 1}
\prod_{j=1}^d
\frac{\lambda_j^{n_j}}{P_j(n_j)^{s_j}} .
\]
A sufficient condition for convergence is that, for every \(1\leq k\leq d\),
\[
|\Lambda_k|<1
\]
or
\[
|\Lambda_k|=1
\quad\text{and}\quad
\sum_{j=1}^k \nu_j\,\operatorname{Re}(s_j)
>
\sum_{\ell=1}^k \mathbf 1_{\Lambda_\ell=1}.
\]
Equivalently, the ordinary convergence test is applied after replacing each
weight \(s_j\) by the effective weight
\[
\nu_j s_j=(\deg P_j)s_j .
\]
The affine case is the special case \(\nu_j=1\) for all \(j\).

\subsection{Absolute convergence}

The preceding tests allow conditional convergence coming from oscillation of the
partial products \(\Lambda_k\).  A stronger but sometimes more convenient absolute
convergence test is obtained by replacing \(\mathbf 1_{\Lambda_\ell=1}\) by
\(\mathbf 1_{|\Lambda_\ell|=1}\).  Thus, in the polynomial-base case, absolute
convergence is guaranteed if, for every \(1\leq k\leq d\),
\[
|\Lambda_k|<1
\]
or
\[
|\Lambda_k|=1
\quad\text{and}\quad
\sum_{j=1}^k \nu_j\,\operatorname{Re}(s_j)
>
\sum_{\ell=1}^k \mathbf 1_{|\Lambda_\ell|=1}.
\]
The corresponding affine and ordinary absolute convergence tests are obtained
by setting \(\nu_j=1\).

\subsection{Proof}

We first recall the standard proof mechanism in the ordinary case.  Put
\[
n_j=m_j+m_{j+1}+\cdots+m_d,
\qquad 1\leq j\leq d,
\]
where \(m_1,\ldots,m_d\geq 1\).  Then
\[
\lambda_1^{n_1}\cdots \lambda_d^{n_d}
=
\Lambda_1^{m_1}\Lambda_2^{m_2}\cdots \Lambda_d^{m_d}.
\]
Thus the summation directions are controlled by the partial products
\(\Lambda_k\).  If \(|\Lambda_k|<1\), the \(m_k\)-direction has exponential
decay.  If \(|\Lambda_k|=1\) but \(\Lambda_k\neq 1\), the corresponding
geometric partial sums are bounded, and one obtains convergence by Dirichlet
summation, provided the relevant power factor tends to zero.  If \(\Lambda_k=1\),
there is no oscillation in that direction, and one loses one full power of
summability.

For the first \(k\) directions, the total polynomial decay is governed by
\[
\operatorname{Re}(s_1+\cdots+s_k).
\]
The number of non-oscillatory unit directions among the first \(k\) variables is
\[
\sum_{\ell=1}^k \mathbf 1_{\Lambda_\ell=1}.
\]
The ordinary criterion follows by applying the multidimensional Dirichlet test,
or equivalently by repeated Abel summation in the variables
\(m_1,\ldots,m_d\).

For the affine case, we have, as \(n\to\infty\),
\[
(\alpha_j n+\beta_j)^{-s_j}
=
\alpha_j^{-s_j} n^{-s_j}\left(1+O\left(\frac1n\right)\right).
\]
Thus each affine denominator has the same tail order as \(n^{s_j}\).  Replacing
\(n_j^{-s_j}\) by \((\alpha_j n_j+\beta_j)^{-s_j}\) therefore changes the
ordinary summand only by lower-order terms.  These lower-order terms are handled
by the same Dirichlet summation argument, with strictly stronger decay.  Hence
the ordinary criterion remains a sufficient convergence criterion for the affine
series.

For the polynomial-base case, write
\[
P_j(n)=c_j n^{\nu_j}\left(1+O\left(\frac1n\right)\right),
\qquad c_j\neq 0,
\qquad \nu_j=\deg P_j .
\]
Therefore
\[
P_j(n)^{-s_j}
=
c_j^{-s_j} n^{-\nu_j s_j}
\left(1+O\left(\frac1n\right)\right).
\]
Consequently, the polynomial denominator contributes the same leading decay as
\(n^{\nu_j s_j}\).  Applying the ordinary convergence criterion to the effective
weights
\[
s_j\longmapsto \nu_j s_j
\]
gives the stated sufficient condition:
\[
\sum_{j=1}^k \nu_j\,\operatorname{Re}(s_j)
>
\sum_{\ell=1}^k \mathbf 1_{\Lambda_\ell=1}
\]
whenever \(|\Lambda_k|=1\).  The lower-order terms in the asymptotic expansion
of \(P_j(n)^{-s_j}\) again have stronger decay and are covered by the same
Abel-summation argument.  This proves the polynomial-base sufficient convergence
test.

\section{Additional infinite-sum reductions}
\label{app:additional-infinite-sum-reductions}

This appendix collects two technical refinements for infinite sums.  They are
kept separate from Section~\ref{sec:infinite-sums} because they are evaluation
mechanisms for special limiting situations rather than part of the main finite
alphabetic closure theory.  Bailey and McPhedran recently gave general formulas
for a large class of Euler sums built from ordinary harmonic numbers
\citep{BaileyMcPhedran2026}; the reductions below are in a related spirit, but
are formulated for colored multiple harmonic numbers rather than only depth-one
harmonic-number factors.

\subsection{Integer-affine multiple-polylogarithm values}
\label{subsec:integer-affine-denominators-direct-mpl-values}

The affine limiting constants described in Subsection~\ref{subsec:finite-closures-limiting-value-spaces}
are the natural closure space for general affine letters.  There is nevertheless
an important integer-affine subcase in which no new affine constants are needed.
When the denominator is an integral affine factor raised to a positive integer
power, a beta-kernel integral converts the sum to an iterated-integral
calculation.  The final constants are ordinary multiple polylogarithm values at
complex arguments; in the algebraic-parameter subcase, these arguments are
algebraic.

\begin{theorem}[Integer-affine denominator method]
\label{thm:integer-affine-denominator-mpl}
Let $p,a,b,q\in\mathbb N$, let $z\in\mathbb C$, and let
\[
 A_N=\prod_{j=1}^{m}\mathcal H_{\alpha_j}(N)^{e_j},
 \qquad e_j\in\mathbb N_0,
\]
where the colors occurring in the words $\alpha_j$ are complex and all relevant
prefix products are nonzero.  If the series converges, then
\begin{equation}
\label{eq:integer-affine-denominator-target}
 \sum_{n=1}^{\infty}\frac{z^n A_{pn}}{(an+b)^q}
\end{equation}
is a finite linear combination of multiple polylogarithm values at complex
arguments.  If the colors and $z$ are algebraic, then the resulting arguments
and coefficients are algebraic after adjoining the roots introduced by the
construction.
\end{theorem}

\begin{proof}
By the quasi-shuffle product, $A_N$ is a finite linear combination of single
colored multiple harmonic numbers $\mathcal H_{\alpha}(N)$.  Hence it is enough to
consider one such summand.  The denominator is represented by the beta kernel
\begin{equation}
\label{eq:integer-affine-beta-kernel}
 \frac{1}{(an+b)^q}
 =
 \frac{p^q}{(q-1)!}
 \int_0^1 x^{p(an+b)-1}(-\log x)^{q-1}\,dx.
\end{equation}
After interchanging summation and integration in the convergence range, the
inner sum has the form
\[
 \sum_{n\ge1}(\rho x^a)^{pn}\mathcal H_{\alpha}(pn),
 \qquad \rho^p=z.
\]
The condition that the upper limit is a multiple of $p$ is imposed by the
root-of-unity filter
\begin{equation}
\label{eq:integer-affine-root-filter}
 \sum_{n\ge1}(\rho x^a)^{pn}\mathcal H_{\alpha}(pn)
 =
 \frac1p\sum_{\ell=0}^{p-1}
 F_{\alpha}(\omega_p^\ell\rho x^a),
 \qquad
 \omega_p=e^{2\pi i/p},
\end{equation}
where
\[
 F_{\alpha}(t)=\sum_{N\ge1}t^N\mathcal H_{\alpha}(N)
\]
is the ordinary generating function of the finite harmonic sums.  This
generating function is a rational factor times a generalized polylogarithm
$G(W_\alpha;t)$, where the word $W_\alpha$ is obtained from the powers and
prefix products of the colors in $\alpha$.

Thus the original sum becomes a finite linear combination of integrals whose
integrands are products of powers of $\log x$, rational factors of the form
$(1-\xi x^a)^{-1}$, and generalized polylogarithms with endpoint $x^a$, with
complex letters.  The endpoint-power expansion rewrites the generalized
polylogarithms with endpoint $x^a$ as generalized polylogarithms with endpoint
$x$ and complex letters.  The factors $(1-\xi x^a)^{-1}$ decompose into simple
GPL letters after adjoining the $a$th roots of $\xi^{-1}$, and powers of
$\log x$ are GPLs with repeated zero letters.  Products are reduced by the
shuffle product, and the integration from $0$ to $1$ appends one further letter.
Therefore the result is a finite linear combination of GPL values at $1$ with
complex letters.  Finally, GPL values at $1$ are equivalent to multiple
polylogarithm values at complex arguments.
\end{proof}

This theorem gives a direct multiple-polylogarithm evaluation route for the special sums
\[
 \sum_{n=1}^{\infty}\frac{z^n}{(an+b)^q}
 \prod_{j=1}^{m}\mathcal H_{\alpha_j}(pn)^{e_j},
\]
whereas the general affine-letter framework naturally leads to affine
multiple-polylogarithmic constants.  In practice, the difference is that the
integer-affine denominator is absorbed by an integral kernel, rather than by
adjoining a new affine summation letter.  No algebraicity of the colors is
needed for this mechanism; algebraicity only specializes the resulting multiple-polylogarithm
arguments to algebraic points.

As an example of explicit evaluation, let
\[
 \alpha=((1,1/2),(2,1/4)).
\]
Then
\begin{align}
\sum_{n=1}^{\infty}\frac{\mathcal H_{\alpha}(n)}{(n+3)^2}
&=\frac1{16}\Biggl(
162+28G(2,0,8;1)-16G(0,1,2,0,8;1)
-805\log\frac87 \notag\\
&\qquad\qquad
+64\Li_{2,2}\!\left(\frac12,\frac14\right)
-444\Li_2\!\left(\frac18\right)
\Biggr).
\label{eq:app-C-integer-affine-example}
\end{align}

\subsection{Shifted denominator cancellations}
\label{subsec:shifted-denominator-cancellation}

A related phenomenon occurs for denominator power one.  Individual summands of
the form
\[
 \sum_{n\ge1}\frac{\mathcal H_{\alpha}(pn)}{an+b}
\]
may diverge, but suitable linear combinations of shifted denominators can be
convergent.  The method below evaluates the convergent combination as a whole;
it does not assign separate values to the divergent pieces.

\begin{theorem}[Shifted-denominator cancellation]
\label{thm:shifted-denominator-cancellation}
Let $p\in\mathbb N$, let $a_j,b_j\in\mathbb N$, and let
$\lambda_j\in\mathbb C$.  Put
\[
 C_j=\frac{\lambda_j}{a_j},
 \qquad
 \theta_j=\frac{b_j}{a_j}.
\]
Let
\[
 A_N=\prod_{i=1}^{m}\mathcal H_{\beta_i}(N)^{e_i},
 \qquad e_i\in\mathbb N_0,
\]
with complex colors whose relevant prefix products are nonzero.  Assume that
the combined series converges and that
\begin{equation}
\label{eq:shifted-denominator-cancellation-condition}
 \sum_{j=1}^{M}C_j=0.
\end{equation}
Then
\begin{equation}
\label{eq:shifted-denominator-cancellation-target}
 \sum_{n=1}^{\infty}
 A_{pn}\sum_{j=1}^{M}\frac{\lambda_j}{a_jn+b_j}
\end{equation}
is a finite linear combination of multiple polylogarithm values at complex
arguments.  If all colors and coefficients are algebraic, then the resulting
arguments and coefficients are algebraic after adjoining the roots introduced
by the construction.
\end{theorem}

\begin{proof}
Again the quasi-shuffle product reduces $A_N$ to a finite linear combination of
single colored multiple harmonic numbers $\mathcal H_\beta(N)$.  We therefore
consider one such term.  Choose $\ell$ to be a common multiple of $p$ and of the
denominators of the rational shifts $\theta_j$, and put $k=\ell/p$.  For each
shift,
\[
 \frac{1}{n+\theta_j}
 =
 \ell\int_0^1 x^{\ell(n+\theta_j)-1}\,dx.
\]
After summing the shifted denominators, the polynomial factor
\[
 P(x)=\sum_{j=1}^{M}C_jx^{\ell\theta_j}
\]
appears.  The condition \eqref{eq:shifted-denominator-cancellation-condition}
is exactly the condition $P(1)=0$.

The root-of-unity filter gives
\[
 \sum_{n\ge1}\mathcal H_\beta(pn)x^{\ell n}
 =
 \frac1p\sum_{r=0}^{p-1}F_\beta(\omega_p^r x^k),
\]
where $F_\beta(t)=\sum_{N\ge1}t^N\mathcal H_\beta(N)$ is the GPL generating
function of the finite harmonic sums.  Consequently the required combination is
a finite linear combination of integrals of the form
\begin{equation}
\label{eq:shifted-denominator-integral}
 \int_0^1
 \frac{P(x)}{x(1-\omega_p^r x^k)}
 G(W_\beta;\omega_p^r x^k)\,dx.
\end{equation}
For $r=0$, the possible pole at $x=1$ is removable because $P(1)=0$ and
$1-x^k=(1-x)(1+x+\cdots+x^{k-1})$.  For $r\ne0$, there is no pole at $x=1$.
Thus the cancellation has removed the common divergent $1/n$ tail before the
integral is evaluated.

As in Theorem~\ref{thm:integer-affine-denominator-mpl}, the endpoint-power
expansion rewrites GPLs with endpoint $x^k$ in endpoint $x$, the rational
kernel decomposes into simple complex GPL letters, products are reduced by the
shuffle product, and integration from $0$ to $1$ gives GPL values at $1$.
These values are then converted to multiple polylogarithm values at complex
arguments.
\end{proof}

In depth zero, the theorem reduces to the classical cancellation
\[
 \sum_{n=1}^{\infty}\sum_{j=1}^{M}\frac{C_j}{n+\theta_j}
 =
 -\sum_{j=1}^{M}C_j H_{\theta_j},
 \qquad
 \sum_{j=1}^{M}C_j=0,
\]
where $H_x=\psi(x+1)+\gamma$.  For rational shifts this is a
combination of logarithms of algebraic numbers, hence a weight-one multiple-polylogarithm value.

As a second example of explicit evaluation, let
\[
 \alpha=((1,1/2),(2,1/4)).
\]
Then
\begin{align}
&\sum_{n=1}^{\infty}
\mathcal H_{\alpha}(n)
\left(\frac{1}{2n+1}-\frac{1}{2n+3}\right)
\notag\\
&\quad=2\Biggl(
-4+8\sqrt2\,\operatorname{arccoth}(2\sqrt2)
+(1+\sqrt2)G(-\sqrt2,0,-2\sqrt2;1)\notag\\
&\qquad\quad
+(1+\sqrt2)G(-\sqrt2,0,2\sqrt2;1)
-(\sqrt2-1)G(\sqrt2,0,-2\sqrt2;1)\notag\\
&\qquad\quad
-(\sqrt2-1)G(\sqrt2,0,2\sqrt2;1)
+\log\frac{49}{64}+\Li_2\!\left(\frac18\right)
\Biggr).
\label{eq:app-C-shifted-cancellation-example}
\end{align}

Together with the termwise-convergent and peeling criteria in
Section~\ref{sec:infinite-sums}, these two reductions give useful sufficient
mechanisms for passing from finite identities to infinite evaluations.  They do
not amount to a complete asymptotic or regularization theory for all convergent
infinite sums.

\paragraph{AI-assisted preparation disclosure}
AI-assisted tools were used in a limited way to improve the clarity and presentation of selected parts of the manuscript. The mathematical ideas, definitions, results, proofs, computations, and overall substance are the author's own. All AI-assisted suggestions were carefully reviewed, and the author remains fully responsible for the accuracy, originality, and final form of the article.

\section*{Declaration of competing interest}
The author declares no competing interests.

\section*{Acknowledgements}
I would like to thank Devendra Kapadia, Tigran Ishkhanyan, and Roger Germundsson of Wolfram Research for many helpful discussions and valuable feedback.

\clearpage
\phantomsection

\end{document}


\maketitle

This supplementary file collects examples of explicit evaluations in the notation used in the paper.  The section and subsection numbers correspond to those in the paper.  In particular,
\(\mathcal H_N^{(r_1,\ldots,r_d)}(s_1,\ldots,s_d)\) denotes the corresponding strict colored multiple harmonic number with upper limit \(N\); when the color list is omitted, all colors are \(1\).  The symbols \(\mathcal A\), \(h^{[m]}\), \(\mathcal G\), and \(\mathcal P\) denote the corresponding alternating, hyperharmonic, affine-letter, and polynomial-letter harmonic numbers, respectively.  The affine and polynomial identities display the original letter data.

\section*{3.1 Basic finite convolution}

\subsection*{Identity 1}
\begin{dmath*}
\sum_{n=1}^{k} z^{n}\,\allowbreak n^{q}\,\allowbreak \bigl(\mathcal{H}_{n}^{(r)}\!(s)\bigr)^{3} = \mathcal{H}_{k}^{(-q +\allowbreak 3\,\allowbreak r)}\!(s^{3}\,\allowbreak z) +\allowbreak \mathcal{H}_{k}^{(-q,3\,\allowbreak r)}\!(z,\allowbreak s^{3}) +\allowbreak 3\,\allowbreak \mathcal{H}_{k}^{(-q +\allowbreak r,2\,\allowbreak r)}\!(s\,\allowbreak z,\allowbreak s^{2}) +\allowbreak 3\,\allowbreak \mathcal{H}_{k}^{(-q +\allowbreak 2\,\allowbreak r,r)}\!(s^{2}\,\allowbreak z,\allowbreak s) +\allowbreak 3\,\allowbreak \mathcal{H}_{k}^{(-q,r,2\,\allowbreak r)}\!(z,\allowbreak s,\allowbreak s^{2}) +\allowbreak 3\,\allowbreak \mathcal{H}_{k}^{(-q,2\,\allowbreak r,r)}\!(z,\allowbreak s^{2},\allowbreak s) +\allowbreak 6\,\allowbreak \mathcal{H}_{k}^{(-q +\allowbreak r,r,r)}\!(s\,\allowbreak z,\allowbreak s,\allowbreak s)
\end{dmath*}
\vspace{-1.1em}
\begin{dmath*}
{}\quad +\allowbreak 6\,\allowbreak \mathcal{H}_{k}^{(-q,r,r,r)}\!(z,\allowbreak s,\allowbreak s,\allowbreak s)
\end{dmath*}
\subsection*{Identity 2}
\begin{dmath*}
\sum_{n=1}^{k} \frac{\mathcal{H}_{n}^{(2)}\,\allowbreak \mathcal{H}_{n}^{(3)}}{n^{4 +\allowbreak 2\,\allowbreak \mathrm{i}}} = -\mathcal{H}_{k}^{(2,7 +\allowbreak 2\,\allowbreak \mathrm{i})} +\allowbreak \mathcal{H}_{k}^{(2)}\,\allowbreak \bigl(\mathcal{H}_{k}^{(3)}\,\allowbreak \mathcal{H}_{k}^{(4 +\allowbreak 2\,\allowbreak \mathrm{i})} -\allowbreak \mathcal{H}_{k}^{(3,4 +\allowbreak 2\,\allowbreak \mathrm{i})}\bigr) -\allowbreak \mathcal{H}_{k}^{(2,4 +\allowbreak 2\,\allowbreak \mathrm{i},3)}
\end{dmath*}
\subsection*{Identity 3}
\begin{dmath*}
\sum_{n=1}^{k} \bigl(\frac{1}{2}\bigr)^{n}\,\allowbreak \mathcal{H}_{n}^{(\mathrm{i})}\,\allowbreak \mathcal{A}_{n}^{(2)}\!(\mathrm{i} +\allowbreak \frac{1}{3}) = \mathcal{H}_{k}^{(\mathrm{i})}\,\allowbreak \bigl(2^{-k}\,\allowbreak \mathcal{H}_{k}^{(2)}\!(-\frac{1}{3} -\allowbreak \mathrm{i}) -\allowbreak 2\,\allowbreak \mathcal{H}_{k}^{(2)}\!(-\frac{1}{6} -\allowbreak \frac{\mathrm{i}}{2})\bigr) -\allowbreak 2\,\allowbreak \mathcal{H}_{k}^{(\mathrm{i},2)}\!(\frac{1}{2},\allowbreak -\frac{1}{3} -\allowbreak \mathrm{i}) +\allowbreak 2\,\allowbreak \mathcal{H}_{k}^{(\mathrm{i},2)}\!(1,\allowbreak -\frac{1}{6} -\allowbreak \frac{\mathrm{i}}{2})
\end{dmath*}
\subsection*{Identity 4}
\begin{dmath*}
\sum_{n=1}^{k} \chi_{3,2}\!(n)\,\allowbreak \mathcal{H}_{n}^{(1,2)}\!(s_{1},\allowbreak s_{2}) = \frac{1}{3}\,\allowbreak \bigl(-1\bigr)^{\frac{1}{3}}\,\allowbreak \bigl(1 +\allowbreak \bigl(-1\bigr)^{\frac{1}{3}}\bigr)\,\allowbreak \bigl(\mathcal{H}_{k}^{(1,2)}\!(-\bigl(\bigl(-1\bigr)^{\frac{1}{3}}\bigr)\,\allowbreak s_{1},\allowbreak s_{2}) -\allowbreak \mathcal{H}_{k}^{(1,2)}\!(\bigl(-1\bigr)^{\frac{2}{3}}\,\allowbreak s_{1},\allowbreak s_{2}) +\allowbreak \mathcal{H}_{k}^{(0,1,2)}\!(-\bigl(\bigl(-1\bigr)^{\frac{1}{3}}\bigr),\allowbreak s_{1},\allowbreak s_{2}) -\allowbreak \mathcal{H}_{k}^{(0,1,2)}\!(\bigl(-1\bigr)^{\frac{2}{3}},\allowbreak s_{1},\allowbreak s_{2})\bigr)
\end{dmath*}
\subsection*{Identity 5}
\begin{dmath*}
\sum_{n=1}^{k} \operatorname{Mod}\!(n,3)\,\allowbreak \mathcal{H}_{n}^{(1,2)}\,\allowbreak H_{n} = \frac{1}{18 +\allowbreak 6\,\allowbreak \mathrm{i}\,\allowbreak \sqrt{3}}\,\allowbreak 2\,\allowbreak \bigl(-1\bigr)^{\frac{1}{6}}\,\allowbreak \bigl(3\,\allowbreak \sqrt{3}\,\allowbreak \bigl(H_{k}\bigr)^{2} -\allowbreak 2^{-k}\,\allowbreak H_{k}\,\allowbreak \bigl(3\,\allowbreak 2^{1 +\allowbreak k}\,\allowbreak \sqrt{3}\,\allowbreak k\,\allowbreak \mathcal{H}_{k}^{(2)} +\allowbreak \bigl(2^{k}\,\allowbreak \bigl(-3\,\allowbreak \mathrm{i} +\allowbreak \sqrt{3}\bigr)\,\allowbreak \mathrm{e}^{\frac{2\,\allowbreak \mathrm{i}\,\allowbreak k\,\allowbreak \pi}{3}} +\allowbreak 2^{k}\,\allowbreak \bigl(5\,\allowbreak \mathrm{i} +\allowbreak \sqrt{3}\bigr)\,\allowbreak \mathrm{e}^{\frac{4\,\allowbreak \mathrm{i}\,\allowbreak k\,\allowbreak \pi}{3}}
\end{dmath*}
\vspace{-1.1em}
\begin{dmath*}
{}\quad -\allowbreak 2\,\allowbreak \bigl(3\,\allowbreak 2^{k}\,\allowbreak \sqrt{3} +\allowbreak \mathrm{i}\,\allowbreak \bigl(-1 -\allowbreak \mathrm{i}\,\allowbreak \sqrt{3}\bigr)^{k} +\allowbreak 3\,\allowbreak 2^{k}\,\allowbreak \sqrt{3}\,\allowbreak k\bigr)\bigr)\,\allowbreak \mathcal{H}_{k}^{(1,2)} +\allowbreak 2^{1 +\allowbreak k}\,\allowbreak \sqrt{3}\,\allowbreak \bigl(6 +\allowbreak \mathcal{H}_{k}^{(1,2)}\!(-\frac{1}{2} -\allowbreak \frac{\mathrm{i}\,\allowbreak \sqrt{3}}{2},\allowbreak 1) +\allowbreak \mathcal{H}_{k}^{(1,2)}\!(\frac{1}{2}\,\allowbreak \mathrm{i}\,\allowbreak \bigl(\mathrm{i} +\allowbreak \sqrt{3}\bigr),\allowbreak 1)\bigr)\bigr)
\end{dmath*}
\vspace{-1.1em}
\begin{dmath*}
{}\quad +\allowbreak \sqrt{3}\,\allowbreak \bigl(\mathcal{H}_{k}^{(2)}\,\allowbreak \bigl(3 +\allowbreak 12\,\allowbreak k +\allowbreak 2\,\allowbreak \mathcal{H}_{k}^{(2)}\!(-\frac{1}{2} -\allowbreak \frac{\mathrm{i}\,\allowbreak \sqrt{3}}{2}) +\allowbreak 2\,\allowbreak \mathcal{H}_{k}^{(2)}\!(\frac{1}{2}\,\allowbreak \mathrm{i}\,\allowbreak \bigl(\mathrm{i} +\allowbreak \sqrt{3}\bigr))\bigr) -\allowbreak 2\,\allowbreak \bigl(\mathcal{H}_{k}^{(4)}\!(-\frac{1}{2} -\allowbreak \frac{\mathrm{i}\,\allowbreak \sqrt{3}}{2}) +\allowbreak \mathcal{H}_{k}^{(4)}\!(\frac{1}{2}\,\allowbreak \mathrm{i}\,\allowbreak \bigl(\mathrm{i} +\allowbreak \sqrt{3}\bigr)) +\allowbreak 3\,\allowbreak \mathcal{H}_{k}^{(1,2)}
\end{dmath*}
\vspace{-1.1em}
\begin{dmath*}
{}\quad +\allowbreak 3\,\allowbreak k\,\allowbreak \mathcal{H}_{k}^{(1,2)} +\allowbreak \mathcal{H}_{k}^{(1)}\!(-\frac{1}{2} -\allowbreak \frac{\mathrm{i}\,\allowbreak \sqrt{3}}{2})\,\allowbreak \mathcal{H}_{k}^{(1,2)} +\allowbreak \mathcal{H}_{k}^{(1)}\!(\frac{1}{2}\,\allowbreak \mathrm{i}\,\allowbreak \bigl(\mathrm{i} +\allowbreak \sqrt{3}\bigr))\,\allowbreak \mathcal{H}_{k}^{(1,2)} -\allowbreak \mathcal{H}_{k}^{(1,3)}\!(1,\allowbreak -\frac{1}{2} -\allowbreak \frac{\mathrm{i}\,\allowbreak \sqrt{3}}{2}) -\allowbreak \mathcal{H}_{k}^{(1,3)}\!(1,\allowbreak \frac{1}{2}\,\allowbreak \mathrm{i}\,\allowbreak \bigl(\mathrm{i} +\allowbreak \sqrt{3}\bigr))
\end{dmath*}
\vspace{-1.1em}
\begin{dmath*}
{}\quad +\allowbreak \mathcal{H}_{k}^{(2,2)}\!(1,\allowbreak -\frac{1}{2} -\allowbreak \frac{\mathrm{i}\,\allowbreak \sqrt{3}}{2}) +\allowbreak \mathcal{H}_{k}^{(2,2)}\!(1,\allowbreak \frac{1}{2}\,\allowbreak \mathrm{i}\,\allowbreak \bigl(\mathrm{i} +\allowbreak \sqrt{3}\bigr)) -\allowbreak 2\,\allowbreak \mathcal{H}_{k}^{(1,1,2)}\!(1,\allowbreak -\frac{1}{2} -\allowbreak \frac{\mathrm{i}\,\allowbreak \sqrt{3}}{2},\allowbreak 1) -\allowbreak 2\,\allowbreak \mathcal{H}_{k}^{(1,1,2)}\!(1,\allowbreak \frac{1}{2}\,\allowbreak \mathrm{i}\,\allowbreak \bigl(\mathrm{i} +\allowbreak \sqrt{3}\bigr),\allowbreak 1) -\allowbreak \mathcal{H}_{k}^{(1,2,1)}\!(1,\allowbreak 1,\allowbreak -\frac{1}{2}
\end{dmath*}
\vspace{-1.1em}
\begin{dmath*}
{}\quad -\allowbreak \frac{\mathrm{i}\,\allowbreak \sqrt{3}}{2}) -\allowbreak \mathcal{H}_{k}^{(1,2,1)}\!(1,\allowbreak 1,\allowbreak \frac{1}{2}\,\allowbreak \mathrm{i}\,\allowbreak \bigl(\mathrm{i} +\allowbreak \sqrt{3}\bigr))\bigr)\bigr)\bigr)
\end{dmath*}
\subsection*{Identity 6}
\begin{dmath*}
\sum_{n=1}^{k} \operatorname{lcm}\!(n,\allowbreak 2,\allowbreak 3)\,\allowbreak \mathcal{H}_{n}^{(1,2)} = \frac{1}{6}\,\allowbreak \bigl(-\bigl(\frac{7}{2}\bigr)\,\allowbreak \bigl(\mathcal{A}_{k}^{(2)} -\allowbreak \bigl(-1\bigr)^{1 +\allowbreak k}\,\allowbreak \mathcal{H}_{k}^{(2)}\bigr) +\allowbreak 21\,\allowbreak \bigl(-H_{k} +\allowbreak k\,\allowbreak \mathcal{H}_{k}^{(2)}\bigr) -\allowbreak \frac{6\,\allowbreak \bigl(-\bigl(\bigl(-\frac{1}{2} -\allowbreak \frac{\mathrm{i}\,\allowbreak \sqrt{3}}{2}\bigr)^{1 +\allowbreak k}\bigr)\,\allowbreak \mathcal{H}_{k}^{(2)} +\allowbreak \bigl(-\frac{1}{2} -\allowbreak \frac{\mathrm{i}\,\allowbreak \sqrt{3}}{2}\bigr)\,\allowbreak \mathcal{H}_{k}^{(2)}\!(-\frac{1}{2} -\allowbreak \frac{\mathrm{i}\,\allowbreak \sqrt{3}}{2})\bigr)}{\frac{3}{2} +\allowbreak \frac{\mathrm{i}\,\allowbreak \sqrt{3}}{2}}
\end{dmath*}
\vspace{-1.1em}
\begin{dmath*}
{}\quad +\allowbreak \frac{2\,\allowbreak \bigl(-\bigl(\bigl(\frac{1}{2} -\allowbreak \frac{\mathrm{i}\,\allowbreak \sqrt{3}}{2}\bigr)^{1 +\allowbreak k}\bigr)\,\allowbreak \mathcal{H}_{k}^{(2)} +\allowbreak \bigl(\frac{1}{2} -\allowbreak \frac{\mathrm{i}\,\allowbreak \sqrt{3}}{2}\bigr)\,\allowbreak \mathcal{H}_{k}^{(2)}\!(\frac{1}{2} -\allowbreak \frac{\mathrm{i}\,\allowbreak \sqrt{3}}{2})\bigr)}{\frac{1}{2} +\allowbreak \frac{\mathrm{i}\,\allowbreak \sqrt{3}}{2}} +\allowbreak \frac{2\,\allowbreak \bigl(-\bigl(\bigl(\frac{1}{2} +\allowbreak \frac{\mathrm{i}\,\allowbreak \sqrt{3}}{2}\bigr)^{1 +\allowbreak k}\bigr)\,\allowbreak \mathcal{H}_{k}^{(2)} +\allowbreak \bigl(\frac{1}{2} +\allowbreak \frac{\mathrm{i}\,\allowbreak \sqrt{3}}{2}\bigr)\,\allowbreak \mathcal{H}_{k}^{(2)}\!(\frac{1}{2} +\allowbreak \frac{\mathrm{i}\,\allowbreak \sqrt{3}}{2})\bigr)}{\frac{1}{2} -\allowbreak \frac{\mathrm{i}\,\allowbreak \sqrt{3}}{2}}
\end{dmath*}
\vspace{-1.1em}
\begin{dmath*}
{}\quad -\allowbreak \frac{6\,\allowbreak \bigl(-\bigl(\bigl(\frac{1}{2}\,\allowbreak \mathrm{i}\,\allowbreak \bigl(\mathrm{i} +\allowbreak \sqrt{3}\bigr)\bigr)^{1 +\allowbreak k}\bigr)\,\allowbreak \mathcal{H}_{k}^{(2)} +\allowbreak \frac{1}{2}\,\allowbreak \mathrm{i}\,\allowbreak \bigl(\mathrm{i} +\allowbreak \sqrt{3}\bigr)\,\allowbreak \mathcal{H}_{k}^{(2)}\!(\frac{1}{2}\,\allowbreak \mathrm{i}\,\allowbreak \bigl(\mathrm{i} +\allowbreak \sqrt{3}\bigr))\bigr)}{1 -\allowbreak \frac{1}{2}\,\allowbreak \mathrm{i}\,\allowbreak \bigl(\mathrm{i} +\allowbreak \sqrt{3}\bigr)} +\allowbreak 21\,\allowbreak \mathcal{H}_{k}^{(-1,1,2)} -\allowbreak 7\,\allowbreak \mathcal{H}_{k}^{(-1,1,2)}\!(-1,\allowbreak 1,\allowbreak 1)
\end{dmath*}
\vspace{-1.1em}
\begin{dmath*}
{}\quad -\allowbreak 6\,\allowbreak \mathcal{H}_{k}^{(-1,1,2)}\!(-\frac{1}{2} -\allowbreak \frac{\mathrm{i}\,\allowbreak \sqrt{3}}{2},\allowbreak 1,\allowbreak 1) +\allowbreak 2\,\allowbreak \mathcal{H}_{k}^{(-1,1,2)}\!(\frac{1}{2} -\allowbreak \frac{\mathrm{i}\,\allowbreak \sqrt{3}}{2},\allowbreak 1,\allowbreak 1) +\allowbreak 2\,\allowbreak \mathcal{H}_{k}^{(-1,1,2)}\!(\frac{1}{2} +\allowbreak \frac{\mathrm{i}\,\allowbreak \sqrt{3}}{2},\allowbreak 1,\allowbreak 1) -\allowbreak 6\,\allowbreak \mathcal{H}_{k}^{(-1,1,2)}\!(\frac{1}{2}\,\allowbreak \mathrm{i}\,\allowbreak \bigl(\mathrm{i} +\allowbreak \sqrt{3}\bigr),\allowbreak 1,\allowbreak 1)\bigr)
\end{dmath*}
\subsection*{Identity 7}
\begin{dmath*}
\sum_{n=1}^{k} \frac{\mathcal{H}_{n}^{(2)}\!(\frac{1}{3})\,\allowbreak h_{n}^{[2]}\!(1;1)}{n^{2}} = \mathcal{H}_{k}^{(5)}\!(\frac{1}{3}) +\allowbreak 3\,\allowbreak \mathcal{H}_{k}^{(2,3)}\!(1,\allowbreak \frac{1}{3}) +\allowbreak \mathcal{H}_{k}^{(3,2)}\!(1,\allowbreak \frac{1}{3}) +\allowbreak 3\,\allowbreak \mathcal{H}_{k}^{(4,1)}\!(\frac{1}{3},\allowbreak 1) +\allowbreak 3\,\allowbreak \mathcal{H}_{k}^{(2,0,3)}\!(1,\allowbreak 1,\allowbreak \frac{1}{3}) +\allowbreak 3\,\allowbreak \mathcal{H}_{k}^{(2,1,2)}\!(1,\allowbreak 1,\allowbreak \frac{1}{3}) +\allowbreak 6\,\allowbreak \mathcal{H}_{k}^{(2,2,1)}\!(1,\allowbreak \frac{1}{3},\allowbreak 1) +\allowbreak 3\,\allowbreak \mathcal{H}_{k}^{(4,0,1)}\!(\frac{1}{3},\allowbreak 1,\allowbreak 1)
\end{dmath*}
\vspace{-1.1em}
\begin{dmath*}
{}\quad +\allowbreak \mathcal{H}_{k}^{(2,0,0,3)}\!(1,\allowbreak 1,\allowbreak 1,\allowbreak \frac{1}{3}) +\allowbreak 3\,\allowbreak \mathcal{H}_{k}^{(2,0,1,2)}\!(1,\allowbreak 1,\allowbreak 1,\allowbreak \frac{1}{3}) +\allowbreak 4\,\allowbreak \mathcal{H}_{k}^{(2,0,2,1)}\!(1,\allowbreak 1,\allowbreak \frac{1}{3},\allowbreak 1) +\allowbreak 4\,\allowbreak \mathcal{H}_{k}^{(2,2,0,1)}\!(1,\allowbreak \frac{1}{3},\allowbreak 1,\allowbreak 1) +\allowbreak \mathcal{H}_{k}^{(4,0,0,1)}\!(\frac{1}{3},\allowbreak 1,\allowbreak 1,\allowbreak 1) +\allowbreak \mathcal{H}_{k}^{(2,0,0,1,2)}\!(1,\allowbreak 1,\allowbreak 1,\allowbreak 1,\allowbreak \frac{1}{3}) +\allowbreak \mathcal{H}_{k}^{(2,0,0,2,1)}\!(1,\allowbreak 1,\allowbreak 1,\allowbreak \frac{1}{3},\allowbreak 1)
\end{dmath*}
\vspace{-1.1em}
\begin{dmath*}
{}\quad +\allowbreak \mathcal{H}_{k}^{(2,0,2,0,1)}\!(1,\allowbreak 1,\allowbreak \frac{1}{3},\allowbreak 1,\allowbreak 1) +\allowbreak \mathcal{H}_{k}^{(2,2,0,0,1)}\!(1,\allowbreak \frac{1}{3},\allowbreak 1,\allowbreak 1,\allowbreak 1)
\end{dmath*}
\subsection*{Identity 8}
\begin{dmath*}
\sum_{n=1}^{k} F_{n,\allowbreak 3}\,\allowbreak \bigl(\mathcal{H}_{n}^{(2\,\allowbreak \mathrm{i})}\bigr)^{2}\,\allowbreak \left\{\begin{matrix}n\\ 2\end{matrix}\right\} = \frac{1}{4\,\allowbreak \sqrt{13}}\,\allowbreak \bigl(-4\,\allowbreak \mathcal{A}_{k}^{(4\,\allowbreak \mathrm{i})}\!(\frac{1}{2}\,\allowbreak \bigl(-3 +\allowbreak \sqrt{13}\bigr)) +\allowbreak 2\,\allowbreak \mathcal{A}_{k}^{(4\,\allowbreak \mathrm{i})}\!(-3 +\allowbreak \sqrt{13}) +\allowbreak \frac{\bigl(3 -\allowbreak \sqrt{13}\bigr)^{2}\,\allowbreak \mathcal{A}_{k}^{(4\,\allowbreak \mathrm{i})}\!(-3 +\allowbreak \sqrt{13})}{2\,\allowbreak \bigl(-2 +\allowbreak \sqrt{13}\bigr)\,\allowbreak \bigl(1 +\allowbreak \frac{1}{2}\,\allowbreak \bigl(-3 +\allowbreak \sqrt{13}\bigr)\bigr)} -\allowbreak \frac{\bigl(3 -\allowbreak \sqrt{13}\bigr)^{1 +\allowbreak k}\,\allowbreak \mathcal{H}_{k}^{(4\,\allowbreak \mathrm{i})}}{2 -\allowbreak \sqrt{13}}
\end{dmath*}
\vspace{-1.1em}
\begin{dmath*}
{}\quad +\allowbreak \frac{\bigl(3 +\allowbreak \sqrt{13}\bigr)^{1 +\allowbreak k}\,\allowbreak \mathcal{H}_{k}^{(4\,\allowbreak \mathrm{i})}}{2 +\allowbreak \sqrt{13}} +\allowbreak \frac{2\,\allowbreak \bigl(-\bigl(\frac{1}{2}\bigr)\,\allowbreak \bigl(3 -\allowbreak \sqrt{13}\bigr)\,\allowbreak \mathcal{A}_{k}^{(4\,\allowbreak \mathrm{i})}\!(\frac{1}{2}\,\allowbreak \bigl(-3 +\allowbreak \sqrt{13}\bigr)) -\allowbreak \bigl(\frac{1}{2}\,\allowbreak \bigl(3 -\allowbreak \sqrt{13}\bigr)\bigr)^{1 +\allowbreak k}\,\allowbreak \mathcal{H}_{k}^{(4\,\allowbreak \mathrm{i})}\bigr)}{1 +\allowbreak \frac{1}{2}\,\allowbreak \bigl(-3 +\allowbreak \sqrt{13}\bigr)}
\end{dmath*}
\vspace{-1.1em}
\begin{dmath*}
{}\quad -\allowbreak \frac{-\bigl(3 -\allowbreak \sqrt{13}\bigr)\,\allowbreak \mathcal{A}_{k}^{(4\,\allowbreak \mathrm{i})}\!(-3 +\allowbreak \sqrt{13}) -\allowbreak \bigl(3 -\allowbreak \sqrt{13}\bigr)^{1 +\allowbreak k}\,\allowbreak \mathcal{H}_{k}^{(4\,\allowbreak \mathrm{i})}}{-2 +\allowbreak \sqrt{13}} +\allowbreak \frac{2^{-k}\,\allowbreak \bigl(3 -\allowbreak \sqrt{13}\bigr)^{1 +\allowbreak k}\,\allowbreak \mathcal{H}_{k}^{(4\,\allowbreak \mathrm{i})}\!(2)}{-1 +\allowbreak \frac{1}{2}\,\allowbreak \bigl(3 -\allowbreak \sqrt{13}\bigr)} -\allowbreak \frac{2^{-k}\,\allowbreak \bigl(3 +\allowbreak \sqrt{13}\bigr)^{1 +\allowbreak k}\,\allowbreak \mathcal{H}_{k}^{(4\,\allowbreak \mathrm{i})}\!(2)}{-1 +\allowbreak \frac{1}{2}\,\allowbreak \bigl(3 +\allowbreak \sqrt{13}\bigr)}
\end{dmath*}
\vspace{-1.1em}
\begin{dmath*}
{}\quad -\allowbreak \frac{-\bigl(\frac{1}{2}\bigr)\,\allowbreak \bigl(3 -\allowbreak \sqrt{13}\bigr)\,\allowbreak \mathcal{A}_{k}^{(4\,\allowbreak \mathrm{i})}\!(-3 +\allowbreak \sqrt{13}) -\allowbreak \bigl(\frac{1}{2}\,\allowbreak \bigl(3 -\allowbreak \sqrt{13}\bigr)\bigr)^{1 +\allowbreak k}\,\allowbreak \mathcal{H}_{k}^{(4\,\allowbreak \mathrm{i})}\!(2)}{1 +\allowbreak \frac{1}{2}\,\allowbreak \bigl(-3 +\allowbreak \sqrt{13}\bigr)} -\allowbreak \frac{1}{-1 +\allowbreak \frac{1}{2}\,\allowbreak \bigl(3 -\allowbreak \sqrt{13}\bigr)}\,\allowbreak \bigl(-\bigl(3 -\allowbreak \sqrt{13}\bigr)\,\allowbreak \mathcal{A}_{k}^{(4\,\allowbreak \mathrm{i})}\!(\frac{1}{2}\,\allowbreak \bigl(-3
\end{dmath*}
\vspace{-1.1em}
\begin{dmath*}
{}\quad +\allowbreak \sqrt{13}\bigr)) +\allowbreak \frac{1}{2}\,\allowbreak \bigl(3 -\allowbreak \sqrt{13}\bigr)\,\allowbreak \mathcal{A}_{k}^{(4\,\allowbreak \mathrm{i})}\!(-3 +\allowbreak \sqrt{13}) -\allowbreak 2^{-k}\,\allowbreak \bigl(3 -\allowbreak \sqrt{13}\bigr)^{1 +\allowbreak k}\,\allowbreak \mathcal{H}_{k}^{(4\,\allowbreak \mathrm{i})} +\allowbreak \bigl(\frac{1}{2}\,\allowbreak \bigl(3 -\allowbreak \sqrt{13}\bigr)\bigr)^{1 +\allowbreak k}\,\allowbreak \mathcal{H}_{k}^{(4\,\allowbreak \mathrm{i})}\!(2)\bigr) -\allowbreak 4\,\allowbreak \mathcal{H}_{k}^{(4\,\allowbreak \mathrm{i})}\!(\frac{1}{2}\,\allowbreak \bigl(3 +\allowbreak \sqrt{13}\bigr))
\end{dmath*}
\vspace{-1.1em}
\begin{dmath*}
{}\quad -\allowbreak \frac{2\,\allowbreak \bigl(-\bigl(\bigl(\frac{2}{3 +\allowbreak \sqrt{13}}\bigr)^{-1 -\allowbreak k}\bigr)\,\allowbreak \mathcal{H}_{k}^{(4\,\allowbreak \mathrm{i})} +\allowbreak \frac{1}{2}\,\allowbreak \bigl(3 +\allowbreak \sqrt{13}\bigr)\,\allowbreak \mathcal{H}_{k}^{(4\,\allowbreak \mathrm{i})}\!(\frac{1}{2}\,\allowbreak \bigl(3 +\allowbreak \sqrt{13}\bigr))\bigr)}{1 +\allowbreak \frac{1}{2}\,\allowbreak \bigl(-3 -\allowbreak \sqrt{13}\bigr)} +\allowbreak 2\,\allowbreak \mathcal{H}_{k}^{(4\,\allowbreak \mathrm{i})}\!(3 +\allowbreak \sqrt{13}) +\allowbreak \frac{\bigl(3 +\allowbreak \sqrt{13}\bigr)^{2}\,\allowbreak \mathcal{H}_{k}^{(4\,\allowbreak \mathrm{i})}\!(3 +\allowbreak \sqrt{13})}{2\,\allowbreak \bigl(-2 -\allowbreak \sqrt{13}\bigr)\,\allowbreak \bigl(1 +\allowbreak \frac{1}{2}\,\allowbreak \bigl(-3 -\allowbreak \sqrt{13}\bigr)\bigr)}
\end{dmath*}
\vspace{-1.1em}
\begin{dmath*}
{}\quad +\allowbreak \frac{1}{-1 +\allowbreak \frac{1}{2}\,\allowbreak \bigl(3 +\allowbreak \sqrt{13}\bigr)}\,\allowbreak \bigl(-\bigl(2^{-k}\bigr)\,\allowbreak \bigl(3 +\allowbreak \sqrt{13}\bigr)^{1 +\allowbreak k}\,\allowbreak \mathcal{H}_{k}^{(4\,\allowbreak \mathrm{i})} +\allowbreak \bigl(\frac{1}{2}\,\allowbreak \bigl(3 +\allowbreak \sqrt{13}\bigr)\bigr)^{1 +\allowbreak k}\,\allowbreak \mathcal{H}_{k}^{(4\,\allowbreak \mathrm{i})}\!(2) +\allowbreak \bigl(3 +\allowbreak \sqrt{13}\bigr)\,\allowbreak \mathcal{H}_{k}^{(4\,\allowbreak \mathrm{i})}\!(\frac{1}{2}\,\allowbreak \bigl(3 +\allowbreak \sqrt{13}\bigr)) -\allowbreak \frac{1}{2}\,\allowbreak \bigl(3
\end{dmath*}
\vspace{-1.1em}
\begin{dmath*}
{}\quad +\allowbreak \sqrt{13}\bigr)\,\allowbreak \mathcal{H}_{k}^{(4\,\allowbreak \mathrm{i})}\!(3 +\allowbreak \sqrt{13})\bigr) +\allowbreak \frac{-\bigl(\bigl(\frac{2}{3 +\allowbreak \sqrt{13}}\bigr)^{-1 -\allowbreak k}\bigr)\,\allowbreak \mathcal{H}_{k}^{(4\,\allowbreak \mathrm{i})}\!(2) +\allowbreak \frac{1}{2}\,\allowbreak \bigl(3 +\allowbreak \sqrt{13}\bigr)\,\allowbreak \mathcal{H}_{k}^{(4\,\allowbreak \mathrm{i})}\!(3 +\allowbreak \sqrt{13})}{1 +\allowbreak \frac{1}{2}\,\allowbreak \bigl(-3 -\allowbreak \sqrt{13}\bigr)} +\allowbreak \frac{-\bigl(\bigl(3 +\allowbreak \sqrt{13}\bigr)^{1 +\allowbreak k}\bigr)\,\allowbreak \mathcal{H}_{k}^{(4\,\allowbreak \mathrm{i})} +\allowbreak \bigl(3 +\allowbreak \sqrt{13}\bigr)\,\allowbreak \mathcal{H}_{k}^{(4\,\allowbreak \mathrm{i})}\!(3 +\allowbreak \sqrt{13})}{-2 -\allowbreak \sqrt{13}}
\end{dmath*}
\vspace{-1.1em}
\begin{dmath*}
{}\quad +\allowbreak 4\,\allowbreak \mathcal{H}_{k}^{(2\,\allowbreak \mathrm{i},2\,\allowbreak \mathrm{i})}\!(\frac{1}{2}\,\allowbreak \bigl(3 -\allowbreak \sqrt{13}\bigr),\allowbreak 1) -\allowbreak 2\,\allowbreak \mathcal{H}_{k}^{(2\,\allowbreak \mathrm{i},2\,\allowbreak \mathrm{i})}\!(\frac{1}{2}\,\allowbreak \bigl(3 -\allowbreak \sqrt{13}\bigr),\allowbreak 2) -\allowbreak 2\,\allowbreak \mathcal{H}_{k}^{(2\,\allowbreak \mathrm{i},2\,\allowbreak \mathrm{i})}\!(3 -\allowbreak \sqrt{13},\allowbreak 1) -\allowbreak 4\,\allowbreak \mathcal{H}_{k}^{(2\,\allowbreak \mathrm{i},2\,\allowbreak \mathrm{i})}\!(\frac{1}{2}\,\allowbreak \bigl(3 +\allowbreak \sqrt{13}\bigr),\allowbreak 1)
\end{dmath*}
\vspace{-1.1em}
\begin{dmath*}
{}\quad +\allowbreak 2\,\allowbreak \mathcal{H}_{k}^{(2\,\allowbreak \mathrm{i},2\,\allowbreak \mathrm{i})}\!(\frac{1}{2}\,\allowbreak \bigl(3 +\allowbreak \sqrt{13}\bigr),\allowbreak 2) +\allowbreak 2\,\allowbreak \mathcal{H}_{k}^{(2\,\allowbreak \mathrm{i},2\,\allowbreak \mathrm{i})}\!(3 +\allowbreak \sqrt{13},\allowbreak 1) +\allowbreak 4\,\allowbreak \mathcal{H}_{k}^{(0,2\,\allowbreak \mathrm{i},2\,\allowbreak \mathrm{i})}\!(\frac{1}{2}\,\allowbreak \bigl(3 -\allowbreak \sqrt{13}\bigr),\allowbreak 1,\allowbreak 1) -\allowbreak 2\,\allowbreak \mathcal{H}_{k}^{(0,2\,\allowbreak \mathrm{i},2\,\allowbreak \mathrm{i})}\!(\frac{1}{2}\,\allowbreak \bigl(3
\end{dmath*}
\vspace{-1.1em}
\begin{dmath*}
{}\quad -\allowbreak \sqrt{13}\bigr),\allowbreak 1,\allowbreak 2) -\allowbreak 2\,\allowbreak \mathcal{H}_{k}^{(0,2\,\allowbreak \mathrm{i},2\,\allowbreak \mathrm{i})}\!(\frac{1}{2}\,\allowbreak \bigl(3 -\allowbreak \sqrt{13}\bigr),\allowbreak 2,\allowbreak 1) -\allowbreak 2\,\allowbreak \mathcal{H}_{k}^{(0,2\,\allowbreak \mathrm{i},2\,\allowbreak \mathrm{i})}\!(3 -\allowbreak \sqrt{13},\allowbreak 1,\allowbreak 1) -\allowbreak 4\,\allowbreak \mathcal{H}_{k}^{(0,2\,\allowbreak \mathrm{i},2\,\allowbreak \mathrm{i})}\!(\frac{1}{2}\,\allowbreak \bigl(3 +\allowbreak \sqrt{13}\bigr),\allowbreak 1,\allowbreak 1) +\allowbreak 2\,\allowbreak \mathcal{H}_{k}^{(0,2\,\allowbreak \mathrm{i},2\,\allowbreak \mathrm{i})}\!(\frac{1}{2}\,\allowbreak \bigl(3
\end{dmath*}
\vspace{-1.1em}
\begin{dmath*}
{}\quad +\allowbreak \sqrt{13}\bigr),\allowbreak 1,\allowbreak 2) +\allowbreak 2\,\allowbreak \mathcal{H}_{k}^{(0,2\,\allowbreak \mathrm{i},2\,\allowbreak \mathrm{i})}\!(\frac{1}{2}\,\allowbreak \bigl(3 +\allowbreak \sqrt{13}\bigr),\allowbreak 2,\allowbreak 1) +\allowbreak 2\,\allowbreak \mathcal{H}_{k}^{(0,2\,\allowbreak \mathrm{i},2\,\allowbreak \mathrm{i})}\!(3 +\allowbreak \sqrt{13},\allowbreak 1,\allowbreak 1) -\allowbreak 2\,\allowbreak \mathcal{H}_{k}^{(2\,\allowbreak \mathrm{i},0,2\,\allowbreak \mathrm{i})}\!(\frac{1}{2}\,\allowbreak \bigl(3 -\allowbreak \sqrt{13}\bigr),\allowbreak 2,\allowbreak 1) +\allowbreak 2\,\allowbreak \mathcal{H}_{k}^{(2\,\allowbreak \mathrm{i},0,2\,\allowbreak \mathrm{i})}\!(\frac{1}{2}\,\allowbreak \bigl(3
\end{dmath*}
\vspace{-1.1em}
\begin{dmath*}
{}\quad +\allowbreak \sqrt{13}\bigr),\allowbreak 2,\allowbreak 1) -\allowbreak 2\,\allowbreak \mathcal{H}_{k}^{(2\,\allowbreak \mathrm{i},2\,\allowbreak \mathrm{i},0)}\!(\frac{1}{2}\,\allowbreak \bigl(3 -\allowbreak \sqrt{13}\bigr),\allowbreak 1,\allowbreak 2) +\allowbreak 2\,\allowbreak \mathcal{H}_{k}^{(2\,\allowbreak \mathrm{i},2\,\allowbreak \mathrm{i},0)}\!(\frac{1}{2}\,\allowbreak \bigl(3 +\allowbreak \sqrt{13}\bigr),\allowbreak 1,\allowbreak 2) -\allowbreak 2\,\allowbreak \mathcal{H}_{k}^{(0,0,2\,\allowbreak \mathrm{i},2\,\allowbreak \mathrm{i})}\!(\frac{1}{2}\,\allowbreak \bigl(3 -\allowbreak \sqrt{13}\bigr),\allowbreak 2,\allowbreak 1,\allowbreak 1)
\end{dmath*}
\vspace{-1.1em}
\begin{dmath*}
{}\quad +\allowbreak 2\,\allowbreak \mathcal{H}_{k}^{(0,0,2\,\allowbreak \mathrm{i},2\,\allowbreak \mathrm{i})}\!(\frac{1}{2}\,\allowbreak \bigl(3 +\allowbreak \sqrt{13}\bigr),\allowbreak 2,\allowbreak 1,\allowbreak 1) -\allowbreak 2\,\allowbreak \mathcal{H}_{k}^{(0,2\,\allowbreak \mathrm{i},0,2\,\allowbreak \mathrm{i})}\!(\frac{1}{2}\,\allowbreak \bigl(3 -\allowbreak \sqrt{13}\bigr),\allowbreak 1,\allowbreak 2,\allowbreak 1) +\allowbreak 2\,\allowbreak \mathcal{H}_{k}^{(0,2\,\allowbreak \mathrm{i},0,2\,\allowbreak \mathrm{i})}\!(\frac{1}{2}\,\allowbreak \bigl(3 +\allowbreak \sqrt{13}\bigr),\allowbreak 1,\allowbreak 2,\allowbreak 1) -\allowbreak 2\,\allowbreak \mathcal{H}_{k}^{(0,2\,\allowbreak \mathrm{i},2\,\allowbreak \mathrm{i},0)}\!(\frac{1}{2}\,\allowbreak \bigl(3
\end{dmath*}
\vspace{-1.1em}
\begin{dmath*}
{}\quad -\allowbreak \sqrt{13}\bigr),\allowbreak 1,\allowbreak 1,\allowbreak 2) +\allowbreak 2\,\allowbreak \mathcal{H}_{k}^{(0,2\,\allowbreak \mathrm{i},2\,\allowbreak \mathrm{i},0)}\!(\frac{1}{2}\,\allowbreak \bigl(3 +\allowbreak \sqrt{13}\bigr),\allowbreak 1,\allowbreak 1,\allowbreak 2)\bigr)
\end{dmath*}
\subsection*{Identity 9}
\begin{dmath*}
\sum_{n=1}^{k} H_{n}\,\allowbreak \bigl(\left\{\begin{matrix}n\\ 2\end{matrix}\right\}\bigr)^{2} = \frac{1}{48}\,\allowbreak \bigl(12\,\allowbreak \bigl(-2 +\allowbreak 2^{k}\bigr)^{2}\,\allowbreak \bigl(1 +\allowbreak k\bigr)\,\allowbreak H_{k} +\allowbreak 12\,\allowbreak \mathcal{H}_{k}^{(-1,1)}\!(2,\allowbreak 1) -\allowbreak 6\,\allowbreak \mathcal{H}_{k}^{(-1,1)}\!(2,\allowbreak 2) -\allowbreak 3\,\allowbreak \mathcal{H}_{k}^{(-1,1)}\!(4,\allowbreak 1) -\allowbreak 2\,\allowbreak \bigl(8\,\allowbreak \bigl(5 -\allowbreak 3\,\allowbreak 2^{1 +\allowbreak k} +\allowbreak 4^{k} +\allowbreak 3\,\allowbreak k\bigr) +\allowbreak 3\,\allowbreak \mathcal{H}_{k}^{(-1,0,1)}\!(2,\allowbreak 2,\allowbreak 1) +\allowbreak 3\,\allowbreak \mathcal{H}_{k}^{(-1,1,0)}\!(2,\allowbreak 1,\allowbreak 2)\bigr)\bigr)
\end{dmath*}
\subsection*{Identity 10}
\begin{dmath*}
\sum_{i=1}^{n} \psi^{(3)}\!(i)\,\allowbreak \mathcal{A}_{i}^{(2\,\allowbreak \mathrm{i})} = \frac{1}{15}\,\allowbreak \bigl(-\bigl(\pi^{4}\bigr)\,\allowbreak \mathcal{A}_{n}^{(-1 +\allowbreak 2\,\allowbreak \mathrm{i})} +\allowbreak \bigl(1 +\allowbreak n\bigr)\,\allowbreak \pi^{4}\,\allowbreak \mathcal{A}_{n}^{(2\,\allowbreak \mathrm{i})} +\allowbreak 90\,\allowbreak \bigl(\mathcal{A}_{n}^{(3 +\allowbreak 2\,\allowbreak \mathrm{i})} -\allowbreak n\,\allowbreak \mathcal{A}_{n}^{(4 +\allowbreak 2\,\allowbreak \mathrm{i})} +\allowbreak \mathcal{H}_{n}^{(2\,\allowbreak \mathrm{i},4)}\!(-1,\allowbreak 1) +\allowbreak \mathcal{H}_{n}^{(0,2\,\allowbreak \mathrm{i},4)}\!(1,\allowbreak -1,\allowbreak 1) +\allowbreak \mathcal{H}_{n}^{(0,4,2\,\allowbreak \mathrm{i})}\!(1,\allowbreak 1,\allowbreak -1)\bigr)\bigr)
\end{dmath*}
\subsection*{Identity 11}
\begin{dmath*}
\sum_{n=1}^{k} \zeta\!(2,n)\,\allowbreak \mathcal{H}_{n}^{(1,2)} = \frac{1}{6}\,\allowbreak \bigl(\pi^{2}\,\allowbreak \bigl(\mathcal{H}_{k}^{(1,2)} +\allowbreak \mathcal{H}_{k}^{(0,1,2)}\bigr) -\allowbreak 6\,\allowbreak \bigl(\mathcal{H}_{k}^{(1,4)} +\allowbreak \mathcal{H}_{k}^{(0,1,4)} +\allowbreak \mathcal{H}_{k}^{(0,3,2)} +\allowbreak 2\,\allowbreak \bigl(\mathcal{H}_{k}^{(1,2,2)} +\allowbreak \mathcal{H}_{k}^{(0,1,2,2)}\bigr) +\allowbreak \mathcal{H}_{k}^{(0,2,1,2)}\bigr)\bigr)
\end{dmath*}
\subsection*{Identity 12}
\begin{dmath*}
\sum_{n=1}^{k} \tau\!(2^{n})\,\allowbreak \mathcal{H}_{n}^{(1,2)} = \frac{1}{493374}\,\allowbreak \bigl(8\,\allowbreak \bigl(\bigl(-12 -\allowbreak 4\,\allowbreak \mathrm{i}\,\allowbreak \sqrt{119}\bigr)^{k}\,\allowbreak \bigl(30821 -\allowbreak 713\,\allowbreak \mathrm{i}\,\allowbreak \sqrt{119}\bigr) +\allowbreak \bigl(30821 +\allowbreak 713\,\allowbreak \mathrm{i}\,\allowbreak \sqrt{119}\bigr)\,\allowbreak \bigl(4\,\allowbreak \mathrm{i}\,\allowbreak \bigl(3\,\allowbreak \mathrm{i} +\allowbreak \sqrt{119}\bigr)\bigr)^{k}\bigr)\,\allowbreak \mathcal{H}_{k}^{(1,2)} +\allowbreak \bigl(119 -\allowbreak 515\,\allowbreak \mathrm{i}\,\allowbreak \sqrt{119}\bigr)\,\allowbreak \mathcal{H}_{k}^{(1,2)}\!(-12 -\allowbreak 4\,\allowbreak \mathrm{i}\,\allowbreak \sqrt{119},\allowbreak 1)
\end{dmath*}
\vspace{-1.1em}
\begin{dmath*}
{}\quad +\allowbreak \bigl(119 +\allowbreak 515\,\allowbreak \mathrm{i}\,\allowbreak \sqrt{119}\bigr)\,\allowbreak \mathcal{H}_{k}^{(1,2)}\!(-12 +\allowbreak 4\,\allowbreak \mathrm{i}\,\allowbreak \sqrt{119},\allowbreak 1)\bigr)
\end{dmath*}
\subsection*{Identity 13}
\begin{dmath*}
\sum_{n=1}^{k} (\frac{n}{15})\,\allowbreak \mathcal{H}_{n}^{(1,2)} = \frac{-1}{\sqrt{15}}\,\allowbreak \mathrm{i}\,\allowbreak \bigl(\mathcal{H}_{k}^{(1,2)}\!(-\bigl(\bigl(-1\bigr)^{\frac{1}{15}}\bigr),\allowbreak 1) +\allowbreak \mathcal{H}_{k}^{(1,2)}\!(\bigl(-1\bigr)^{\frac{2}{15}},\allowbreak 1) +\allowbreak \mathcal{H}_{k}^{(1,2)}\!(\bigl(-1\bigr)^{\frac{4}{15}},\allowbreak 1) -\allowbreak \mathcal{H}_{k}^{(1,2)}\!(-\bigl(\bigl(-1\bigr)^{\frac{7}{15}}\bigr),\allowbreak 1) +\allowbreak \mathcal{H}_{k}^{(1,2)}\!(\bigl(-1\bigr)^{\frac{8}{15}},\allowbreak 1) -\allowbreak \mathcal{H}_{k}^{(1,2)}\!(-\bigl(\bigl(-1\bigr)^{\frac{11}{15}}\bigr),\allowbreak 1) -\allowbreak \mathcal{H}_{k}^{(1,2)}\!(-\bigl(\bigl(-1\bigr)^{\frac{13}{15}}\bigr),\allowbreak 1)
\end{dmath*}
\vspace{-1.1em}
\begin{dmath*}
{}\quad -\allowbreak \mathcal{H}_{k}^{(1,2)}\!(\bigl(-1\bigr)^{\frac{14}{15}},\allowbreak 1) +\allowbreak \mathcal{H}_{k}^{(0,1,2)}\!(-\bigl(\bigl(-1\bigr)^{\frac{1}{15}}\bigr),\allowbreak 1,\allowbreak 1) +\allowbreak \mathcal{H}_{k}^{(0,1,2)}\!(\bigl(-1\bigr)^{\frac{2}{15}},\allowbreak 1,\allowbreak 1) +\allowbreak \mathcal{H}_{k}^{(0,1,2)}\!(\bigl(-1\bigr)^{\frac{4}{15}},\allowbreak 1,\allowbreak 1) -\allowbreak \mathcal{H}_{k}^{(0,1,2)}\!(-\bigl(\bigl(-1\bigr)^{\frac{7}{15}}\bigr),\allowbreak 1,\allowbreak 1) +\allowbreak \mathcal{H}_{k}^{(0,1,2)}\!(\bigl(-1\bigr)^{\frac{8}{15}},\allowbreak 1,\allowbreak 1) -\allowbreak \mathcal{H}_{k}^{(0,1,2)}\!(-\bigl(\bigl(-1\bigr)^{\frac{11}{15}}\bigr),\allowbreak 1,\allowbreak 1)
\end{dmath*}
\vspace{-1.1em}
\begin{dmath*}
{}\quad -\allowbreak \mathcal{H}_{k}^{(0,1,2)}\!(-\bigl(\bigl(-1\bigr)^{\frac{13}{15}}\bigr),\allowbreak 1,\allowbreak 1) -\allowbreak \mathcal{H}_{k}^{(0,1,2)}\!(\bigl(-1\bigr)^{\frac{14}{15}},\allowbreak 1,\allowbreak 1)\bigr)
\end{dmath*}
\subsection*{Identity 14}
\begin{dmath*}
\sum_{n=1}^{k} (\frac{n}{8})_K\,\allowbreak \mathcal{H}_{n}^{(1,2)} = \frac{-1}{2\,\allowbreak \sqrt{2}}\,\allowbreak \bigl(\mathcal{H}_{k}^{(1,2)}\!(-\bigl(\bigl(-1\bigr)^{\frac{1}{4}}\bigr),\allowbreak 1) -\allowbreak \mathcal{H}_{k}^{(1,2)}\!(\bigl(-1\bigr)^{\frac{1}{4}},\allowbreak 1) -\allowbreak \mathcal{H}_{k}^{(1,2)}\!(-\bigl(\bigl(-1\bigr)^{\frac{3}{4}}\bigr),\allowbreak 1) +\allowbreak \mathcal{H}_{k}^{(1,2)}\!(\bigl(-1\bigr)^{\frac{3}{4}},\allowbreak 1) +\allowbreak \mathcal{H}_{k}^{(0,1,2)}\!(-\bigl(\bigl(-1\bigr)^{\frac{1}{4}}\bigr),\allowbreak 1,\allowbreak 1) -\allowbreak \mathcal{H}_{k}^{(0,1,2)}\!(\bigl(-1\bigr)^{\frac{1}{4}},\allowbreak 1,\allowbreak 1) -\allowbreak \mathcal{H}_{k}^{(0,1,2)}\!(-\bigl(\bigl(-1\bigr)^{\frac{3}{4}}\bigr),\allowbreak 1,\allowbreak 1)
\end{dmath*}
\vspace{-1.1em}
\begin{dmath*}
{}\quad +\allowbreak \mathcal{H}_{k}^{(0,1,2)}\!(\bigl(-1\bigr)^{\frac{3}{4}},\allowbreak 1,\allowbreak 1)\bigr)
\end{dmath*}
\subsection*{Identity 15}
\begin{dmath*}
\sum_{n=1}^{k} \operatorname{Quotient}\!(2\,\allowbreak n +\allowbreak 1,5)\,\allowbreak \mathcal{H}_{n}^{(1,2)} = \frac{1}{25}\,\allowbreak \bigl(10\,\allowbreak \bigl(-H_{k} +\allowbreak k\,\allowbreak \mathcal{H}_{k}^{(2)}\bigr) -\allowbreak 5\,\allowbreak \mathcal{H}_{k}^{(1,2)} +\allowbreak 10\,\allowbreak \mathcal{H}_{k}^{(-1,1,2)} -\allowbreak 5\,\allowbreak \mathcal{H}_{k}^{(0,1,2)} -\allowbreak \mathcal{H}_{k}^{(1,2)}\!(-\bigl(\bigl(-1\bigr)^{\frac{1}{5}}\bigr),\allowbreak 1) +\allowbreak 4\,\allowbreak \bigl(-1\bigr)^{\frac{1}{5}}\,\allowbreak \mathcal{H}_{k}^{(1,2)}\!(-\bigl(\bigl(-1\bigr)^{\frac{1}{5}}\bigr),\allowbreak 1) -\allowbreak 2\,\allowbreak \bigl(-1\bigr)^{\frac{2}{5}}\,\allowbreak \mathcal{H}_{k}^{(1,2)}\!(-\bigl(\bigl(-1\bigr)^{\frac{1}{5}}\bigr),\allowbreak 1)
\end{dmath*}
\vspace{-1.1em}
\begin{dmath*}
{}\quad -\allowbreak 3\,\allowbreak \bigl(-1\bigr)^{\frac{4}{5}}\,\allowbreak \mathcal{H}_{k}^{(1,2)}\!(-\bigl(\bigl(-1\bigr)^{\frac{1}{5}}\bigr),\allowbreak 1) -\allowbreak \mathcal{H}_{k}^{(1,2)}\!(\bigl(-1\bigr)^{\frac{2}{5}},\allowbreak 1) -\allowbreak 4\,\allowbreak \bigl(-1\bigr)^{\frac{2}{5}}\,\allowbreak \mathcal{H}_{k}^{(1,2)}\!(\bigl(-1\bigr)^{\frac{2}{5}},\allowbreak 1) +\allowbreak 3\,\allowbreak \bigl(-1\bigr)^{\frac{3}{5}}\,\allowbreak \mathcal{H}_{k}^{(1,2)}\!(\bigl(-1\bigr)^{\frac{2}{5}},\allowbreak 1) -\allowbreak 2\,\allowbreak \bigl(-1\bigr)^{\frac{4}{5}}\,\allowbreak \mathcal{H}_{k}^{(1,2)}\!(\bigl(-1\bigr)^{\frac{2}{5}},\allowbreak 1)
\end{dmath*}
\vspace{-1.1em}
\begin{dmath*}
{}\quad -\allowbreak \mathcal{H}_{k}^{(1,2)}\!(-\bigl(\bigl(-1\bigr)^{\frac{3}{5}}\bigr),\allowbreak 1) +\allowbreak 2\,\allowbreak \bigl(-1\bigr)^{\frac{1}{5}}\,\allowbreak \mathcal{H}_{k}^{(1,2)}\!(-\bigl(\bigl(-1\bigr)^{\frac{3}{5}}\bigr),\allowbreak 1) -\allowbreak 3\,\allowbreak \bigl(-1\bigr)^{\frac{2}{5}}\,\allowbreak \mathcal{H}_{k}^{(1,2)}\!(-\bigl(\bigl(-1\bigr)^{\frac{3}{5}}\bigr),\allowbreak 1) +\allowbreak 4\,\allowbreak \bigl(-1\bigr)^{\frac{3}{5}}\,\allowbreak \mathcal{H}_{k}^{(1,2)}\!(-\bigl(\bigl(-1\bigr)^{\frac{3}{5}}\bigr),\allowbreak 1) -\allowbreak \mathcal{H}_{k}^{(1,2)}\!(\bigl(-1\bigr)^{\frac{4}{5}},\allowbreak 1) +\allowbreak 3\,\allowbreak \bigl(-1\bigr)^{\frac{1}{5}}\,\allowbreak \mathcal{H}_{k}^{(1,2)}\!(\bigl(-1\bigr)^{\frac{4}{5}},\allowbreak 1)
\end{dmath*}
\vspace{-1.1em}
\begin{dmath*}
{}\quad +\allowbreak 2\,\allowbreak \bigl(-1\bigr)^{\frac{3}{5}}\,\allowbreak \mathcal{H}_{k}^{(1,2)}\!(\bigl(-1\bigr)^{\frac{4}{5}},\allowbreak 1) -\allowbreak 4\,\allowbreak \bigl(-1\bigr)^{\frac{4}{5}}\,\allowbreak \mathcal{H}_{k}^{(1,2)}\!(\bigl(-1\bigr)^{\frac{4}{5}},\allowbreak 1) -\allowbreak \mathcal{H}_{k}^{(0,1,2)}\!(-\bigl(\bigl(-1\bigr)^{\frac{1}{5}}\bigr),\allowbreak 1,\allowbreak 1) +\allowbreak 4\,\allowbreak \bigl(-1\bigr)^{\frac{1}{5}}\,\allowbreak \mathcal{H}_{k}^{(0,1,2)}\!(-\bigl(\bigl(-1\bigr)^{\frac{1}{5}}\bigr),\allowbreak 1,\allowbreak 1) -\allowbreak 2\,\allowbreak \bigl(-1\bigr)^{\frac{2}{5}}\,\allowbreak \mathcal{H}_{k}^{(0,1,2)}\!(-\bigl(\bigl(-1\bigr)^{\frac{1}{5}}\bigr),\allowbreak 1,\allowbreak 1)
\end{dmath*}
\vspace{-1.1em}
\begin{dmath*}
{}\quad -\allowbreak 3\,\allowbreak \bigl(-1\bigr)^{\frac{4}{5}}\,\allowbreak \mathcal{H}_{k}^{(0,1,2)}\!(-\bigl(\bigl(-1\bigr)^{\frac{1}{5}}\bigr),\allowbreak 1,\allowbreak 1) -\allowbreak \mathcal{H}_{k}^{(0,1,2)}\!(\bigl(-1\bigr)^{\frac{2}{5}},\allowbreak 1,\allowbreak 1) -\allowbreak 4\,\allowbreak \bigl(-1\bigr)^{\frac{2}{5}}\,\allowbreak \mathcal{H}_{k}^{(0,1,2)}\!(\bigl(-1\bigr)^{\frac{2}{5}},\allowbreak 1,\allowbreak 1) +\allowbreak 3\,\allowbreak \bigl(-1\bigr)^{\frac{3}{5}}\,\allowbreak \mathcal{H}_{k}^{(0,1,2)}\!(\bigl(-1\bigr)^{\frac{2}{5}},\allowbreak 1,\allowbreak 1) -\allowbreak 2\,\allowbreak \bigl(-1\bigr)^{\frac{4}{5}}\,\allowbreak \mathcal{H}_{k}^{(0,1,2)}\!(\bigl(-1\bigr)^{\frac{2}{5}},\allowbreak 1,\allowbreak 1)
\end{dmath*}
\vspace{-1.1em}
\begin{dmath*}
{}\quad -\allowbreak \mathcal{H}_{k}^{(0,1,2)}\!(-\bigl(\bigl(-1\bigr)^{\frac{3}{5}}\bigr),\allowbreak 1,\allowbreak 1) +\allowbreak 2\,\allowbreak \bigl(-1\bigr)^{\frac{1}{5}}\,\allowbreak \mathcal{H}_{k}^{(0,1,2)}\!(-\bigl(\bigl(-1\bigr)^{\frac{3}{5}}\bigr),\allowbreak 1,\allowbreak 1) -\allowbreak 3\,\allowbreak \bigl(-1\bigr)^{\frac{2}{5}}\,\allowbreak \mathcal{H}_{k}^{(0,1,2)}\!(-\bigl(\bigl(-1\bigr)^{\frac{3}{5}}\bigr),\allowbreak 1,\allowbreak 1) +\allowbreak 4\,\allowbreak \bigl(-1\bigr)^{\frac{3}{5}}\,\allowbreak \mathcal{H}_{k}^{(0,1,2)}\!(-\bigl(\bigl(-1\bigr)^{\frac{3}{5}}\bigr),\allowbreak 1,\allowbreak 1) -\allowbreak \mathcal{H}_{k}^{(0,1,2)}\!(\bigl(-1\bigr)^{\frac{4}{5}},\allowbreak 1,\allowbreak 1)
\end{dmath*}
\vspace{-1.1em}
\begin{dmath*}
{}\quad +\allowbreak 3\,\allowbreak \bigl(-1\bigr)^{\frac{1}{5}}\,\allowbreak \mathcal{H}_{k}^{(0,1,2)}\!(\bigl(-1\bigr)^{\frac{4}{5}},\allowbreak 1,\allowbreak 1) +\allowbreak 2\,\allowbreak \bigl(-1\bigr)^{\frac{3}{5}}\,\allowbreak \mathcal{H}_{k}^{(0,1,2)}\!(\bigl(-1\bigr)^{\frac{4}{5}},\allowbreak 1,\allowbreak 1) -\allowbreak 4\,\allowbreak \bigl(-1\bigr)^{\frac{4}{5}}\,\allowbreak \mathcal{H}_{k}^{(0,1,2)}\!(\bigl(-1\bigr)^{\frac{4}{5}},\allowbreak 1,\allowbreak 1)\bigr)
\end{dmath*}
\subsection*{Identity 16}
\begin{dmath*}
\sum_{n=1}^{k} \operatorname{ord}_{8}\!(2\,\allowbreak n +\allowbreak 1)\,\allowbreak \mathcal{H}_{n}^{(1,2)} = \frac{1}{4}\,\allowbreak \bigl(7\,\allowbreak \mathcal{H}_{k}^{(1,2)} +\allowbreak 7\,\allowbreak \mathcal{H}_{k}^{(0,1,2)} -\allowbreak \mathcal{H}_{k}^{(1,2)}\!(-1,\allowbreak 1) -\allowbreak \mathcal{H}_{k}^{(1,2)}\!(-\mathrm{i},\allowbreak 1) -\allowbreak \mathcal{H}_{k}^{(1,2)}\!(\mathrm{i},\allowbreak 1) -\allowbreak \mathcal{H}_{k}^{(0,1,2)}\!(-1,\allowbreak 1,\allowbreak 1) -\allowbreak \mathcal{H}_{k}^{(0,1,2)}\!(-\mathrm{i},\allowbreak 1,\allowbreak 1) -\allowbreak \mathcal{H}_{k}^{(0,1,2)}\!(\mathrm{i},\allowbreak 1,\allowbreak 1)\bigr)
\end{dmath*}
\subsection*{Identity 17}
\begin{dmath*}
\sum_{n=1}^{k} \operatorname{PowerMod}\!(2,\allowbreak n,\allowbreak 5)\,\allowbreak \mathcal{H}_{n}^{(1,2)} = \bigl(-\frac{1}{4} +\allowbreak \frac{\mathrm{i}}{4}\bigr)\,\allowbreak \bigl(\bigl(-5 -\allowbreak 5\,\allowbreak \mathrm{i}\bigr)\,\allowbreak \mathcal{H}_{k}^{(1,2)} -\allowbreak \bigl(5 +\allowbreak 5\,\allowbreak \mathrm{i}\bigr)\,\allowbreak \mathcal{H}_{k}^{(0,1,2)} +\allowbreak \bigl(1 +\allowbreak 2\,\allowbreak \mathrm{i}\bigr)\,\allowbreak \mathcal{H}_{k}^{(1,2)}\!(-\mathrm{i},\allowbreak 1) +\allowbreak \bigl(2 +\allowbreak \mathrm{i}\bigr)\,\allowbreak \mathcal{H}_{k}^{(1,2)}\!(\mathrm{i},\allowbreak 1) +\allowbreak \bigl(1 +\allowbreak 2\,\allowbreak \mathrm{i}\bigr)\,\allowbreak \mathcal{H}_{k}^{(0,1,2)}\!(-\mathrm{i},\allowbreak 1,\allowbreak 1)
\end{dmath*}
\vspace{-1.1em}
\begin{dmath*}
{}\quad +\allowbreak \bigl(2 +\allowbreak \mathrm{i}\bigr)\,\allowbreak \mathcal{H}_{k}^{(0,1,2)}\!(\mathrm{i},\allowbreak 1,\allowbreak 1)\bigr)
\end{dmath*}
\subsection*{Identity 18}
\begin{dmath*}
\sum_{n=1}^{k} \nu_{4}\!(12\,\allowbreak 2^{n})\,\allowbreak \mathcal{H}_{n}^{(1,2)} = \frac{1}{4}\,\allowbreak \bigl(2\,\allowbreak \bigl(-H_{k} +\allowbreak k\,\allowbreak \mathcal{H}_{k}^{(2)}\bigr) +\allowbreak 3\,\allowbreak \mathcal{H}_{k}^{(1,2)} +\allowbreak 2\,\allowbreak \mathcal{H}_{k}^{(-1,1,2)} +\allowbreak 3\,\allowbreak \mathcal{H}_{k}^{(0,1,2)} +\allowbreak \mathcal{H}_{k}^{(1,2)}\!(-1,\allowbreak 1) +\allowbreak \mathcal{H}_{k}^{(0,1,2)}\!(-1,\allowbreak 1,\allowbreak 1)\bigr)
\end{dmath*}
\subsection*{Identity 19}
\begin{dmath*}
\sum_{n=1}^{k} \sigma_{2}\!(2^{n})\,\allowbreak \mathcal{H}_{n}^{(1,2)} = \frac{1}{3}\,\allowbreak \bigl(-\mathcal{H}_{k}^{(1,2)} -\allowbreak \mathcal{H}_{k}^{(0,1,2)} +\allowbreak 4\,\allowbreak \mathcal{H}_{k}^{(1,2)}\!(4,\allowbreak 1) +\allowbreak 4\,\allowbreak \mathcal{H}_{k}^{(0,1,2)}\!(4,\allowbreak 1,\allowbreak 1)\bigr)
\end{dmath*}
\subsection*{Identity 20}
\begin{dmath*}
\sum_{n_{1}=1}^{k}\sum_{n_{2}=1}^{n_{1}-1}\sum_{n_{3}=1}^{n_{2}-1}
\frac{\chi_{4,2}\!(n_{1})\,\chi_{4,2}\!(n_{2})\,\chi_{4,2}\!(n_{3})}{n_{1}^{2}\,n_{2}^{2}\,n_{3}^{2}}
= -\frac{\mathrm{i}}{48}\,\bigl(\bigl(6\,\mathcal{A}_{k}^{(4)}+6\,\mathcal{H}_{k}^{(4)}+\bigl(\mathcal{H}_{k}^{(2)}\!(-\mathrm{i})-\mathcal{H}_{k}^{(2)}\!(\mathrm{i})\bigr)^{2}\bigr)\,\bigl(\mathcal{H}_{k}^{(2)}\!(-\mathrm{i})-\mathcal{H}_{k}^{(2)}\!(\mathrm{i})\bigr)-8\,\mathcal{H}_{k}^{(6)}\!(-\mathrm{i})+8\,\mathcal{H}_{k}^{(6)}\!(\mathrm{i})\bigr)
\end{dmath*}
\section*{3.2 Aligned affine case}

\subsection*{Identity 21}
\begin{dmath*}
\sum_{n=1}^{k} \frac{\mathcal{H}_{2\,\allowbreak n +\allowbreak 3}^{(1,2)}\!(\frac{1}{4},\allowbreak \frac{1}{5})}{\bigl(2\,\allowbreak n +\allowbreak 3\bigr)^{2}} = \frac{1}{2}\,\allowbreak \bigl(-\frac{47}{28800} -\allowbreak \mathcal{H}_{3 +\allowbreak 2\,\allowbreak k}^{(3,2)}\!(-\bigl(\frac{1}{4}\bigr),\allowbreak \frac{1}{5}) +\allowbreak \mathcal{H}_{3 +\allowbreak 2\,\allowbreak k}^{(3,2)}\!(\frac{1}{4},\allowbreak \frac{1}{5}) -\allowbreak \mathcal{H}_{3 +\allowbreak 2\,\allowbreak k}^{(2,1,2)}\!(-1,\allowbreak \frac{1}{4},\allowbreak \frac{1}{5}) +\allowbreak \mathcal{H}_{3 +\allowbreak 2\,\allowbreak k}^{(2,1,2)}\!(1,\allowbreak \frac{1}{4},\allowbreak \frac{1}{5})\bigr)
\end{dmath*}
\subsection*{Identity 22}
\begin{dmath*}
\sum_{n=1}^{k} \bigl(2\,\allowbreak n +\allowbreak 5\bigr)^{2\,\allowbreak \mathrm{i}}\,\allowbreak \bigl(h_{2\,\allowbreak n +\allowbreak 5}^{[1]}\!(1;1)\bigr)^{2} = -1 -\allowbreak 169\,\allowbreak 3^{-2 +\allowbreak 2\,\allowbreak \mathrm{i}} -\allowbreak \frac{7569}{4}\,\allowbreak 5^{-2 +\allowbreak 2\,\allowbreak \mathrm{i}} +\allowbreak \frac{1}{2}\,\allowbreak \mathcal{A}_{5 +\allowbreak 2\,\allowbreak k}^{(2 -\allowbreak 2\,\allowbreak \mathrm{i})} -\allowbreak \frac{1}{2}\,\allowbreak \zeta\!(2 -\allowbreak 2\,\allowbreak \mathrm{i},6 +\allowbreak 2\,\allowbreak k) +\allowbreak 2\,\allowbreak \mathcal{H}_{5 +\allowbreak 2\,\allowbreak k}^{(-2\,\allowbreak \mathrm{i},2)} +\allowbreak 2\,\allowbreak \mathcal{H}_{5 +\allowbreak 2\,\allowbreak k}^{(1 -\allowbreak 2\,\allowbreak \mathrm{i},1)} +\allowbreak \frac{5}{2}\,\allowbreak \mathcal{H}_{5 +\allowbreak 2\,\allowbreak k}^{(-2\,\allowbreak \mathrm{i},0,2)}
\end{dmath*}
\vspace{-1.1em}
\begin{dmath*}
{}\quad +\allowbreak 6\,\allowbreak \mathcal{H}_{5 +\allowbreak 2\,\allowbreak k}^{(-2\,\allowbreak \mathrm{i},1,1)} +\allowbreak \mathcal{H}_{5 +\allowbreak 2\,\allowbreak k}^{(1 -\allowbreak 2\,\allowbreak \mathrm{i},0,1)} +\allowbreak \mathcal{H}_{5 +\allowbreak 2\,\allowbreak k}^{(-2\,\allowbreak \mathrm{i},0,0,2)} +\allowbreak 6\,\allowbreak \mathcal{H}_{5 +\allowbreak 2\,\allowbreak k}^{(-2\,\allowbreak \mathrm{i},0,1,1)} +\allowbreak 2\,\allowbreak \mathcal{H}_{5 +\allowbreak 2\,\allowbreak k}^{(-2\,\allowbreak \mathrm{i},1,0,1)} +\allowbreak 2\,\allowbreak \mathcal{H}_{5 +\allowbreak 2\,\allowbreak k}^{(-2\,\allowbreak \mathrm{i},0,0,1,1)} +\allowbreak \mathcal{H}_{5 +\allowbreak 2\,\allowbreak k}^{(-2\,\allowbreak \mathrm{i},0,1,0,1)}
\end{dmath*}
\vspace{-1.1em}
\begin{dmath*}
{}\quad -\allowbreak 2\,\allowbreak \mathcal{H}_{5 +\allowbreak 2\,\allowbreak k}^{(-2\,\allowbreak \mathrm{i},2)}\!(-1,\allowbreak 1) -\allowbreak 2\,\allowbreak \mathcal{H}_{5 +\allowbreak 2\,\allowbreak k}^{(1 -\allowbreak 2\,\allowbreak \mathrm{i},1)}\!(-1,\allowbreak 1) -\allowbreak \frac{5}{2}\,\allowbreak \mathcal{H}_{5 +\allowbreak 2\,\allowbreak k}^{(-2\,\allowbreak \mathrm{i},0,2)}\!(-1,\allowbreak 1,\allowbreak 1) -\allowbreak 6\,\allowbreak \mathcal{H}_{5 +\allowbreak 2\,\allowbreak k}^{(-2\,\allowbreak \mathrm{i},1,1)}\!(-1,\allowbreak 1,\allowbreak 1) -\allowbreak \mathcal{H}_{5 +\allowbreak 2\,\allowbreak k}^{(1 -\allowbreak 2\,\allowbreak \mathrm{i},0,1)}\!(-1,\allowbreak 1,\allowbreak 1)
\end{dmath*}
\vspace{-1.1em}
\begin{dmath*}
{}\quad -\allowbreak \mathcal{H}_{5 +\allowbreak 2\,\allowbreak k}^{(-2\,\allowbreak \mathrm{i},0,0,2)}\!(-1,\allowbreak 1,\allowbreak 1,\allowbreak 1) -\allowbreak 6\,\allowbreak \mathcal{H}_{5 +\allowbreak 2\,\allowbreak k}^{(-2\,\allowbreak \mathrm{i},0,1,1)}\!(-1,\allowbreak 1,\allowbreak 1,\allowbreak 1) -\allowbreak 2\,\allowbreak \mathcal{H}_{5 +\allowbreak 2\,\allowbreak k}^{(-2\,\allowbreak \mathrm{i},1,0,1)}\!(-1,\allowbreak 1,\allowbreak 1,\allowbreak 1) -\allowbreak 2\,\allowbreak \mathcal{H}_{5 +\allowbreak 2\,\allowbreak k}^{(-2\,\allowbreak \mathrm{i},0,0,1,1)}\!(-1,\allowbreak 1,\allowbreak 1,\allowbreak 1,\allowbreak 1) -\allowbreak \mathcal{H}_{5 +\allowbreak 2\,\allowbreak k}^{(-2\,\allowbreak \mathrm{i},0,1,0,1)}\!(-1,\allowbreak 1,\allowbreak 1,\allowbreak 1,\allowbreak 1)
\end{dmath*}
\vspace{-1.1em}
\begin{dmath*}
{}\quad +\allowbreak \frac{1}{2}\,\allowbreak \zeta\!(2 -\allowbreak 2\,\allowbreak \mathrm{i})
\end{dmath*}
\subsection*{Identity 23}
\begin{dmath*}
\sum_{n=1}^{k} \frac{\mathcal{H}_{2\,\allowbreak n +\allowbreak 5}^{(2)}\!(\frac{1}{3})\,\allowbreak \mathcal{A}_{2\,\allowbreak n +\allowbreak 5}}{\bigl(2\,\allowbreak n +\allowbreak 5\bigr)^{2}} = \frac{1}{2}\,\allowbreak \bigl(-\frac{496829993}{656100000} +\allowbreak \mathcal{A}_{5 +\allowbreak 2\,\allowbreak k}^{(5)}\!(\frac{1}{3}) +\allowbreak \mathcal{H}_{5 +\allowbreak 2\,\allowbreak k}^{(5)}\!(\frac{1}{3}) +\allowbreak \mathcal{H}_{5 +\allowbreak 2\,\allowbreak k}^{(2,3)}\!(-1,\allowbreak -\bigl(\frac{1}{3}\bigr)) -\allowbreak \mathcal{H}_{5 +\allowbreak 2\,\allowbreak k}^{(2,3)}\!(1,\allowbreak -\bigl(\frac{1}{3}\bigr)) -\allowbreak \mathcal{H}_{5 +\allowbreak 2\,\allowbreak k}^{(3,2)}\!(-1,\allowbreak \frac{1}{3}) +\allowbreak \mathcal{H}_{5 +\allowbreak 2\,\allowbreak k}^{(3,2)}\!(1,\allowbreak \frac{1}{3}) +\allowbreak \mathcal{H}_{5 +\allowbreak 2\,\allowbreak k}^{(4,1)}\!(-\bigl(\frac{1}{3}\bigr),\allowbreak -1)
\end{dmath*}
\vspace{-1.1em}
\begin{dmath*}
{}\quad -\allowbreak \mathcal{H}_{5 +\allowbreak 2\,\allowbreak k}^{(4,1)}\!(\frac{1}{3},\allowbreak -1) +\allowbreak \mathcal{H}_{5 +\allowbreak 2\,\allowbreak k}^{(2,1,2)}\!(-1,\allowbreak -1,\allowbreak \frac{1}{3}) -\allowbreak \mathcal{H}_{5 +\allowbreak 2\,\allowbreak k}^{(2,1,2)}\!(1,\allowbreak -1,\allowbreak \frac{1}{3}) +\allowbreak \mathcal{H}_{5 +\allowbreak 2\,\allowbreak k}^{(2,2,1)}\!(-1,\allowbreak \frac{1}{3},\allowbreak -1) -\allowbreak \mathcal{H}_{5 +\allowbreak 2\,\allowbreak k}^{(2,2,1)}\!(1,\allowbreak \frac{1}{3},\allowbreak -1)\bigr)
\end{dmath*}
\subsection*{Identity 24}
\begin{dmath*}
\sum_{n=1}^{k} \frac{\bigl(\frac{1}{2}\bigr)^{n}\,\allowbreak H_{2\,\allowbreak n +\allowbreak 3}\,\allowbreak \mathcal{H}_{2\,\allowbreak n +\allowbreak 3}^{(1,2)}}{\bigl(2\,\allowbreak n +\allowbreak 3\bigr)^{2}} = \sqrt{2}\,\allowbreak \bigl(-\frac{121}{648\,\allowbreak \sqrt{2}} -\allowbreak \mathcal{H}_{3 +\allowbreak 2\,\allowbreak k}^{(3,3)}\!(-\bigl(\frac{1}{\sqrt{2}}\bigr),\allowbreak 1) +\allowbreak \mathcal{H}_{3 +\allowbreak 2\,\allowbreak k}^{(3,3)}\!(\frac{1}{\sqrt{2}},\allowbreak 1) -\allowbreak \mathcal{H}_{3 +\allowbreak 2\,\allowbreak k}^{(4,2)}\!(-\bigl(\frac{1}{\sqrt{2}}\bigr),\allowbreak 1) +\allowbreak \mathcal{H}_{3 +\allowbreak 2\,\allowbreak k}^{(4,2)}\!(\frac{1}{\sqrt{2}},\allowbreak 1) -\allowbreak \mathcal{H}_{3 +\allowbreak 2\,\allowbreak k}^{(2,1,3)}\!(-\bigl(\frac{1}{\sqrt{2}}\bigr),\allowbreak 1,\allowbreak 1) +\allowbreak \mathcal{H}_{3 +\allowbreak 2\,\allowbreak k}^{(2,1,3)}\!(\frac{1}{\sqrt{2}},\allowbreak 1,\allowbreak 1)
\end{dmath*}
\vspace{-1.1em}
\begin{dmath*}
{}\quad -\allowbreak \mathcal{H}_{3 +\allowbreak 2\,\allowbreak k}^{(2,2,2)}\!(-\bigl(\frac{1}{\sqrt{2}}\bigr),\allowbreak 1,\allowbreak 1) +\allowbreak \mathcal{H}_{3 +\allowbreak 2\,\allowbreak k}^{(2,2,2)}\!(\frac{1}{\sqrt{2}},\allowbreak 1,\allowbreak 1) -\allowbreak 2\,\allowbreak \mathcal{H}_{3 +\allowbreak 2\,\allowbreak k}^{(3,1,2)}\!(-\bigl(\frac{1}{\sqrt{2}}\bigr),\allowbreak 1,\allowbreak 1) +\allowbreak 2\,\allowbreak \mathcal{H}_{3 +\allowbreak 2\,\allowbreak k}^{(3,1,2)}\!(\frac{1}{\sqrt{2}},\allowbreak 1,\allowbreak 1) -\allowbreak \mathcal{H}_{3 +\allowbreak 2\,\allowbreak k}^{(3,2,1)}\!(-\bigl(\frac{1}{\sqrt{2}}\bigr),\allowbreak 1,\allowbreak 1)
\end{dmath*}
\vspace{-1.1em}
\begin{dmath*}
{}\quad +\allowbreak \mathcal{H}_{3 +\allowbreak 2\,\allowbreak k}^{(3,2,1)}\!(\frac{1}{\sqrt{2}},\allowbreak 1,\allowbreak 1) -\allowbreak 2\,\allowbreak \mathcal{H}_{3 +\allowbreak 2\,\allowbreak k}^{(2,1,1,2)}\!(-\bigl(\frac{1}{\sqrt{2}}\bigr),\allowbreak 1,\allowbreak 1,\allowbreak 1) +\allowbreak 2\,\allowbreak \mathcal{H}_{3 +\allowbreak 2\,\allowbreak k}^{(2,1,1,2)}\!(\frac{1}{\sqrt{2}},\allowbreak 1,\allowbreak 1,\allowbreak 1) -\allowbreak \mathcal{H}_{3 +\allowbreak 2\,\allowbreak k}^{(2,1,2,1)}\!(-\bigl(\frac{1}{\sqrt{2}}\bigr),\allowbreak 1,\allowbreak 1,\allowbreak 1) +\allowbreak \mathcal{H}_{3 +\allowbreak 2\,\allowbreak k}^{(2,1,2,1)}\!(\frac{1}{\sqrt{2}},\allowbreak 1,\allowbreak 1,\allowbreak 1)\bigr)
\end{dmath*}
\section*{3.3 Arithmetic-progression shifted powers}

\subsection*{Identity 25}
\begin{dmath*}
\sum_{n=1}^{k} \frac{\mathcal{H}_{n}^{(r_{1},r_{2},r_{3})}\!(s_{1},\allowbreak s_{2},\allowbreak s_{3})}{2\,\allowbreak n +\allowbreak 1} = 2^{-4 +\allowbreak r_{1} +\allowbreak r_{2} +\allowbreak r_{3}}\,\allowbreak \bigl(-\mathcal{H}_{1 +\allowbreak 2\,\allowbreak k}^{(1,r_{1},r_{2},r_{3})}\!(-1,\allowbreak -\sqrt{s_{1}},\allowbreak -\sqrt{s_{2}},\allowbreak -\sqrt{s_{3}}) -\allowbreak \mathcal{H}_{1 +\allowbreak 2\,\allowbreak k}^{(1,r_{1},r_{2},r_{3})}\!(-1,\allowbreak -\sqrt{s_{1}},\allowbreak -\sqrt{s_{2}},\allowbreak \sqrt{s_{3}}) -\allowbreak \mathcal{H}_{1 +\allowbreak 2\,\allowbreak k}^{(1,r_{1},r_{2},r_{3})}\!(-1,\allowbreak -\sqrt{s_{1}},\allowbreak \sqrt{s_{2}},\allowbreak -\sqrt{s_{3}}) -\allowbreak \mathcal{H}_{1 +\allowbreak 2\,\allowbreak k}^{(1,r_{1},r_{2},r_{3})}\!(-1,\allowbreak -\sqrt{s_{1}},\allowbreak \sqrt{s_{2}},\allowbreak \sqrt{s_{3}})
\end{dmath*}
\vspace{-1.1em}
\begin{dmath*}
{}\quad -\allowbreak \mathcal{H}_{1 +\allowbreak 2\,\allowbreak k}^{(1,r_{1},r_{2},r_{3})}\!(-1,\allowbreak \sqrt{s_{1}},\allowbreak -\sqrt{s_{2}},\allowbreak -\sqrt{s_{3}}) -\allowbreak \mathcal{H}_{1 +\allowbreak 2\,\allowbreak k}^{(1,r_{1},r_{2},r_{3})}\!(-1,\allowbreak \sqrt{s_{1}},\allowbreak -\sqrt{s_{2}},\allowbreak \sqrt{s_{3}}) -\allowbreak \mathcal{H}_{1 +\allowbreak 2\,\allowbreak k}^{(1,r_{1},r_{2},r_{3})}\!(-1,\allowbreak \sqrt{s_{1}},\allowbreak \sqrt{s_{2}},\allowbreak -\sqrt{s_{3}}) -\allowbreak \mathcal{H}_{1 +\allowbreak 2\,\allowbreak k}^{(1,r_{1},r_{2},r_{3})}\!(-1,\allowbreak \sqrt{s_{1}},\allowbreak \sqrt{s_{2}},\allowbreak \sqrt{s_{3}})
\end{dmath*}
\vspace{-1.1em}
\begin{dmath*}
{}\quad +\allowbreak \mathcal{H}_{1 +\allowbreak 2\,\allowbreak k}^{(1,r_{1},r_{2},r_{3})}\!(1,\allowbreak -\sqrt{s_{1}},\allowbreak -\sqrt{s_{2}},\allowbreak -\sqrt{s_{3}}) +\allowbreak \mathcal{H}_{1 +\allowbreak 2\,\allowbreak k}^{(1,r_{1},r_{2},r_{3})}\!(1,\allowbreak -\sqrt{s_{1}},\allowbreak -\sqrt{s_{2}},\allowbreak \sqrt{s_{3}}) +\allowbreak \mathcal{H}_{1 +\allowbreak 2\,\allowbreak k}^{(1,r_{1},r_{2},r_{3})}\!(1,\allowbreak -\sqrt{s_{1}},\allowbreak \sqrt{s_{2}},\allowbreak -\sqrt{s_{3}}) +\allowbreak \mathcal{H}_{1 +\allowbreak 2\,\allowbreak k}^{(1,r_{1},r_{2},r_{3})}\!(1,\allowbreak -\sqrt{s_{1}},\allowbreak \sqrt{s_{2}},\allowbreak \sqrt{s_{3}})
\end{dmath*}
\vspace{-1.1em}
\begin{dmath*}
{}\quad +\allowbreak \mathcal{H}_{1 +\allowbreak 2\,\allowbreak k}^{(1,r_{1},r_{2},r_{3})}\!(1,\allowbreak \sqrt{s_{1}},\allowbreak -\sqrt{s_{2}},\allowbreak -\sqrt{s_{3}}) +\allowbreak \mathcal{H}_{1 +\allowbreak 2\,\allowbreak k}^{(1,r_{1},r_{2},r_{3})}\!(1,\allowbreak \sqrt{s_{1}},\allowbreak -\sqrt{s_{2}},\allowbreak \sqrt{s_{3}}) +\allowbreak \mathcal{H}_{1 +\allowbreak 2\,\allowbreak k}^{(1,r_{1},r_{2},r_{3})}\!(1,\allowbreak \sqrt{s_{1}},\allowbreak \sqrt{s_{2}},\allowbreak -\sqrt{s_{3}}) +\allowbreak \mathcal{H}_{1 +\allowbreak 2\,\allowbreak k}^{(1,r_{1},r_{2},r_{3})}\!(1,\allowbreak \sqrt{s_{1}},\allowbreak \sqrt{s_{2}},\allowbreak \sqrt{s_{3}})\bigr)
\end{dmath*}
\subsection*{Identity 26}
\begin{dmath*}
\sum_{n=1}^{k} \frac{\mathcal{H}_{n +\allowbreak 1}^{(2)}\!(2\,\allowbreak \mathrm{i})}{n^{2}} = -6\,\allowbreak \mathrm{i} -\allowbreak 4\,\allowbreak \mathrm{i}\,\allowbreak \mathcal{H}_{k}^{(1)}\!(2\,\allowbreak \mathrm{i}) +\allowbreak \bigl(2\,\allowbreak \mathrm{i} +\allowbreak \mathcal{H}_{k}^{(2)}\bigr)\,\allowbreak \mathcal{H}_{k}^{(2)}\!(2\,\allowbreak \mathrm{i}) +\allowbreak 2\,\allowbreak \mathcal{H}_{1 +\allowbreak k}^{(1)}\!(2\,\allowbreak \mathrm{i}) +\allowbreak \mathcal{H}_{1 +\allowbreak k}^{(2)}\!(2\,\allowbreak \mathrm{i}) -\allowbreak \mathcal{H}_{k}^{(2,2)}\!(2\,\allowbreak \mathrm{i},\allowbreak 1)
\end{dmath*}
\subsection*{Identity 27}
\begin{dmath*}
\sum_{n=1}^{k} \mathrm{i}^{n}\,\allowbreak \mathcal{A}_{n +\allowbreak 1}^{(\sqrt{6} +\allowbreak \mathrm{i})}\!(\frac{2}{3})\,\allowbreak H_{n +\allowbreak 1} = \bigl(\frac{1}{6} +\allowbreak \frac{\mathrm{i}}{6}\bigr)\,\allowbreak \bigl(-2 +\allowbreak 2\,\allowbreak \mathrm{i} +\allowbreak H_{k}\,\allowbreak \bigl(-3\,\allowbreak \mathrm{i}\,\allowbreak \mathrm{i}^{k}\,\allowbreak \mathcal{A}_{k}^{(\mathrm{i} +\allowbreak \sqrt{6})}\!(\frac{2}{3}) -\allowbreak 3\,\allowbreak \mathcal{H}_{k}^{(\mathrm{i} +\allowbreak \sqrt{6})}\!(-\bigl(\frac{2\,\allowbreak \mathrm{i}}{3}\bigr))\bigr) +\allowbreak \bigl(3 +\allowbreak 3\,\allowbreak \mathrm{i}\bigr)\,\allowbreak \mathcal{H}_{1 +\allowbreak k}^{(1 +\allowbreak \mathrm{i} +\allowbreak \sqrt{6})}\!(-\bigl(\frac{2\,\allowbreak \mathrm{i}}{3}\bigr)) -\allowbreak 3\,\allowbreak \mathcal{H}_{k}^{(1,\mathrm{i} +\allowbreak \sqrt{6})}\!(\mathrm{i},\allowbreak -\bigl(\frac{2}{3}\bigr))
\end{dmath*}
\vspace{-1.1em}
\begin{dmath*}
{}\quad +\allowbreak 3\,\allowbreak \mathcal{H}_{k}^{(1,\mathrm{i} +\allowbreak \sqrt{6})}\!(1,\allowbreak -\bigl(\frac{2\,\allowbreak \mathrm{i}}{3}\bigr)) +\allowbreak \bigl(3 +\allowbreak 3\,\allowbreak \mathrm{i}\bigr)\,\allowbreak \mathcal{H}_{1 +\allowbreak k}^{(1,\mathrm{i} +\allowbreak \sqrt{6})}\!(\mathrm{i},\allowbreak -\bigl(\frac{2}{3}\bigr)) +\allowbreak \bigl(3 +\allowbreak 3\,\allowbreak \mathrm{i}\bigr)\,\allowbreak \mathcal{H}_{1 +\allowbreak k}^{(\mathrm{i} +\allowbreak \sqrt{6},1)}\!(-\bigl(\frac{2\,\allowbreak \mathrm{i}}{3}\bigr),\allowbreak 1)\bigr)
\end{dmath*}
\subsection*{Identity 28}
\begin{dmath*}
\sum_{n=1}^{k} \frac{\mathcal{H}_{n}^{(1,2)}}{\bigl(3\,\allowbreak n -\allowbreak 1\bigr)^{2}} = -3\,\allowbreak \mathcal{H}_{-1 +\allowbreak 3\,\allowbreak k}^{(3)} +\allowbreak 3\,\allowbreak \bigl(-1\bigr)^{\frac{1}{3}}\,\allowbreak \mathcal{H}_{-1 +\allowbreak 3\,\allowbreak k}^{(3)} -\allowbreak 3\,\allowbreak \bigl(-1\bigr)^{\frac{2}{3}}\,\allowbreak \mathcal{H}_{-1 +\allowbreak 3\,\allowbreak k}^{(3)} +\allowbreak 3\,\allowbreak \mathcal{H}_{-1 +\allowbreak 3\,\allowbreak k}^{(4)} -\allowbreak 3\,\allowbreak \bigl(-1\bigr)^{\frac{1}{3}}\,\allowbreak \mathcal{H}_{-1 +\allowbreak 3\,\allowbreak k}^{(4)} +\allowbreak 3\,\allowbreak \bigl(-1\bigr)^{\frac{2}{3}}\,\allowbreak \mathcal{H}_{-1 +\allowbreak 3\,\allowbreak k}^{(4)} -\allowbreak 3\,\allowbreak \mathcal{H}_{-1 +\allowbreak 3\,\allowbreak k}^{(3)}\!(-\bigl(\bigl(-1\bigr)^{\frac{1}{3}}\bigr))
\end{dmath*}
\vspace{-1.1em}
\begin{dmath*}
{}\quad +\allowbreak 3\,\allowbreak \bigl(-1\bigr)^{\frac{1}{3}}\,\allowbreak \mathcal{H}_{-1 +\allowbreak 3\,\allowbreak k}^{(3)}\!(-\bigl(\bigl(-1\bigr)^{\frac{1}{3}}\bigr)) -\allowbreak 3\,\allowbreak \bigl(-1\bigr)^{\frac{2}{3}}\,\allowbreak \mathcal{H}_{-1 +\allowbreak 3\,\allowbreak k}^{(3)}\!(-\bigl(\bigl(-1\bigr)^{\frac{1}{3}}\bigr)) -\allowbreak 3\,\allowbreak \mathcal{H}_{-1 +\allowbreak 3\,\allowbreak k}^{(3)}\!(\bigl(-1\bigr)^{\frac{2}{3}}) +\allowbreak 3\,\allowbreak \bigl(-1\bigr)^{\frac{1}{3}}\,\allowbreak \mathcal{H}_{-1 +\allowbreak 3\,\allowbreak k}^{(3)}\!(\bigl(-1\bigr)^{\frac{2}{3}}) -\allowbreak 3\,\allowbreak \bigl(-1\bigr)^{\frac{2}{3}}\,\allowbreak \mathcal{H}_{-1 +\allowbreak 3\,\allowbreak k}^{(3)}\!(\bigl(-1\bigr)^{\frac{2}{3}})
\end{dmath*}
\vspace{-1.1em}
\begin{dmath*}
{}\quad +\allowbreak 3\,\allowbreak \mathcal{H}_{-1 +\allowbreak 3\,\allowbreak k}^{(4)}\!(-\bigl(\bigl(-1\bigr)^{\frac{1}{3}}\bigr)) -\allowbreak 3\,\allowbreak \bigl(-1\bigr)^{\frac{1}{3}}\,\allowbreak \mathcal{H}_{-1 +\allowbreak 3\,\allowbreak k}^{(4)}\!(-\bigl(\bigl(-1\bigr)^{\frac{1}{3}}\bigr)) +\allowbreak 3\,\allowbreak \bigl(-1\bigr)^{\frac{2}{3}}\,\allowbreak \mathcal{H}_{-1 +\allowbreak 3\,\allowbreak k}^{(4)}\!(-\bigl(\bigl(-1\bigr)^{\frac{1}{3}}\bigr)) +\allowbreak 3\,\allowbreak \mathcal{H}_{-1 +\allowbreak 3\,\allowbreak k}^{(4)}\!(\bigl(-1\bigr)^{\frac{2}{3}}) -\allowbreak 3\,\allowbreak \bigl(-1\bigr)^{\frac{1}{3}}\,\allowbreak \mathcal{H}_{-1 +\allowbreak 3\,\allowbreak k}^{(4)}\!(\bigl(-1\bigr)^{\frac{2}{3}})
\end{dmath*}
\vspace{-1.1em}
\begin{dmath*}
{}\quad +\allowbreak 3\,\allowbreak \bigl(-1\bigr)^{\frac{2}{3}}\,\allowbreak \mathcal{H}_{-1 +\allowbreak 3\,\allowbreak k}^{(4)}\!(\bigl(-1\bigr)^{\frac{2}{3}}) +\allowbreak 3\,\allowbreak \mathcal{H}_{3\,\allowbreak k}^{(1,2)} -\allowbreak 3\,\allowbreak \mathcal{H}_{-1 +\allowbreak 3\,\allowbreak k}^{(1,2)} +\allowbreak 3\,\allowbreak \mathcal{H}_{-1 +\allowbreak 3\,\allowbreak k}^{(2,2)} +\allowbreak \mathcal{H}_{-1 +\allowbreak 3\,\allowbreak k}^{(3,2)} -\allowbreak \bigl(-1\bigr)^{\frac{1}{3}}\,\allowbreak \mathcal{H}_{-1 +\allowbreak 3\,\allowbreak k}^{(3,2)} +\allowbreak \bigl(-1\bigr)^{\frac{2}{3}}\,\allowbreak \mathcal{H}_{-1 +\allowbreak 3\,\allowbreak k}^{(3,2)}
\end{dmath*}
\vspace{-1.1em}
\begin{dmath*}
{}\quad +\allowbreak \mathcal{H}_{-1 +\allowbreak 3\,\allowbreak k}^{(2,1,2)} +\allowbreak 3\,\allowbreak \mathcal{H}_{3\,\allowbreak k}^{(1,2)}\!(1,\allowbreak -\bigl(\bigl(-1\bigr)^{\frac{1}{3}}\bigr)) +\allowbreak 3\,\allowbreak \mathcal{H}_{3\,\allowbreak k}^{(1,2)}\!(1,\allowbreak \bigl(-1\bigr)^{\frac{2}{3}}) +\allowbreak 3\,\allowbreak \mathcal{H}_{3\,\allowbreak k}^{(1,2)}\!(-\bigl(\bigl(-1\bigr)^{\frac{1}{3}}\bigr),\allowbreak 1) +\allowbreak 3\,\allowbreak \mathcal{H}_{3\,\allowbreak k}^{(1,2)}\!(-\bigl(\bigl(-1\bigr)^{\frac{1}{3}}\bigr),\allowbreak -\bigl(\bigl(-1\bigr)^{\frac{1}{3}}\bigr)) +\allowbreak 3\,\allowbreak \mathcal{H}_{3\,\allowbreak k}^{(1,2)}\!(-\bigl(\bigl(-1\bigr)^{\frac{1}{3}}\bigr),\allowbreak \bigl(-1\bigr)^{\frac{2}{3}})
\end{dmath*}
\vspace{-1.1em}
\begin{dmath*}
{}\quad +\allowbreak 3\,\allowbreak \mathcal{H}_{3\,\allowbreak k}^{(1,2)}\!(\bigl(-1\bigr)^{\frac{2}{3}},\allowbreak 1) +\allowbreak 3\,\allowbreak \mathcal{H}_{3\,\allowbreak k}^{(1,2)}\!(\bigl(-1\bigr)^{\frac{2}{3}},\allowbreak -\bigl(\bigl(-1\bigr)^{\frac{1}{3}}\bigr)) +\allowbreak 3\,\allowbreak \mathcal{H}_{3\,\allowbreak k}^{(1,2)}\!(\bigl(-1\bigr)^{\frac{2}{3}},\allowbreak \bigl(-1\bigr)^{\frac{2}{3}}) -\allowbreak 3\,\allowbreak \mathcal{H}_{-1 +\allowbreak 3\,\allowbreak k}^{(1,2)}\!(1,\allowbreak -\bigl(\bigl(-1\bigr)^{\frac{1}{3}}\bigr)) -\allowbreak 3\,\allowbreak \mathcal{H}_{-1 +\allowbreak 3\,\allowbreak k}^{(1,2)}\!(1,\allowbreak \bigl(-1\bigr)^{\frac{2}{3}})
\end{dmath*}
\vspace{-1.1em}
\begin{dmath*}
{}\quad +\allowbreak 3\,\allowbreak \bigl(-1\bigr)^{\frac{1}{3}}\,\allowbreak \mathcal{H}_{-1 +\allowbreak 3\,\allowbreak k}^{(1,2)}\!(-\bigl(\bigl(-1\bigr)^{\frac{1}{3}}\bigr),\allowbreak 1) +\allowbreak 3\,\allowbreak \bigl(-1\bigr)^{\frac{1}{3}}\,\allowbreak \mathcal{H}_{-1 +\allowbreak 3\,\allowbreak k}^{(1,2)}\!(-\bigl(\bigl(-1\bigr)^{\frac{1}{3}}\bigr),\allowbreak -\bigl(\bigl(-1\bigr)^{\frac{1}{3}}\bigr)) +\allowbreak 3\,\allowbreak \bigl(-1\bigr)^{\frac{1}{3}}\,\allowbreak \mathcal{H}_{-1 +\allowbreak 3\,\allowbreak k}^{(1,2)}\!(-\bigl(\bigl(-1\bigr)^{\frac{1}{3}}\bigr),\allowbreak \bigl(-1\bigr)^{\frac{2}{3}}) -\allowbreak 3\,\allowbreak \bigl(-1\bigr)^{\frac{2}{3}}\,\allowbreak \mathcal{H}_{-1 +\allowbreak 3\,\allowbreak k}^{(1,2)}\!(\bigl(-1\bigr)^{\frac{2}{3}},\allowbreak 1)
\end{dmath*}
\vspace{-1.1em}
\begin{dmath*}
{}\quad -\allowbreak 3\,\allowbreak \bigl(-1\bigr)^{\frac{2}{3}}\,\allowbreak \mathcal{H}_{-1 +\allowbreak 3\,\allowbreak k}^{(1,2)}\!(\bigl(-1\bigr)^{\frac{2}{3}},\allowbreak -\bigl(\bigl(-1\bigr)^{\frac{1}{3}}\bigr)) -\allowbreak 3\,\allowbreak \bigl(-1\bigr)^{\frac{2}{3}}\,\allowbreak \mathcal{H}_{-1 +\allowbreak 3\,\allowbreak k}^{(1,2)}\!(\bigl(-1\bigr)^{\frac{2}{3}},\allowbreak \bigl(-1\bigr)^{\frac{2}{3}}) +\allowbreak 3\,\allowbreak \mathcal{H}_{-1 +\allowbreak 3\,\allowbreak k}^{(2,2)}\!(1,\allowbreak -\bigl(\bigl(-1\bigr)^{\frac{1}{3}}\bigr)) +\allowbreak 3\,\allowbreak \mathcal{H}_{-1 +\allowbreak 3\,\allowbreak k}^{(2,2)}\!(1,\allowbreak \bigl(-1\bigr)^{\frac{2}{3}})
\end{dmath*}
\vspace{-1.1em}
\begin{dmath*}
{}\quad -\allowbreak 3\,\allowbreak \bigl(-1\bigr)^{\frac{1}{3}}\,\allowbreak \mathcal{H}_{-1 +\allowbreak 3\,\allowbreak k}^{(2,2)}\!(-\bigl(\bigl(-1\bigr)^{\frac{1}{3}}\bigr),\allowbreak 1) -\allowbreak 3\,\allowbreak \bigl(-1\bigr)^{\frac{1}{3}}\,\allowbreak \mathcal{H}_{-1 +\allowbreak 3\,\allowbreak k}^{(2,2)}\!(-\bigl(\bigl(-1\bigr)^{\frac{1}{3}}\bigr),\allowbreak -\bigl(\bigl(-1\bigr)^{\frac{1}{3}}\bigr)) -\allowbreak 3\,\allowbreak \bigl(-1\bigr)^{\frac{1}{3}}\,\allowbreak \mathcal{H}_{-1 +\allowbreak 3\,\allowbreak k}^{(2,2)}\!(-\bigl(\bigl(-1\bigr)^{\frac{1}{3}}\bigr),\allowbreak \bigl(-1\bigr)^{\frac{2}{3}}) +\allowbreak 3\,\allowbreak \bigl(-1\bigr)^{\frac{2}{3}}\,\allowbreak \mathcal{H}_{-1 +\allowbreak 3\,\allowbreak k}^{(2,2)}\!(\bigl(-1\bigr)^{\frac{2}{3}},\allowbreak 1)
\end{dmath*}
\vspace{-1.1em}
\begin{dmath*}
{}\quad +\allowbreak 3\,\allowbreak \bigl(-1\bigr)^{\frac{2}{3}}\,\allowbreak \mathcal{H}_{-1 +\allowbreak 3\,\allowbreak k}^{(2,2)}\!(\bigl(-1\bigr)^{\frac{2}{3}},\allowbreak -\bigl(\bigl(-1\bigr)^{\frac{1}{3}}\bigr)) +\allowbreak 3\,\allowbreak \bigl(-1\bigr)^{\frac{2}{3}}\,\allowbreak \mathcal{H}_{-1 +\allowbreak 3\,\allowbreak k}^{(2,2)}\!(\bigl(-1\bigr)^{\frac{2}{3}},\allowbreak \bigl(-1\bigr)^{\frac{2}{3}}) +\allowbreak \mathcal{H}_{-1 +\allowbreak 3\,\allowbreak k}^{(3,2)}\!(1,\allowbreak -\bigl(\bigl(-1\bigr)^{\frac{1}{3}}\bigr)) -\allowbreak \bigl(-1\bigr)^{\frac{1}{3}}\,\allowbreak \mathcal{H}_{-1 +\allowbreak 3\,\allowbreak k}^{(3,2)}\!(1,\allowbreak -\bigl(\bigl(-1\bigr)^{\frac{1}{3}}\bigr))
\end{dmath*}
\vspace{-1.1em}
\begin{dmath*}
{}\quad +\allowbreak \bigl(-1\bigr)^{\frac{2}{3}}\,\allowbreak \mathcal{H}_{-1 +\allowbreak 3\,\allowbreak k}^{(3,2)}\!(1,\allowbreak -\bigl(\bigl(-1\bigr)^{\frac{1}{3}}\bigr)) +\allowbreak \mathcal{H}_{-1 +\allowbreak 3\,\allowbreak k}^{(3,2)}\!(1,\allowbreak \bigl(-1\bigr)^{\frac{2}{3}}) -\allowbreak \bigl(-1\bigr)^{\frac{1}{3}}\,\allowbreak \mathcal{H}_{-1 +\allowbreak 3\,\allowbreak k}^{(3,2)}\!(1,\allowbreak \bigl(-1\bigr)^{\frac{2}{3}}) +\allowbreak \bigl(-1\bigr)^{\frac{2}{3}}\,\allowbreak \mathcal{H}_{-1 +\allowbreak 3\,\allowbreak k}^{(3,2)}\!(1,\allowbreak \bigl(-1\bigr)^{\frac{2}{3}}) +\allowbreak \mathcal{H}_{-1 +\allowbreak 3\,\allowbreak k}^{(3,2)}\!(-\bigl(\bigl(-1\bigr)^{\frac{1}{3}}\bigr),\allowbreak 1)
\end{dmath*}
\vspace{-1.1em}
\begin{dmath*}
{}\quad -\allowbreak \bigl(-1\bigr)^{\frac{1}{3}}\,\allowbreak \mathcal{H}_{-1 +\allowbreak 3\,\allowbreak k}^{(3,2)}\!(-\bigl(\bigl(-1\bigr)^{\frac{1}{3}}\bigr),\allowbreak 1) +\allowbreak \bigl(-1\bigr)^{\frac{2}{3}}\,\allowbreak \mathcal{H}_{-1 +\allowbreak 3\,\allowbreak k}^{(3,2)}\!(-\bigl(\bigl(-1\bigr)^{\frac{1}{3}}\bigr),\allowbreak 1) +\allowbreak \mathcal{H}_{-1 +\allowbreak 3\,\allowbreak k}^{(3,2)}\!(-\bigl(\bigl(-1\bigr)^{\frac{1}{3}}\bigr),\allowbreak -\bigl(\bigl(-1\bigr)^{\frac{1}{3}}\bigr)) -\allowbreak \bigl(-1\bigr)^{\frac{1}{3}}\,\allowbreak \mathcal{H}_{-1 +\allowbreak 3\,\allowbreak k}^{(3,2)}\!(-\bigl(\bigl(-1\bigr)^{\frac{1}{3}}\bigr),\allowbreak -\bigl(\bigl(-1\bigr)^{\frac{1}{3}}\bigr))
\end{dmath*}
\vspace{-1.1em}
\begin{dmath*}
{}\quad +\allowbreak \bigl(-1\bigr)^{\frac{2}{3}}\,\allowbreak \mathcal{H}_{-1 +\allowbreak 3\,\allowbreak k}^{(3,2)}\!(-\bigl(\bigl(-1\bigr)^{\frac{1}{3}}\bigr),\allowbreak -\bigl(\bigl(-1\bigr)^{\frac{1}{3}}\bigr)) +\allowbreak \mathcal{H}_{-1 +\allowbreak 3\,\allowbreak k}^{(3,2)}\!(-\bigl(\bigl(-1\bigr)^{\frac{1}{3}}\bigr),\allowbreak \bigl(-1\bigr)^{\frac{2}{3}}) -\allowbreak \bigl(-1\bigr)^{\frac{1}{3}}\,\allowbreak \mathcal{H}_{-1 +\allowbreak 3\,\allowbreak k}^{(3,2)}\!(-\bigl(\bigl(-1\bigr)^{\frac{1}{3}}\bigr),\allowbreak \bigl(-1\bigr)^{\frac{2}{3}}) +\allowbreak \bigl(-1\bigr)^{\frac{2}{3}}\,\allowbreak \mathcal{H}_{-1 +\allowbreak 3\,\allowbreak k}^{(3,2)}\!(-\bigl(\bigl(-1\bigr)^{\frac{1}{3}}\bigr),\allowbreak \bigl(-1\bigr)^{\frac{2}{3}})
\end{dmath*}
\vspace{-1.1em}
\begin{dmath*}
{}\quad +\allowbreak \mathcal{H}_{-1 +\allowbreak 3\,\allowbreak k}^{(3,2)}\!(\bigl(-1\bigr)^{\frac{2}{3}},\allowbreak 1) -\allowbreak \bigl(-1\bigr)^{\frac{1}{3}}\,\allowbreak \mathcal{H}_{-1 +\allowbreak 3\,\allowbreak k}^{(3,2)}\!(\bigl(-1\bigr)^{\frac{2}{3}},\allowbreak 1) +\allowbreak \bigl(-1\bigr)^{\frac{2}{3}}\,\allowbreak \mathcal{H}_{-1 +\allowbreak 3\,\allowbreak k}^{(3,2)}\!(\bigl(-1\bigr)^{\frac{2}{3}},\allowbreak 1) +\allowbreak \mathcal{H}_{-1 +\allowbreak 3\,\allowbreak k}^{(3,2)}\!(\bigl(-1\bigr)^{\frac{2}{3}},\allowbreak -\bigl(\bigl(-1\bigr)^{\frac{1}{3}}\bigr)) -\allowbreak \bigl(-1\bigr)^{\frac{1}{3}}\,\allowbreak \mathcal{H}_{-1 +\allowbreak 3\,\allowbreak k}^{(3,2)}\!(\bigl(-1\bigr)^{\frac{2}{3}},\allowbreak -\bigl(\bigl(-1\bigr)^{\frac{1}{3}}\bigr))
\end{dmath*}
\vspace{-1.1em}
\begin{dmath*}
{}\quad +\allowbreak \bigl(-1\bigr)^{\frac{2}{3}}\,\allowbreak \mathcal{H}_{-1 +\allowbreak 3\,\allowbreak k}^{(3,2)}\!(\bigl(-1\bigr)^{\frac{2}{3}},\allowbreak -\bigl(\bigl(-1\bigr)^{\frac{1}{3}}\bigr)) +\allowbreak \mathcal{H}_{-1 +\allowbreak 3\,\allowbreak k}^{(3,2)}\!(\bigl(-1\bigr)^{\frac{2}{3}},\allowbreak \bigl(-1\bigr)^{\frac{2}{3}}) -\allowbreak \bigl(-1\bigr)^{\frac{1}{3}}\,\allowbreak \mathcal{H}_{-1 +\allowbreak 3\,\allowbreak k}^{(3,2)}\!(\bigl(-1\bigr)^{\frac{2}{3}},\allowbreak \bigl(-1\bigr)^{\frac{2}{3}}) +\allowbreak \bigl(-1\bigr)^{\frac{2}{3}}\,\allowbreak \mathcal{H}_{-1 +\allowbreak 3\,\allowbreak k}^{(3,2)}\!(\bigl(-1\bigr)^{\frac{2}{3}},\allowbreak \bigl(-1\bigr)^{\frac{2}{3}})
\end{dmath*}
\vspace{-1.1em}
\begin{dmath*}
{}\quad +\allowbreak \mathcal{H}_{-1 +\allowbreak 3\,\allowbreak k}^{(2,1,2)}\!(1,\allowbreak 1,\allowbreak -\bigl(\bigl(-1\bigr)^{\frac{1}{3}}\bigr)) +\allowbreak \mathcal{H}_{-1 +\allowbreak 3\,\allowbreak k}^{(2,1,2)}\!(1,\allowbreak 1,\allowbreak \bigl(-1\bigr)^{\frac{2}{3}}) +\allowbreak \mathcal{H}_{-1 +\allowbreak 3\,\allowbreak k}^{(2,1,2)}\!(1,\allowbreak -\bigl(\bigl(-1\bigr)^{\frac{1}{3}}\bigr),\allowbreak 1) +\allowbreak \mathcal{H}_{-1 +\allowbreak 3\,\allowbreak k}^{(2,1,2)}\!(1,\allowbreak -\bigl(\bigl(-1\bigr)^{\frac{1}{3}}\bigr),\allowbreak -\bigl(\bigl(-1\bigr)^{\frac{1}{3}}\bigr)) +\allowbreak \mathcal{H}_{-1 +\allowbreak 3\,\allowbreak k}^{(2,1,2)}\!(1,\allowbreak -\bigl(\bigl(-1\bigr)^{\frac{1}{3}}\bigr),\allowbreak \bigl(-1\bigr)^{\frac{2}{3}})
\end{dmath*}
\vspace{-1.1em}
\begin{dmath*}
{}\quad +\allowbreak \mathcal{H}_{-1 +\allowbreak 3\,\allowbreak k}^{(2,1,2)}\!(1,\allowbreak \bigl(-1\bigr)^{\frac{2}{3}},\allowbreak 1) +\allowbreak \mathcal{H}_{-1 +\allowbreak 3\,\allowbreak k}^{(2,1,2)}\!(1,\allowbreak \bigl(-1\bigr)^{\frac{2}{3}},\allowbreak -\bigl(\bigl(-1\bigr)^{\frac{1}{3}}\bigr)) +\allowbreak \mathcal{H}_{-1 +\allowbreak 3\,\allowbreak k}^{(2,1,2)}\!(1,\allowbreak \bigl(-1\bigr)^{\frac{2}{3}},\allowbreak \bigl(-1\bigr)^{\frac{2}{3}}) -\allowbreak \bigl(-1\bigr)^{\frac{1}{3}}\,\allowbreak \mathcal{H}_{-1 +\allowbreak 3\,\allowbreak k}^{(2,1,2)}\!(-\bigl(\bigl(-1\bigr)^{\frac{1}{3}}\bigr),\allowbreak 1,\allowbreak 1) -\allowbreak \bigl(-1\bigr)^{\frac{1}{3}}\,\allowbreak \mathcal{H}_{-1 +\allowbreak 3\,\allowbreak k}^{(2,1,2)}\!(-\bigl(\bigl(-1\bigr)^{\frac{1}{3}}\bigr),\allowbreak 1,\allowbreak -\bigl(\bigl(-1\bigr)^{\frac{1}{3}}\bigr))
\end{dmath*}
\vspace{-1.1em}
\begin{dmath*}
{}\quad -\allowbreak \bigl(-1\bigr)^{\frac{1}{3}}\,\allowbreak \mathcal{H}_{-1 +\allowbreak 3\,\allowbreak k}^{(2,1,2)}\!(-\bigl(\bigl(-1\bigr)^{\frac{1}{3}}\bigr),\allowbreak 1,\allowbreak \bigl(-1\bigr)^{\frac{2}{3}}) -\allowbreak \bigl(-1\bigr)^{\frac{1}{3}}\,\allowbreak \mathcal{H}_{-1 +\allowbreak 3\,\allowbreak k}^{(2,1,2)}\!(-\bigl(\bigl(-1\bigr)^{\frac{1}{3}}\bigr),\allowbreak -\bigl(\bigl(-1\bigr)^{\frac{1}{3}}\bigr),\allowbreak 1) -\allowbreak \bigl(-1\bigr)^{\frac{1}{3}}\,\allowbreak \mathcal{H}_{-1 +\allowbreak 3\,\allowbreak k}^{(2,1,2)}\!(-\bigl(\bigl(-1\bigr)^{\frac{1}{3}}\bigr),\allowbreak -\bigl(\bigl(-1\bigr)^{\frac{1}{3}}\bigr),\allowbreak -\bigl(\bigl(-1\bigr)^{\frac{1}{3}}\bigr))
\end{dmath*}
\vspace{-1.1em}
\begin{dmath*}
{}\quad -\allowbreak \bigl(-1\bigr)^{\frac{1}{3}}\,\allowbreak \mathcal{H}_{-1 +\allowbreak 3\,\allowbreak k}^{(2,1,2)}\!(-\bigl(\bigl(-1\bigr)^{\frac{1}{3}}\bigr),\allowbreak -\bigl(\bigl(-1\bigr)^{\frac{1}{3}}\bigr),\allowbreak \bigl(-1\bigr)^{\frac{2}{3}}) -\allowbreak \bigl(-1\bigr)^{\frac{1}{3}}\,\allowbreak \mathcal{H}_{-1 +\allowbreak 3\,\allowbreak k}^{(2,1,2)}\!(-\bigl(\bigl(-1\bigr)^{\frac{1}{3}}\bigr),\allowbreak \bigl(-1\bigr)^{\frac{2}{3}},\allowbreak 1) -\allowbreak \bigl(-1\bigr)^{\frac{1}{3}}\,\allowbreak \mathcal{H}_{-1 +\allowbreak 3\,\allowbreak k}^{(2,1,2)}\!(-\bigl(\bigl(-1\bigr)^{\frac{1}{3}}\bigr),\allowbreak \bigl(-1\bigr)^{\frac{2}{3}},\allowbreak -\bigl(\bigl(-1\bigr)^{\frac{1}{3}}\bigr))
\end{dmath*}
\vspace{-1.1em}
\begin{dmath*}
{}\quad -\allowbreak \bigl(-1\bigr)^{\frac{1}{3}}\,\allowbreak \mathcal{H}_{-1 +\allowbreak 3\,\allowbreak k}^{(2,1,2)}\!(-\bigl(\bigl(-1\bigr)^{\frac{1}{3}}\bigr),\allowbreak \bigl(-1\bigr)^{\frac{2}{3}},\allowbreak \bigl(-1\bigr)^{\frac{2}{3}}) +\allowbreak \bigl(-1\bigr)^{\frac{2}{3}}\,\allowbreak \mathcal{H}_{-1 +\allowbreak 3\,\allowbreak k}^{(2,1,2)}\!(\bigl(-1\bigr)^{\frac{2}{3}},\allowbreak 1,\allowbreak 1) +\allowbreak \bigl(-1\bigr)^{\frac{2}{3}}\,\allowbreak \mathcal{H}_{-1 +\allowbreak 3\,\allowbreak k}^{(2,1,2)}\!(\bigl(-1\bigr)^{\frac{2}{3}},\allowbreak 1,\allowbreak -\bigl(\bigl(-1\bigr)^{\frac{1}{3}}\bigr)) +\allowbreak \bigl(-1\bigr)^{\frac{2}{3}}\,\allowbreak \mathcal{H}_{-1 +\allowbreak 3\,\allowbreak k}^{(2,1,2)}\!(\bigl(-1\bigr)^{\frac{2}{3}},\allowbreak 1,\allowbreak \bigl(-1\bigr)^{\frac{2}{3}})
\end{dmath*}
\vspace{-1.1em}
\begin{dmath*}
{}\quad +\allowbreak \bigl(-1\bigr)^{\frac{2}{3}}\,\allowbreak \mathcal{H}_{-1 +\allowbreak 3\,\allowbreak k}^{(2,1,2)}\!(\bigl(-1\bigr)^{\frac{2}{3}},\allowbreak -\bigl(\bigl(-1\bigr)^{\frac{1}{3}}\bigr),\allowbreak 1) +\allowbreak \bigl(-1\bigr)^{\frac{2}{3}}\,\allowbreak \mathcal{H}_{-1 +\allowbreak 3\,\allowbreak k}^{(2,1,2)}\!(\bigl(-1\bigr)^{\frac{2}{3}},\allowbreak -\bigl(\bigl(-1\bigr)^{\frac{1}{3}}\bigr),\allowbreak -\bigl(\bigl(-1\bigr)^{\frac{1}{3}}\bigr)) +\allowbreak \bigl(-1\bigr)^{\frac{2}{3}}\,\allowbreak \mathcal{H}_{-1 +\allowbreak 3\,\allowbreak k}^{(2,1,2)}\!(\bigl(-1\bigr)^{\frac{2}{3}},\allowbreak -\bigl(\bigl(-1\bigr)^{\frac{1}{3}}\bigr),\allowbreak \bigl(-1\bigr)^{\frac{2}{3}})
\end{dmath*}
\vspace{-1.1em}
\begin{dmath*}
{}\quad +\allowbreak \bigl(-1\bigr)^{\frac{2}{3}}\,\allowbreak \mathcal{H}_{-1 +\allowbreak 3\,\allowbreak k}^{(2,1,2)}\!(\bigl(-1\bigr)^{\frac{2}{3}},\allowbreak \bigl(-1\bigr)^{\frac{2}{3}},\allowbreak 1) +\allowbreak \bigl(-1\bigr)^{\frac{2}{3}}\,\allowbreak \mathcal{H}_{-1 +\allowbreak 3\,\allowbreak k}^{(2,1,2)}\!(\bigl(-1\bigr)^{\frac{2}{3}},\allowbreak \bigl(-1\bigr)^{\frac{2}{3}},\allowbreak -\bigl(\bigl(-1\bigr)^{\frac{1}{3}}\bigr)) +\allowbreak \bigl(-1\bigr)^{\frac{2}{3}}\,\allowbreak \mathcal{H}_{-1 +\allowbreak 3\,\allowbreak k}^{(2,1,2)}\!(\bigl(-1\bigr)^{\frac{2}{3}},\allowbreak \bigl(-1\bigr)^{\frac{2}{3}},\allowbreak \bigl(-1\bigr)^{\frac{2}{3}})
\end{dmath*}
\subsection*{Identity 29}
\begin{dmath*}
\sum_{n=1}^{k} \frac{h_{2\,\allowbreak n +\allowbreak 1}^{[1]}\!(2;2\,\allowbreak \mathrm{i})}{n^{2}} = 2\,\allowbreak \bigl(-8\,\allowbreak \mathrm{i} +\allowbreak 4\,\allowbreak \mathrm{i}\,\allowbreak \mathcal{A}_{k}^{(1)}\!(4) -\allowbreak \mathrm{i}\,\allowbreak \mathcal{A}_{k}^{(2)}\!(4) +\allowbreak 4\,\allowbreak \mathrm{i}\,\allowbreak \arctan\!(2) +\allowbreak 2\,\allowbreak \mathcal{H}_{2\,\allowbreak k}^{(4)}\!(-2\,\allowbreak \mathrm{i}) +\allowbreak 2\,\allowbreak \mathcal{H}_{2\,\allowbreak k}^{(4)}\!(2\,\allowbreak \mathrm{i}) +\allowbreak \mathrm{i}\,\allowbreak \bigl(-1\bigr)^{k}\,\allowbreak 4^{2 +\allowbreak k}\,\allowbreak \Phi\!(-4,1,\frac{3}{2} +\allowbreak k) -\allowbreak 4\,\allowbreak \mathrm{i}\,\allowbreak \Phi\!(-4,2,\frac{3}{2})
\end{dmath*}
\vspace{-1.1em}
\begin{dmath*}
{}\quad +\allowbreak \mathrm{i}\,\allowbreak \bigl(-1\bigr)^{k}\,\allowbreak 4^{1 +\allowbreak k}\,\allowbreak \Phi\!(-4,2,\frac{3}{2} +\allowbreak k) +\allowbreak \mathcal{H}_{2\,\allowbreak k}^{(1,2)}\!(-1,\allowbreak 2\,\allowbreak \mathrm{i}) +\allowbreak \mathcal{H}_{2\,\allowbreak k}^{(1,2)}\!(1,\allowbreak 2\,\allowbreak \mathrm{i}) -\allowbreak \mathcal{H}_{2\,\allowbreak k}^{(2,1)}\!(-1,\allowbreak 2\,\allowbreak \mathrm{i}) -\allowbreak \mathcal{H}_{2\,\allowbreak k}^{(2,1)}\!(1,\allowbreak 2\,\allowbreak \mathrm{i}) +\allowbreak 2\,\allowbreak \mathcal{H}_{2\,\allowbreak k}^{(2,2)}\!(-1,\allowbreak 2\,\allowbreak \mathrm{i}) +\allowbreak 2\,\allowbreak \mathcal{H}_{2\,\allowbreak k}^{(2,2)}\!(1,\allowbreak 2\,\allowbreak \mathrm{i})\bigr)
\end{dmath*}
\section*{4 Affine-letter extensions}

\subsection*{Identity 30}
\begin{dmath*}
\sum_{n=1}^{k} \frac{\mathcal{H}_{n}^{(2)}\!(\frac{1}{3})}{\bigl(2\,\allowbreak n +\allowbreak 10\bigr)^{\frac{1}{2} +\allowbreak \mathrm{i}}} = \mathcal{G}_{n}\!(\{\{-\frac{1}{2} -\allowbreak \mathrm{i},\allowbreak 2\}\};\allowbreak \{\frac{1}{3}\};\allowbreak \{\{\{2,\allowbreak 10\},\allowbreak \{0,\allowbreak 1\}\}\}) +\allowbreak \mathcal{G}_{n}\!(\{\{-\frac{1}{2} -\allowbreak \mathrm{i}\},\allowbreak \{2\}\};\allowbreak \{1,\allowbreak \frac{1}{3}\};\allowbreak \{\{\{2,\allowbreak 10\}\},\allowbreak \{\{0,\allowbreak 1\}\}\})
\end{dmath*}
\subsection*{Identity 31}
\begin{dmath*}
\sum_{n=1}^{k} \frac{\operatorname{Mod}\!(n,2)\,\allowbreak \mathcal{A}_{n}^{(2\,\allowbreak \mathrm{i})}}{\bigl(n +\allowbreak \mathrm{i}\bigr)^{3 +\allowbreak 2\,\allowbreak \mathrm{i}}} = \frac{1}{2}\,\allowbreak \bigl(-\mathcal{G}_{k}\!(\{\{3 +\allowbreak 2\,\allowbreak \mathrm{i},\allowbreak 2\,\allowbreak \mathrm{i}\}\};\allowbreak \{1\};\allowbreak \{\{\{1,\allowbreak \mathrm{i}\},\allowbreak \{1,\allowbreak 0\}\}\}) -\allowbreak \mathcal{G}_{k}\!(\{\{3 +\allowbreak 2\,\allowbreak \mathrm{i}\},\allowbreak \{2\,\allowbreak \mathrm{i}\}\};\allowbreak \{-1,\allowbreak -1\};\allowbreak \{\{\{1,\allowbreak \mathrm{i}\}\},\allowbreak \{\{1,\allowbreak 0\}\}\})\bigr) +\allowbreak \frac{1}{2}\,\allowbreak \bigl(\mathcal{G}_{k}\!(\{\{3 +\allowbreak 2\,\allowbreak \mathrm{i},\allowbreak 2\,\allowbreak \mathrm{i}\}\};\allowbreak \{-1\};\allowbreak \{\{\{1,\allowbreak \mathrm{i}\},\allowbreak \{1,\allowbreak 0\}\}\})
\end{dmath*}
\vspace{-1.1em}
\begin{dmath*}
{}\quad +\allowbreak \mathcal{G}_{k}\!(\{\{3 +\allowbreak 2\,\allowbreak \mathrm{i}\},\allowbreak \{2\,\allowbreak \mathrm{i}\}\};\allowbreak \{1,\allowbreak -1\};\allowbreak \{\{\{1,\allowbreak \mathrm{i}\}\},\allowbreak \{\{1,\allowbreak 0\}\}\})\bigr)
\end{dmath*}
\subsection*{Identity 32}
\begin{dmath*}
\sum_{n=1}^{10} \bigl(-1\bigr)^{n}\,\allowbreak \bigl(7\,\allowbreak n +\allowbreak 3\bigr)^{2\,\allowbreak \mathrm{i}}\,\allowbreak \mathcal{H}_{n}^{(2,3)}\!(\frac{\mathrm{i}}{6},\allowbreak \frac{\mathrm{i}}{8}) = \mathcal{G}_{k}\!(\{\{-2\,\allowbreak \mathrm{i},\allowbreak 2\},\allowbreak \{3\}\};\allowbreak \{-\bigl(\frac{\mathrm{i}}{6}\bigr),\allowbreak \frac{\mathrm{i}}{8}\};\allowbreak \{\{\{7,\allowbreak 3\},\allowbreak \{1,\allowbreak 0\}\},\allowbreak \{\{1,\allowbreak 0\}\}\}) +\allowbreak \mathcal{G}_{k}\!(\{\{-2\,\allowbreak \mathrm{i}\},\allowbreak \{2\},\allowbreak \{3\}\};\allowbreak \{-1,\allowbreak \frac{\mathrm{i}}{6},\allowbreak \frac{\mathrm{i}}{8}\};\allowbreak \{\{\{7,\allowbreak 3\}\},\allowbreak \{\{1,\allowbreak 0\}\},\allowbreak \{\{1,\allowbreak 0\}\}\})
\end{dmath*}
\subsection*{Identity 33}
\begin{dmath*}
\sum_{n=1}^{k} \frac{\bigl(\frac{1}{2}\bigr)^{n}\,\allowbreak \mathcal{H}_{n}^{(2,3)}\!(\mathrm{i},\allowbreak \frac{7}{3})}{\bigl(6\,\allowbreak n +\allowbreak 4\bigr)^{\mathrm{i}}\,\allowbreak \bigl(8\,\allowbreak n -\allowbreak 2\bigr)^{2\,\allowbreak \mathrm{i}}\,\allowbreak \bigl(10\,\allowbreak n -\allowbreak 8\bigr)^{3\,\allowbreak \mathrm{i}}} = \mathcal{G}_{k}\!(\{\{\mathrm{i},\allowbreak 2\,\allowbreak \mathrm{i},\allowbreak 3\,\allowbreak \mathrm{i},\allowbreak 2\},\allowbreak \{3\}\};\allowbreak \{\frac{\mathrm{i}}{2},\allowbreak \frac{7}{3}\};\allowbreak \{\{\{6,\allowbreak 4\},\allowbreak \{8,\allowbreak -2\},\allowbreak \{10,\allowbreak -8\},\allowbreak \{1,\allowbreak 0\}\},\allowbreak \{\{1,\allowbreak 0\}\}\}) +\allowbreak \mathcal{G}_{k}\!(\{\{\mathrm{i},\allowbreak 2\,\allowbreak \mathrm{i},\allowbreak 3\,\allowbreak \mathrm{i}\},\allowbreak \{2\},\allowbreak \{3\}\};\allowbreak \{\frac{1}{2},\allowbreak \mathrm{i},\allowbreak \frac{7}{3}\};\allowbreak \{\{\{6,\allowbreak 4\},\allowbreak \{8,\allowbreak -2\},\allowbreak \{10,\allowbreak -8\}\},\allowbreak \{\{1,\allowbreak 0\}\},\allowbreak \{\{1,\allowbreak 0\}\}\})
\end{dmath*}
\subsection*{Identity 34}
\begin{dmath*}
\sum_{n=1}^{10} \frac{\bigl(-12\bigr)^{n}\,\allowbreak \bigl(3\,\allowbreak n +\allowbreak 2\bigr)^{\frac{1}{2} +\allowbreak \mathrm{i}}\,\allowbreak \mathcal{G}_{n}\!(\{\{1 +\allowbreak \mathrm{i},\allowbreak 2\},\allowbreak \{3 -\allowbreak \mathrm{i}\}\};\allowbreak \{\frac{2}{3},\allowbreak \frac{-\mathrm{i}}{5}\};\allowbreak \{\{\{1,\allowbreak 1\},\allowbreak \{3,\allowbreak -2\}\},\allowbreak \{\{2,\allowbreak 1 -\allowbreak \mathrm{i}\}\}\})}{\bigl(5\,\allowbreak n +\allowbreak 7\bigr)^{2\,\allowbreak \mathrm{i}}\,\allowbreak \bigl(9\,\allowbreak n -\allowbreak 1\bigr)^{1 +\allowbreak \mathrm{i}}} = \mathcal{G}_{k}\!(\{\{-\frac{1}{2} -\allowbreak \mathrm{i},\allowbreak 2\,\allowbreak \mathrm{i},\allowbreak 1 +\allowbreak \mathrm{i},\allowbreak 1 +\allowbreak \mathrm{i},\allowbreak 2\},\allowbreak \{3 -\allowbreak \mathrm{i}\}\};\allowbreak \{-8,\allowbreak -\bigl(\frac{\mathrm{i}}{5}\bigr)\};\allowbreak \{\{\{3,\allowbreak 2\},\allowbreak \{5,\allowbreak 7\},\allowbreak \{9,\allowbreak -1\},\allowbreak \{1,\allowbreak 1\},\allowbreak \{3,\allowbreak -2\}\},\allowbreak \{\{2,\allowbreak 1 -\allowbreak \mathrm{i}\}\}\}) +\allowbreak \mathcal{G}_{k}\!(\{\{-\frac{1}{2} -\allowbreak \mathrm{i},\allowbreak 2\,\allowbreak \mathrm{i},\allowbreak 1 +\allowbreak \mathrm{i}\},\allowbreak \{1 +\allowbreak \mathrm{i},\allowbreak 2\},\allowbreak \{3 -\allowbreak \mathrm{i}\}\};\allowbreak \{-12,\allowbreak \frac{2}{3},\allowbreak -\bigl(\frac{\mathrm{i}}{5}\bigr)\};\allowbreak \{\{\{3,\allowbreak 2\},\allowbreak \{5,\allowbreak 7\},\allowbreak \{9,\allowbreak -1\}\},\allowbreak \{\{1,\allowbreak 1\},\allowbreak \{3,\allowbreak -2\}\},\allowbreak \{\{2,\allowbreak 1 -\allowbreak \mathrm{i}\}\}\})
\end{dmath*}
\section*{5 Polynomial-letter extensions}

\subsection*{Identity 35}
\begin{dmath*}
\sum_{n=1}^{10} \frac{\bigl(\frac{1}{3}\bigr)^{n}\,\allowbreak \mathcal{H}_{n}^{(2)}\!(\frac{1}{5})}{\bigl(1 +\allowbreak 6\,\allowbreak n +\allowbreak 3\,\allowbreak n^{2}\bigr)^{\frac{1}{2} +\allowbreak \mathrm{i}}} = \mathcal{P}_{k}\!(\{\{\frac{1}{2} +\allowbreak \mathrm{i},\allowbreak 2\}\};\allowbreak \{\frac{1}{15}\};\allowbreak \{\{\{1,\allowbreak 6,\allowbreak 3\},\allowbreak \{0,\allowbreak 1\}\}\}) +\allowbreak \mathcal{P}_{k}\!(\{\{\frac{1}{2} +\allowbreak \mathrm{i}\},\allowbreak \{2\}\};\allowbreak \{\frac{1}{3},\allowbreak \frac{1}{5}\};\allowbreak \{\{\{1,\allowbreak 6,\allowbreak 3\}\},\allowbreak \{\{0,\allowbreak 1\}\}\})
\end{dmath*}
\subsection*{Identity 36}
\begin{dmath*}
\sum_{n=1}^{10} \bigl(-1\bigr)^{n}\,\allowbreak \bigl(2 +\allowbreak 3\,\allowbreak n +\allowbreak n^{2}\bigr)^{\frac{1}{2} +\allowbreak \mathrm{i}}\,\allowbreak \mathcal{H}_{n}^{(2,3)}\!(\frac{\mathrm{i}}{6},\allowbreak \frac{\mathrm{i}}{8}) = \mathcal{P}_{k}\!(\{\{-\frac{1}{2} -\allowbreak \mathrm{i},\allowbreak 2\},\allowbreak \{3\}\};\allowbreak \{-\bigl(\frac{\mathrm{i}}{6}\bigr),\allowbreak \frac{\mathrm{i}}{8}\};\allowbreak \{\{\{2,\allowbreak 3,\allowbreak 1\},\allowbreak \{0,\allowbreak 1\}\},\allowbreak \{\{0,\allowbreak 1\}\}\}) +\allowbreak \mathcal{P}_{k}\!(\{\{-\frac{1}{2} -\allowbreak \mathrm{i}\},\allowbreak \{2\},\allowbreak \{3\}\};\allowbreak \{-1,\allowbreak \frac{\mathrm{i}}{6},\allowbreak \frac{\mathrm{i}}{8}\};\allowbreak \{\{\{2,\allowbreak 3,\allowbreak 1\}\},\allowbreak \{\{0,\allowbreak 1\}\},\allowbreak \{\{0,\allowbreak 1\}\}\})
\end{dmath*}
\subsection*{Identity 37}
\begin{dmath*}
\sum_{n=1}^{k} \frac{\bigl(\frac{1}{2}\bigr)^{n}\,\allowbreak \mathcal{H}_{n}^{(2,3)}\!(\mathrm{i},\allowbreak \frac{7}{3})}{\bigl(1 +\allowbreak n +\allowbreak n^{2}\bigr)^{\mathrm{i}}\,\allowbreak \bigl(3 -\allowbreak 2\,\allowbreak n +\allowbreak 5\,\allowbreak n^{2}\bigr)^{2\,\allowbreak \mathrm{i}}\,\allowbreak \bigl(7 +\allowbreak n^{3}\bigr)^{3\,\allowbreak \mathrm{i}}} = \mathcal{P}_{k}\!(\{\{\mathrm{i},\allowbreak 2\,\allowbreak \mathrm{i},\allowbreak 3\,\allowbreak \mathrm{i},\allowbreak 2\},\allowbreak \{3\}\};\allowbreak \{\frac{\mathrm{i}}{2},\allowbreak \frac{7}{3}\};\allowbreak \{\{\{1,\allowbreak 1,\allowbreak 1\},\allowbreak \{3,\allowbreak -2,\allowbreak 5\},\allowbreak \{7,\allowbreak 0,\allowbreak 0,\allowbreak 1\},\allowbreak \{0,\allowbreak 1\}\},\allowbreak \{\{0,\allowbreak 1\}\}\}) +\allowbreak \mathcal{P}_{k}\!(\{\{\mathrm{i},\allowbreak 2\,\allowbreak \mathrm{i},\allowbreak 3\,\allowbreak \mathrm{i}\},\allowbreak \{2\},\allowbreak \{3\}\};\allowbreak \{\frac{1}{2},\allowbreak \mathrm{i},\allowbreak \frac{7}{3}\};\allowbreak \{\{\{1,\allowbreak 1,\allowbreak 1\},\allowbreak \{3,\allowbreak -2,\allowbreak 5\},\allowbreak \{7,\allowbreak 0,\allowbreak 0,\allowbreak 1\}\},\allowbreak \{\{0,\allowbreak 1\}\},\allowbreak \{\{0,\allowbreak 1\}\}\})
\end{dmath*}
\subsection*{Identity 38}
\begin{dmath*}
\sum_{n=1}^{k} \frac{\bigl(\frac{-2}{5}\bigr)^{n}\,\allowbreak \mathcal{P}_{n}\!(\{\{2,\allowbreak 1\},\allowbreak \{3\}\};\allowbreak \{\frac{1}{2},\allowbreak \frac{-\mathrm{i}}{3}\};\allowbreak \{\{\{1,\allowbreak 1\},\allowbreak \{2,\allowbreak 0,\allowbreak 3\}\},\allowbreak \{\{-1,\allowbreak 2,\allowbreak 1\}\}\})}{\bigl(1 +\allowbreak 4\,\allowbreak n +\allowbreak n^{2}\bigr)^{\frac{3}{2}}\,\allowbreak \bigl(2 -\allowbreak \mathrm{i} +\allowbreak \bigl(1 +\allowbreak \mathrm{i}\bigr)\,\allowbreak n +\allowbreak n^{2}\bigr)^{\mathrm{i}}} = \mathcal{P}_{k}\!(\{\{\frac{3}{2},\allowbreak \mathrm{i},\allowbreak 2,\allowbreak 1\},\allowbreak \{3\}\};\allowbreak \{-\bigl(\frac{1}{5}\bigr),\allowbreak -\bigl(\frac{\mathrm{i}}{3}\bigr)\};\allowbreak \{\{\{1,\allowbreak 4,\allowbreak 1\},\allowbreak \{2 -\allowbreak \mathrm{i},\allowbreak 1 +\allowbreak \mathrm{i},\allowbreak 1\},\allowbreak \{1,\allowbreak 1\},\allowbreak \{2,\allowbreak 0,\allowbreak 3\}\},\allowbreak \{\{-1,\allowbreak 2,\allowbreak 1\}\}\}) +\allowbreak \mathcal{P}_{k}\!(\{\{\frac{3}{2},\allowbreak \mathrm{i}\},\allowbreak \{2,\allowbreak 1\},\allowbreak \{3\}\};\allowbreak \{-\bigl(\frac{2}{5}\bigr),\allowbreak \frac{1}{2},\allowbreak -\bigl(\frac{\mathrm{i}}{3}\bigr)\};\allowbreak \{\{\{1,\allowbreak 4,\allowbreak 1\},\allowbreak \{2 -\allowbreak \mathrm{i},\allowbreak 1 +\allowbreak \mathrm{i},\allowbreak 1\}\},\allowbreak \{\{1,\allowbreak 1\},\allowbreak \{2,\allowbreak 0,\allowbreak 3\}\},\allowbreak \{\{-1,\allowbreak 2,\allowbreak 1\}\}\})
\end{dmath*}
\subsection*{Identity 39}
\begin{dmath*}
\sum_{n=1}^{k} \frac{\bigl(\frac{1}{3}\bigr)^{n}\,\allowbreak \bigl(1 +\allowbreak n +\allowbreak n^{2}\bigr)^{\frac{1}{2} +\allowbreak \mathrm{i}}\,\allowbreak \mathcal{P}_{n}\!(\{\{2,\allowbreak 1\},\allowbreak \{3 -\allowbreak \mathrm{i}\}\};\allowbreak \{\frac{1}{2},\allowbreak \frac{-\mathrm{i}}{3}\};\allowbreak \{\{\{1,\allowbreak 1\},\allowbreak \{2,\allowbreak 0,\allowbreak 3\}\},\allowbreak \{\{-1,\allowbreak 2,\allowbreak 1\}\}\})}{\bigl(5 -\allowbreak n +\allowbreak 2\,\allowbreak n^{2} +\allowbreak n^{4}\bigr)^{2 +\allowbreak \mathrm{i}}} = \mathcal{P}_{k}\!(\{\{-\frac{1}{2} -\allowbreak \mathrm{i},\allowbreak 2 +\allowbreak \mathrm{i},\allowbreak 2,\allowbreak 1\},\allowbreak \{3 -\allowbreak \mathrm{i}\}\};\allowbreak \{\frac{1}{6},\allowbreak -\bigl(\frac{\mathrm{i}}{3}\bigr)\};\allowbreak \{\{\{1,\allowbreak 1,\allowbreak 1\},\allowbreak \{5,\allowbreak -1,\allowbreak 2,\allowbreak 0,\allowbreak 1\},\allowbreak \{1,\allowbreak 1\},\allowbreak \{2,\allowbreak 0,\allowbreak 3\}\},\allowbreak \{\{-1,\allowbreak 2,\allowbreak 1\}\}\}) +\allowbreak \mathcal{P}_{k}\!(\{\{-\frac{1}{2} -\allowbreak \mathrm{i},\allowbreak 2 +\allowbreak \mathrm{i}\},\allowbreak \{2,\allowbreak 1\},\allowbreak \{3 -\allowbreak \mathrm{i}\}\};\allowbreak \{\frac{1}{3},\allowbreak \frac{1}{2},\allowbreak -\bigl(\frac{\mathrm{i}}{3}\bigr)\};\allowbreak \{\{\{1,\allowbreak 1,\allowbreak 1\},\allowbreak \{5,\allowbreak -1,\allowbreak 2,\allowbreak 0,\allowbreak 1\}\},\allowbreak \{\{1,\allowbreak 1\},\allowbreak \{2,\allowbreak 0,\allowbreak 3\}\},\allowbreak \{\{-1,\allowbreak 2,\allowbreak 1\}\}\})
\end{dmath*}
\section*{6 Scaled-index sums}

\subsection*{Identity 40}
\begin{dmath*}
\sum_{n=1}^{k} \frac{H_{3\,\allowbreak n}}{n^{4}} = 27\,\allowbreak \bigl(\mathcal{H}_{3\,\allowbreak k}^{(5)} +\allowbreak \mathcal{H}_{3\,\allowbreak k}^{(5)}\!(-\bigl(\bigl(-1\bigr)^{\frac{1}{3}}\bigr)) +\allowbreak \mathcal{H}_{3\,\allowbreak k}^{(5)}\!(\bigl(-1\bigr)^{\frac{2}{3}}) +\allowbreak \mathcal{H}_{3\,\allowbreak k}^{(4,1)} +\allowbreak \mathcal{H}_{3\,\allowbreak k}^{(4,1)}\!(-\bigl(\bigl(-1\bigr)^{\frac{1}{3}}\bigr),\allowbreak 1) +\allowbreak \mathcal{H}_{3\,\allowbreak k}^{(4,1)}\!(\bigl(-1\bigr)^{\frac{2}{3}},\allowbreak 1)\bigr)
\end{dmath*}
\subsection*{Identity 41}
\begin{dmath*}
\sum_{n=1}^{k} \frac{H_{\frac{n}{2}}\,\allowbreak H_{2\,\allowbreak n}}{n^{4}} = 8\,\allowbreak \bigl(-4\,\allowbreak \mathcal{A}_{2\,\allowbreak k}^{(6)} +\allowbreak 4\,\allowbreak \mathcal{H}_{2\,\allowbreak k}^{(6)} +\allowbreak \mathcal{H}_{2\,\allowbreak k}^{(4,2)} +\allowbreak 5\,\allowbreak \mathcal{H}_{2\,\allowbreak k}^{(5,1)} +\allowbreak 2\,\allowbreak \mathcal{H}_{2\,\allowbreak k}^{(4,1,1)} +\allowbreak \log\!(2)\,\allowbreak \bigl(\mathcal{A}_{2\,\allowbreak k}^{(5)} -\allowbreak \mathcal{H}_{2\,\allowbreak k}^{(5)} +\allowbreak \mathcal{H}_{2\,\allowbreak k}^{(5)}\!(-\mathrm{i}) +\allowbreak \mathcal{H}_{2\,\allowbreak k}^{(5)}\!(\mathrm{i}) -\allowbreak \mathcal{H}_{2\,\allowbreak k}^{(4,1)} -\allowbreak \mathcal{H}_{2\,\allowbreak k}^{(4,1)}\!(-1,\allowbreak 1)
\end{dmath*}
\vspace{-1.1em}
\begin{dmath*}
{}\quad +\allowbreak \mathcal{H}_{2\,\allowbreak k}^{(4,1)}\!(-\mathrm{i},\allowbreak 1) +\allowbreak \mathcal{H}_{2\,\allowbreak k}^{(4,1)}\!(\mathrm{i},\allowbreak 1)\bigr) +\allowbreak \mathcal{H}_{2\,\allowbreak k}^{(4,2)}\!(-1,\allowbreak -1) +\allowbreak \mathcal{H}_{2\,\allowbreak k}^{(4,2)}\!(-1,\allowbreak 1) +\allowbreak \mathcal{H}_{2\,\allowbreak k}^{(4,2)}\!(-\mathrm{i},\allowbreak -\mathrm{i}) +\allowbreak \mathcal{H}_{2\,\allowbreak k}^{(4,2)}\!(-\mathrm{i},\allowbreak \mathrm{i}) +\allowbreak \mathcal{H}_{2\,\allowbreak k}^{(4,2)}\!(\mathrm{i},\allowbreak -\mathrm{i}) +\allowbreak \mathcal{H}_{2\,\allowbreak k}^{(4,2)}\!(\mathrm{i},\allowbreak \mathrm{i})
\end{dmath*}
\vspace{-1.1em}
\begin{dmath*}
{}\quad +\allowbreak \mathcal{H}_{2\,\allowbreak k}^{(4,2)}\!(1,\allowbreak -1) +\allowbreak \mathcal{H}_{2\,\allowbreak k}^{(5,1)}\!(-1,\allowbreak -1) +\allowbreak 5\,\allowbreak \mathcal{H}_{2\,\allowbreak k}^{(5,1)}\!(-1,\allowbreak 1) +\allowbreak \mathcal{H}_{2\,\allowbreak k}^{(5,1)}\!(-\mathrm{i},\allowbreak -\mathrm{i}) +\allowbreak \mathcal{H}_{2\,\allowbreak k}^{(5,1)}\!(-\mathrm{i},\allowbreak \mathrm{i}) +\allowbreak \mathcal{H}_{2\,\allowbreak k}^{(5,1)}\!(\mathrm{i},\allowbreak -\mathrm{i}) +\allowbreak \mathcal{H}_{2\,\allowbreak k}^{(5,1)}\!(\mathrm{i},\allowbreak \mathrm{i}) +\allowbreak \mathcal{H}_{2\,\allowbreak k}^{(5,1)}\!(1,\allowbreak -1)
\end{dmath*}
\vspace{-1.1em}
\begin{dmath*}
{}\quad +\allowbreak \mathcal{H}_{2\,\allowbreak k}^{(4,1,1)}\!(-1,\allowbreak -1,\allowbreak 1) +\allowbreak \mathcal{H}_{2\,\allowbreak k}^{(4,1,1)}\!(-1,\allowbreak 1,\allowbreak -1) +\allowbreak 2\,\allowbreak \mathcal{H}_{2\,\allowbreak k}^{(4,1,1)}\!(-1,\allowbreak 1,\allowbreak 1) +\allowbreak \mathcal{H}_{2\,\allowbreak k}^{(4,1,1)}\!(-\mathrm{i},\allowbreak -\mathrm{i},\allowbreak 1) +\allowbreak \mathcal{H}_{2\,\allowbreak k}^{(4,1,1)}\!(-\mathrm{i},\allowbreak \mathrm{i},\allowbreak 1) +\allowbreak \mathcal{H}_{2\,\allowbreak k}^{(4,1,1)}\!(-\mathrm{i},\allowbreak 1,\allowbreak -\mathrm{i}) +\allowbreak \mathcal{H}_{2\,\allowbreak k}^{(4,1,1)}\!(-\mathrm{i},\allowbreak 1,\allowbreak \mathrm{i})
\end{dmath*}
\vspace{-1.1em}
\begin{dmath*}
{}\quad +\allowbreak \mathcal{H}_{2\,\allowbreak k}^{(4,1,1)}\!(\mathrm{i},\allowbreak -\mathrm{i},\allowbreak 1) +\allowbreak \mathcal{H}_{2\,\allowbreak k}^{(4,1,1)}\!(\mathrm{i},\allowbreak \mathrm{i},\allowbreak 1) +\allowbreak \mathcal{H}_{2\,\allowbreak k}^{(4,1,1)}\!(\mathrm{i},\allowbreak 1,\allowbreak -\mathrm{i}) +\allowbreak \mathcal{H}_{2\,\allowbreak k}^{(4,1,1)}\!(\mathrm{i},\allowbreak 1,\allowbreak \mathrm{i}) +\allowbreak \mathcal{H}_{2\,\allowbreak k}^{(4,1,1)}\!(1,\allowbreak -1,\allowbreak 1) +\allowbreak \mathcal{H}_{2\,\allowbreak k}^{(4,1,1)}\!(1,\allowbreak 1,\allowbreak -1)\bigr)
\end{dmath*}
\subsection*{Identity 42}
\begin{dmath*}
\sum_{n=1}^{k} n^{4 +\allowbreak 2\,\allowbreak \mathrm{i}}\,\allowbreak H_{n}\,\allowbreak \mathcal{A}_{2\,\allowbreak n}^{(\mathrm{i})} = -\bigl(2^{-5 -\allowbreak 2\,\allowbreak \mathrm{i}}\bigr)\,\allowbreak \bigl(-2\,\allowbreak \mathcal{A}_{2\,\allowbreak k}^{(-3 -\allowbreak \mathrm{i})} +\allowbreak 2\,\allowbreak \mathcal{H}_{2\,\allowbreak k}^{(-3 -\allowbreak \mathrm{i})} +\allowbreak \mathcal{H}_{2\,\allowbreak k}^{(-4 -\allowbreak \mathrm{i},1)} +\allowbreak \mathcal{H}_{2\,\allowbreak k}^{(-4 -\allowbreak 2\,\allowbreak \mathrm{i},1 +\allowbreak \mathrm{i})} +\allowbreak \mathcal{H}_{2\,\allowbreak k}^{(-4 -\allowbreak \mathrm{i},1)}\!(-1,\allowbreak -1) +\allowbreak \mathcal{H}_{2\,\allowbreak k}^{(-4 -\allowbreak \mathrm{i},1)}\!(-1,\allowbreak 1) +\allowbreak \mathcal{H}_{2\,\allowbreak k}^{(-4 -\allowbreak \mathrm{i},1)}\!(1,\allowbreak -1)
\end{dmath*}
\vspace{-1.1em}
\begin{dmath*}
{}\quad +\allowbreak \mathcal{H}_{2\,\allowbreak k}^{(-4 -\allowbreak 2\,\allowbreak \mathrm{i},1 +\allowbreak \mathrm{i})}\!(-1,\allowbreak -1) +\allowbreak \mathcal{H}_{2\,\allowbreak k}^{(-4 -\allowbreak 2\,\allowbreak \mathrm{i},1 +\allowbreak \mathrm{i})}\!(-1,\allowbreak 1) +\allowbreak \mathcal{H}_{2\,\allowbreak k}^{(-4 -\allowbreak 2\,\allowbreak \mathrm{i},1 +\allowbreak \mathrm{i})}\!(1,\allowbreak -1) +\allowbreak 2\,\allowbreak \mathcal{H}_{2\,\allowbreak k}^{(-3 -\allowbreak 2\,\allowbreak \mathrm{i},\mathrm{i})}\!(-1,\allowbreak -1) +\allowbreak 2\,\allowbreak \mathcal{H}_{2\,\allowbreak k}^{(-3 -\allowbreak 2\,\allowbreak \mathrm{i},\mathrm{i})}\!(1,\allowbreak -1)
\end{dmath*}
\vspace{-1.1em}
\begin{dmath*}
{}\quad +\allowbreak \mathcal{H}_{2\,\allowbreak k}^{(-4 -\allowbreak 2\,\allowbreak \mathrm{i},\mathrm{i},1)}\!(-1,\allowbreak -1,\allowbreak -1) +\allowbreak \mathcal{H}_{2\,\allowbreak k}^{(-4 -\allowbreak 2\,\allowbreak \mathrm{i},\mathrm{i},1)}\!(-1,\allowbreak -1,\allowbreak 1) +\allowbreak \mathcal{H}_{2\,\allowbreak k}^{(-4 -\allowbreak 2\,\allowbreak \mathrm{i},\mathrm{i},1)}\!(1,\allowbreak -1,\allowbreak -1) +\allowbreak \mathcal{H}_{2\,\allowbreak k}^{(-4 -\allowbreak 2\,\allowbreak \mathrm{i},\mathrm{i},1)}\!(1,\allowbreak -1,\allowbreak 1) +\allowbreak \mathcal{H}_{2\,\allowbreak k}^{(-4 -\allowbreak 2\,\allowbreak \mathrm{i},1,\mathrm{i})}\!(-1,\allowbreak -1,\allowbreak -1)
\end{dmath*}
\vspace{-1.1em}
\begin{dmath*}
{}\quad +\allowbreak \mathcal{H}_{2\,\allowbreak k}^{(-4 -\allowbreak 2\,\allowbreak \mathrm{i},1,\mathrm{i})}\!(-1,\allowbreak 1,\allowbreak -1) +\allowbreak \mathcal{H}_{2\,\allowbreak k}^{(-4 -\allowbreak 2\,\allowbreak \mathrm{i},1,\mathrm{i})}\!(1,\allowbreak -1,\allowbreak -1) +\allowbreak \mathcal{H}_{2\,\allowbreak k}^{(-4 -\allowbreak 2\,\allowbreak \mathrm{i},1,\mathrm{i})}\!(1,\allowbreak 1,\allowbreak -1)\bigr)
\end{dmath*}
\subsection*{Identity 43}
\begin{dmath*}
\sum_{n=1}^{k} \frac{\bigl(\frac{1}{2}\bigr)^{n}\,\allowbreak h_{2\,\allowbreak n}^{[2]}\!(1;1)}{n^{4}} = 8\,\allowbreak \bigl(-\mathcal{A}_{2\,\allowbreak k}^{(5)}\!(\frac{1}{\sqrt{2}}) +\allowbreak \mathcal{H}_{2\,\allowbreak k}^{(5)}\!(\frac{1}{\sqrt{2}}) +\allowbreak 3\,\allowbreak \mathcal{H}_{2\,\allowbreak k}^{(4,1)}\!(-\bigl(\frac{1}{\sqrt{2}}\bigr),\allowbreak 1) +\allowbreak 3\,\allowbreak \mathcal{H}_{2\,\allowbreak k}^{(4,1)}\!(\frac{1}{\sqrt{2}},\allowbreak 1) +\allowbreak 3\,\allowbreak \mathcal{H}_{2\,\allowbreak k}^{(4,0,1)}\!(-\bigl(\frac{1}{\sqrt{2}}\bigr),\allowbreak 1,\allowbreak 1) +\allowbreak 3\,\allowbreak \mathcal{H}_{2\,\allowbreak k}^{(4,0,1)}\!(\frac{1}{\sqrt{2}},\allowbreak 1,\allowbreak 1) +\allowbreak \mathcal{H}_{2\,\allowbreak k}^{(4,0,0,1)}\!(-\bigl(\frac{1}{\sqrt{2}}\bigr),\allowbreak 1,\allowbreak 1,\allowbreak 1)
\end{dmath*}
\vspace{-1.1em}
\begin{dmath*}
{}\quad +\allowbreak \mathcal{H}_{2\,\allowbreak k}^{(4,0,0,1)}\!(\frac{1}{\sqrt{2}},\allowbreak 1,\allowbreak 1,\allowbreak 1)\bigr)
\end{dmath*}
\subsection*{Identity 44}
\begin{dmath*}
\sum_{n=1}^{k} \frac{\bigl(\frac{1}{2}\bigr)^{n}\,\allowbreak \mathcal{H}_{2\,\allowbreak n}^{(\mathrm{i},2\,\allowbreak \mathrm{i})}}{n^{4}} = 8\,\allowbreak \bigl(\mathcal{H}_{2\,\allowbreak k}^{(4 +\allowbreak \mathrm{i},2\,\allowbreak \mathrm{i})}\!(-\bigl(\frac{1}{\sqrt{2}}\bigr),\allowbreak 1) +\allowbreak \mathcal{H}_{2\,\allowbreak k}^{(4 +\allowbreak \mathrm{i},2\,\allowbreak \mathrm{i})}\!(\frac{1}{\sqrt{2}},\allowbreak 1) +\allowbreak \mathcal{H}_{2\,\allowbreak k}^{(4,\mathrm{i},2\,\allowbreak \mathrm{i})}\!(-\bigl(\frac{1}{\sqrt{2}}\bigr),\allowbreak 1,\allowbreak 1) +\allowbreak \mathcal{H}_{2\,\allowbreak k}^{(4,\mathrm{i},2\,\allowbreak \mathrm{i})}\!(\frac{1}{\sqrt{2}},\allowbreak 1,\allowbreak 1)\bigr)
\end{dmath*}
\subsection*{Identity 45}
\begin{dmath*}
\sum_{n=1}^{k} H_{n}\,\allowbreak \mathcal{H}_{3\,\allowbreak n}^{(2)}\!(\frac{1}{3}) = \frac{1}{3}\,\allowbreak \bigl(-\mathcal{H}_{3\,\allowbreak k}^{(2)}\!(\frac{1}{3}) -\allowbreak \mathcal{H}_{3\,\allowbreak k}^{(2)}\!(-\bigl(\frac{1}{3}\bigr)\,\allowbreak \bigl(-1\bigr)^{\frac{1}{3}}) -\allowbreak \mathcal{H}_{3\,\allowbreak k}^{(2)}\!(\frac{1}{3}\,\allowbreak \bigl(-1\bigr)^{\frac{2}{3}}) +\allowbreak 3\,\allowbreak \mathcal{H}_{3\,\allowbreak k}^{(3)}\!(\frac{1}{3}) +\allowbreak 3\,\allowbreak k\,\allowbreak \mathcal{H}_{3\,\allowbreak k}^{(3)}\!(\frac{1}{3}) +\allowbreak 3\,\allowbreak \mathcal{H}_{3\,\allowbreak k}^{(3)}\!(-\bigl(\frac{1}{3}\bigr)\,\allowbreak \bigl(-1\bigr)^{\frac{1}{3}}) +\allowbreak 3\,\allowbreak k\,\allowbreak \mathcal{H}_{3\,\allowbreak k}^{(3)}\!(-\bigl(\frac{1}{3}\bigr)\,\allowbreak \bigl(-1\bigr)^{\frac{1}{3}})
\end{dmath*}
\vspace{-1.1em}
\begin{dmath*}
{}\quad +\allowbreak \frac{-\bigl(\bigl(-\bigl(\bigl(-1\bigr)^{\frac{1}{3}}\bigr)\bigr)^{1 +\allowbreak 3\,\allowbreak k}\bigr)\,\allowbreak \mathcal{H}_{3\,\allowbreak k}^{(3)}\!(\frac{1}{3}) -\allowbreak \bigl(-1\bigr)^{\frac{1}{3}}\,\allowbreak \mathcal{H}_{3\,\allowbreak k}^{(3)}\!(-\bigl(\frac{1}{3}\bigr)\,\allowbreak \bigl(-1\bigr)^{\frac{1}{3}})}{1 +\allowbreak \bigl(-1\bigr)^{\frac{1}{3}}} +\allowbreak \frac{\bigl(-1\bigr)^{\frac{2}{3}}\,\allowbreak \mathcal{H}_{3\,\allowbreak k}^{(3)}\!(\frac{1}{3}) -\allowbreak \bigl(-1\bigr)^{\frac{2}{3}\,\allowbreak \bigl(1 +\allowbreak 3\,\allowbreak k\bigr)}\,\allowbreak \mathcal{H}_{3\,\allowbreak k}^{(3)}\!(-\bigl(\frac{1}{3}\bigr)\,\allowbreak \bigl(-1\bigr)^{\frac{1}{3}})}{1 -\allowbreak \bigl(-1\bigr)^{\frac{2}{3}}}
\end{dmath*}
\vspace{-1.1em}
\begin{dmath*}
{}\quad +\allowbreak 3\,\allowbreak \mathcal{H}_{3\,\allowbreak k}^{(3)}\!(\frac{1}{3}\,\allowbreak \bigl(-1\bigr)^{\frac{2}{3}}) +\allowbreak 3\,\allowbreak k\,\allowbreak \mathcal{H}_{3\,\allowbreak k}^{(3)}\!(\frac{1}{3}\,\allowbreak \bigl(-1\bigr)^{\frac{2}{3}}) +\allowbreak \frac{-\bigl(\bigl(-\bigl(\bigl(-1\bigr)^{\frac{1}{3}}\bigr)\bigr)^{1 +\allowbreak 3\,\allowbreak k}\bigr)\,\allowbreak \mathcal{H}_{3\,\allowbreak k}^{(3)}\!(-\bigl(\frac{1}{3}\bigr)\,\allowbreak \bigl(-1\bigr)^{\frac{1}{3}}) -\allowbreak \bigl(-1\bigr)^{\frac{1}{3}}\,\allowbreak \mathcal{H}_{3\,\allowbreak k}^{(3)}\!(\frac{1}{3}\,\allowbreak \bigl(-1\bigr)^{\frac{2}{3}})}{1 +\allowbreak \bigl(-1\bigr)^{\frac{1}{3}}}
\end{dmath*}
\vspace{-1.1em}
\begin{dmath*}
{}\quad +\allowbreak \frac{-\bigl(\bigl(-1\bigr)^{\frac{2}{3}\,\allowbreak \bigl(1 +\allowbreak 3\,\allowbreak k\bigr)}\bigr)\,\allowbreak \mathcal{H}_{3\,\allowbreak k}^{(3)}\!(\frac{1}{3}) +\allowbreak \bigl(-1\bigr)^{\frac{2}{3}}\,\allowbreak \mathcal{H}_{3\,\allowbreak k}^{(3)}\!(\frac{1}{3}\,\allowbreak \bigl(-1\bigr)^{\frac{2}{3}})}{1 -\allowbreak \bigl(-1\bigr)^{\frac{2}{3}}} +\allowbreak \frac{\bigl(-1\bigr)^{\frac{2}{3}}\,\allowbreak \mathcal{H}_{3\,\allowbreak k}^{(3)}\!(-\bigl(\frac{1}{3}\bigr)\,\allowbreak \bigl(-1\bigr)^{\frac{1}{3}}) -\allowbreak \bigl(-1\bigr)^{\frac{2}{3}\,\allowbreak \bigl(1 +\allowbreak 3\,\allowbreak k\bigr)}\,\allowbreak \mathcal{H}_{3\,\allowbreak k}^{(3)}\!(\frac{1}{3}\,\allowbreak \bigl(-1\bigr)^{\frac{2}{3}})}{1 -\allowbreak \bigl(-1\bigr)^{\frac{2}{3}}}
\end{dmath*}
\vspace{-1.1em}
\begin{dmath*}
{}\quad +\allowbreak \frac{-\bigl(\bigl(-1\bigr)^{\frac{1}{3}}\bigr)\,\allowbreak \mathcal{H}_{3\,\allowbreak k}^{(3)}\!(\frac{1}{3}) -\allowbreak \bigl(-\bigl(\bigl(-1\bigr)^{\frac{1}{3}}\bigr)\bigr)^{1 +\allowbreak 3\,\allowbreak k}\,\allowbreak \mathcal{H}_{3\,\allowbreak k}^{(3)}\!(\frac{1}{3}\,\allowbreak \bigl(-1\bigr)^{\frac{2}{3}})}{1 +\allowbreak \bigl(-1\bigr)^{\frac{1}{3}}} -\allowbreak \bigl(-\mathcal{H}_{3\,\allowbreak k}^{(1)}\!(\frac{1}{3}) +\allowbreak 3\,\allowbreak k\,\allowbreak \mathcal{H}_{3\,\allowbreak k}^{(2)}\!(\frac{1}{3})\bigr)\,\allowbreak \log\!(3) -\allowbreak \frac{\bigl(-\bigl(\bigl(-\bigl(\bigl(-1\bigr)^{\frac{1}{3}}\bigr)\bigr)^{1 +\allowbreak 3\,\allowbreak k}\bigr)\,\allowbreak \mathcal{H}_{3\,\allowbreak k}^{(2)}\!(\frac{1}{3}) -\allowbreak \bigl(-1\bigr)^{\frac{1}{3}}\,\allowbreak \mathcal{H}_{3\,\allowbreak k}^{(2)}\!(-\bigl(\frac{1}{3}\bigr)\,\allowbreak \bigl(-1\bigr)^{\frac{1}{3}})\bigr)\,\allowbreak \log\!(3)}{1 +\allowbreak \bigl(-1\bigr)^{\frac{1}{3}}}
\end{dmath*}
\vspace{-1.1em}
\begin{dmath*}
{}\quad -\allowbreak \mathcal{H}_{3\,\allowbreak k}^{(2)}\!(\frac{1}{3}\,\allowbreak \bigl(-1\bigr)^{\frac{2}{3}})\,\allowbreak \log\!(3) -\allowbreak \frac{\bigl(-\bigl(\bigl(-1\bigr)^{\frac{2}{3}\,\allowbreak \bigl(1 +\allowbreak 3\,\allowbreak k\bigr)}\bigr)\,\allowbreak \mathcal{H}_{3\,\allowbreak k}^{(2)}\!(\frac{1}{3}) +\allowbreak \bigl(-1\bigr)^{\frac{2}{3}}\,\allowbreak \mathcal{H}_{3\,\allowbreak k}^{(2)}\!(\frac{1}{3}\,\allowbreak \bigl(-1\bigr)^{\frac{2}{3}})\bigr)\,\allowbreak \log\!(3)}{1 -\allowbreak \bigl(-1\bigr)^{\frac{2}{3}}} +\allowbreak \bigl(-\mathcal{H}_{3\,\allowbreak k}^{(1)}\!(\frac{1}{3}) +\allowbreak 3\,\allowbreak k\,\allowbreak \mathcal{H}_{3\,\allowbreak k}^{(2)}\!(\frac{1}{3})\bigr)\,\allowbreak \log\!(1
\end{dmath*}
\vspace{-1.1em}
\begin{dmath*}
{}\quad +\allowbreak \bigl(-1\bigr)^{\frac{1}{3}}) +\allowbreak \frac{\bigl(-\bigl(\bigl(-\bigl(\bigl(-1\bigr)^{\frac{1}{3}}\bigr)\bigr)^{1 +\allowbreak 3\,\allowbreak k}\bigr)\,\allowbreak \mathcal{H}_{3\,\allowbreak k}^{(2)}\!(\frac{1}{3}) -\allowbreak \bigl(-1\bigr)^{\frac{1}{3}}\,\allowbreak \mathcal{H}_{3\,\allowbreak k}^{(2)}\!(-\bigl(\frac{1}{3}\bigr)\,\allowbreak \bigl(-1\bigr)^{\frac{1}{3}})\bigr)\,\allowbreak \log\!(1 +\allowbreak \bigl(-1\bigr)^{\frac{1}{3}})}{1 +\allowbreak \bigl(-1\bigr)^{\frac{1}{3}}} +\allowbreak \mathcal{H}_{3\,\allowbreak k}^{(2)}\!(\frac{1}{3}\,\allowbreak \bigl(-1\bigr)^{\frac{2}{3}})\,\allowbreak \log\!(1 +\allowbreak \bigl(-1\bigr)^{\frac{1}{3}})
\end{dmath*}
\vspace{-1.1em}
\begin{dmath*}
{}\quad +\allowbreak \frac{\bigl(-\bigl(\bigl(-1\bigr)^{\frac{2}{3}\,\allowbreak \bigl(1 +\allowbreak 3\,\allowbreak k\bigr)}\bigr)\,\allowbreak \mathcal{H}_{3\,\allowbreak k}^{(2)}\!(\frac{1}{3}) +\allowbreak \bigl(-1\bigr)^{\frac{2}{3}}\,\allowbreak \mathcal{H}_{3\,\allowbreak k}^{(2)}\!(\frac{1}{3}\,\allowbreak \bigl(-1\bigr)^{\frac{2}{3}})\bigr)\,\allowbreak \log\!(1 +\allowbreak \bigl(-1\bigr)^{\frac{1}{3}})}{1 -\allowbreak \bigl(-1\bigr)^{\frac{2}{3}}} +\allowbreak \bigl(-\mathcal{H}_{3\,\allowbreak k}^{(1)}\!(\frac{1}{3}) +\allowbreak 3\,\allowbreak k\,\allowbreak \mathcal{H}_{3\,\allowbreak k}^{(2)}\!(\frac{1}{3})\bigr)\,\allowbreak \log\!(1 -\allowbreak \bigl(-1\bigr)^{\frac{2}{3}})
\end{dmath*}
\vspace{-1.1em}
\begin{dmath*}
{}\quad +\allowbreak \frac{\bigl(-\bigl(\bigl(-\bigl(\bigl(-1\bigr)^{\frac{1}{3}}\bigr)\bigr)^{1 +\allowbreak 3\,\allowbreak k}\bigr)\,\allowbreak \mathcal{H}_{3\,\allowbreak k}^{(2)}\!(\frac{1}{3}) -\allowbreak \bigl(-1\bigr)^{\frac{1}{3}}\,\allowbreak \mathcal{H}_{3\,\allowbreak k}^{(2)}\!(-\bigl(\frac{1}{3}\bigr)\,\allowbreak \bigl(-1\bigr)^{\frac{1}{3}})\bigr)\,\allowbreak \log\!(1 -\allowbreak \bigl(-1\bigr)^{\frac{2}{3}})}{1 +\allowbreak \bigl(-1\bigr)^{\frac{1}{3}}} +\allowbreak \mathcal{H}_{3\,\allowbreak k}^{(2)}\!(\frac{1}{3}\,\allowbreak \bigl(-1\bigr)^{\frac{2}{3}})\,\allowbreak \log\!(1 -\allowbreak \bigl(-1\bigr)^{\frac{2}{3}}) +\allowbreak \frac{\bigl(-\bigl(\bigl(-1\bigr)^{\frac{2}{3}\,\allowbreak \bigl(1 +\allowbreak 3\,\allowbreak k\bigr)}\bigr)\,\allowbreak \mathcal{H}_{3\,\allowbreak k}^{(2)}\!(\frac{1}{3}) +\allowbreak \bigl(-1\bigr)^{\frac{2}{3}}\,\allowbreak \mathcal{H}_{3\,\allowbreak k}^{(2)}\!(\frac{1}{3}\,\allowbreak \bigl(-1\bigr)^{\frac{2}{3}})\bigr)\,\allowbreak \log\!(1 -\allowbreak \bigl(-1\bigr)^{\frac{2}{3}})}{1 -\allowbreak \bigl(-1\bigr)^{\frac{2}{3}}}
\end{dmath*}
\vspace{-1.1em}
\begin{dmath*}
{}\quad +\allowbreak \mathcal{H}_{3\,\allowbreak k}^{(2)}\!(\frac{1}{3})\,\allowbreak \bigl(\log\!(\frac{1}{3}\,\allowbreak \bigl(1 +\allowbreak \bigl(-1\bigr)^{\frac{1}{3}}\bigr)) +\allowbreak \log\!(1 -\allowbreak \bigl(-1\bigr)^{\frac{2}{3}})\bigr) +\allowbreak \mathcal{H}_{3\,\allowbreak k}^{(2)}\!(-\bigl(\frac{1}{3}\bigr)\,\allowbreak \bigl(-1\bigr)^{\frac{1}{3}})\,\allowbreak \bigl(\log\!(\frac{1}{3}\,\allowbreak \bigl(1 +\allowbreak \bigl(-1\bigr)^{\frac{1}{3}}\bigr)) +\allowbreak \log\!(1 -\allowbreak \bigl(-1\bigr)^{\frac{2}{3}})\bigr) +\allowbreak 3\,\allowbreak \mathcal{H}_{3\,\allowbreak k}^{(1,2)}\!(1,\allowbreak \frac{1}{3}) +\allowbreak 3\,\allowbreak \mathcal{H}_{3\,\allowbreak k}^{(1,2)}\!(-\bigl(\bigl(-1\bigr)^{\frac{1}{3}}\bigr),\allowbreak \frac{1}{3})
\end{dmath*}
\vspace{-1.1em}
\begin{dmath*}
{}\quad +\allowbreak 3\,\allowbreak \mathcal{H}_{3\,\allowbreak k}^{(1,2)}\!(\bigl(-1\bigr)^{\frac{2}{3}},\allowbreak \frac{1}{3}) +\allowbreak \mathcal{H}_{3\,\allowbreak k}^{(2,1)}\!(\frac{1}{3},\allowbreak 1) +\allowbreak \mathcal{H}_{3\,\allowbreak k}^{(2,1)}\!(\frac{1}{3},\allowbreak -\bigl(\bigl(-1\bigr)^{\frac{1}{3}}\bigr)) +\allowbreak \mathcal{H}_{3\,\allowbreak k}^{(2,1)}\!(\frac{1}{3},\allowbreak \bigl(-1\bigr)^{\frac{2}{3}}) +\allowbreak \mathcal{H}_{3\,\allowbreak k}^{(2,1)}\!(-\bigl(\frac{1}{3}\bigr)\,\allowbreak \bigl(-1\bigr)^{\frac{1}{3}},\allowbreak 1) +\allowbreak \mathcal{H}_{3\,\allowbreak k}^{(2,1)}\!(-\bigl(\frac{1}{3}\bigr)\,\allowbreak \bigl(-1\bigr)^{\frac{1}{3}},\allowbreak -\bigl(\bigl(-1\bigr)^{\frac{1}{3}}\bigr))
\end{dmath*}
\vspace{-1.1em}
\begin{dmath*}
{}\quad +\allowbreak \mathcal{H}_{3\,\allowbreak k}^{(2,1)}\!(-\bigl(\frac{1}{3}\bigr)\,\allowbreak \bigl(-1\bigr)^{\frac{1}{3}},\allowbreak \bigl(-1\bigr)^{\frac{2}{3}}) +\allowbreak \mathcal{H}_{3\,\allowbreak k}^{(2,1)}\!(\frac{1}{3}\,\allowbreak \bigl(-1\bigr)^{\frac{2}{3}},\allowbreak 1) +\allowbreak \mathcal{H}_{3\,\allowbreak k}^{(2,1)}\!(\frac{1}{3}\,\allowbreak \bigl(-1\bigr)^{\frac{2}{3}},\allowbreak -\bigl(\bigl(-1\bigr)^{\frac{1}{3}}\bigr)) +\allowbreak \mathcal{H}_{3\,\allowbreak k}^{(2,1)}\!(\frac{1}{3}\,\allowbreak \bigl(-1\bigr)^{\frac{2}{3}},\allowbreak \bigl(-1\bigr)^{\frac{2}{3}}) +\allowbreak \mathcal{H}_{3\,\allowbreak k}^{(0,1,2)}\!(1,\allowbreak 1,\allowbreak \frac{1}{3})
\end{dmath*}
\vspace{-1.1em}
\begin{dmath*}
{}\quad +\allowbreak \mathcal{H}_{3\,\allowbreak k}^{(0,1,2)}\!(1,\allowbreak -\bigl(\bigl(-1\bigr)^{\frac{1}{3}}\bigr),\allowbreak \frac{1}{3}) +\allowbreak \mathcal{H}_{3\,\allowbreak k}^{(0,1,2)}\!(1,\allowbreak \bigl(-1\bigr)^{\frac{2}{3}},\allowbreak \frac{1}{3}) +\allowbreak \mathcal{H}_{3\,\allowbreak k}^{(0,1,2)}\!(-\bigl(\bigl(-1\bigr)^{\frac{1}{3}}\bigr),\allowbreak 1,\allowbreak \frac{1}{3}) +\allowbreak \mathcal{H}_{3\,\allowbreak k}^{(0,1,2)}\!(-\bigl(\bigl(-1\bigr)^{\frac{1}{3}}\bigr),\allowbreak -\bigl(\bigl(-1\bigr)^{\frac{1}{3}}\bigr),\allowbreak \frac{1}{3}) +\allowbreak \mathcal{H}_{3\,\allowbreak k}^{(0,1,2)}\!(-\bigl(\bigl(-1\bigr)^{\frac{1}{3}}\bigr),\allowbreak \bigl(-1\bigr)^{\frac{2}{3}},\allowbreak \frac{1}{3})
\end{dmath*}
\vspace{-1.1em}
\begin{dmath*}
{}\quad +\allowbreak \mathcal{H}_{3\,\allowbreak k}^{(0,1,2)}\!(\bigl(-1\bigr)^{\frac{2}{3}},\allowbreak 1,\allowbreak \frac{1}{3}) +\allowbreak \mathcal{H}_{3\,\allowbreak k}^{(0,1,2)}\!(\bigl(-1\bigr)^{\frac{2}{3}},\allowbreak -\bigl(\bigl(-1\bigr)^{\frac{1}{3}}\bigr),\allowbreak \frac{1}{3}) +\allowbreak \mathcal{H}_{3\,\allowbreak k}^{(0,1,2)}\!(\bigl(-1\bigr)^{\frac{2}{3}},\allowbreak \bigl(-1\bigr)^{\frac{2}{3}},\allowbreak \frac{1}{3}) +\allowbreak \mathcal{H}_{3\,\allowbreak k}^{(0,2,1)}\!(1,\allowbreak \frac{1}{3},\allowbreak 1) +\allowbreak \mathcal{H}_{3\,\allowbreak k}^{(0,2,1)}\!(1,\allowbreak \frac{1}{3},\allowbreak -\bigl(\bigl(-1\bigr)^{\frac{1}{3}}\bigr))
\end{dmath*}
\vspace{-1.1em}
\begin{dmath*}
{}\quad +\allowbreak \mathcal{H}_{3\,\allowbreak k}^{(0,2,1)}\!(1,\allowbreak \frac{1}{3},\allowbreak \bigl(-1\bigr)^{\frac{2}{3}}) +\allowbreak \mathcal{H}_{3\,\allowbreak k}^{(0,2,1)}\!(-\bigl(\bigl(-1\bigr)^{\frac{1}{3}}\bigr),\allowbreak \frac{1}{3},\allowbreak 1) +\allowbreak \mathcal{H}_{3\,\allowbreak k}^{(0,2,1)}\!(-\bigl(\bigl(-1\bigr)^{\frac{1}{3}}\bigr),\allowbreak \frac{1}{3},\allowbreak -\bigl(\bigl(-1\bigr)^{\frac{1}{3}}\bigr)) +\allowbreak \mathcal{H}_{3\,\allowbreak k}^{(0,2,1)}\!(-\bigl(\bigl(-1\bigr)^{\frac{1}{3}}\bigr),\allowbreak \frac{1}{3},\allowbreak \bigl(-1\bigr)^{\frac{2}{3}}) +\allowbreak \mathcal{H}_{3\,\allowbreak k}^{(0,2,1)}\!(\bigl(-1\bigr)^{\frac{2}{3}},\allowbreak \frac{1}{3},\allowbreak 1)
\end{dmath*}
\vspace{-1.1em}
\begin{dmath*}
{}\quad +\allowbreak \mathcal{H}_{3\,\allowbreak k}^{(0,2,1)}\!(\bigl(-1\bigr)^{\frac{2}{3}},\allowbreak \frac{1}{3},\allowbreak -\bigl(\bigl(-1\bigr)^{\frac{1}{3}}\bigr)) +\allowbreak \mathcal{H}_{3\,\allowbreak k}^{(0,2,1)}\!(\bigl(-1\bigr)^{\frac{2}{3}},\allowbreak \frac{1}{3},\allowbreak \bigl(-1\bigr)^{\frac{2}{3}})\bigr)
\end{dmath*}
\subsection*{Identity 46}
\begin{dmath*}
\sum_{n_{1}=1}^{k}\sum_{n_{2}=1}^{n_{1}-1}\frac{1}{(2\,n_{1}-1)^{2}\,(2\,n_{2}-1)^{3}}
= \frac{1}{32}\,\bigl(16\,\mathcal{A}_{-1+2\,k}^{(5)}+16\,\mathcal{H}_{-1+2\,k}^{(5)}+\zeta\!(5,\frac{1}{2}+k)+8\,\mathcal{H}_{-1+2\,k}^{(2,3)}+8\,\mathcal{H}_{-1+2\,k}^{(2,3)}\!(-1,-1)-8\,\mathcal{H}_{-1+2\,k}^{(2,3)}\!(-1,1)-8\,\mathcal{H}_{-1+2\,k}^{(2,3)}\!(1,-1)-31\,\zeta\!(5)\bigr)
\end{dmath*}
\section*{7 Nested sums}

\subsection*{Identity 47}
\begin{dmath*}
\sum_{n_{1}=1}^{k}\sum_{n_{2}=1}^{n_{1}} H_{2\,\allowbreak n_{2}} = \frac{1}{8}\,\allowbreak \bigl(1 -\allowbreak \bigl(-1\bigr)^{2\,\allowbreak k} -\allowbreak 3\,\allowbreak k\,\allowbreak \bigl(3 +\allowbreak 2\,\allowbreak k\bigr) -\allowbreak \bigl(1 +\allowbreak \bigl(-1\bigr)^{2\,\allowbreak k} +\allowbreak 4\,\allowbreak k\bigr)\,\allowbreak \mathcal{A}_{2\,\allowbreak k} +\allowbreak \bigl(3 +\allowbreak 3\,\allowbreak \bigl(-1\bigr)^{2\,\allowbreak k} +\allowbreak 10\,\allowbreak k +\allowbreak 4\,\allowbreak k^{2}\bigr)\,\allowbreak H_{2\,\allowbreak k} +\allowbreak 2\,\allowbreak \mathcal{H}_{2\,\allowbreak k}^{(0,0,1)}\!(-1,\allowbreak -1,\allowbreak 1) +\allowbreak 2\,\allowbreak \mathcal{H}_{2\,\allowbreak k}^{(0,0,1)}\!(-1,\allowbreak 1,\allowbreak 1)
\end{dmath*}
\vspace{-1.1em}
\begin{dmath*}
{}\quad +\allowbreak 2\,\allowbreak \mathcal{H}_{2\,\allowbreak k}^{(0,0,1)}\!(1,\allowbreak -1,\allowbreak 1)\bigr)
\end{dmath*}
\subsection*{Identity 48}
\begin{dmath*}
\sum_{n_{1}=1}^{k}\sum_{n_{2}=1}^{n_{1}}\sum_{n_{3}=1}^{n_{2}} \frac{n_{1}^{2}}{n_{2}^{3}}\,\allowbreak n_{3}^{4}\,\allowbreak h_{n_{3}}^{[1]}\!(1;1) = \frac{k^{2}}{4} +\allowbreak \frac{k^{3}}{2} +\allowbreak \frac{k^{4}}{4} +\allowbreak 2\,\allowbreak \mathcal{H}_{k}^{(-3,1)} +\allowbreak \mathcal{H}_{k}^{(1,-3)} +\allowbreak \mathcal{H}_{k}^{(-3,0,1)} +\allowbreak 2\,\allowbreak \mathcal{H}_{k}^{(-2,-1,1)} +\allowbreak \mathcal{H}_{k}^{(-2,3,-3)} +\allowbreak 2\,\allowbreak \mathcal{H}_{k}^{(1,-4,1)} +\allowbreak \mathcal{H}_{k}^{(-2,-1,0,1)} +\allowbreak 2\,\allowbreak \mathcal{H}_{k}^{(-2,3,-4,1)} +\allowbreak \mathcal{H}_{k}^{(1,-4,0,1)} +\allowbreak \mathcal{H}_{k}^{(-2,3,-4,0,1)}
\end{dmath*}
\subsection*{Identity 49}
\begin{dmath*}
\sum_{n_{1}=1}^{k}\sum_{n_{2}=1}^{n_{1}} H_{n_{1}}\,\allowbreak \mathcal{A}_{n_{2}}^{(\mathrm{i})}\!(2\,\allowbreak \mathrm{i}) = \bigl(-2 -\allowbreak k +\allowbreak \bigl(1 +\allowbreak k\bigr)\,\allowbreak H_{k}\bigr)\,\allowbreak \mathcal{H}_{k}^{(-1 +\allowbreak \mathrm{i})}\!(-2\,\allowbreak \mathrm{i}) -\allowbreak \bigl(1 +\allowbreak k\bigr)^{2}\,\allowbreak \bigl(-1 +\allowbreak H_{k}\bigr)\,\allowbreak \mathcal{H}_{k}^{(\mathrm{i})}\!(-2\,\allowbreak \mathrm{i}) -\allowbreak \mathcal{H}_{k}^{(-1,\mathrm{i})}\!(1,\allowbreak -2\,\allowbreak \mathrm{i}) +\allowbreak \mathcal{H}_{k}^{(-1,1 +\allowbreak \mathrm{i})}\!(1,\allowbreak -2\,\allowbreak \mathrm{i}) +\allowbreak \mathcal{H}_{k}^{(-1 +\allowbreak \mathrm{i},1)}\!(-2\,\allowbreak \mathrm{i},\allowbreak 1)
\end{dmath*}
\vspace{-1.1em}
\begin{dmath*}
{}\quad +\allowbreak \mathcal{H}_{k}^{(-1,\mathrm{i},1)}\!(1,\allowbreak -2\,\allowbreak \mathrm{i},\allowbreak 1) +\allowbreak \mathcal{H}_{k}^{(-1,1,\mathrm{i})}\!(1,\allowbreak 1,\allowbreak -2\,\allowbreak \mathrm{i})
\end{dmath*}
\subsection*{Identity 50}
\begin{dmath*}
\sum_{n_{1}=1}^{k}\sum_{n_{2}=1}^{n_{1}}\sum_{n_{3}=1}^{n_{2}} \frac{H_{n_{1}}\,\allowbreak H_{n_{2}}\,\allowbreak H_{n_{3}}}{n_{1}^{2}\,\allowbreak n_{2}^{3}\,\allowbreak n_{3}^{4}} = \mathcal{H}_{k}^{(12)} +\allowbreak \mathcal{H}_{k}^{(2,10)} +\allowbreak \mathcal{H}_{k}^{(3,9)} +\allowbreak \mathcal{H}_{k}^{(5,7)} +\allowbreak 2\,\allowbreak \mathcal{H}_{k}^{(6,6)} +\allowbreak \mathcal{H}_{k}^{(7,5)} +\allowbreak \mathcal{H}_{k}^{(9,3)} +\allowbreak 3\,\allowbreak \mathcal{H}_{k}^{(10,2)} +\allowbreak 3\,\allowbreak \mathcal{H}_{k}^{(11,1)} +\allowbreak \mathcal{H}_{k}^{(2,1,9)} +\allowbreak \mathcal{H}_{k}^{(2,3,7)} +\allowbreak 2\,\allowbreak \mathcal{H}_{k}^{(2,4,6)} +\allowbreak \mathcal{H}_{k}^{(2,5,5)} +\allowbreak \mathcal{H}_{k}^{(2,7,3)} +\allowbreak 3\,\allowbreak \mathcal{H}_{k}^{(2,8,2)} +\allowbreak 3\,\allowbreak \mathcal{H}_{k}^{(2,9,1)}
\end{dmath*}
\vspace{-1.1em}
\begin{dmath*}
{}\quad +\allowbreak \mathcal{H}_{k}^{(3,3,6)} +\allowbreak \mathcal{H}_{k}^{(3,4,5)} +\allowbreak \mathcal{H}_{k}^{(3,7,2)} +\allowbreak 2\,\allowbreak \mathcal{H}_{k}^{(3,8,1)} +\allowbreak 2\,\allowbreak \mathcal{H}_{k}^{(5,1,6)} +\allowbreak \mathcal{H}_{k}^{(5,2,5)} +\allowbreak \mathcal{H}_{k}^{(5,4,3)} +\allowbreak 3\,\allowbreak \mathcal{H}_{k}^{(5,5,2)} +\allowbreak 3\,\allowbreak \mathcal{H}_{k}^{(5,6,1)} +\allowbreak 2\,\allowbreak \mathcal{H}_{k}^{(6,1,5)} +\allowbreak 2\,\allowbreak \mathcal{H}_{k}^{(6,4,2)} +\allowbreak 4\,\allowbreak \mathcal{H}_{k}^{(6,5,1)} +\allowbreak \mathcal{H}_{k}^{(7,4,1)} +\allowbreak 3\,\allowbreak \mathcal{H}_{k}^{(9,1,2)}
\end{dmath*}
\vspace{-1.1em}
\begin{dmath*}
{}\quad +\allowbreak 3\,\allowbreak \mathcal{H}_{k}^{(9,2,1)} +\allowbreak 6\,\allowbreak \mathcal{H}_{k}^{(10,1,1)} +\allowbreak \mathcal{H}_{k}^{(2,1,3,6)} +\allowbreak \mathcal{H}_{k}^{(2,1,4,5)} +\allowbreak \mathcal{H}_{k}^{(2,1,7,2)} +\allowbreak 2\,\allowbreak \mathcal{H}_{k}^{(2,1,8,1)} +\allowbreak 2\,\allowbreak \mathcal{H}_{k}^{(2,3,1,6)} +\allowbreak \mathcal{H}_{k}^{(2,3,2,5)} +\allowbreak \mathcal{H}_{k}^{(2,3,4,3)} +\allowbreak 3\,\allowbreak \mathcal{H}_{k}^{(2,3,5,2)} +\allowbreak 3\,\allowbreak \mathcal{H}_{k}^{(2,3,6,1)} +\allowbreak 2\,\allowbreak \mathcal{H}_{k}^{(2,4,1,5)} +\allowbreak 2\,\allowbreak \mathcal{H}_{k}^{(2,4,4,2)} +\allowbreak 4\,\allowbreak \mathcal{H}_{k}^{(2,4,5,1)}
\end{dmath*}
\vspace{-1.1em}
\begin{dmath*}
{}\quad +\allowbreak \mathcal{H}_{k}^{(2,5,4,1)} +\allowbreak 3\,\allowbreak \mathcal{H}_{k}^{(2,7,1,2)} +\allowbreak 3\,\allowbreak \mathcal{H}_{k}^{(2,7,2,1)} +\allowbreak 6\,\allowbreak \mathcal{H}_{k}^{(2,8,1,1)} +\allowbreak \mathcal{H}_{k}^{(3,3,1,5)} +\allowbreak \mathcal{H}_{k}^{(3,3,4,2)} +\allowbreak 2\,\allowbreak \mathcal{H}_{k}^{(3,3,5,1)} +\allowbreak \mathcal{H}_{k}^{(3,4,4,1)} +\allowbreak 2\,\allowbreak \mathcal{H}_{k}^{(3,7,1,1)} +\allowbreak 2\,\allowbreak \mathcal{H}_{k}^{(5,1,1,5)} +\allowbreak 2\,\allowbreak \mathcal{H}_{k}^{(5,1,4,2)} +\allowbreak 4\,\allowbreak \mathcal{H}_{k}^{(5,1,5,1)} +\allowbreak \mathcal{H}_{k}^{(5,2,4,1)}
\end{dmath*}
\vspace{-1.1em}
\begin{dmath*}
{}\quad +\allowbreak 3\,\allowbreak \mathcal{H}_{k}^{(5,4,1,2)} +\allowbreak 3\,\allowbreak \mathcal{H}_{k}^{(5,4,2,1)} +\allowbreak 6\,\allowbreak \mathcal{H}_{k}^{(5,5,1,1)} +\allowbreak 2\,\allowbreak \mathcal{H}_{k}^{(6,1,4,1)} +\allowbreak 4\,\allowbreak \mathcal{H}_{k}^{(6,4,1,1)} +\allowbreak 6\,\allowbreak \mathcal{H}_{k}^{(9,1,1,1)} +\allowbreak \mathcal{H}_{k}^{(2,1,3,1,5)} +\allowbreak \mathcal{H}_{k}^{(2,1,3,4,2)} +\allowbreak 2\,\allowbreak \mathcal{H}_{k}^{(2,1,3,5,1)} +\allowbreak \mathcal{H}_{k}^{(2,1,4,4,1)} +\allowbreak 2\,\allowbreak \mathcal{H}_{k}^{(2,1,7,1,1)} +\allowbreak 2\,\allowbreak \mathcal{H}_{k}^{(2,3,1,1,5)} +\allowbreak 2\,\allowbreak \mathcal{H}_{k}^{(2,3,1,4,2)}
\end{dmath*}
\vspace{-1.1em}
\begin{dmath*}
{}\quad +\allowbreak 4\,\allowbreak \mathcal{H}_{k}^{(2,3,1,5,1)} +\allowbreak \mathcal{H}_{k}^{(2,3,2,4,1)} +\allowbreak 3\,\allowbreak \mathcal{H}_{k}^{(2,3,4,1,2)} +\allowbreak 3\,\allowbreak \mathcal{H}_{k}^{(2,3,4,2,1)} +\allowbreak 6\,\allowbreak \mathcal{H}_{k}^{(2,3,5,1,1)} +\allowbreak 2\,\allowbreak \mathcal{H}_{k}^{(2,4,1,4,1)} +\allowbreak 4\,\allowbreak \mathcal{H}_{k}^{(2,4,4,1,1)} +\allowbreak 6\,\allowbreak \mathcal{H}_{k}^{(2,7,1,1,1)} +\allowbreak \mathcal{H}_{k}^{(3,3,1,4,1)} +\allowbreak 2\,\allowbreak \mathcal{H}_{k}^{(3,3,4,1,1)} +\allowbreak 2\,\allowbreak \mathcal{H}_{k}^{(5,1,1,4,1)} +\allowbreak 4\,\allowbreak \mathcal{H}_{k}^{(5,1,4,1,1)}
\end{dmath*}
\vspace{-1.1em}
\begin{dmath*}
{}\quad +\allowbreak 6\,\allowbreak \mathcal{H}_{k}^{(5,4,1,1,1)} +\allowbreak \mathcal{H}_{k}^{(2,1,3,1,4,1)} +\allowbreak 2\,\allowbreak \mathcal{H}_{k}^{(2,1,3,4,1,1)} +\allowbreak 2\,\allowbreak \mathcal{H}_{k}^{(2,3,1,1,4,1)} +\allowbreak 4\,\allowbreak \mathcal{H}_{k}^{(2,3,1,4,1,1)} +\allowbreak 6\,\allowbreak \mathcal{H}_{k}^{(2,3,4,1,1,1)}
\end{dmath*}
\subsection*{Identity 51}
\begin{dmath*}
\sum_{n_{1}=1}^{k}\sum_{n_{2}=1}^{n_{1}} F_{n_{1}}\,\allowbreak H_{2\,\allowbreak n_{2}} = \frac{-1}{4\,\allowbreak \sqrt{5}}\,\allowbreak \bigl(-\frac{\mathrm{i}^{1 +\allowbreak 2\,\allowbreak k}\,\allowbreak 2^{-1 +\allowbreak \frac{1}{2}\,\allowbreak \bigl(1 +\allowbreak 2\,\allowbreak k\bigr)}\,\allowbreak \bigl(1 +\allowbreak \sqrt{5}\bigr)^{\frac{1}{2}\,\allowbreak \bigl(-1 -\allowbreak 2\,\allowbreak k\bigr)}\,\allowbreak \mathcal{A}_{2\,\allowbreak k}}{1 -\allowbreak \mathrm{i}\,\allowbreak \sqrt{\frac{2}{1 +\allowbreak \sqrt{5}}}} -\allowbreak \frac{\bigl(-\mathrm{i}\bigr)^{1 +\allowbreak 2\,\allowbreak k}\,\allowbreak 2^{-1 +\allowbreak \frac{1}{2}\,\allowbreak \bigl(1 +\allowbreak 2\,\allowbreak k\bigr)}\,\allowbreak \bigl(1 +\allowbreak \sqrt{5}\bigr)^{\frac{1}{2}\,\allowbreak \bigl(-1 -\allowbreak 2\,\allowbreak k\bigr)}\,\allowbreak \mathcal{A}_{2\,\allowbreak k}}{1 +\allowbreak \mathrm{i}\,\allowbreak \sqrt{\frac{2}{1 +\allowbreak \sqrt{5}}}}
\end{dmath*}
\vspace{-1.1em}
\begin{dmath*}
{}\quad +\allowbreak \frac{2^{-1 +\allowbreak \frac{1}{2}\,\allowbreak \bigl(-1 -\allowbreak 2\,\allowbreak k\bigr)}\,\allowbreak \bigl(1 +\allowbreak \sqrt{5}\bigr)^{\frac{1}{2}\,\allowbreak \bigl(1 +\allowbreak 2\,\allowbreak k\bigr)}\,\allowbreak \mathcal{A}_{2\,\allowbreak k}}{1 -\allowbreak \sqrt{\frac{1}{2}\,\allowbreak \bigl(1 +\allowbreak \sqrt{5}\bigr)}} +\allowbreak \frac{\bigl(-1\bigr)^{1 +\allowbreak 2\,\allowbreak k}\,\allowbreak 2^{-1 +\allowbreak \frac{1}{2}\,\allowbreak \bigl(-1 -\allowbreak 2\,\allowbreak k\bigr)}\,\allowbreak \bigl(1 +\allowbreak \sqrt{5}\bigr)^{\frac{1}{2}\,\allowbreak \bigl(1 +\allowbreak 2\,\allowbreak k\bigr)}\,\allowbreak \mathcal{A}_{2\,\allowbreak k}}{1 +\allowbreak \sqrt{\frac{1}{2}\,\allowbreak \bigl(1 +\allowbreak \sqrt{5}\bigr)}}
\end{dmath*}
\vspace{-1.1em}
\begin{dmath*}
{}\quad +\allowbreak 2\,\allowbreak \mathcal{A}_{2\,\allowbreak k}^{(1)}\!(\sqrt{\frac{1}{2}\,\allowbreak \bigl(1 +\allowbreak \sqrt{5}\bigr)}) -\allowbreak \frac{\bigl(1 +\allowbreak \sqrt{5}\bigr)\,\allowbreak \mathcal{A}_{2\,\allowbreak k}^{(1)}\!(\sqrt{\frac{1}{2}\,\allowbreak \bigl(1 +\allowbreak \sqrt{5}\bigr)})}{2\,\allowbreak \bigl(1 -\allowbreak \sqrt{\frac{1}{2}\,\allowbreak \bigl(1 +\allowbreak \sqrt{5}\bigr)}\bigr)\,\allowbreak \bigl(1 +\allowbreak \sqrt{\frac{1}{2}\,\allowbreak \bigl(1 +\allowbreak \sqrt{5}\bigr)}\bigr)} -\allowbreak \frac{1}{1 -\allowbreak \sqrt{\frac{1}{2}\,\allowbreak \bigl(1 +\allowbreak \sqrt{5}\bigr)}}\,\allowbreak \bigl(2^{\frac{1}{2}\,\allowbreak \bigl(-1 -\allowbreak 2\,\allowbreak k\bigr)}\,\allowbreak \bigl(1
\end{dmath*}
\vspace{-1.1em}
\begin{dmath*}
{}\quad +\allowbreak \sqrt{5}\bigr)^{\frac{1}{2}\,\allowbreak \bigl(1 +\allowbreak 2\,\allowbreak k\bigr)}\,\allowbreak \mathcal{A}_{2\,\allowbreak k} -\allowbreak \sqrt{\frac{1}{2}\,\allowbreak \bigl(1 +\allowbreak \sqrt{5}\bigr)}\,\allowbreak \mathcal{A}_{2\,\allowbreak k}^{(1)}\!(\sqrt{\frac{1}{2}\,\allowbreak \bigl(1 +\allowbreak \sqrt{5}\bigr)})\bigr) +\allowbreak \frac{\bigl(-1\bigr)^{1 +\allowbreak 2\,\allowbreak k}\,\allowbreak \bigl(-\mathrm{i}\bigr)^{1 +\allowbreak 2\,\allowbreak k}\,\allowbreak 2^{-1 +\allowbreak \frac{1}{2}\,\allowbreak \bigl(1 +\allowbreak 2\,\allowbreak k\bigr)}\,\allowbreak \bigl(1 +\allowbreak \sqrt{5}\bigr)^{\frac{1}{2}\,\allowbreak \bigl(-1 -\allowbreak 2\,\allowbreak k\bigr)}\,\allowbreak H_{2\,\allowbreak k}}{1 -\allowbreak \mathrm{i}\,\allowbreak \sqrt{\frac{2}{1 +\allowbreak \sqrt{5}}}}
\end{dmath*}
\vspace{-1.1em}
\begin{dmath*}
{}\quad +\allowbreak \frac{\bigl(-1\bigr)^{1 +\allowbreak 2\,\allowbreak k}\,\allowbreak \mathrm{i}^{1 +\allowbreak 2\,\allowbreak k}\,\allowbreak 2^{-1 +\allowbreak \frac{1}{2}\,\allowbreak \bigl(1 +\allowbreak 2\,\allowbreak k\bigr)}\,\allowbreak \bigl(1 +\allowbreak \sqrt{5}\bigr)^{\frac{1}{2}\,\allowbreak \bigl(-1 -\allowbreak 2\,\allowbreak k\bigr)}\,\allowbreak H_{2\,\allowbreak k}}{1 +\allowbreak \mathrm{i}\,\allowbreak \sqrt{\frac{2}{1 +\allowbreak \sqrt{5}}}} -\allowbreak \frac{\bigl(-1\bigr)^{2 +\allowbreak 4\,\allowbreak k}\,\allowbreak 2^{-1 +\allowbreak \frac{1}{2}\,\allowbreak \bigl(-1 -\allowbreak 2\,\allowbreak k\bigr)}\,\allowbreak \bigl(1 +\allowbreak \sqrt{5}\bigr)^{\frac{1}{2}\,\allowbreak \bigl(1 +\allowbreak 2\,\allowbreak k\bigr)}\,\allowbreak H_{2\,\allowbreak k}}{1 -\allowbreak \sqrt{\frac{1}{2}\,\allowbreak \bigl(1 +\allowbreak \sqrt{5}\bigr)}}
\end{dmath*}
\vspace{-1.1em}
\begin{dmath*}
{}\quad -\allowbreak \frac{\bigl(-1\bigr)^{1 +\allowbreak 2\,\allowbreak k}\,\allowbreak 2^{-1 +\allowbreak \frac{1}{2}\,\allowbreak \bigl(-1 -\allowbreak 2\,\allowbreak k\bigr)}\,\allowbreak \bigl(1 +\allowbreak \sqrt{5}\bigr)^{\frac{1}{2}\,\allowbreak \bigl(1 +\allowbreak 2\,\allowbreak k\bigr)}\,\allowbreak H_{2\,\allowbreak k}}{1 +\allowbreak \sqrt{\frac{1}{2}\,\allowbreak \bigl(1 +\allowbreak \sqrt{5}\bigr)}} -\allowbreak \frac{1}{1 +\allowbreak \sqrt{\frac{1}{2}\,\allowbreak \bigl(1 +\allowbreak \sqrt{5}\bigr)}}\,\allowbreak 3\,\allowbreak \bigl(\sqrt{\frac{1}{2}\,\allowbreak \bigl(1 +\allowbreak \sqrt{5}\bigr)}\,\allowbreak \mathcal{A}_{2\,\allowbreak k}^{(1)}\!(\sqrt{\frac{1}{2}\,\allowbreak \bigl(1 +\allowbreak \sqrt{5}\bigr)})
\end{dmath*}
\vspace{-1.1em}
\begin{dmath*}
{}\quad -\allowbreak \bigl(-1\bigr)^{1 +\allowbreak 2\,\allowbreak k}\,\allowbreak 2^{\frac{1}{2}\,\allowbreak \bigl(-1 -\allowbreak 2\,\allowbreak k\bigr)}\,\allowbreak \bigl(1 +\allowbreak \sqrt{5}\bigr)^{\frac{1}{2}\,\allowbreak \bigl(1 +\allowbreak 2\,\allowbreak k\bigr)}\,\allowbreak H_{2\,\allowbreak k}\bigr) +\allowbreak 2\,\allowbreak \mathcal{H}_{2\,\allowbreak k}^{(1)}\!(-\mathrm{i}\,\allowbreak \sqrt{\frac{2}{1 +\allowbreak \sqrt{5}}}) +\allowbreak \frac{2\,\allowbreak \mathcal{H}_{2\,\allowbreak k}^{(1)}\!(-\mathrm{i}\,\allowbreak \sqrt{\frac{2}{1 +\allowbreak \sqrt{5}}})}{\bigl(1 +\allowbreak \sqrt{5}\bigr)\,\allowbreak \bigl(1 -\allowbreak \mathrm{i}\,\allowbreak \sqrt{\frac{2}{1 +\allowbreak \sqrt{5}}}\bigr)\,\allowbreak \bigl(1 +\allowbreak \mathrm{i}\,\allowbreak \sqrt{\frac{2}{1 +\allowbreak \sqrt{5}}}\bigr)}
\end{dmath*}
\vspace{-1.1em}
\begin{dmath*}
{}\quad +\allowbreak \frac{3\,\allowbreak \bigl(-\bigl(\bigl(-\mathrm{i}\bigr)^{1 +\allowbreak 2\,\allowbreak k}\bigr)\,\allowbreak 2^{\frac{1}{2}\,\allowbreak \bigl(1 +\allowbreak 2\,\allowbreak k\bigr)}\,\allowbreak \bigl(1 +\allowbreak \sqrt{5}\bigr)^{\frac{1}{2}\,\allowbreak \bigl(-1 -\allowbreak 2\,\allowbreak k\bigr)}\,\allowbreak H_{2\,\allowbreak k} -\allowbreak \mathrm{i}\,\allowbreak \sqrt{\frac{2}{1 +\allowbreak \sqrt{5}}}\,\allowbreak \mathcal{H}_{2\,\allowbreak k}^{(1)}\!(-\mathrm{i}\,\allowbreak \sqrt{\frac{2}{1 +\allowbreak \sqrt{5}}})\bigr)}{1 +\allowbreak \mathrm{i}\,\allowbreak \sqrt{\frac{2}{1 +\allowbreak \sqrt{5}}}} +\allowbreak \frac{\mathrm{i}^{1 +\allowbreak 2\,\allowbreak k}\,\allowbreak 2^{\frac{1}{2}\,\allowbreak \bigl(1 +\allowbreak 2\,\allowbreak k\bigr)}\,\allowbreak \bigl(1 +\allowbreak \sqrt{5}\bigr)^{\frac{1}{2}\,\allowbreak \bigl(-1 -\allowbreak 2\,\allowbreak k\bigr)}\,\allowbreak \mathcal{A}_{2\,\allowbreak k} +\allowbreak \mathrm{i}\,\allowbreak \sqrt{\frac{2}{1 +\allowbreak \sqrt{5}}}\,\allowbreak \mathcal{H}_{2\,\allowbreak k}^{(1)}\!(-\mathrm{i}\,\allowbreak \sqrt{\frac{2}{1 +\allowbreak \sqrt{5}}})}{1 -\allowbreak \mathrm{i}\,\allowbreak \sqrt{\frac{2}{1 +\allowbreak \sqrt{5}}}}
\end{dmath*}
\vspace{-1.1em}
\begin{dmath*}
{}\quad +\allowbreak 2\,\allowbreak \mathcal{H}_{2\,\allowbreak k}^{(1)}\!(\mathrm{i}\,\allowbreak \sqrt{\frac{2}{1 +\allowbreak \sqrt{5}}}) +\allowbreak \frac{2\,\allowbreak \mathcal{H}_{2\,\allowbreak k}^{(1)}\!(\mathrm{i}\,\allowbreak \sqrt{\frac{2}{1 +\allowbreak \sqrt{5}}})}{\bigl(1 +\allowbreak \sqrt{5}\bigr)\,\allowbreak \bigl(1 -\allowbreak \mathrm{i}\,\allowbreak \sqrt{\frac{2}{1 +\allowbreak \sqrt{5}}}\bigr)\,\allowbreak \bigl(1 +\allowbreak \mathrm{i}\,\allowbreak \sqrt{\frac{2}{1 +\allowbreak \sqrt{5}}}\bigr)} +\allowbreak \frac{\bigl(-\mathrm{i}\bigr)^{1 +\allowbreak 2\,\allowbreak k}\,\allowbreak 2^{\frac{1}{2}\,\allowbreak \bigl(1 +\allowbreak 2\,\allowbreak k\bigr)}\,\allowbreak \bigl(1 +\allowbreak \sqrt{5}\bigr)^{\frac{1}{2}\,\allowbreak \bigl(-1 -\allowbreak 2\,\allowbreak k\bigr)}\,\allowbreak \mathcal{A}_{2\,\allowbreak k} -\allowbreak \mathrm{i}\,\allowbreak \sqrt{\frac{2}{1 +\allowbreak \sqrt{5}}}\,\allowbreak \mathcal{H}_{2\,\allowbreak k}^{(1)}\!(\mathrm{i}\,\allowbreak \sqrt{\frac{2}{1 +\allowbreak \sqrt{5}}})}{1 +\allowbreak \mathrm{i}\,\allowbreak \sqrt{\frac{2}{1 +\allowbreak \sqrt{5}}}}
\end{dmath*}
\vspace{-1.1em}
\begin{dmath*}
{}\quad +\allowbreak \frac{3\,\allowbreak \bigl(-\bigl(\mathrm{i}^{1 +\allowbreak 2\,\allowbreak k}\bigr)\,\allowbreak 2^{\frac{1}{2}\,\allowbreak \bigl(1 +\allowbreak 2\,\allowbreak k\bigr)}\,\allowbreak \bigl(1 +\allowbreak \sqrt{5}\bigr)^{\frac{1}{2}\,\allowbreak \bigl(-1 -\allowbreak 2\,\allowbreak k\bigr)}\,\allowbreak H_{2\,\allowbreak k} +\allowbreak \mathrm{i}\,\allowbreak \sqrt{\frac{2}{1 +\allowbreak \sqrt{5}}}\,\allowbreak \mathcal{H}_{2\,\allowbreak k}^{(1)}\!(\mathrm{i}\,\allowbreak \sqrt{\frac{2}{1 +\allowbreak \sqrt{5}}})\bigr)}{1 -\allowbreak \mathrm{i}\,\allowbreak \sqrt{\frac{2}{1 +\allowbreak \sqrt{5}}}} -\allowbreak 2\,\allowbreak \mathcal{H}_{2\,\allowbreak k}^{(1)}\!(\sqrt{\frac{1}{2}\,\allowbreak \bigl(1 +\allowbreak \sqrt{5}\bigr)})
\end{dmath*}
\vspace{-1.1em}
\begin{dmath*}
{}\quad +\allowbreak \frac{\bigl(1 +\allowbreak \sqrt{5}\bigr)\,\allowbreak \mathcal{H}_{2\,\allowbreak k}^{(1)}\!(\sqrt{\frac{1}{2}\,\allowbreak \bigl(1 +\allowbreak \sqrt{5}\bigr)})}{2\,\allowbreak \bigl(1 -\allowbreak \sqrt{\frac{1}{2}\,\allowbreak \bigl(1 +\allowbreak \sqrt{5}\bigr)}\bigr)\,\allowbreak \bigl(1 +\allowbreak \sqrt{\frac{1}{2}\,\allowbreak \bigl(1 +\allowbreak \sqrt{5}\bigr)}\bigr)} -\allowbreak \frac{\bigl(-1\bigr)^{1 +\allowbreak 2\,\allowbreak k}\,\allowbreak 2^{\frac{1}{2}\,\allowbreak \bigl(-1 -\allowbreak 2\,\allowbreak k\bigr)}\,\allowbreak \bigl(1 +\allowbreak \sqrt{5}\bigr)^{\frac{1}{2}\,\allowbreak \bigl(1 +\allowbreak 2\,\allowbreak k\bigr)}\,\allowbreak \mathcal{A}_{2\,\allowbreak k} -\allowbreak \sqrt{\frac{1}{2}\,\allowbreak \bigl(1 +\allowbreak \sqrt{5}\bigr)}\,\allowbreak \mathcal{H}_{2\,\allowbreak k}^{(1)}\!(\sqrt{\frac{1}{2}\,\allowbreak \bigl(1 +\allowbreak \sqrt{5}\bigr)})}{1 +\allowbreak \sqrt{\frac{1}{2}\,\allowbreak \bigl(1 +\allowbreak \sqrt{5}\bigr)}}
\end{dmath*}
\vspace{-1.1em}
\begin{dmath*}
{}\quad -\allowbreak \frac{3\,\allowbreak \bigl(-\bigl(2^{\frac{1}{2}\,\allowbreak \bigl(-1 -\allowbreak 2\,\allowbreak k\bigr)}\bigr)\,\allowbreak \bigl(1 +\allowbreak \sqrt{5}\bigr)^{\frac{1}{2}\,\allowbreak \bigl(1 +\allowbreak 2\,\allowbreak k\bigr)}\,\allowbreak H_{2\,\allowbreak k} +\allowbreak \sqrt{\frac{1}{2}\,\allowbreak \bigl(1 +\allowbreak \sqrt{5}\bigr)}\,\allowbreak \mathcal{H}_{2\,\allowbreak k}^{(1)}\!(\sqrt{\frac{1}{2}\,\allowbreak \bigl(1 +\allowbreak \sqrt{5}\bigr)})\bigr)}{1 -\allowbreak \sqrt{\frac{1}{2}\,\allowbreak \bigl(1 +\allowbreak \sqrt{5}\bigr)}} +\allowbreak \mathcal{H}_{2\,\allowbreak k}^{(0,0,1)}\!(-\mathrm{i}\,\allowbreak \sqrt{\frac{2}{1 +\allowbreak \sqrt{5}}},\allowbreak 1,\allowbreak 1)
\end{dmath*}
\vspace{-1.1em}
\begin{dmath*}
{}\quad +\allowbreak \mathcal{H}_{2\,\allowbreak k}^{(0,0,1)}\!(\mathrm{i}\,\allowbreak \sqrt{\frac{2}{1 +\allowbreak \sqrt{5}}},\allowbreak 1,\allowbreak 1) -\allowbreak \mathcal{H}_{2\,\allowbreak k}^{(0,0,1)}\!(-\sqrt{\frac{1}{2}\,\allowbreak \bigl(1 +\allowbreak \sqrt{5}\bigr)},\allowbreak 1,\allowbreak 1) -\allowbreak \mathcal{H}_{2\,\allowbreak k}^{(0,0,1)}\!(\sqrt{\frac{1}{2}\,\allowbreak \bigl(1 +\allowbreak \sqrt{5}\bigr)},\allowbreak 1,\allowbreak 1)\bigr)
\end{dmath*}
\subsection*{Identity 52}
\begin{dmath*}
\sum_{n_{1}=1}^{k}\sum_{n_{2}=1}^{n_{1}} \frac{\mathcal{P}_{n_{1}}\!(\{\{2,\allowbreak 1\},\allowbreak \{3\}\};\allowbreak \{\frac{1}{2},\allowbreak \frac{-\mathrm{i}}{3}\};\allowbreak \{\{\{1,\allowbreak 1\},\allowbreak \{2,\allowbreak 0,\allowbreak 3\}\},\allowbreak \{\{-1,\allowbreak 2,\allowbreak 1\}\}\})}{\bigl(2 +\allowbreak 3\,\allowbreak n_{1} +\allowbreak 4\,\allowbreak n_{1}^{2}\bigr)^{2}\,\allowbreak \bigl(6 +\allowbreak 9\,\allowbreak n_{2}^{2}\bigr)^{2}} = \mathcal{P}_{k}\!(\{\{2,\allowbreak 2,\allowbreak 1\},\allowbreak \{3,\allowbreak 2\}\};\allowbreak \{\frac{1}{2},\allowbreak -\bigl(\frac{\mathrm{i}}{3}\bigr)\};\allowbreak \{\{\{2,\allowbreak 3,\allowbreak 4\},\allowbreak \{1,\allowbreak 1\},\allowbreak \{2,\allowbreak 0,\allowbreak 3\}\},\allowbreak \{\{-1,\allowbreak 2,\allowbreak 1\},\allowbreak \{6,\allowbreak 0,\allowbreak 9\}\}\}) +\allowbreak \mathcal{P}_{k}\!(\{\{2,\allowbreak 2,\allowbreak 1,\allowbreak 2\},\allowbreak \{3\}\};\allowbreak \{\frac{1}{2},\allowbreak -\bigl(\frac{\mathrm{i}}{3}\bigr)\};\allowbreak \{\{\{2,\allowbreak 3,\allowbreak 4\},\allowbreak \{1,\allowbreak 1\},\allowbreak \{2,\allowbreak 0,\allowbreak 3\},\allowbreak \{6,\allowbreak 0,\allowbreak 9\}\},\allowbreak \{\{-1,\allowbreak 2,\allowbreak 1\}\}\})
\end{dmath*}
\vspace{-1.1em}
\begin{dmath*}
{}\quad +\allowbreak \mathcal{P}_{k}\!(\{\{2\},\allowbreak \{2,\allowbreak 1\},\allowbreak \{3,\allowbreak 2\}\};\allowbreak \{1,\allowbreak \frac{1}{2},\allowbreak -\bigl(\frac{\mathrm{i}}{3}\bigr)\};\allowbreak \{\{\{2,\allowbreak 3,\allowbreak 4\}\},\allowbreak \{\{1,\allowbreak 1\},\allowbreak \{2,\allowbreak 0,\allowbreak 3\}\},\allowbreak \{\{-1,\allowbreak 2,\allowbreak 1\},\allowbreak \{6,\allowbreak 0,\allowbreak 9\}\}\}) +\allowbreak \mathcal{P}_{k}\!(\{\{2\},\allowbreak \{2,\allowbreak 1,\allowbreak 2\},\allowbreak \{3\}\};\allowbreak \{1,\allowbreak \frac{1}{2},\allowbreak -\bigl(\frac{\mathrm{i}}{3}\bigr)\};\allowbreak \{\{\{2,\allowbreak 3,\allowbreak 4\}\},\allowbreak \{\{1,\allowbreak 1\},\allowbreak \{2,\allowbreak 0,\allowbreak 3\},\allowbreak \{6,\allowbreak 0,\allowbreak 9\}\},\allowbreak \{\{-1,\allowbreak 2,\allowbreak 1\}\}\})
\end{dmath*}
\vspace{-1.1em}
\begin{dmath*}
{}\quad +\allowbreak \mathcal{P}_{k}\!(\{\{2,\allowbreak 2\},\allowbreak \{2,\allowbreak 1\},\allowbreak \{3\}\};\allowbreak \{1,\allowbreak \frac{1}{2},\allowbreak -\bigl(\frac{\mathrm{i}}{3}\bigr)\};\allowbreak \{\{\{2,\allowbreak 3,\allowbreak 4\},\allowbreak \{6,\allowbreak 0,\allowbreak 9\}\},\allowbreak \{\{1,\allowbreak 1\},\allowbreak \{2,\allowbreak 0,\allowbreak 3\}\},\allowbreak \{\{-1,\allowbreak 2,\allowbreak 1\}\}\}) +\allowbreak \mathcal{P}_{k}\!(\{\{2,\allowbreak 2,\allowbreak 1\},\allowbreak \{2\},\allowbreak \{3\}\};\allowbreak \{\frac{1}{2},\allowbreak 1,\allowbreak -\bigl(\frac{\mathrm{i}}{3}\bigr)\};\allowbreak \{\{\{2,\allowbreak 3,\allowbreak 4\},\allowbreak \{1,\allowbreak 1\},\allowbreak \{2,\allowbreak 0,\allowbreak 3\}\},\allowbreak \{\{6,\allowbreak 0,\allowbreak 9\}\},\allowbreak \{\{-1,\allowbreak 2,\allowbreak 1\}\}\})
\end{dmath*}
\vspace{-1.1em}
\begin{dmath*}
{}\quad +\allowbreak \mathcal{P}_{k}\!(\{\{2,\allowbreak 2,\allowbreak 1\},\allowbreak \{3\},\allowbreak \{2\}\};\allowbreak \{\frac{1}{2},\allowbreak -\bigl(\frac{\mathrm{i}}{3}\bigr),\allowbreak 1\};\allowbreak \{\{\{2,\allowbreak 3,\allowbreak 4\},\allowbreak \{1,\allowbreak 1\},\allowbreak \{2,\allowbreak 0,\allowbreak 3\}\},\allowbreak \{\{-1,\allowbreak 2,\allowbreak 1\}\},\allowbreak \{\{6,\allowbreak 0,\allowbreak 9\}\}\}) +\allowbreak \mathcal{P}_{k}\!(\{\{2\},\allowbreak \{2\},\allowbreak \{2,\allowbreak 1\},\allowbreak \{3\}\};\allowbreak \{1,\allowbreak 1,\allowbreak \frac{1}{2},\allowbreak -\bigl(\frac{\mathrm{i}}{3}\bigr)\};\allowbreak \{\{\{2,\allowbreak 3,\allowbreak 4\}\},\allowbreak \{\{6,\allowbreak 0,\allowbreak 9\}\},\allowbreak \{\{1,\allowbreak 1\},\allowbreak \{2,\allowbreak 0,\allowbreak 3\}\},\allowbreak \{\{-1,\allowbreak 2,\allowbreak 1\}\}\})
\end{dmath*}
\vspace{-1.1em}
\begin{dmath*}
{}\quad +\allowbreak \mathcal{P}_{k}\!(\{\{2\},\allowbreak \{2,\allowbreak 1\},\allowbreak \{2\},\allowbreak \{3\}\};\allowbreak \{1,\allowbreak \frac{1}{2},\allowbreak 1,\allowbreak -\bigl(\frac{\mathrm{i}}{3}\bigr)\};\allowbreak \{\{\{2,\allowbreak 3,\allowbreak 4\}\},\allowbreak \{\{1,\allowbreak 1\},\allowbreak \{2,\allowbreak 0,\allowbreak 3\}\},\allowbreak \{\{6,\allowbreak 0,\allowbreak 9\}\},\allowbreak \{\{-1,\allowbreak 2,\allowbreak 1\}\}\}) +\allowbreak \mathcal{P}_{k}\!(\{\{2\},\allowbreak \{2,\allowbreak 1\},\allowbreak \{3\},\allowbreak \{2\}\};\allowbreak \{1,\allowbreak \frac{1}{2},\allowbreak -\bigl(\frac{\mathrm{i}}{3}\bigr),\allowbreak 1\};\allowbreak \{\{\{2,\allowbreak 3,\allowbreak 4\}\},\allowbreak \{\{1,\allowbreak 1\},\allowbreak \{2,\allowbreak 0,\allowbreak 3\}\},\allowbreak \{\{-1,\allowbreak 2,\allowbreak 1\}\},\allowbreak \{\{6,\allowbreak 0,\allowbreak 9\}\}\})
\end{dmath*}
\section*{8 Infinite sums}

\subsection*{Identity 53}
\begin{dmath*}
\sum_{n=1}^{\infty} \bigl(\frac{H_{n}}{n}\bigr)^{4} = \frac{40651\,\allowbreak \pi^{8} -\allowbreak 70761600\,\allowbreak \zeta\!(5,\allowbreak 3) -\allowbreak 2268000\,\allowbreak \pi^{2}\,\allowbreak \bigl(\zeta\!(3)\bigr)^{2} -\allowbreak 272160000\,\allowbreak \zeta\!(3)\,\allowbreak \zeta\!(5)}{6804000}
\end{dmath*}
\subsection*{Identity 54}
\begin{dmath*}
\sum_{n=1}^{\infty} \bigl(\frac{H_{n}}{n}\bigr)^{6} = \frac{1}{1225944720000}\,\allowbreak \bigl(6116077813\,\allowbreak \pi^{12} -\allowbreak 4012925716800\,\allowbreak \pi^{4}\,\allowbreak \zeta\!(5,\allowbreak 3) -\allowbreak 108507053655000\,\allowbreak \pi^{2}\,\allowbreak \zeta\!(7,\allowbreak 3) +\allowbreak 804219736320000\,\allowbreak \zeta\!(9,\allowbreak 3) +\allowbreak 208410602400000\,\allowbreak \zeta\!(6,\allowbreak 4,\allowbreak 1,\allowbreak 1) +\allowbreak 1289998710000\,\allowbreak \pi^{6}\,\allowbreak \bigl(\zeta\!(3)\bigr)^{2} -\allowbreak 8172964800000\,\allowbreak \bigl(\zeta\!(3)\bigr)^{4} -\allowbreak 4236320088000\,\allowbreak \pi^{4}\,\allowbreak \zeta\!(3)\,\allowbreak \zeta\!(5)
\end{dmath*}
\vspace{-1.1em}
\begin{dmath*}
{}\quad -\allowbreak 151947821025000\,\allowbreak \pi^{2}\,\allowbreak \bigl(\zeta\!(5)\bigr)^{2} -\allowbreak 387960422850000\,\allowbreak \pi^{2}\,\allowbreak \zeta\!(3)\,\allowbreak \zeta\!(7) -\allowbreak 2161493784450000\,\allowbreak \zeta\!(5)\,\allowbreak \zeta\!(7) +\allowbreak 1348164597780000\,\allowbreak \zeta\!(3)\,\allowbreak \zeta\!(9)\bigr)
\end{dmath*}
\subsection*{Identity 55}
\begin{dmath*}
\sum_{n=1}^{\infty} \frac{\mathcal{A}_{n}\,\allowbreak \mathcal{A}_{n}^{(2)}}{n^{4}} = \frac{1}{960}\,\allowbreak \bigl(-960\,\allowbreak \bigl(\operatorname{HPL}_{4,-2,1}\!(1) +\allowbreak \operatorname{HPL}_{4,-1,2}\!(1)\bigr) +\allowbreak \pi^{6}\,\allowbreak \log\!(4) -\allowbreak 8\,\allowbreak \pi^{4}\,\allowbreak \zeta\!(3) -\allowbreak 2200\,\allowbreak \pi^{2}\,\allowbreak \zeta\!(5) +\allowbreak 22935\,\allowbreak \zeta\!(7)\bigr)
\end{dmath*}
\subsection*{Identity 56}
\begin{dmath*}
\sum_{n=1}^{\infty} \frac{\bigl(\mathcal{H}_{n}^{(1,2)}\bigr)^{2}}{n^{2}} = \frac{319\,\allowbreak \pi^{8} +\allowbreak 1587600\,\allowbreak \zeta\!(5,\allowbreak 3) +\allowbreak 189000\,\allowbreak \pi^{2}\,\allowbreak \bigl(\zeta\!(3)\bigr)^{2} -\allowbreak 2268000\,\allowbreak \zeta\!(3)\,\allowbreak \zeta\!(5)}{1134000}
\end{dmath*}
\subsection*{Identity 57}
\begin{dmath*}
\sum_{n=1}^{\infty} \frac{H_{n}}{n^{10 +\allowbreak 8\,\allowbreak \mathrm{i}}} = \zeta\!(10 +\allowbreak 8\,\allowbreak \mathrm{i},\allowbreak 1) +\allowbreak \zeta\!(11 +\allowbreak 8\,\allowbreak \mathrm{i})
\end{dmath*}
\subsection*{Identity 58}
\begin{dmath*}
\sum_{n=1}^{\infty} \frac{\mathcal{A}_{n}^{(\frac{1}{2})}\,\allowbreak \mathcal{A}_{n}^{(2)}}{n^{4 +\allowbreak \frac{2}{3}\,\allowbreak \mathrm{i}}} = \operatorname{Li}_{\frac{9}{2} +\allowbreak \frac{2\,\allowbreak \mathrm{i}}{3},2}\!(-1,\allowbreak -1) +\allowbreak \operatorname{Li}_{6 +\allowbreak \frac{2\,\allowbreak \mathrm{i}}{3},\frac{1}{2}}\!(-1,\allowbreak -1) +\allowbreak \operatorname{Li}_{4 +\allowbreak \frac{2\,\allowbreak \mathrm{i}}{3},\frac{1}{2},2}\!(1,\allowbreak -1,\allowbreak -1) +\allowbreak \operatorname{Li}_{4 +\allowbreak \frac{2\,\allowbreak \mathrm{i}}{3},2,\frac{1}{2}}\!(1,\allowbreak -1,\allowbreak -1) +\allowbreak \zeta\!(4 +\allowbreak \frac{2\,\allowbreak \mathrm{i}}{3},\allowbreak \frac{5}{2}) +\allowbreak \zeta\!(\frac{13}{2} +\allowbreak \frac{2\,\allowbreak \mathrm{i}}{3})
\end{dmath*}
\subsection*{Identity 59}
\begin{dmath*}
\sum_{n=1}^{\infty} \frac{\bigl(\frac{1}{2}\bigr)^{n}\,\allowbreak h_{n}^{[1]}\!(1;\frac{1}{2})\,\allowbreak h_{n}^{[1]}\!(1;1)}{n^{2 -\allowbreak \mathrm{i}}} = 4\,\allowbreak \operatorname{Li}_{2 -\allowbreak \mathrm{i},2}\!(\frac{1}{2},\allowbreak \frac{1}{2}) +\allowbreak 2\,\allowbreak \operatorname{Li}_{3 -\allowbreak \mathrm{i},1}\!(\frac{1}{4},\allowbreak 1) +\allowbreak 2\,\allowbreak \operatorname{Li}_{3 -\allowbreak \mathrm{i},1}\!(\frac{1}{2},\allowbreak \frac{1}{2}) +\allowbreak 5\,\allowbreak \operatorname{Li}_{2 -\allowbreak \mathrm{i},0,2}\!(\frac{1}{2},\allowbreak 1,\allowbreak \frac{1}{2}) +\allowbreak 6\,\allowbreak \operatorname{Li}_{2 -\allowbreak \mathrm{i},1,1}\!(\frac{1}{2},\allowbreak \frac{1}{2},\allowbreak 1) +\allowbreak 6\,\allowbreak \operatorname{Li}_{2 -\allowbreak \mathrm{i},1,1}\!(\frac{1}{2},\allowbreak 1,\allowbreak \frac{1}{2})
\end{dmath*}
\vspace{-1.1em}
\begin{dmath*}
{}\quad +\allowbreak \operatorname{Li}_{3 -\allowbreak \mathrm{i},0,1}\!(\frac{1}{4},\allowbreak 1,\allowbreak 1) +\allowbreak \operatorname{Li}_{3 -\allowbreak \mathrm{i},0,1}\!(\frac{1}{2},\allowbreak 1,\allowbreak \frac{1}{2}) +\allowbreak 2\,\allowbreak \operatorname{Li}_{2 -\allowbreak \mathrm{i},0,0,2}\!(\frac{1}{2},\allowbreak 1,\allowbreak 1,\allowbreak \frac{1}{2}) +\allowbreak 6\,\allowbreak \operatorname{Li}_{2 -\allowbreak \mathrm{i},0,1,1}\!(\frac{1}{2},\allowbreak 1,\allowbreak \frac{1}{2},\allowbreak 1) +\allowbreak 6\,\allowbreak \operatorname{Li}_{2 -\allowbreak \mathrm{i},0,1,1}\!(\frac{1}{2},\allowbreak 1,\allowbreak 1,\allowbreak \frac{1}{2})
\end{dmath*}
\vspace{-1.1em}
\begin{dmath*}
{}\quad +\allowbreak 2\,\allowbreak \operatorname{Li}_{2 -\allowbreak \mathrm{i},1,0,1}\!(\frac{1}{2},\allowbreak \frac{1}{2},\allowbreak 1,\allowbreak 1) +\allowbreak 2\,\allowbreak \operatorname{Li}_{2 -\allowbreak \mathrm{i},1,0,1}\!(\frac{1}{2},\allowbreak 1,\allowbreak 1,\allowbreak \frac{1}{2}) +\allowbreak 2\,\allowbreak \operatorname{Li}_{2 -\allowbreak \mathrm{i},0,0,1,1}\!(\frac{1}{2},\allowbreak 1,\allowbreak 1,\allowbreak \frac{1}{2},\allowbreak 1) +\allowbreak 2\,\allowbreak \operatorname{Li}_{2 -\allowbreak \mathrm{i},0,0,1,1}\!(\frac{1}{2},\allowbreak 1,\allowbreak 1,\allowbreak 1,\allowbreak \frac{1}{2}) +\allowbreak \operatorname{Li}_{2 -\allowbreak \mathrm{i},0,1,0,1}\!(\frac{1}{2},\allowbreak 1,\allowbreak \frac{1}{2},\allowbreak 1,\allowbreak 1)
\end{dmath*}
\vspace{-1.1em}
\begin{dmath*}
{}\quad +\allowbreak \operatorname{Li}_{2 -\allowbreak \mathrm{i},0,1,0,1}\!(\frac{1}{2},\allowbreak 1,\allowbreak 1,\allowbreak 1,\allowbreak \frac{1}{2}) +\allowbreak \operatorname{Li}_{4 -\allowbreak \mathrm{i}}\!(\frac{1}{4})
\end{dmath*}
\subsection*{Identity 60}
\begin{dmath*}
\sum_{n=1}^{\infty} \frac{\bigl(\mathcal{H}_{n}^{(1,2)}\bigr)^{2}}{n^{3 -\allowbreak 6\,\allowbreak \mathrm{i}}} = \zeta\!(5 -\allowbreak 6\,\allowbreak \mathrm{i},\allowbreak 4) +\allowbreak \zeta\!(3 -\allowbreak 6\,\allowbreak \mathrm{i},\allowbreak 2,\allowbreak 4) +\allowbreak 2\,\allowbreak \bigl(\zeta\!(4 -\allowbreak 6\,\allowbreak \mathrm{i},\allowbreak 1,\allowbreak 4) +\allowbreak \zeta\!(4 -\allowbreak 6\,\allowbreak \mathrm{i},\allowbreak 3,\allowbreak 2) +\allowbreak \zeta\!(5 -\allowbreak 6\,\allowbreak \mathrm{i},\allowbreak 2,\allowbreak 2) +\allowbreak \zeta\!(3 -\allowbreak 6\,\allowbreak \mathrm{i},\allowbreak 1,\allowbreak 1,\allowbreak 4) +\allowbreak \zeta\!(3 -\allowbreak 6\,\allowbreak \mathrm{i},\allowbreak 1,\allowbreak 3,\allowbreak 2)
\end{dmath*}
\vspace{-1.1em}
\begin{dmath*}
{}\quad +\allowbreak \zeta\!(3 -\allowbreak 6\,\allowbreak \mathrm{i},\allowbreak 2,\allowbreak 2,\allowbreak 2) +\allowbreak 2\,\allowbreak \zeta\!(4 -\allowbreak 6\,\allowbreak \mathrm{i},\allowbreak 1,\allowbreak 2,\allowbreak 2) +\allowbreak \zeta\!(4 -\allowbreak 6\,\allowbreak \mathrm{i},\allowbreak 2,\allowbreak 1,\allowbreak 2) +\allowbreak 2\,\allowbreak \zeta\!(3 -\allowbreak 6\,\allowbreak \mathrm{i},\allowbreak 1,\allowbreak 1,\allowbreak 2,\allowbreak 2) +\allowbreak \zeta\!(3 -\allowbreak 6\,\allowbreak \mathrm{i},\allowbreak 1,\allowbreak 2,\allowbreak 1,\allowbreak 2)\bigr)
\end{dmath*}
\subsection*{Identity 61}
\begin{dmath*}
\sum_{n=1}^{\infty} \frac{\bigl(\frac{1}{2}\bigr)^{n}\,\allowbreak \mathcal{A}_{n}\,\allowbreak \mathcal{A}_{n}^{(2)}}{n^{4}} = \operatorname{HPL}_{4,3}\!(\frac{1}{2}) -\allowbreak \operatorname{HPL}_{5,-2}\!(-\bigl(\frac{1}{2}\bigr)) -\allowbreak \operatorname{HPL}_{6,-1}\!(-\bigl(\frac{1}{2}\bigr)) +\allowbreak \operatorname{HPL}_{-4,1,-2}\!(-\bigl(\frac{1}{2}\bigr)) +\allowbreak \operatorname{HPL}_{-4,2,-1}\!(-\bigl(\frac{1}{2}\bigr)) +\allowbreak \operatorname{Li}_{7}\!(\frac{1}{2})
\end{dmath*}
\subsection*{Identity 62}
\begin{dmath*}
\sum_{n=1}^{\infty} \frac{\mathcal{H}_{n}^{(2)}\!(-1)\,\allowbreak \mathcal{H}_{n}^{(3)}\!(-\bigl(\frac{1}{2}\bigr))}{n^{6 +\allowbreak \frac{\mathrm{i}}{2}}}\,\allowbreak \bigl(-1\bigr)^{n} = \operatorname{Li}_{6 +\allowbreak \frac{\mathrm{i}}{2},5}\!(-1,\allowbreak \frac{1}{2}) +\allowbreak \operatorname{Li}_{8 +\allowbreak \frac{\mathrm{i}}{2},3}\!(1,\allowbreak -\bigl(\frac{1}{2}\bigr)) +\allowbreak \operatorname{Li}_{9 +\allowbreak \frac{\mathrm{i}}{2},2}\!(\frac{1}{2},\allowbreak -1) +\allowbreak \operatorname{Li}_{6 +\allowbreak \frac{\mathrm{i}}{2},2,3}\!(-1,\allowbreak -1,\allowbreak -\bigl(\frac{1}{2}\bigr)) +\allowbreak \operatorname{Li}_{6 +\allowbreak \frac{\mathrm{i}}{2},3,2}\!(-1,\allowbreak -\bigl(\frac{1}{2}\bigr),\allowbreak -1) +\allowbreak \operatorname{Li}_{11 +\allowbreak \frac{\mathrm{i}}{2}}\!(-\bigl(\frac{1}{2}\bigr))
\end{dmath*}
\subsection*{Identity 63}
\begin{dmath*}
\sum_{n=1}^{\infty} \frac{\mathcal{H}_{2\,\allowbreak n}^{(3)}\!(\frac{1}{2})}{n^{4}} = 8\,\allowbreak \bigl(\operatorname{Li}_{4,3}\!(-1,\allowbreak \frac{1}{2}) +\allowbreak \operatorname{Li}_{4,3}\!(1,\allowbreak \frac{1}{2}) +\allowbreak \operatorname{Li}_{7}\!(-\bigl(\frac{1}{2}\bigr)) +\allowbreak \operatorname{Li}_{7}\!(\frac{1}{2})\bigr)
\end{dmath*}
\subsection*{Identity 64}
\begin{dmath*}
\sum_{n=1}^{\infty} \frac{\mathcal{A}_{4\,\allowbreak n}^{(2)}}{n^{4}} = -\frac{3107\,\allowbreak \pi^{6}}{45360} +\allowbreak 64\,\allowbreak \bigl(\operatorname{HPL}_{-4,2}\!(\mathrm{i}) +\allowbreak \operatorname{HPL}_{4,-2}\!(\mathrm{i}) +\allowbreak \operatorname{HPL}_{4,-2}\!(1)\bigr) +\allowbreak 48\,\allowbreak \bigl(\zeta\!(3)\bigr)^{2}
\end{dmath*}
\subsection*{Identity 65}
\begin{dmath*}
\sum_{n_{1}=1}^{\infty}\sum_{n_{2}=1}^{n_{1}} \frac{H_{n_{1}}\,\allowbreak H_{n_{2}}}{n_{1}^{2}\,\allowbreak n_{2}^{3}} = -\bigl(\frac{23}{180}\bigr)\,\allowbreak \pi^{4}\,\allowbreak \zeta\!(3) +\allowbreak \frac{3}{4}\,\allowbreak \pi^{2}\,\allowbreak \zeta\!(5) +\allowbreak 10\,\allowbreak \zeta\!(7)
\end{dmath*}
\subsection*{Identity 66}
\begin{dmath*}
\sum_{n_{1}=1}^{\infty}\sum_{n_{2}=1}^{n_{1}}\sum_{n_{3}=1}^{n_{2}} \frac{H_{n_{1}}\,\allowbreak H_{n_{2}}\,\allowbreak H_{n_{3}}}{n_{1}^{2}\,\allowbreak n_{2}^{3}\,\allowbreak n_{3}^{4}} = \frac{169364029\,\allowbreak \pi^{12}}{891596160000} +\allowbreak \frac{5567\,\allowbreak \pi^{6}\,\allowbreak \bigl(\zeta\!(3)\bigr)^{2}}{90720} -\allowbreak \frac{\pi^{4}\,\allowbreak \bigl(843\,\allowbreak \zeta\!(5,\allowbreak 3) +\allowbreak 1255\,\allowbreak \zeta\!(3)\,\allowbreak \zeta\!(5)\bigr)}{2700} -\allowbreak \frac{\pi^{2}\,\allowbreak \bigl(35277\,\allowbreak \zeta\!(7,\allowbreak 3) +\allowbreak 34291\,\allowbreak \bigl(\zeta\!(5)\bigr)^{2} +\allowbreak 112700\,\allowbreak \zeta\!(3)\,\allowbreak \zeta\!(7)\bigr)}{4032} +\allowbreak \frac{1}{216}\,\allowbreak \bigl(16686\,\allowbreak \zeta\!(9,\allowbreak 3) +\allowbreak 3852\,\allowbreak \zeta\!(6,\allowbreak 4,\allowbreak 1,\allowbreak 1)
\end{dmath*}
\vspace{-1.1em}
\begin{dmath*}
{}\quad -\allowbreak 429\,\allowbreak \bigl(\zeta\!(3)\bigr)^{4} +\allowbreak 2367\,\allowbreak \zeta\!(5)\,\allowbreak \zeta\!(7) +\allowbreak 39004\,\allowbreak \zeta\!(3)\,\allowbreak \zeta\!(9)\bigr)
\end{dmath*}
\subsection*{Identity 67}
\begin{dmath*}
\sum_{n_{1}=1}^{\infty}\sum_{n_{2}=1}^{n_{1}} \frac{\bigl(\frac{1}{2}\bigr)^{n_{1}}\,\allowbreak \bigl(\frac{1}{3}\bigr)^{n_{2}}\,\allowbreak H_{n_{1}}\,\allowbreak H_{n_{2}}}{n_{1}^{2}\,\allowbreak n_{2}^{3}} = \operatorname{HPL}_{5,2}\!(\frac{1}{6}) +\allowbreak \operatorname{Li}_{2,5}\!(\frac{1}{2},\allowbreak \frac{1}{3}) +\allowbreak \operatorname{Li}_{3,4}\!(\frac{1}{2},\allowbreak \frac{1}{3}) +\allowbreak \operatorname{Li}_{2,1,4}\!(\frac{1}{2},\allowbreak 1,\allowbreak \frac{1}{3}) +\allowbreak \operatorname{Li}_{2,3,2}\!(\frac{1}{2},\allowbreak \frac{1}{3},\allowbreak 1) +\allowbreak 2\,\allowbreak \operatorname{Li}_{2,4,1}\!(\frac{1}{2},\allowbreak \frac{1}{3},\allowbreak 1) +\allowbreak \operatorname{Li}_{3,3,1}\!(\frac{1}{2},\allowbreak \frac{1}{3},\allowbreak 1) +\allowbreak \operatorname{Li}_{2,1,3,1}\!(\frac{1}{2},\allowbreak 1,\allowbreak \frac{1}{3},\allowbreak 1)
\end{dmath*}
\vspace{-1.1em}
\begin{dmath*}
{}\quad +\allowbreak 2\,\allowbreak \operatorname{Li}_{2,3,1,1}\!(\frac{1}{2},\allowbreak \frac{1}{3},\allowbreak 1,\allowbreak 1) +\allowbreak \operatorname{Li}_{7}\!(\frac{1}{6}) +\allowbreak 2\,\allowbreak \operatorname{Li}_{4,3}\!(\frac{1}{6}) +\allowbreak 2\,\allowbreak \operatorname{Li}_{5,2}\!(\frac{1}{6})
\end{dmath*}